\documentclass[12pt]{amsart}
\usepackage{bbm}
\usepackage{amsfonts}
\usepackage{amssymb, eucal, amsfonts, amsmath, xypic, latexsym}
\usepackage{pifont}
\usepackage{mathrsfs}
\usepackage{amsthm,indentfirst,bm,fancyhdr,dsfont}
\usepackage{graphicx}
\usepackage[perpage]{footmisc}
\numberwithin{equation}{section}

\theoremstyle{plain}
\newtheorem{theorem}{Theorem}[section]
\newtheorem{corollary}{Corollary}[section]
\newtheorem{prop}{Proposition}[section]
\newtheorem{lemma}{Lemma}[section]

\theoremstyle{definition}
\newtheorem{defn}{Definition}[section]

\newtheorem{conj}{Conjecture}[section]

\theoremstyle{remark}

\newtheorem{rem}{Remark}[section]

\def\subtitle#1. {{\medskip\bf#1\par\nobreak\smallskip}}
\def\proclaim#1. {\medbreak\bgroup\noindent\bf#1. \it}

\def\endproclaim{\egroup
\ifdim\lastskip<\medskipamount\removelastskip\medskip\fi}
\newcount
=0
\def\citedef#1 {\advance\citation by1
  \expandafter\edef\csname#1\endcsname{{\the\citation}}
  \checkendcitedef}
\def\checkendcitedef#1{\ifx#1\endcitedef\else\citedef#1\fi}
\def\cite#1{\csname#1\endcsname}
\citedef  BBG2 BG3 BGK BK BK2 BR CW EK FG GG G3 GRU Gow H J J3 K K2 KW KL Ko L1 L3 L5 Ly M Peng2 Peng3 PS PS2 P1 P2 P3 P4 P6 P7 P8 SW S2 S3 SK WZ WZ2 W Zeng Z1 Z2
\endcitedef
\newtoks\nextauth
\newif\iffirstauth
\def\checkendauth#1{\ifx\endauth#1
        \iffirstauth\the\nextauth
        \else{} and \the\nextauth\fi,
    \else\iffirstauth\the\nextauth\firstauthfalse
        \else, \the\nextauth\fi
        \expandafter\auth\expandafter#1\fi}
\def\auth#1 #2 {\nextauth={#1 #2}\checkendauth}
\newif\ifinbook
\newif\ifbookref
\def\nextref#1 {\bookreffalse\inbookfalse
    \bibitem[\cite{#1}]{}
    \firstauthtrue
    \ignorespaces}

\def\ggg{\mathfrak{g}}

\def\bbc{\mathbb{C}}
\def\bbk{\mathds{k}}
\def\bbf{\mathbb{F}}

\def\Lie{\text{Lie}}

\begin{document}
\title{Finite $W$-superalgebras for basic classical Lie superalgebras}
\author{Yang Zeng and Bin Shu}
\thanks{\nonumber{{\it{Mathematics Subject Classification}} (2000 {\it{revision}})
Primary 17B35. Secondary 17B81. This work is  supported partially by the NSF of China (No. 11271130; 11201293; 111126062),  the Innovation Program of Shanghai Municipal Education Commission (No. 12zz038). }}
\address{School of Science, Nanjing Audit University, Nanjing, Jiangsu Province 211815, China}
\email{zengyang214@163.com}
\address{Department of Mathematics, East China Normal University, Shanghai 200241, China}
\email{bshu@math.ecnu.edu.cn}


\begin{abstract}
\noindent
We consider the finite $W$-superalgebra $U(\mathfrak{g_\bbf},e)$ for a basic classical Lie superalgebra $\mathfrak{g}_\bbf$ associated with an even nilpotent element $e\in (\ggg_\bbf)_{\bar0}$ both over the field of complex numbers $\bbf=\mathbb{C}$ and and over a filed $\bbf=\mathds{k}$ of positive characteristic. We present the PBW theorem for $U(\mathfrak{g}_\bbf,e)$ and show that the construction of $U(\mathfrak{g}_\bbf,e)$  can be divided into two cases in virtue of the parity of $\text{dim}~\mathfrak{g}(-1)_{\bar1}$. Then we formulate a conjecture about the minimal dimensional representations of $U(\mathfrak{g}_\bbc,e)$ and demonstrate it with some examples. Under the assumption that the conjecture holds, we finally show that the lower bound of dimensions predicted in the super version of Kac-Weisfeiler Conjecture formulated and proved by Wang-Zhao in [\cite{WZ}] for the modular representations of the basic classical Lie superalgebra ${\ggg}_{\mathds{k}}$ with arbitrary $p$-characters can be reached.
\end{abstract}
\maketitle
\tableofcontents
\section{Introduction}
\noindent{\bf 1.1.} A finite $W$-algebra $U(\ggg,e)$ is a certain associative algebra associated to a complex semisimple Lie algebra $\mathfrak{g}$ and a nilpotent element $e\in\mathfrak{g}$. The study of finite $W$-algebras can be traced back to Kostant's work in the case when $e$ is regular [\cite{Ko}], whose construction was generalized to arbitrary even nilpotent elements by Lynch [\cite{Ly}]. Afterwards, Premet developed the finite $W$-algebras in full generality in [\cite{P2}]. On his way of proving the celebrated Kac-Weisfeiler conjecture for Lie algebras of reductive groups in [\cite{P1}], Premet first constructed the modular version of finite $W$-algebras (which is called reduced $W$-algebras there). By means of a complicated but natural ``admissible'' procedure, the finite $W$-algebras over the field of complex numbers were introduced in [\cite{P2}], which showed that they are filtrated deformations of the coordinate rings of Slodowy slices. In the extreme case when $e=0$, the corresponding finite $W$-algebra is isomorphic to $U(\mathfrak{g})$, the universal enveloping algebra of $\ggg$. In the other extreme case when $e$ is a principal nilpotent element, Kostant proved that the associated finite $W$-algebra is isomorphic to the center of the universal enveloping algebra $U(\mathfrak{g})$ in [\cite{Ko}]. On the other hand, Brundan-Kleshchev showed that finite $W$-algebras can be realized as shifted Yangians for the case of type $A$, which provides a powerful tool for the study of finite $W$-algebras; see [\cite{BK2}]. Finite $W$-algebras theory becomes a very active area, and we refer to the survey papers [\cite{L1}], [\cite{P6}] and [\cite{W}] and references therein for more details.

Aside from the advances in finite $W$-algebras over the field of complex numbers, the modular theory of finite $W$-algebras is also developed excitingly. As a most remarkable work, Premet proved in [\cite{P7}] that under the assumption $p\gg 0$, if $U(\ggg,e)$ has a one-dimensional representations, then the reduced enveloping algebra $U_\chi(\ggg_\bbk)$ of the modular counterpart $\ggg_\bbk$ of $\ggg$ possesses a simple module of dimension $d(e)$, where $\chi$ is the linear function on $\ggg_{\mathds{k}}$ corresponding to $e$, $d(\chi)$ is the half dimension of the orbit $G\cdot \chi$ for the simple, simply-connected algebraic group $G$ with $\ggg_\bbk=\Lie(G)$, which is a lower bound predicted by Kac-Weisfeiler conjecture mentioned above. This assumption is conjectured to be always true by Premt himself, which has been proved in the case of classical groups by Losev (cf. [\cite{L3}]), and by Goodwin-R\"{o}hrle-Ubly for the case $E_6,E_7,F_4,G_2$, or $E_8$ with $e$ not rigid (cf. [\cite{GRU}]).

~

\noindent{\bf 1.2.} The theory of finite $W$-superalgebras were developed in the same time.  In the work of De Sole and Kac [\cite{SK}], finite $W$-superalgebras were defined in terms of BRST cohomology under the background of vertex algebras and quantum reduction. The theory of finite $W$-superalgebras for the queer Lie superalgebras (it is notable that which are not basic classical Lie superalgebras) over an algebraically closed field of characteristic $p>2$ was firstly introduced and discussed  by Wang and Zhao in [\cite{WZ2}], then studied by Zhao over the field of complex numbers in [\cite{Z2}]. The connection between super Yangians and finite $W$-superalgebras was first obtained by Broit and Ragoucy in [\cite{BR}]. In the paper [\cite{BBG2}], the connection between the finite $W$-superalgebra associated to a principal nilpotent element was developed by Brown-Brundan-Goodwin. Some related results in this situation were also obtained independently by Poletaeva and Serganova in [\cite{PS}], where they described the finite $W$-superalgebras in the regular case for some classical and exceptional Lie superalgebras of defect one. In [\cite{Peng2}], Peng established a connection between finite $W$-superalgebras and super Yangians explicitly in type $A$ where the Jordan type of the nilpotent element $e$ satisfies certain condition. Now the theory of $W$-superalgebras related to super Yangians is still under investigating.

In a very recent paper [\cite{PS2}], Poletaeva and Serganova studied some generalities of the finite $W$-superalgebras associated with an even nilpotent element over the field of complex numbers, and proved that the finite $W$-superalgebras for basic classical Lie superalgebras or the queer Lie superalgebras associated with regular nilpotent elements satisfy the Amitsur-Levitzki identity. Then the related topics on the finite $W$-superalgebras for the queer Lie superalgebra $\mathfrak{q}_n$ with regular nilpotent elements were studied in detail in the same paper.

~

\noindent{\bf 1.3.}  The main purposes of the present paper are both to develop the general theory of finite $W$-superalgebras, and to exploit their applications to modular representations of Lie superalgebras.  Our approach is roughly generalizing Premet's arguments in the case of finite $W$-algebras in [\cite{P2}] and [\cite{P7}] to the case of finite $W$-superalgebras based on the results given by Wang and Zhao in [\cite{WZ}], and some new results are obtained either.

In the first part of the present paper, we develop the theory of finite $W$-superalgebras for basic classical Lie superalgebras both over the field of complex numbers and in prime characteristic. Let $\mathfrak{g}=\mathfrak{g}_{\bar0}+\mathfrak{g}_{\bar1}$ be a basic classical Lie superalgebra over $\mathbb{C}$ except for type $D(2,1;a)(a\notin\mathbb{Q})$. Let $e\in\mathfrak{g}_{\bar0}$ be an even nilpotent element, and we fix an $\mathfrak{sl}_2$-triple $f,h,e$. Denote by $\mathfrak{g}^e:=\text{Ker}(ad\,e)$ in $\mathfrak{g}$. The linear
operator ad~$h$ defines a $\mathbb{Z}$-grading $\mathfrak{g}=\bigoplus\limits_{i\in\mathbb{Z}}\mathfrak{g}(i)$. Define the Kazhdan degree on $\mathfrak{g}$ by declaring $x\in\mathfrak{g}(j)$ is $(j+2)$. We construct a $\mathbb{C}$-algebra (which is called a finite $W$-superalgebra)$$U(\mathfrak{g},e)=(\text{End}_\mathfrak{g}Q_{\chi})^{\text{op}}.$$
Lemma~\ref{hwc} and Theorem~\ref{PBWC} show that

\begin{theorem}\label{graded W}
Under the Kazhdan grading we have
\[\begin{array}{rll}
(1)& \text{gr}\,U(\mathfrak{g},e)\cong S(\mathfrak{g}^e) & \text{as}~\mathbb{C}\text{-algebras when dim}~\mathfrak{g}(-1)_{\bar1}~\text{is even};\\
(2)& \text{gr}\,U(\mathfrak{g},e)\cong S(\mathfrak{g}^e)\otimes \mathbb{C}[\bar\theta]&\text{as vector spaces when dim}~\mathfrak{g}(-1)_{\bar1}~\text{is odd},
\end{array}\]
\noindent where $\mathbb{C}[\bar\theta]$ is the exterior algebra generated by one element $\bar\theta$ for the case when dim~$\mathfrak{g}(-1)_{\bar1}$ is odd.
\end{theorem}

The main tool applied for the proof of Theorem~\ref{graded W} is the ``modular $p$ reduction'' method introduced by Premet for the finite $W$-algebra case in [\cite{P2}] and the results on basic classical Lie superalgebras obtained by Wang and Zhao in [\cite{WZ}]. It is remarkable that after the draft of this paper has been written, we know from [\cite{PS}] and [\cite{PS2}] that Poletaeva and Serganova also noticed that Theorem~\ref{graded W} maybe true and formulated the corresponding conjecture in ([\cite{PS2}], Conjecture 2.8). At the same time, they also realized (see [\cite{PS}]) that the theorem can be obtained possibly by analoging Premet's treatment for the finite $W$-algebra case, but did not give a proof therein. In [\cite{PS2}] they proved that for any element $y\in\mathfrak{g}^e$ if one can find $Y\in U(\mathfrak{g},e)$ such that $\text{gr}\,Y(1_\chi)=y$ under the Kazhdan grading, then Theorem~\ref{graded W}  establishes. As a special case, they constructed the generators of the finite $W$-superalgebras for the queer Lie superalgebra $\mathfrak{q}_n$ associated with regular nilpotent elements, and proved that Theorem~\ref{graded W}~establishes in that situation.

~

\noindent{\bf 1.4.}  In virtue of the incompleteness of related topics on Lie superalgebras, Wang and Zhao bypassed the support variety machinery completely and adopted pure algebraic method as Skryabin's treatment for the finite $W$-algebra case [\cite{S3}] when they established some critical lemma for the super Kac-Weisfeiler property. Therefore, the tool of support variety machinery can not be applied in the establishment of finite $W$-superalgebra theory. Moveover, there are cases where the dimension of odd part in the critical graded subspace of the Lie superalgebra is odd under the Dynkin grading, which lead to the existence of odd isomorphism for the related superalgebras. This significant difference has great impact on the structure of finite $W$-superalgebras. Of course, the appearance of super structure also makes the situation more complicated. Therefore, the establishment of finite $W$-superalgebras theory is no longer simple promotion of the theory on finite $W$-algebras, and some technical methods are needed especially for the case when the dimension of the critical graded subspace is odd. In light of Gan-Ginzburg's definition of finite $W$-algebras over the field of complex numbers in [\cite{GG}], Wang defined the reduced $W$-superalgebras in positive characteristic in a new way ([\cite{W}], Remark 70), which he thought makes better sense. We also discuss the construction of these algebras for the version of characteristic zero in the paper.

~

\noindent{\bf 1.5.} Let $\mathfrak{g}_\mathds{k}$ be the corresponding Lie superalgebra over positive characteristic field $\mathds{k}$. After studying some related topics on the reduced $W$-superalgebra $U_\chi(\mathfrak{g}_\mathds{k},e)$ in positive characteristic  in Section 5, we introduce the PBW theorem (Theorem~\ref{PBWC}) for the $\mathbb{C}$-algebra $U(\mathfrak{g},e)$ based on the parity of dim~$\mathfrak{g}(-1)_{\bar1}$ respectively, and also the relations between the generators of $U(\mathfrak{g},e)$ (Theorem~\ref{relationc}). All these completely characterize the structure of finite $W$-superalgebra $U(\mathfrak{g},e)$.

The finite $W$-superalgebra $\widehat{U}(\mathfrak{g}_\mathds{k},e)$ in positive characteristic is introduced in Section 7. In Theorem~\ref{keyisotheorem}(3) we obtain that
\begin{theorem}\label{graded W2}
$\widehat{U}(\mathfrak{g}_\mathds{k},e)\cong U(\mathfrak{g}_\mathds{k},e)\otimes_\mathds{k}Z_p(\mathfrak{a}_\mathds{k})$ as $\mathds{k}$-algebras.
\end{theorem}
In the above theorem, $Z_p(\mathfrak{a}_\mathds{k})$ is a subalgebra of the $p$-center of $\widehat{U}(\mathfrak{g}_\mathds{k},e)$ and $U(\mathfrak{g}_\mathds{k},e)$ is the translation subalgebra of $U(\mathfrak{g}_\mathds{k},e)$.

~

\noindent{\bf 1.6.}  The second part of the present paper is to exploit some applications of finite-$W$ superalgebras to modular representations. One of the multi-purposes of the present paper is to provide  a super version of Premet's work, as shown before, on the reachable property of up-bounds of dimensions of modular representations of reductive Lie algebras predicted by  Kac-Weisfeiler conjecture. For this, we will formulate a conjecture about the minimal dimensional representations of $U(\mathfrak{g}_\bbc,e)$. Under the assumption that the conjecture holds,
we complete this analogue of Premet's work in basic Lie superalgebebras. Let us explain it roughly as below.

In Section 8 and Section 9, we first formulate a conjecture (Conjecture~\ref{con}) on the minimal dimensional representations of finite $W$-superalgebras over the field of complex numbers:
\begin{conj}\label{conjecture}
Let $\mathfrak{g}$ be a basic classical Lie superalgebra over $\mathbb{C}$, then the following are true:

(i) when $d_1$ is even, the finite $W$-superalgebra $U(\mathfrak{g},e)$ affords a $1$-dimensional representation;

(ii) when $d_1$ is odd, the finite $W$-superalgebra $U(\mathfrak{g},e)$ affords a $2$-dimensional representation.
\end{conj}

We show that Conjecture~\ref{conjecture} is true for some special cases. Based on the conjecture, we first prove that the lower bound of dimensions in the Super Kac-Weisfeiler Conjecture with any nilpotent $p$-characters ([\cite{WZ}], Theorem 4.3) can be reached. Explicitly speaking, let $\mathfrak{g}$ be a basic classical Lie superalgebra over $\mathbb{C}$ and $\mathfrak{g}_\mathds{k}$ be the corresponding Lie superalgebra over positive characteristic field $\mathds{k}$. Let $\chi\in(\mathfrak{g}_\mathds{k}^*)_{\bar0}$ be the $p$-character of $\mathfrak{g}_\mathds{k}$ such that $\chi(\bar y)=(e,\bar y)$ for any $\bar y\in\mathfrak{g}_\mathds{k}$. Denote by $d_i:=\text{dim}\mathfrak{g}_i-\text{dim}\mathfrak{g}^e_i$ for $i\in\{\bar0,\bar1\}$ where $\mathfrak{g}^e$ is the centralizer of $e$ in $\mathfrak{g}$.
\begin{theorem}\label{nilg}
If Conjecture~\ref{conjecture} establishes, the following are true:

(1) when $d_1$ is even, for $p\gg0$ the reduced enveloping algebra $U_\chi(\mathfrak{g}_\mathds{k})$ over $\mathds{k}=\overline{\mathbb{F}}_p$ admits irreducible representations of dimension $p^{\frac{d_0}{2}}2^{\frac{d_1}{2}}$;

(2) when $d_1$ is odd, for $p\gg0$ the reduced enveloping algebra $U_\chi(\mathfrak{g}_\mathds{k})$ over $\mathds{k}=\overline{\mathbb{F}}_p$ admits irreducible representations of dimension $p^{\frac{d_0}{2}}2^{\frac{d_1+1}{2}}$.
\end{theorem}

In virtue of Theorem~\ref{nilg}, we further show that the lower bound of dimensions in the Super Kac-Weisfeiler Conjecture with arbitrary $p$-characters ([\cite{WZ}], Theorem 5.6) is reachable under the assumption of Conjecture~\ref{conjecture}. Explicitly speaking, let $\xi\in(\mathfrak{g}_\mathds{k}^*)_{\bar0}$ be any $p$-character of $\mathfrak{g}_\mathds{k}$ corresponding to an element $\bar x\in (\mathfrak{g}_\mathds{k})_{\bar0}$ such that $\xi(\bar y)=(\bar x,\bar y)$ for any $\bar y\in\mathfrak{g}_\mathds{k}$. Set $d_0:=\text{dim}(\mathfrak{g}_\mathds{k})_{\bar 0}-\text{dim}(\mathfrak{g}_\mathds{k}^{\bar x})_{\bar 0}$ and $d_1:=\text{dim}(\mathfrak{g}_\mathds{k})_{\bar 1}-\text{dim}(\mathfrak{g}_\mathds{k}^{\bar x})_{\bar 1}$, where $\mathfrak{g}_\mathds{k}^{\bar x}$ denotes the centralizer of $\bar x$ in $\mathfrak{g}_\mathds{k}$. Then we have

\begin{theorem}
Let $\mathfrak{g}_\mathds{k}$ be a basic classical Lie superalgebra over $\mathds{k}=\overline{\mathbb{F}}_p$, and let $\xi\in(\mathfrak{g}_\mathds{k}^*)_{\bar0}$. If Conjecture~\ref{conjecture} establishes, then for $p\gg0$ the reduced enveloping algebra $U_\xi(\mathfrak{g}_\mathds{k})$ admits irreducible representations of dimension $p^{\frac{d_0}{2}}2^{\lfloor\frac{d_1}{2}\rfloor}$.
\end{theorem}

\noindent{\bf 1.7.} The paper is organized as follows.

In Section 2, we recall some basics about the algebraic supergroups.

In Section 3, three equivalent definitions for finite $W$-superalgebras $U(\mathfrak{g},e)$ over $\mathbb{C}$ are introduced. Then follows the Kazhdan filtration and the Skryabin equivalence theorem. The restricted root decomposition is discussed in the final part.

In Section 4, the finite $W$-superalgebra $\widehat{U}(\mathfrak{g}_\mathds{k},e)$ over positive characteristic field $\mathds{k}$  and the reduced $W$-superalgebra $U_\chi(\mathfrak{g}_\mathds{k},e)$ are defined. The Morita equivalence theorem between $\mathds{k}$-algebras $U_\chi(\mathfrak{g}_\mathds{k},e)$ and $U_\chi(\mathfrak{g}_\mathds{k})$ is introduced.

In Section 5, we introduce the generators and their relations for the $\mathds{k}$-algebra $U_\chi(\mathfrak{g}_\mathds{k},e)$, then follows the PBW Theorem. Notably, we find that the construction of reduced $W$-superalgebra $U_\chi(\mathfrak{g}_\mathds{k},e)$ critically depends on the parity of dim~$\mathfrak{g}_\mathds{k}(-1)_{\bar1}$. Based on which, the construction of reduced $W$-superalgebras can be divided into two cases, which never happens for the reduced $W$-algebra case.

In virtue of the results obtained in Section 5, we introduce the PBW Theorem for finite $W$-superalgebra $U(\mathfrak{g},e)$ over $\mathbb{C}$ in Section 6 by means of the ``admissible'' procedure. The relationship between the refined finite $W$-superalgebra $Q_\chi^{\text{ad}\mathfrak{m}'}$ by Wang in ([\cite{W}], Remark 70) and the finite $W$-superalgebra $U(\mathfrak{g},e)$ over $\mathbb{C}$ is also discussed.

In Section 7 the translation subalgebra $U(\mathfrak{g}_\mathds{k},e)$ over positive characteristic field is introduced and the relationship between finite $W$-superalgebra $\widehat{U}(\mathfrak{g}_\mathds{k},e)$, its $p$-center $Z_p(\widetilde{\mathfrak{p}}_\mathds{k})$ and the translation subalgebra $U(\mathfrak{g}_\mathds{k},e)$ is discussed.

As the construction of finite $W$-superalgebra $U(\mathfrak{g},e)$ can be divided into two cases in virtue of the parity of dim~$\mathfrak{g}(-1)_{\bar1}$, for each case the minimal dimension for the representations of $U(\mathfrak{g},e)$ over $\mathbb{C}$ is estimated and reasonable conjectured in Section 8, respectively. We also show that these representations can be translated into all the common zeros of some polynomials.

In Section 9, we first show that the conjecture given in Section 8 establishes for some special cases. Under the assumption of the conjecture, we prove that the lower bound in the Super Kac-Weisfeiler Property for a basic classical Lie superalgebra given by Wang-Zhao can be reached for the cases with any nilpotent $p$-characters. Consequently, the Super Kac-Weisfeiler Property for a direct sum of basic classical Lie superalgebras with nilpotent $p$-characters introduced by Wang-Zhao is refined and reachability of the lower bound given in the refined version is proved. In virtue of this consequence, we further show that the lower bound in the Super Kac-Weisfeiler Property for a basic classical Lie superalgebra with arbitrary $p$-characters is also reachable. The main tool applied there is the geometric method caused by the nilpotent orbits of algebraic groups.

~

\noindent{\bf 1.8.} Throughout we work with the field of complex numbers $\mathbb{C}$ or the algebraically closed field $\mathds{k}=\overline{\mathbb{F}}_p$ in positive characteristic as the ground field.

Let $\mathbb{Z}_+$ be the set of all the non-negative integers in $\mathbb{Z}$, and denote by $\mathbb{Z}_2$ the residue class ring modulo $2$ in $\mathbb{Z}$.

A superspace is a $\mathbb{Z}_2$-graded vector space $V=V_{\bar0}\oplus V_{\bar1}$, in which we call elements in $V_{\bar0}$ and $V_{\bar1}$ even and odd, respectively. Write $|v|\in\mathbb{Z}_2$ for the parity (or degree) of $v\in V$, which is implicitly assumed to be $\mathbb{Z}_2$-homogeneous. We will use the notations
$$\text{\underline{dim}}V=(\text{dim}V_{\bar0},\text{dim}V_{\bar1}),\quad\text{dim}V=\text{dim}V_{\bar0}+\text{dim}V_{\bar1}.$$
All Lie superalgebras $\mathfrak{g}$ will be assumed to be finite dimensional.

Recall that a superalgebra analog of Schur's Lemma states that the endomorphism ring of an irreducible module of a superalgebra is either $1$-dimensional or $2$-dimensional (in the latter case it is isomorphic to a Clifford algebra), cf. for example, Kleshchev ([\cite{KL}], Chapter 12). An irreducible module is of type $M$ if its endomorphism ring is $1$-dimensional and it is of type $Q$ otherwise.

By vector spaces, subalgebras, ideals, modules, and submodules etc. we mean in the super sense unless otherwise specified, throughout the paper.

\section{Preliminaries}
The materials in this section are standard results about algebraic supergroups and Lie superalgebras.
\subsection{Algebraic supergroups}\label{2.1}

We first briefly recall the generalities on algebraic supergroups, following ([\cite{SW}], Section 2) by Shu and Wang. One can also refer to [\cite{BK}], [\cite{J}], [\cite{M}].

Let $\mathbb{F}$ be a fixed algebraically closed field of characteristic $p>2$. All objects in this section will be defined over $\mathbb{F}$ unless otherwise specified. Let $B=B_{\bar0}+B_{\bar1}$ be a commutative $\mathbb{Z}_2$-graded superalgebra over $\mathbb{F}$, that is, $ab=(-1)^{|a||b|}ba$ for all homogeneous elements $a,b\in B$ of degree $|a|,|b|\in\mathbb{Z}_2$. In the sequel, we assume that all formulas are defined via
the homogeneous elements and extended by linearity. An element in $B_{\bar0}$ is called even and an element in $B_{\bar1}$ is called odd. From the supercommutativity it follows that $b^2=0$ for all $b\in B_{\bar1}$. We will denote by $\mathfrak{salg}$ the category of commutative superalgebras over $\mathbb{F}$ and
even homomorphisms. A fundamental object in this category is the free commutative superalgebra $\mathbb{F}[x_1,\cdots,x_n;\xi_1,\cdots,\xi_m]$ in even generators $x_1,\cdots,x_n$ and odd generators $\xi_1,\cdots,\xi_m$.

\begin{defn}\label{superscheme}
An affine superscheme $X$ will be identified with its associated functor in the category of superschemes $$\text{Hom}(Spec(-),X):\mathfrak{salg}\rightarrow\mathfrak{sets}.$$
\end{defn}

For an affine superscheme $X$, its coordinate superalgebra $\mathbb{F}[X]$ is the superalgebra $\text{Mor}(X,\mathbb{A}^{1|1})$ of all natural transformations from the functor $X$ to $\mathbb{A}^{1|1}$.

\begin{defn}\label{supergroup}  An affine algebraic supergroup $G$ is a functor from the category $\mathfrak{salg}$ to the category of groups which associates to a commutative superalgebra $B$ a group $G(B)$ functorially, and
which has a coordinator algebra $\mathbb{F}[G]$ that is finitely generated.
\end{defn}

For an algebraic supergroup $G$, we have $\mathbb{F}[G]$ admits a canonical structure of Hopf superalgebra, with comultiplication
$\Delta: \mathbb{F}[G]\rightarrow\mathbb{F}[G]\otimes \mathbb{F}[G]$, the antipode $S:\mathbb{F}[G]\rightarrow\mathbb{F}[G]$, and the counit $\varepsilon:\mathbb{F}[G]\rightarrow\mathbb{F}$. Set $\mathscr{J}:=\text{ker}(\varepsilon)$. A closed subgroup of $G$ is an affine supergroup scheme with coordinate algebra that is a quotient
of $\mathbb{F}[G]$ by a Hopf ideal. In particular, the underlying purely even group of $G$, denoted by $G_{\text{ev}}$, corresponds to the Hopf ideal $\mathbb{F}[G]\mathbb{F}[G]_{\bar1}$. That is, $\mathbb{F}[G_{\text{ev}}]\cong\mathbb{F}[G]/\mathbb{F}[G]\mathbb{F}[G]_{\bar1}$.

The superspace of distributions (at the identity $e\in G$) is $$\text{Dist}(G):=\cup_{n\geqslant 0}\text{Dist}_n(G),$$
where $\text{Dist}_n(G):=\{X\in\mathbb{F}[G]^*|X(\mathscr{J}^{n+1})=0\}\cong(\mathbb{F}[G]/\mathscr{J}^{n+1})^*.$ For any $X\in\text{Dist}_s(G)$ and $Y\in\text{Dist}_t(G)$, define
$$[X,Y]:=X*Y-(-1)^{|X||Y|}Y*X\in\text{Dist}_{s+t-1}(G).$$Hence, the tangent space at the identity$$T_e(G):=\{X\in\text{Dist}_1(G)|X(1)=0\}\cong(\mathscr{J}/\mathscr{J}^2)^*$$carries a Lie superalgebra structure; it is called the Lie superalgebra of $G$ and will be denoted by $\text{Lie}(G)$. The canonical map$$\pi:\mathbb{F}[G]\rightarrow\mathbb{F}[G_{\text{ev}}]=\mathbb{F}[G]/\mathbb{F}[G]\mathbb{F}[G]_{\bar1}$$ sends $\mathscr{J}$ to the kernel $\mathscr{J}_{\text{ev}}$ of $\varepsilon_{\text{ev}}:\mathbb{F}[G_{\text{ev}}]\rightarrow\mathbb{F}$ and $\pi(\mathscr{J}^i)\subset\mathscr{J}_{\text{ev}}^i$ for $i\geqslant1$. This induces an injective algebra homomorphism $$\pi^*:\text{Dist}(G_{\text{ev}})\rightarrow\text{Dist}(G).$$

\begin{lemma}$^{[\cite{SW}]}$\label{Dist}
The superalgebra homomorphism $\pi^*$ induces an isomorphism of Lie algebras from $\text{Lie}(G_{\text{ev}})$ onto $\text{Lie}(G)_{\bar0}=\text{Lie}(G)\cap\text{Dist}(G)_{\bar0}$.
\end{lemma}

\subsection{Lie superalgebras}\label{2.2}

In this part we will recall some basics on Lie superalgebras.

\begin{defn}\label{form}
Let $V=V_{\bar0}\oplus V_{\bar1}$ be a $\mathbb{Z}_2$-graded space and $(\cdot,\cdot)$ be a bilinear form on $V$.

(1) If $(a,b)=0$ for any $a\in V_{\bar0}, b\in V_{\bar1}$, then $(\cdot,\cdot)$ is called even;

(2) If $(a,b)=(-1)^{|a||b|}(b,a)$ for any homogeneous elements $a,b\in V$, then $(\cdot,\cdot)$ is called supersymmetric;

(3) If $([a,b],c)=(a,[b,c])$ for any homogeneous elements $a,b,c\in V$, then $(\cdot,\cdot)$ is called invariant;

(4) If it follows from $(a,V)=0$ that $a=0$, then $(\cdot,\cdot)$ is called non-degenerated.

\end{defn}

\begin{defn}\label{Lie superalgebra}
Let $\mathfrak{g}$ be a simple Lie superalgebra. If

(1) there exists a non-degenerated, supersymmetric and invariant even bilinear form on $\mathfrak{g}$;

(2) the even part $\mathfrak{g}_{\bar0}$ of $\mathfrak{g}$ is a Lie algebra of reductive group,

\noindent then $\mathfrak{g}$ is called a basic classical Lie superalgebra.
\end{defn}

In other words, a finite-dimensional Lie superalgebra $\mathfrak{g}=\mathfrak{g}_{\bar0}\oplus\mathfrak{g}_{\bar1}$ is called classical if
it is simple and the representation of $\mathfrak{g}_{\bar0}$ on $\mathfrak{g}_{\bar1}$ is completely reducible.

Now let $\mathbb{F}$ be the field $\mathds{k}:=\overline{\mathbb{F}}_p$ of positive characteristic $p>0$.

\begin{defn}\label{restricted}
A Lie superalgebra $\mathfrak{g}_\mathds{k}=(\mathfrak{g}_\mathds{k})_{\bar{0}}\oplus(\mathfrak{g}_\mathds{k})_{\bar{1}}$ over $\mathds{k}$ is called a restricted Lie superalgebra,
if there is a $p$-th power map $(\mathfrak{g}_\mathds{k})_{\bar{0}}\rightarrow(\mathfrak{g}_\mathds{k})_{\bar{0}}$, denoted as $^{[p]}$, satisfying

(a) $(kx)^{[p]}=k^px^{[p]}$ for all $k\in\mathds{k}$ and $x\in(\mathfrak{g}_\mathds{k})_{\bar{0}}$;

(b) $[x^{[p]},y]=(\text{ad}x)^p(y)$ for all $x\in(\mathfrak{g}_\mathds{k})_{\bar{0}}$ and $y\in\mathfrak{g}_\mathds{k}$;

(c) $(x+y)^{[p]}=x^{[p]}+y^{[p]}+\sum\limits_{i=1}^{p-1}s_i(x,y)$ for all $x,y\in(\mathfrak{g}_\mathds{k})_{\bar{0}}$, where $is_i(x,y)$ is the
coefficient of $\lambda^{i-1}$ in $(\text{ad}(\lambda x+y))^{p-1}(x)$.
\end{defn}

In short, a restricted Lie superalgebra is a Lie superalgebra whose even subalgebra
is a restricted Lie algebra and the odd part is a restricted module by the
adjoint action of the even subalgebra.

In ([\cite{SW}], Proposition 2.3) Shu and Wang introduced the following consequence:

\begin{lemma}$^{[\cite{SW}]}$\label{Dist(G)}
Let $G$ be a supergroup. Then, $\text{Lie}(G)$ is a restricted Lie superalgebra with the $p$-mapping: $X\mapsto X^{[p]}$ for $X\in\text{Lie}(G)_{\bar{0}}$, where $X^{[p]}:=\overbrace{X*\cdots*X}^{p}$ is defined in
$\text{Dist}(G)$. Moreover, the restricted structure on $\text{Lie}(G_{\text{ev}})$ as a subalgebra of $\text{Lie}(G)$ coincides with the one induced as Lie algebra of the algebraic group $G_{\text{ev}}$.
\end{lemma}

All the Lie superalgebras over positive characteristic field $\mathds{k}$ in this paper will be assumed to be restricted.

\section{Finite $W$-superalgebras over the field of complex numbers}

In this section we will introduce the equivalent definitions of finite $W$-superalgebras over $\mathbb{C}$.

\subsection{The definition of finite $W$-superalgebras over $\mathbb{C}$}

Let $\mathfrak{g}$ be a basic classical Lie superalgebra over $\mathbb{C}$ and $\mathfrak{h}$ be a Cartan subalgebra of $\mathfrak{g}$. Let $\Phi$ be a root system of $\mathfrak{g}$ relative to $\mathfrak{h}$ whose simple roots system $\Pi=\{\alpha_1,\cdots,\alpha_l\}$ is distinguished (which is defined in ([\cite{K2}], Proposition 1.5)). Let $\Phi^+$ be the corresponding positive system in $\Phi$, and put $\Phi^-:=\Phi^+$. Let $\mathfrak{g}=\mathfrak{n}^-\oplus\mathfrak{h}\oplus\mathfrak{n}^+$ be the corresponding triangular decomposition of $\mathfrak{g}$. By [\cite{FG}], we can choose a Chevalley basis $B=\{e_\gamma|\gamma\in\Phi\}\cup\{h_\alpha|\alpha\in\Pi\}$ of $\mathfrak{g}$. Let $\mathfrak{g}_\mathbb{Z}$ denote the Chevalley $\mathbb{Z}$-form in $\mathfrak{g}$ and $U_\mathbb{Z}$ the Kostant $\mathbb{Z}$-form of $U(\mathfrak{g})$ associated to $B$. Given a $\mathbb{Z}$-module $V$ and a $\mathbb{Z}$-algebra $A$, we write $V_A:=V\otimes_\mathbb{Z}A$.

Let $G$ be the algebraic supergroup of $\mathfrak{g}$. It is immediate from Definition~\ref{Lie superalgebra} that the even part of $G$ is reductive and denote it by $G_{\text{ev}}$. Let $e\in\mathfrak{g}_{\bar0}$ be an even nilpotent in $\mathfrak{g}$. By the Dynkin-Kostant theory, $\text{ad}G_{\text{ev}}.e$ interacts with $(\mathfrak{g}_\mathbb{Z})_{\bar{0}}$ nonempty. Therefore, we can assume that all the even nilpotent elements considered are in $(\mathfrak{g}_\mathbb{Z})_{\bar{0}}$. By the same discussion as ([\cite{P2}], Section 4.2), for any nilpotent element $e\in(\mathfrak{g}_\mathbb{Z})_{\bar{0}}$ we can find $f,h\in(\mathfrak{g}_\mathbb{Q})_{\bar{0}}$ such that $(e,h,f)$ is a $\mathfrak{sl}_2$-triple in $\mathfrak{g}$ (i.e. $[h,e]=2e,[h,f]=-2f,[e,f]=h$).

\begin{prop}\label{invariant bilinear}
Let $\mathfrak{g}$ be a basic classical Lie superalgebra (except for type $D(2,1;a)$\\$(a\notin\mathbb{Q}$)) over $\mathbb{C}$, then there exists an even non-degenerated supersymmetric invariant bilinear form $(\cdot,\cdot)$ and a Chevalley basis of $\mathfrak{g}$ under which the invariant bilinear form takes value in $\mathbb{Q}$.
\end{prop}

\begin{proof} Firstly, it is well known from Kac's classification theorem ([\cite{K2}], Proposition 1.1(a)) that the basic classical Lie superalgebras can be divided into seven types, i.e. $A(m, n),B(m, n),$ $C(n),D(m, n), D(2,1;a) (a\in\mathbb{C}\backslash\{0,-1\}), F(4), G(3)$. For each case a Chevalley basis of $\mathfrak{g}$ was constructed by R. Fioresi and F. Gavarini in ([\cite{FG}], Section 3.3) explicitly (which was firstly introduced by Shu-Wang for the orthogonal-symplectic case in [\cite{SW}]). We will choose these vectors as a basis of $\mathfrak{g}$. Now for each case we will consider separately:

(1) It follows from ([\cite{K}], Section 2.3 \& Section 2.4) that the Killing form $\kappa(\cdot,\cdot)$ is non-zero for all basic classical Lie superalgebras except for $A(n,n), D(n+1,n)$ and $D(2,1;a)$. Since $\kappa(\mathfrak{g}_\mathbb{Z},\mathfrak{g}_\mathbb{Z})\in\mathbb{Q}$, we can choose $\kappa(\cdot,\cdot)$ as the desired bilinear form.

(2) For the case $A(n,n)$ and $D(n+1,n)$, note that each element in the Chevalley basis of $\mathfrak{g}$ given in ([\cite{FG}], Section 3.3) is a linear combination of matrix vectors. Therefore, the super-trace $\text{str}(\cdot,\cdot)$ associated to the natural representation of $\mathfrak{g}$ takes value in $\mathbb{Q}$. It follow from ([\cite{K}], Proposition 1.1.2(a)) that $\text{str}(\cdot,\cdot)$ is non-degenerated, supersymmetric and invariant. As $\text{str}(\cdot,\cdot)$ is non-zero, we can choose $\text{str}(\cdot,\cdot)$ as the desired bilinear form (in fact, $\text{str}(\cdot,\cdot)$ can also be selected as the desired linear form in case (1)).

(3) For the case $D(2,1;a)$ ($a\in\mathbb{C}\backslash\{0,-1\}$), a set of generators for $\mathfrak{g}$ was formulated in ([\cite{FG}], Section 3.3), i.e. $\{h_i,e_i,f_i\}\,(i\in\{1,2,3\})$ where $e_1,f_1\in\mathfrak{g}_{\bar1}$, and the rest elements are in $\mathfrak{g}_{\bar0}$. Under which the Cartan matrix is:
$$(a_{i,j})_{i,j=1,2,3}=\left(
\begin{array}{c@{\hspace{6pt}}c@{\hspace{6pt}}c@{\hspace{5.5pt}}}
0&1&a\\-1&2&0\\-1&0&2
\end{array}\right),$$
and the relations between them are:
\[\begin{array}{lll}
~[h_i,h_j]=0,\quad &[e_1,e_1]=0,\quad &[f_1,f_1]=0,\\~[h_i,e_j]=a_{i,j}e_j,\quad &[h_i,f_j]=-a_{i,j}f_j,\quad &[e_i,f_j]=\delta_{i,j}h_i.
\end{array}\]

Define\[\begin{array}{llll}
e_{1,2}:=[e_1,e_2],\quad&e_{1,3}:=[e_1,e_3],\quad&e_{1,2,3}:=[e_{1,2},e_3],\quad&e_{1,1,2,3}:=\frac{[e_1,e_{1,2,3}]}{(1+a)},\\
f_{2,1}:=[f_2,f_1],\quad&f_{3,1}:=[f_3,f_1],\quad&f_{3,2,1}:=[f_3,f_{2,1}],\quad&f_{3,2,1,1}:=-\frac{[f_{3,2,1},f_1]}{(1+a)},\\
H_1:=h_1,\quad&H_2:=\frac{(2h_1-h_2-ah_3)}{(1+a)},\quad&H_3:=h_3,
\end{array}\]
then the set $\{H_i,e_i,f_i\}_{i=1,2,3}\cup\{e_{1,2},e_{1,3},e_{1,2,3},e_{1,1,2,3},f_{2,1},f_{3,1},f_{3,2,1},f_{3,2,1,1}\}$ is a Chevalley basis of $D(2,1;a)$.

If $a\neq0,-1$, we can define the bilinear form $(\cdot,\cdot)$ for the generators of $\mathfrak{g}$ by
\[\begin{array}{llll}(e_1,f_1)=1,& (e_2,f_2)=-1,& (e_3,f_3)=-\frac{1}{a},& (h_1,h_2)=-1,\\
(h_1,h_3)=-1,& (h_2,h_2)=-2,& (h_3,h_3)=-\frac{2}{a},\end{array}\]
(the unwritten ones are all zero) and expand it to the Chevalley basis of $D(2,1;a)$ by linearity. When $a\in\mathbb{Q}\backslash\{0,-1\}$, it is easy to verify that the bilinear form $(\cdot,\cdot)$ takes value in $\mathbb{Q}$.
\end{proof}

\begin{rem}\label{except}
In this paper basic classical Lie superalgebras over $\mathbb{C}$ will be referred to all except for type $D(2,1;a)(a\notin\mathbb{Q})$.
\end{rem}

It follows from Proposition~\ref{invariant bilinear} and the discussion earlier that $(e,f)\in\mathbb{Q}$. ([\cite{K}], Proposition 2.5.5(c)) shows that the non-degenerated, supersymmetric and invariant bilinear form on any basic classical Lie superalgebra is uniquely determined up to a constant factor. Therefore, we can assume $(e,f)=1$ and $(\cdot,\cdot)$ is in $\mathbb{Q}$ under the Chevalley basis of $\mathfrak{g}$ given in [\cite{K}]. Define $\chi\in\mathfrak{g}^{*}$ by letting $\chi(x)=(e,x)$ for $x\in\mathfrak{g}$, and it follows that $\chi(\mathfrak{g}_{\bar{1}})=0$.

\begin{defn}\label{admissible}
We call a commutative (in the usual sense, not super) ring $A$ admissible if $A$ is a finitely generated $\mathbb{Z}$-subalgebra of $\mathbb{C}$, $(e,f)\in A^{\times}(=A\backslash \{0\})$, and all bad primes of the root system of $\mathfrak{g}$ and the determinant of the Gram matrix of ($\cdot,\cdot$) relative to a Chevalley basis of $\mathfrak{g}$ are invertible in $A$.
\end{defn}

It is clear by definition that every admissible ring is a Noetherian domain. Given a finitely generated $\mathbb{Z}$-subalgebra $A$ of $\mathbb{C}$ and an element $\mathfrak{P}$ in the maximal spectrum of $A$, it is well known that for every $\mathfrak{P}\in\text{Specm}A$ the residue field $A/\mathfrak{P}$ is isomorphic to $\mathbb{F}_{q}$, where $q$ is a $p$-power depending on $\mathfrak{P}$. We denote by $\Pi(A)$ the set of all primes $p\in\mathbb{N}$ that occur in this way.

Since the choice of $A$ does not depend on the super structure of $\mathfrak{g}$, it follows from ([\cite{P8}], Section 2.1) that the set $\Pi(A)$ contains almost all primes in $\mathbb{N}$. For example, we can take $A=\mathbb{Z}[\frac{1}{N!}]$ for any sufficiently large integer $N$, then $A$ is an admissible algebra. Let $p$ be a prime with $p\gg N$, i.e. $p\gg0$, then $p\in\Pi(A)$.

Let $\mathfrak{g}(i)=\{x\in\mathfrak{g}|[h,x]=ix\}$ be the decomposition of $\mathfrak{g}$ under the Dynkin grading, then $\mathfrak{g}=\bigoplus\limits_{i\in\mathbb{Z}}\mathfrak{g}(i).$ By the $\mathfrak{sl}_2$-theory, all subspaces $\mathfrak{g}(i)$ defined are over $\mathbb{Q}$. Also, $e\in\mathfrak{g}(2)_{\bar{0}}$ and $f\in\mathfrak{g}(-2)_{\bar{0}}$. By ([\cite{H}], Lemma 2.7(i)) we know that if the integers $i$ and $j$ satisfy $i+j\neq0$, then $(\mathfrak{g}(i),\mathfrak{g}(j))=0$. Moreover, there exist symplectic and symmetric bilinear forms $\langle\cdot,\cdot\rangle$ on the $\mathbb{Z}_2$-graded subspaces $\mathfrak{g}(-1)_{\bar{0}}$ and $\mathfrak{g}(-1)_{\bar{1}}$ given by $$\langle x,y\rangle:=(e,[x,y])=\chi([x,y]),$$ respectively.

It follows from ([\cite{WZ}], Section 4.1) that dim~$\mathfrak{g}(-1)_{\bar{0}}$ is even. Take $\mathfrak{g}(-1)_{\bar{0}}^{\prime}\subset\mathfrak{g}(-1)_{\bar{0}}$ be a maximal isotropic subspace with respect to $\langle\cdot,\cdot\rangle$, then dim~$\mathfrak{g}(-1)_{\bar{0}}^{\prime}$= dim~$\mathfrak{g}(-1)_{\bar{0}}/2=s$. Let $u_{s+1},\cdots,u_{2s}$ of be a basis of $\mathfrak{g}(-1)_{\bar{0}}^{\prime}$, then we can choose a basis $u_1,\cdots,u_s$ of $\mathfrak{g}(-1)_{\bar{0}}\cap(\mathfrak{g}(-1)_{\bar{0}}^{\prime})^{\bot }$ (with respect to $\langle\cdot,\cdot\rangle$) such that $u_1,\cdots,u_s,u_{s+1},\cdots,u_{2s}$ is a basis of $\mathfrak{g}(-1)_{\bar{0}}$ under which the symplectic form $\langle\cdot,\cdot\rangle$ has matrix form
\newcommand*{\adots}{\mathinner{\mkern2mu\raisebox{0.1em}{.}
   \mkern2mu\raisebox{0.4em}{.}\mkern2mu\raisebox{0.7em}{.}\mkern1mu}}
\[\left(
\begin{array}{llllll}
&&&&&-1\\
&&&&\adots&\\
&&&-1&&\\
&&1&&&\\
&\adots&&&&\\
1&&&&&
\end{array}
\right),
\]
i.e. for any $1\leqslant i\leqslant 2s$, if we define \[i^*=\left\{\begin{array}{ll}-1&\text{if}~1\leqslant i\leqslant s;\\ 1&\text{if}~s+1\leqslant i\leqslant 2s,\end{array}\right.\] then $\langle u_i, u_j\rangle =i^*\delta_{i+j,2s+1}$, where $\delta_{i,j}$ is the kronecker symbol.

Accordingly, if $\text{dim}~\mathfrak{g}(-1)_{\bar1}=r$, we can choose a basis $v_1,\cdots,v_r$ of $\mathfrak{g}(-1)_{\bar{1}}$ under which the symmetric form $\langle\cdot,\cdot\rangle$ has matrix form
\[\left(
\begin{array}{lll}
&&1\\
&\adots&\\
1&&
\end{array}
\right),
\]
i.e. for any $1\leqslant i,j\leqslant r$, $\langle v_i,v_j\rangle=\delta_{i+j,r+1}$.

Since the bilinear form $\langle\cdot,\cdot\rangle$ on $\mathfrak{g}(-1)_{\bar{1}}$ is symmetric, the dimension of $\mathfrak{g}(-1)_{\bar{1}}$ is not necessary an even number. If $r$ is even, then take $\mathfrak{g}(-1)_{\bar{1}}^{\prime}\subseteq\mathfrak{g}(-1)_{\bar{1}}$ be the subspace spanned by $v_{\frac{r}{2}+1},\cdots,v_r$. If $r$ is odd, then take $\mathfrak{g}(-1)_{\bar{1}}^{\prime}\subseteq\mathfrak{g}(-1)_{\bar{1}}$ be the subspace spanned by $v_{\frac{r+3}{2}},\cdots,v_r$. Set $\mathfrak{g}(-1)^{\prime}=\mathfrak{g}(-1)'_{\bar{0}}\oplus\mathfrak{g}(-1)'_{\bar{1}}$ and introduce the subalgebras
$$\mathfrak{m}=\bigoplus_{i\leqslant -2}\mathfrak{g}(i)\oplus\mathfrak{g}(-1)^{\prime},\qquad \mathfrak{p}=\bigoplus_{i\geqslant 0}\mathfrak{g}(i),$$
\[\mathfrak{m}^{\prime}=\left\{\begin{array}{ll}\mathfrak{m}&\text{if}~r~\text{is even;}\\
\mathfrak{m}\oplus \mathbb{C}v_{\frac{r+1}{2}}&\text{if}~r~\text{is odd.}\end{array}\right.\]

\begin{rem}\label{bound}
From now on, we will {\bf denote the dimension of $\mathfrak{g}(-1)_{\bar{1}}$ by $r$}. For any real number $a\in\mathbb{R}$, let $\lceil a\rceil$ denote the largest integer lower bound of $a$, and $\lfloor a\rfloor$ the least integer upper bound of $a$. In particular, $\lceil a\rceil=\lfloor a\rfloor=a$ when $a\in\mathbb{Z}$. {\bf We will denote $\lfloor\frac{r}{2}\rfloor$ by $t$} (which equals to the dimension of $\mathfrak{g}(-1)_{\bar{1}}\cap(\mathfrak{g}(-1)_{\bar{1}}^{\prime})^{\bot }$) for convenience in this paper.
\end{rem}

\begin{rem}\label{centralizer}
Write $\mathfrak{g}^e$ for the centralizer of $e$ in $\mathfrak{g}$, and $\mathfrak{g}^f$ the centralizer of $f$ in $\mathfrak{g}$. For any $i\in\mathbb{Z}_2$, denote $d_i:=\text{dim}~\mathfrak{g}_i-\text{dim}~\mathfrak{g}^e_i$. It follows from ([\cite{WZ}], Theorem 4.3) that
$$\text{\underline{dim}}~\mathfrak{g}-\text{\underline{dim}}~\mathfrak{g}^e=\sum\limits_{k\geqslant2}
2\text{\underline{dim}}~\mathfrak{g}(-k)+\text{\underline{dim}}~\mathfrak{g}(-1).$$
\end{rem}

In particular, $\text{dim}~\mathfrak{g}(-1)_{\bar1}$ and $d_1$ always have the same parity. It follows from the definition of $\mathfrak{m}$ that either (1) $(\frac{d_0}{2},\frac{d_1}{2})=\text{\underline{dim}}~\mathfrak{m}$ when $\text{dim}~\mathfrak{g}(-1)_{\bar1}$ (or $d_1$, equivalently) is even, or (2) $(\frac{d_0}{2},\frac{d_1-1}{2})=\text{\underline{dim}}~\mathfrak{m}$ when $\text{dim}~\mathfrak{g}(-1)_{\bar1}$ (or $d_1$) is odd.

By the same discussion as ([\cite{P7}], Section 2.1), we can assume $\mathfrak{g}_A=\bigoplus\limits_{i\in\mathbb{Z}}\mathfrak{g}_A(i)$ after enlarging $A$ if need be, and each $\mathfrak{g}_A(i):=\mathfrak{g}_A\cap\mathfrak{g}(i)$ is freely generated over $A$ by a basis of the vector space $\mathfrak{g}(i)$. Then $\{u_1,\cdots,u_{2s}\}$ and $\{v_1,\cdots,v_r\}$ are free basis of $A$-module $\mathfrak{g}_A(-1)_{\bar0}$ and $\mathfrak{g}_A(-1)_{\bar1}$, respectively. By the assumptions on $A$ we can obtain that
$\mathfrak{m}_A:=\mathfrak{g}_A\cap\mathfrak{m}$, $\mathfrak{m}^{\prime}_A:=\mathfrak{g}_A\cap\mathfrak{m}^{\prime}$ and $\mathfrak{p}_A:=\mathfrak{g}_A\cap\mathfrak{p}$ are free $A$-modules and direct summands of $\mathfrak{g}_A$. More precisely,
$$\mathfrak{m}_A=\mathfrak{g}_A(-1)^{\prime}\oplus\bigoplus\limits_{i\leqslant -2}\mathfrak{g}_A(i),~ \text{where}
~\mathfrak{g}_A(-1)^{\prime}=\mathfrak{g}_A\cap\mathfrak{g}(-1)^{\prime},\quad \mathfrak{p}_A=\bigoplus_{i\geqslant 0}\mathfrak{g}_A(i),$$
\[\mathfrak{m}^{\prime}_A=\left\{\begin{array}{ll}\mathfrak{m}_A&\text{if}~r~\text{is even;}\\\mathfrak{m}_A\oplus A
v_{\frac{r+1}{2}}&\text{if}~r~\text{is odd.}\end{array}\right.\]

Let $\mathfrak{g}^*$ be the $\mathbb{C}$-module dual to $\mathfrak{g}$ and let $\mathfrak{m}^\perp$ denote the set of all linear functions on $\mathfrak{g}$ vanishing on $\mathfrak{m}$. By the discussion at the beginning of Section 3.1 we have that $e\in(\mathfrak{g}_\mathbb{Z})_{\bar{0}}, f\in(\mathfrak{g}_\mathbb{Q})_{\bar{0}}$. Hence we can assume $e,f\in(\mathfrak{g}_A)_{\bar0}$ after enlarging $A$ possibly (for example, if the admissible algebra is chosen as $\mathbb{Z}[\frac{1}{N!}]$, then one can just select a sufficiently large positive integer $N\gg0$) and that $[e,\mathfrak{g}_A(i)]$ and $[f,\mathfrak{g}_A(i)]$ are direct summands of $\mathfrak{g}_A(i+2)$ and $\mathfrak{g}_A(i-2)$, respectively. By the $\mathfrak{sl}_2$-theory we have $\mathfrak{g}_A(i+2)=[e,\mathfrak{g}_A(i)]$ for each $i\geqslant -1$.

Since the vectors in $\mathfrak{g}$ can be identified with their dual vectors in $\mathfrak{g}^*$ by the non-degenerated bilinear form $(\cdot,\cdot)$, we will identify the functions on $\mathfrak{g}$ naturally with the vectors in $\mathfrak{g}$.

\begin{lemma}\label{m'}
For the nilpotent subalgebra $\mathfrak{m}^\perp$ of Lie superalgebra $\mathfrak{g}$, we have
$$\mathfrak{m}^\perp=[\mathfrak{m}',e]\oplus\mathfrak{g}^f.$$
\end{lemma}

\begin{proof}
When $\text{dim}~\mathfrak{g}(-1)_{\bar1}$ is even, i.e. $\mathfrak{m}'=\mathfrak{m}$, the proof is the same as the Lie algebra case (see e.g. [\cite{W}], Lemma 26). When $\text{dim}~\mathfrak{g}(-1)_{\bar1}$ is odd, i.e. $\mathfrak{m}'\neq\mathfrak{m}$, minor modifications are needed for the proof. We will just sketch the proof as follows:

(1) $\mathfrak{g}^f\subseteq\mathfrak{m}^\perp$. This follows from $\mathfrak{g}^f\subseteq\bigoplus\limits_{i\leqslant 0}\mathfrak{g}(i)\subseteq\mathfrak{m}^\perp$.

(2) $[\mathfrak{m}',e]\subseteq\mathfrak{m}^\perp$. This can be seen by $([e,\mathfrak{m}'],\mathfrak{m})=(e,[\mathfrak{m}',\mathfrak{m}])=\chi([\mathfrak{m}',\mathfrak{m}])=0$.

(3) $\text{im}(\text{ad}e)\cap\mathfrak{g}^f=0$. This follows from the $\mathfrak{sl}_2$-representation theory.

(4) $\underline{\text{dim}}~\mathfrak{m}^\perp=\underline{\text{dim}}~\mathfrak{m}'+\underline{\text{dim}}~\mathfrak{g}(0)+\underline{\text{dim}}~\mathfrak{g}(-1)=\underline{\text{dim}}~[\mathfrak{m}',e]+\underline{\text{dim}}~\mathfrak{g}^f$. This follows by the bijection $\mathfrak{m}'\rightarrow [\mathfrak{m}',e], x\mapsto [x,e]$, by (2), and the $\mathfrak{sl}_2$-representation theory.
\end{proof}

\begin{lemma}\label{p}
For the subalgebra $\mathfrak{p}$ of Lie superalgebra $\mathfrak{g}$, we have
$$\mathfrak{p}=\bigoplus\limits_{j\geqslant 2}[f,\mathfrak{g}(j)]\oplus\mathfrak{g}^e.$$
\end{lemma}

\begin{proof}
The proof is straightforward and is the same as the Lie algebra case. (see e.g. [\cite{BGK}], Lemma 2.2)
\end{proof}

By Lemma~\ref{p} and the assumptions on $A$, we can choose a basis $x_1,\cdots,x_l,x_{l+1},\cdots,x_m$\\$\in\mathfrak{p}
_{\bar{0}}, y_1,\cdots,y_q,y_{q+1},$ $\cdots,y_n\in\mathfrak{p}_{\bar{1}}$ of $\mathfrak{p}=\bigoplus\limits_{i\geqslant 0}\mathfrak{g}(i)$ such that

(a) $x_i\in\mathfrak{g}(k_i)_{\bar{0}}, y_j\in\mathfrak{g}(k'_j)_{\bar{1}}$, where $k_i,k'_j\in\mathbb{Z}_+$;

(b) $x_1,\cdots,x_l$ is a basis of $\mathfrak{g}^e_{\bar{0}}$ and $y_1,\cdots,y_q$ is a basis of $\mathfrak{g}^e_{\bar{1}}$;

(c) $x_{l+1},\cdots,x_m\in[f,\mathfrak{g}_{\bar{0}}]$ and $ y_{q+1},\cdots,y_n\in[f,\mathfrak{g}_{\bar{1}}]$\\
and the corresponding elements of (a), (b) and (c) in $A$ form a basis of the free $A$-module $\mathfrak{p}_A=\bigoplus\limits_{i\geqslant 0}\mathfrak{g}_A(i)$ after enlarging admissible algebra $A$ if needed.

\begin{defn}\label{Gelfand-Graev}
Define the generalized Gelfand-Graev $\mathfrak{g}$-module associated to $\chi$ by $$Q_\chi=U(\mathfrak{g})\otimes_{U(\mathfrak{m})}\mathbb{C}_\chi,$$ where $\mathbb{C}_\chi=\mathbb{C}1_\chi$ is a $1$-dimensional $\mathfrak{m}$-module such that $x.1_\chi=\chi(x)1_\chi$ for all $x\in\mathfrak{m}$.
\end{defn}

Given $(\mathbf{a},\mathbf{b},\mathbf{c},\mathbf{d})\in\mathbb{Z}^m_+\times\mathbb{Z}^n_2\times\mathbb{Z}^s_+\times\mathbb{Z}^t_2$, let $x^\mathbf{a}y^\mathbf{b}u^\mathbf{c}v^\mathbf{d}$ denote the monomial $$x_1^{a_1}\cdots x_m^{a_m}y_1^{b_1}\cdots y_n^{b_n}u_1^{c_1}\cdots u_s^{c_s}v_1^{d_1}\cdots v_t^{d_t}$$ in $U(\mathfrak{g})$.

\begin{defn}\label{W-C}
Define the finite $W$-superalgebra over $\mathbb{C}$ by $$U(\mathfrak{g},e):=(\text{End}_\mathfrak{g}Q_{\chi})^{\text{op}},$$
where $(\text{End}_\mathfrak{g}Q_{\chi})^{\text{op}}$ denotes the opposite algebra of the endomorphism algebra of $\mathfrak{g}$-module $Q_{\chi}$.
\end{defn}

It can be easily concluded by definition that if two even nilpotent elements $E,E'\in\mathfrak{g}_{\bar0}$ are conjugate under the action of $\text{Ad}G_{\text{ev}}$, then there is an isomorphism between finite $W$-superalgebras $U(\mathfrak{g},E)$ and $U(\mathfrak{g},E')$. Therefore, the construction of finite $W$-superalgebras only depends on the adjoint orbit Ad$G_{\text{ev}}.e $ of $e$ up to isomorphism.

Let $N_\chi$ denote the $\mathbb{Z}_2$-graded ideal of codimension $1$ in $U(\mathfrak{m})$ generated by all $\langle x-\chi(x)|x\in\mathfrak{m}\rangle$ with $x\in\mathfrak{m}_i, i\in\mathbb{Z}_2$. Then $Q_\chi\cong U(\mathfrak{g})/U(\mathfrak{g})N_\chi$ as $\mathfrak{g}$-modules. By construction, the left ideal $I_\chi:=U(\mathfrak{g})N_\chi$ of $U(\mathfrak{g})$ is a $(U(\mathfrak{g}),U(\mathfrak{m}))$-bimodule. The fixed point space $(U(\mathfrak{g})/I_\chi)^{\text{ad}\mathfrak{m}}$ carries a natural algebra structure given by $$(x+I_\chi)\cdot(y+I_\chi):=(xy+I_\chi)$$ for all $x,y\in U(\mathfrak{g})$.

\begin{theorem}\label{W-C2}
There is an isomorphism between $\mathbb{C}$-algebras
\[\begin{array}{lcll}
\phi:&(\text{End}_\mathfrak{g}Q_{\chi})^{\text{op}}&\rightarrow&Q_{\chi}^{\text{ad}\mathfrak{m}}\\ &\Theta&\mapsto&\Theta(1_\chi),
\end{array}
\]
where $Q_{\chi}^{\text{ad}\mathfrak{m}}$ is the invariant subalgebra of $U(\mathfrak{g})/I_\chi\cong Q_{\chi}$ under the adjoint action of $\mathfrak{m}$.
\end{theorem}

\begin{proof}

Since each element in $(\text{End}_\mathfrak{g}Q_{\chi})^{\text{op}}$ is uniquely determined by its effect on $1_\chi\in Q_{\chi}$, it is easy to verify that the mapping $\phi$ is well-defined, both injective and surjective, and keeps the $\mathbb{Z}_2$-graded algebra structure. It remains to prove that $\phi$ is an isomorphism between $\mathbb{C}$-algebras.

For $\Theta=\Theta_{\bar{0}}+\Theta_{\bar{1}}, \Theta'=\Theta'_{\bar{0}}+\Theta'_{\bar{1}}\in(\text{End}_\mathfrak{g}Q_{\chi})^{\text{op}}$,
\[\begin{array}{rcl}
\phi(\Theta\cdot\Theta')&=&\phi((\Theta_{\bar{0}}+\Theta_{\bar{1}})\cdot(\Theta'_{\bar{0}}+\Theta'_{\bar{1}}))\\
&=&\phi(\Theta_{\bar{0}}\cdot\Theta'_{\bar{0}}+\Theta_{\bar{1}}\cdot\Theta'_{\bar{0}}+\Theta_{\bar{0}}\cdot\Theta'_{\bar{1}}+\Theta_{\bar{1}}\cdot\Theta'_{\bar{1}})\\
&=&(-1)^{|\Theta'_{\bar{0}}||\Theta_{\bar{0}}|}(\Theta'_{\bar{0}}\circ\Theta_{\bar{0}})(1_\chi)+(-1)^{|\Theta'_{\bar{0}}||\Theta_{\bar{1}}|}(\Theta'_{\bar{0}}\circ\Theta_{\bar{1}})(1_\chi)+\\
&&(-1)^{|\Theta'_{\bar{1}}||\Theta_{\bar{0}}|}(\Theta'_{\bar{1}}\circ\Theta_{\bar{0}})(1_\chi)+(-1)^{|\Theta'_{\bar{1}}||\Theta_{\bar{1}}|}(\Theta'_{\bar{1}}\circ\Theta_{\bar{1}})(1_\chi)\\
&=&\Theta'_{\bar{0}}(\Theta_{\bar{0}}(1_\chi))+\Theta'_{\bar{0}}(\Theta_{\bar{1}}(1_\chi))+\Theta'_{\bar{1}}(\Theta_{\bar{0}}(1_\chi))-\Theta'_{\bar{1}}(\Theta_{\bar{1}}(1_\chi))\\
&=&(-1)^{|\Theta'_{\bar{0}}||\Theta_{\bar{0}}|}\Theta_{\bar{0}}(1_\chi)\Theta'_{\bar{0}}(1_\chi)+(-1)^{|\Theta'_{\bar{0}}||\Theta_{\bar{1}}|}\Theta_{\bar{1}}(1_\chi)\Theta'_{\bar{0}}(1_\chi)+\\
&&(-1)^{|\Theta'_{\bar{1}}||\Theta_{\bar{0}}|}\Theta_{\bar{0}}(1_\chi)\Theta'_{\bar{1}}(1_\chi)-(-1)^{|\Theta'_{\bar{1}}||\Theta_{\bar{1}}|}\Theta_{\bar{1}}(1_\chi)\Theta'_{\bar{1}}(1_\chi)\\
&=&\Theta_{\bar{0}}(1_\chi)\Theta'_{\bar{0}}(1_\chi)+\Theta_{\bar{1}}(1_\chi)\Theta'_{\bar{0}}(1_\chi)+
\Theta_{\bar{0}}(1_\chi)\Theta'_{\bar{1}}(1_\chi)+\Theta_{\bar{1}}(1_\chi)\Theta'_{\bar{1}}(1_\chi),
\end{array}\]
but$$\phi(\Theta)\phi(\Theta')=(\Theta_{\bar{0}}(1_\chi)+\Theta_{\bar{1}}(1_\chi))(\Theta'_{\bar{0}}(1_\chi)+\Theta'_{\bar{1}}(1_\chi)),$$
therefore$$\phi(\Theta\cdot\Theta')=\phi(\Theta)\phi(\Theta').$$

It follows from all the discussions above that $\phi$ is an isomorphism.
\end{proof}

\begin{rem}\label{a-d}
We get an equivalent definition for the finite $W$-superalgebras over $\mathbb{C}$ due to Theorem~\ref{W-C2}. If we take $e=0$, then the finite  $W$-superalgebra is simply the enveloping algebra $U(\mathfrak{g})$. Hence the finite $W$-superalgebra $U(\mathfrak{g},e)$ can be considered as a generalization of universal enveloping algebra $U(\mathfrak{g})$.
\end{rem}

By the PBW theorem, there is a vector space decomposition: $$U(\mathfrak{g})=U(\mathfrak{\widetilde{p}})\oplus I_\chi,$$ where $\mathfrak{\widetilde{p}}:=\mathfrak{p}\oplus \mathbb{C}\langle x_1,\cdots,x_l, y_1,\cdots,y_q\rangle(\oplus\mathbb{C}v_{\frac{r+1}{2}})$ (recall that $x_1,\cdots,x_l\in\bigoplus\limits_{j\geqslant 2}[f,\mathfrak{g}(j)_{\bar0}]$ and $ y_1,\cdots,y_q\in\bigoplus\limits_{j\geqslant 2}[f,\mathfrak{g}(j)_{\bar1}]$), and the term $\mathbb{C}v_{\frac{r+1}{2}}$ occurs only for the case when $\text{dim}\,\mathfrak{g}(-1)_{\bar1}$ is odd. Let $\text{Pr}: U(\mathfrak{g})\longrightarrow U(\mathfrak{\widetilde{p}})$ denote the corresponding linear projection.

\begin{defn}\label{W-C3}
Define the subalgebra $W_\chi$ of $U(\mathfrak{\widetilde{p}})$ over $\mathbb{C}$ by
$$W_\chi:=\{u\in U(\mathfrak{\widetilde{p}})~|~\text{Pr}([x,u])=0~\text{for any}\,x\in\mathfrak{m}\}.$$
\end{defn}

\begin{theorem}\label{an iso}
There is an isomorphism between $\mathbb{C}$-algebras
\[\begin{array}{lcll}\varphi:&W_\chi&\longrightarrow&Q_\chi^{\text{ad}\mathfrak{m}}\\ &u&\mapsto&u(1+I_\chi).
\end{array}\]
\end{theorem}

The proof is straightforward and thus will be omitted here.

\begin{rem}\label{td}
It follows from Theorem~\ref{an iso} that we obtained another equivalent definition for finite $W$-superalgebras over $\mathbb{C}$ by Definition~\ref{W-C3}.
\end{rem}

\subsection{Kazhdan filtration}

To study the structure theory of finite $W$-algebras, Premet firstly introduced the ``$e$-degree'' (i.e. the Kazhdan degree) for the enveloping algebra $U(\mathfrak{g})$, then the Kazhdan filtration. In virtue of this filtration, the PBW theorem for finite $W$-algebras was obtained in [\cite{P2}]. Following Premet's treatment and the summary on finite $W$-algebras given by Brundan-Goodwin-Kleshchev in ([\cite{BGK}], Section 3.2), we will introduce the Kazhdan filtration for finite $W$-superalgebras in this part.

Let $\mathfrak{g}=\bigoplus\limits_{i\in\mathbb{Z}}\mathfrak{g}(i)$ denote the root decomposition of $\mathfrak{g}$ under the action of $\text{ad}h$. Define the Kazhdan degree on $\mathfrak{g}$ by declaring $x\in\mathfrak{g}(j)$ is $(j+2)$. Let $\text{F}_iU(\mathfrak{g})$ denote the span of monomials $x_1\cdots x_n$ for  $n\geqslant 0 $ with $ x_1\in\mathfrak{g}(j_1),\cdots,x_n\in\mathfrak{g}(j_n)$ in $U(\mathfrak{g})$ such that $(j_1+2)+\cdots+(j_n+2)\leqslant  i$. Then we get the Kazhdan filtration on $U(\mathfrak{g})$:
$$\cdots\subseteq \text{F}_iU(\mathfrak{g})\subseteq \text{F}_{i+1}U(\mathfrak{g})\subseteq\cdots.$$
The associated graded algebra $\text{gr}U(\mathfrak{g})$ is the supersymmetric algebra $S(\mathfrak{g})$ via the Kazhdan filtration on $\mathfrak{g}$ in which $x\in\mathfrak{g}(j)$ is of degree $(j+2)$.

The Kazhdan filtration on $U(\mathfrak{g})$ induces a filtration on its subalgebras. Setting $\text{F}_iU(\mathfrak{\widetilde{p}}):=U(\mathfrak{\widetilde{p}})\cap \text{F}_iU(\mathfrak{g})$, we get an induced Kazhdan filtration on the subalgebra $U(\mathfrak{\widetilde{p}})$. The Kazhdan filtration on $\mathfrak{\widetilde{p}}$ only involves positive degrees, so the Kazhdan filtration on $U(\mathfrak{\widetilde{p}})$ is strictly positive in the sense that $\text{F}_0U(\mathfrak{\widetilde{p}})=\mathbb{C}$ and $\text{F}_iU(\mathfrak{\widetilde{p}})=0$ for $i<0$. The Kazhdan filtration on $U(\mathfrak{g})$ also induces a filtration on the $\mathbb{Z}_2$-graded left ideal $I_\chi$ and on the quotient $Q_\chi=U(\mathfrak{g})/I_\chi$. By definition it is obvious that $Q_\chi$ is isomorphic to $U(\mathfrak{\widetilde{p}})$ as vector spaces. Hence $\text{gr}Q_\chi=S(\mathfrak{g})/\text{gr}I_\chi$ is a super-commutative $\mathbb{Z}_+$-graded algebra under the Kazhdan grading.

By above discussion we have identified $Q_\chi$ with $U(\mathfrak{\widetilde{p}})$, then we obtain an induced strictly positive filtration
$$\text{F}_0U(\mathfrak{g},e)\subseteq \text{F}_{1}U(\mathfrak{g},e)\subseteq\cdots$$ on $U(\mathfrak{g},e)$ such that $U(\mathfrak{g},e)$ is a subalgebra of $U(\mathfrak{\widetilde{p}})$ by the third definition of finite $W$-superalgebras (see Definition~\ref{W-C3}).

In virtue of the bilinear form $(\cdot,\cdot)$, we can identify $S(\mathfrak{g})$ with the polynomial superalgebra $\mathbb{C}[\mathfrak{g}]$ of regular functions on $\mathfrak{g}$. Then $\text{gr}I_\chi$ is the ideal generated by the functions $\{x-\chi(x)|x\in\mathfrak{m}\}$, i.e. the left ideal of all functions in $\mathbb{C}[\mathfrak{g}]$ vanishing on $e+\mathfrak{m}^\perp$ of $\mathfrak{g}$. Hence $\text{gr}~Q_\chi$ can be identified with $\mathbb{C}[e+\mathfrak{m}^\perp]$. Since $Q_\chi$ is identified with $U(\mathfrak{\widetilde{p}})$ as vector spaces, we have that $S(\mathfrak{\widetilde{p}})\cong\text{gr}~Q_\chi$. Then it follows that $$S(\mathfrak{\widetilde{p}})\cong\mathbb{C}[e+\mathfrak{m}^\perp].$$
If just considering the even part, we can get an isomorphism between $\mathbb{C}$-algebras (in the usual sense, not super)$$S(\mathfrak{\widetilde{p}}_{\bar{0}})\cong\mathbb{C}[e+\mathfrak{m}_{\bar{0}}^\perp],$$ where $\mathfrak{m}^\perp_{\bar{0}}:=\{f\in\mathfrak{g}^*_{\bar{0}}|f(\mathfrak{m}_{\bar{0}})=0\}$, $\mathfrak{m}^\perp_{\bar{0}}$ is identified with the subalgebra of $\mathfrak{g}_{\bar0}$ by the bilinear form $(\cdot,\cdot)$, and $\mathbb{C}[e+\mathfrak{m}_{\bar{0}}^\perp]$ the regular functions on affine variety $e+\mathfrak{m}_{\bar{0}}^\perp$.

\subsection{Whittaker functor and Skryabin equivalence}

Whittaker category is an important part in the representation theory of Lie algebras, and a great deal of infinite dimensional representations are included in this category. Following Skryabin's treatment to the Lie algebra case in [\cite{S2}], we will firstly introduce the Whittaker modules for Lie superalgebras, then establish the Skryabin equivalence between certain representation category of Lie superalgebras and the representation category of finite $W$-superalgebras. All these provide a powerful tool for the study on the infinite dimensional representations theory of Lie superalgebras.

\begin{defn}\label{Whittaker}
A $\mathfrak{g}$-module $L$ is called a Whittaker module if $a-\chi(a),~~\forall a\in\mathfrak{m},$ acts on $L$ locally nilpotently. A Whittaker vector in a Whittaker $\mathfrak{g}$-module $L$ is a vector $v\in L$ which satisfies $(a-\chi(a))v=0,~\forall a\in\mathfrak{m}$.
\end{defn}

Let $\mathfrak{g}\text{-}W\text{mod}^\chi$ denote the category of finitely generated Whittaker $\mathfrak{g}$-modules, and assume all the morphisms are even. Write $$\text{Wh(L)}=\{v\in L|(a-\chi(a))v=0,\forall a\in\mathfrak{m}\}$$ the subspace of all Whittaker vectors in $L$.

Recall the second definition of finite $W$-superalgebras (see Theorem~\ref{W-C2}) shows that $U(\mathfrak{g},e)\cong(U(\mathfrak{g})/I_\chi)^{\text{ad}\mathfrak{m}}$. Denote by $\bar{y}\in U(\mathfrak{g})/I_\chi$ the coset associated to $y\in U(\mathfrak{g})$.

\begin{theorem}\label{W-n}
(1) Given a Whittaker $\mathfrak{g}$-module $L$ with an action map $\rho$, $\text{Wh}(L)$ is naturally a $U(\mathfrak{g},e)$-module by letting $$\bar{y}.v=\rho(y)v$$ for $v\in\text{Wh}(L)$ and $\bar{y}\in U(\mathfrak{g})/I_\chi$.

(2) For $M\in U(\mathfrak{g},e)$, $Q_\chi\otimes_{U(\mathfrak{g},e)}M$ is a Whittaker $\mathfrak{g}$-module by letting $$y.(q\otimes v)=(y.q)\otimes v$$ for $y\in U(\mathfrak{g})$ and $q\in Q_\chi,~v\in V$.
\end{theorem}

\begin{proof}
The proof is straightforward and is the same as the Lie algebra case. (see e.g. proof of ([\cite{W}], Lemma 35)).
\end{proof}

Given a $\mathfrak{g}\text{-}W\text{mod}^\chi$ $M$, define $M^{\mathfrak{m}}$ by $$M^{\mathfrak{m}}=\{m\in M | x.m=\chi(x)m \quad\text{for all}~x\in\mathfrak{m}\},$$
then $M^{\mathfrak{m}}$ can be considered as a $U(\mathfrak{g},e)$-module. The following theorem shows that there exists an equivalence of categories between the $\mathfrak{g}\text{-}W\text{mod}^\chi$ and the $U(\mathfrak{g},e)$-modules.

\begin{theorem}\label{functor}
The functor $Q_\chi\otimes_{U(\mathfrak{g},e)}-:U(\mathfrak{g},e)\text{-mod}\longrightarrow
\mathfrak{g}\text{-}W\text{mod}^\chi$ is an equivalence of categories, with $\text{Wh}:\mathfrak{g}\text{-}W\text{mod}^\chi\longrightarrow U(\mathfrak{g},e)\text{-mod}$ as its quasi-inverse.
\end{theorem}

The theorem generalizes the situation of the Lie algebra case. For the Lie algebra case, Skryabin firstly defined a partial ordered set on the basis of $U(\mathfrak{m})$ by the Kazhdan degree in [\cite{S2}], then gave a proof by induction. Applying the finite $W$-algbras' BRST cohomology definition, Gan and Ginzburg gave an alternative proof for the Skryabin's theorem in [\cite{GG}]. However, the cohomology theory related to the finite $W$-superalgebras has yet to be developed. Especially when $\text{dim}~\mathfrak{g}(-1)_{\bar1}$ is odd, there is even no BRST cohomology definition for the finite $W$-superalgebras. Hence we will follow Skryabin's treatment here. The proof is just sketched, and more specific details refer to ([\cite{S2}], Therorem 1).

\begin{proof}
Define the map $\mu:Q_\chi\otimes_{U(\mathfrak{g},e)}M^{\mathfrak{m}}\longrightarrow M$ by the rule $u.1_\chi\otimes v\mapsto uv$ for $u\in U(\mathfrak{g})$ and $v\in M^{\mathfrak{m}}$. For any $U(\mathfrak{g},e)$-module $V'$, define $\nu:V'\longrightarrow (Q_\chi\otimes_{U(\mathfrak{g},e)}V')^{\mathfrak{m}}$ by the rule $\nu(v')=1_\chi\otimes v'$ for $v'\in V'$. Following Skryabin's discussion in ([\cite{S2}], Therorem 1) (Some of the details need to be improved. See the proof of Proposition 4.2 in [\cite{WZ}] by Wang and Zhao), it is immediate that $\mu$ is an isomorphism of $\mathfrak{g}$-modules, and $\nu$ is an isomorphism of $U(\mathfrak{g},e)$-modules.
\end{proof}

\subsection{Restricted root system and restricted root decomposition}

In this part we will introduce the restricted root system and restricted root decomposition for the basic classical Lie superalgebra associated to an even nilpotent element, which lays the foundation for further study on the highest weight theory of finite $W$-superalgebras. As for the Lie algebra case, we refer to [\cite{BG3}] and [\cite{BGK}].

It is well known that even part $\mathfrak{g}_{\bar{0}}$ of a basic classical Lie superalgebra $\mathfrak{g}$ over $\mathbb{C}$ is a reductive Lie algebra, and $\mathfrak{g}_{\bar{0}}^h\cap\mathfrak{g}_{\bar{0}}^e$ is a Levi factor of $\mathfrak{g}_{\bar{0}}^e$. Recall that the root system of $\mathfrak{g}$ we have chosen is distinguished (see Section 3.1), and the vector space decomposition $\mathfrak{g}=\bigoplus\limits_{i\in\mathbb{Z}}\mathfrak{g}(i)$ is obtained by the action of ad$h$. Hence we can pick a maximal toral subalgebra $\mathfrak{t}^e\subseteq\mathfrak{g}(0)_{\bar0}$ in $\mathfrak{g}_{\bar{0}}^h\cap\mathfrak{g}_{\bar{0}}^e$, and a maximal toral subalgebra $\mathfrak{h}$ of $\mathfrak{g}$ containing $\mathfrak{t}^e$ and $h$.

We can obtain the following result, which can be proved by the same method as the Lie algebra case in ([\cite{BG3}], Lemma 13).

\begin{prop}\label{weight}
The set of weights of $\mathfrak{t}^e$ on $\mathfrak{g}^e$ is equal to the set of weights of $\mathfrak{t}^e$ on $\mathfrak{g}$ under the action of ad$h$.
\end{prop}

For $\alpha\in(\mathfrak{t}^e)^*$, let $\mathfrak{g}_\alpha=\bigoplus\limits_{i\in\mathbb{Z}}\mathfrak{g}_\alpha(i)$ denote the $\alpha$-weight space of $\mathfrak{g}$ with respect to $\mathfrak{t}^e$. Hence $$\mathfrak{g}=\mathfrak{g}_0\oplus\bigoplus_{\alpha\in\Phi^e}\mathfrak{g}_\alpha$$ where ${\alpha\in\Phi^e}\subset(\mathfrak{t}^e)^*$ denotes the set of non-zero weights of $\mathfrak{t}^e$ on $\mathfrak{g}$. Similarly, each of the spaces $\mathfrak{m}, \mathfrak{m}',\mathfrak{p}, \mathfrak{\widetilde{p}}$ can also be decomposed into $\mathfrak{t}^e$-weight spaces. $\Phi^e$ is called a restricted root system. Notably, $\Phi^e$ is not a root system in the usual sense; for example, for $\alpha\in\Phi^e$ there may be multiples of $\alpha$ other than $\pm\alpha$ that belongs to $\Phi^e$. By Proposition~\ref{weight} it is immediate that $\Phi^e$ is also the non-zero weights of $\mathfrak{t}^e$ on $\mathfrak{g}^e$, i.e. there is an induced restricted root decomposition: $$\mathfrak{g}^e=\mathfrak{g}^e_0\oplus\bigoplus_{\alpha\in\Phi^e}\mathfrak{g}^e_\alpha.$$

\section{Finite $W$-superalgebras in positive characteristic}
In this part, we will introduce the finite $W$-superalgebra associated to basic classical Lie superalgebra $\mathfrak{g}$ (which also includes the case $D(2,1;\bar a)(\bar a\in\mathds{k})$) over positive characteristic field $\mathds{k}:=\overline{\mathbb{F}}_p $ with $p\in\Pi(A)$.

\subsection{The definition of finite $W$-superalgebras in positive characteristic}

Given an admissible algebra $A$, set $Q_{\chi,A}:=U(\mathfrak{g}_A)\otimes_{U(\mathfrak{m}_A)}A_\chi$, where $A_\chi=A1_\chi$. It follows by definition that $Q_{\chi,A}$ is a $\mathfrak{g}_A$-stable $A$-lattice in $Q_{\chi}$ with $$\{x^\mathbf{a}y^\mathbf{b}u^\mathbf{c}v^\mathbf{d}\otimes1_\chi|(\mathbf{a},\mathbf{b},\mathbf{c},\mathbf{d})\in\mathbb{Z}^m_+\times\mathbb{Z}^n_2\times\mathbb{Z}^s_+\times\mathbb{Z}^t_2\}$$
as a free basis. Let $N_{\chi,A}$ denote the homogeneous ideal of codimension $1$ in $A$-subalgebra $U(\mathfrak{m}_A)$ generated by all $x-\chi(x)$ with $x\in(\mathfrak{m}_A)_{i}$ where $i\in\mathbb{Z}_2$. Set $I_{\chi,A}:=U(\mathfrak{g}_A)N_{\chi,A}$, the left ideal of $U(\mathfrak{g}_A)$. Then $Q_{\chi,A}\cong U(\mathfrak{g}_A)/I_{\chi,A}$ as $\mathfrak{g}_A$-modules.

Pick a prime $p\in\Pi(A)$ and denote by $\mathds{k}=\overline{\mathbb{F}}_p$ the algebraic closure of $\mathbb{F}_p$. By Definition~\ref{admissible} and the discussion thereafter, we can assume that $(\cdot,\cdot)$ is $A$-valued on $\mathfrak{g}_A$ after enlarging $A$, possibly. The bilinear form $(\cdot,\cdot)$ induces a bilinear form on the Lie superalgebra $\mathfrak{g}_\mathds{k}\cong\mathfrak{g}_A\otimes_A\mathds{k}$. In the following we still denote this bilinear form by $(\cdot,\cdot)$.

If we denote by $G_\mathds{k}$ the algebraic $\mathds{k}$-supergroup of hyperalgebra $U_\mathds{k}=U_\mathbb{Z}\otimes_\mathbb{Z}\mathds{k}$, then $\mathfrak{g}_\mathds{k}=\text{Lie}(G_\mathds{k})$ by the discussion in Section 3.1. Note that the bilinear form $(\cdot,\cdot)$ is non-degenerated and $\text{Ad}(G_\mathds{k})_{\text{ev}}$-invariant. For $x\in\mathfrak{g}_A$, set $\bar{x}:=x\otimes1$, an element of $\mathfrak{g}_\mathds{k}$. To ease notation we identify $e,f,h$ with the nilpotent elements $\bar{e}=e\otimes1,~\bar{f}=f\otimes1$ and $\bar{h}=h\otimes1$ in $\mathfrak{g}_\mathds{k}$, and $\chi$ with the linear function $(e,\cdot)$ on $\mathfrak{g}_\mathds{k}$. Obviously this will not cause confusion.  Set $\mathfrak{m}_\mathds{k}:=\mathfrak{m}_A\otimes_A\mathds{k}$, $\mathfrak{m}'_\mathds{k}:=\mathfrak{m}'_A\otimes_A\mathds{k}$.

Let $\mathfrak{g}_\mathds{k}$ be a restricted Lie superalgebra (see Definition~\ref{restricted}). For each $\bar x\in(\mathfrak{g}_\mathds{k})_{\bar{0}}$, we can obtain that $\bar x^p-\bar x^{[p]}\in U(\mathfrak{g}_\mathds{k})$ is contained in the center of $U(\mathfrak{g}_\mathds{k})$ by definition. The subalgebra $\mathds{k}\langle \bar x^p-\bar x^{[p]}|\bar x\in(\mathfrak{g}_\mathds{k})_{\bar{0}}\rangle$ of $U(\mathfrak{g}_\mathds{k})$ is called the $p$-center of $U(\mathfrak{g}_\mathds{k})$ and denote $Z_p(\mathfrak{g}_\mathds{k})$ for short. It follows from the PBW theorem of $U(\mathfrak{g}_\mathds{k})$ that $Z_p(\mathfrak{g}_\mathds{k})$ is isomorphic to a polynomial algebra (in the usual sense) in $\text{dim}~(\mathfrak{g}_\mathds{k})_{\bar{0}}$ variables. For every maximal ideal $J$ of $Z_p(\mathfrak{g}_\mathds{k})$ there is a unique linear function $\eta=\eta_J\in(\mathfrak{g}_\mathds{k})_{\bar{0}}^*$ such that $$J=\langle \bar x^p-\bar x^{[p]}-\eta(\bar x)^p|\bar x\in(\mathfrak{g}_\mathds{k})_{\bar{0}}\rangle.$$
Since the Frobenius map of $\mathds{k}$ is bijective, this enables us to identify the maximal spectrum $\text{Specm}(Z_p(\mathfrak{g}_\mathds{k}))$ of $Z_p(\mathfrak{g}_\mathds{k})$ with $(\mathfrak{g}_\mathds{k})_{\bar{0}}^*$.

For any $\xi\in(\mathfrak{g}_\mathds{k})_{\bar{0}}^*$ we write $J_\xi$ the two-sided ideal of $U(\mathfrak{g}_\mathds{k})$ generated by the even central elements $$\{\bar x^p-\bar x^{[p]}-\xi(\bar x)^p|\bar x\in(\mathfrak{g}_\mathds{k})_{\bar{0}}\}.$$ Then the quotient algebra $U_\xi(\mathfrak{g}_\mathds{k}):=U(\mathfrak{g}_\mathds{k})/J_\xi$ is a $\mathfrak{g}_\mathds{k}$-module, which is called the reduced enveloping algebra with $p$-character $\xi$. We often regard $\xi\in\mathfrak{g}_\mathds{k}^*$ by letting $\xi((\mathfrak{g}_\mathds{k})_{\bar{1}})=0$. By the classical theory of Lie superalgebras, we have $$\text{dim}~U_\xi(\mathfrak{g}_\mathds{k})=p^{\text{dim}~(\mathfrak{g}_\mathds{k})_{\bar{0}}}2^{\text{dim}~(\mathfrak{g}_\mathds{k})_{\bar{1}}}.$$

For $i\in\mathbb{Z}$, define the graded subspaces of $\mathfrak{g}_\mathds{k}$ over $\mathds{k}$ by $$\mathfrak{g}_\mathds{k}(i):=\mathfrak{g}_A(i)\otimes_A\mathds{k},\quad \mathfrak{m}_\mathds{k}(i):=\mathfrak{m}_A(i)\otimes_A\mathds{k}.$$
Due to our assumptions on $A$, the elements $\bar{x}_1,\cdots,\bar{x}_l$ and $\bar{y}_1,\cdots,\bar{y}_q$ form a basis of the centralizer $(\mathfrak{g}^e_\mathds{k})_{\bar{0}}$ and $(\mathfrak{g}^e_\mathds{k})_{\bar{1}}$ of $e$ in $\mathfrak{g}_\mathds{k}$, respectively.

It follows from ([\cite{WZ}], Section 4.1) that the subalgebra $\mathfrak{m}_\mathds{k}$ is $p$-nilpotent, and the linear function $\chi$ vanishes on the $p$-closure of $[\mathfrak{m}_\mathds{k},\mathfrak{m}_\mathds{k}]$. Set $$Q_{\chi,\mathds{k}}:=U(\mathfrak{g}_\mathds{k})\otimes_{U(\mathfrak{m}_\mathds{k})}\mathds{k}_\chi,$$ where $\mathds{k}_\chi=A_\chi\otimes_{A}\mathds{k}=\mathds{k}1_\chi$. Clearly, $\mathds{k}1_\chi$ is a $1$-dimensional $\mathfrak{m}_\mathds{k}$-module with the property $\bar x.1_\chi=\chi(\bar x)1_\chi$ for all $\bar x\in\mathfrak{m}_\mathds{k}$ and it is obvious that $Q_{\chi,\mathds{k}}\cong Q_{\chi,A}\otimes_A\mathds{k}$. Define $N_{\chi,\mathds{k}}:=N_{\chi,A}\otimes_A\mathds{k}$ and $I_{\chi,\mathds{k}}:=I_{\chi,A}\otimes_A\mathds{k}$.

\begin{defn}\label{W-k}
Define the finite $W$-superalgebra over $\mathds{k}$ by $$\widehat{U}(\mathfrak{g}_\mathds{k},e):=(\text{End}_{\mathfrak{g}_\mathds{k}}Q_{\chi,\mathds{k}})^{\text{op}},$$
where $(\text{End}_{\mathfrak{g}_\mathds{k}}Q_{\chi,\mathds{k}})^{\text{op}}$ denotes the opposite algebra of $\text{End}_{\mathfrak{g}_\mathds{k}}Q_{\chi,\mathds{k}}$.
\end{defn}

\begin{rem}\label{D(2,1)}
Recall that when we call $\mathfrak{g}$ a basic classical Lie superalgebra over $\mathbb{C}$, the ones of type $D(2,1;a)( a\notin\mathbb{Q})$ are excluded (see Remark~\ref{except}). However, it is notable that the procedure of ``modular $p$ reduction'' goes smoothly for the case $D(2,1;\bar a)(\bar a\in\mathds{k}\backslash\{\bar0,\overline {-1}\})$, i.e. the $\mathds{k}$-algebra $\widehat{U}(\mathfrak{g}_\mathds{k},e)$ with $\mathfrak{g}_\mathds{k}\in D(2,1;\bar a)(\bar a\in\mathds{k}\backslash\{\bar0,\overline {-1}\})$ can also be induced from the $\mathbb{C}$-algebra $U(\mathfrak{g},e)$ associated to the Lie superalgebra $D(2,1;a)(a\in\mathbb{Q}\backslash\{0,1\})$. Since all the consequences obtained in Section 3.1 still establish for $D(2,1;a)(a\in\mathbb{Q}\backslash\{0,1\})$, the basic classical Lie superalgebras over $\mathds{k}=\overline{\mathbb{F}}_p$ will be referred to all types with $p\in\Pi(A)$ in this paper, which also include the case $D(2,1;\bar a)(\bar a\in\mathds{k}\backslash\{\bar0,\overline {-1}\})$.
\end{rem}

Let $\mathfrak{g}_A^*$ be the $A$-module dual to $\mathfrak{g}_A$, so that $\mathfrak{g}^*=\mathfrak{g}_A^*\otimes_A\mathbb{C}, ~\mathfrak{g}_\mathds{k}^*=\mathfrak{g}_A^*\otimes_A\mathds{k}$. Let $(\mathfrak{m}_A^\perp)_{\bar{0}}$ denote the set of all linear functions on $(\mathfrak{g}_A)_{\bar0}$ vanishing on $(\mathfrak{m}_A)_{\bar{0}}$. By the assumptions on $A$,  $(\mathfrak{m}_A^\perp)_{\bar{0}}$ is a free $A$-submodule and a direct summand of $\mathfrak{g}^*_A$. Note that
$(\mathfrak{m}_A^\perp\otimes_A\mathbb{C})_{\bar{0}}$ and $(\mathfrak{m}_A^\perp\otimes_A\mathds{k})_{\bar{0}}$ can be identified with the annihilators $\mathfrak{m}^\perp_{\bar{0}}
:=\{f\in\mathfrak{g}^*_{\bar{0}}|f(\mathfrak{m}_{\bar{0}})=0\}$ and $(\mathfrak{m}_\mathds{k}^\perp)_{\bar{0}}:=\{f\in(\mathfrak{g}_\mathds{k}^*)_{\bar{0}}|f((\mathfrak{m}_\mathds{k})_{\bar{0}})=0\}$, respectively.

Given a linear function $\eta\in\chi+(\mathfrak{m}_\mathds{k}^\perp)_{\bar{0}}$, set $\mathfrak{g}_\mathds{k}$-module $$Q_{\chi}^\eta:=Q_{\chi,\mathds{k}}/J_\eta Q_{\chi,\mathds{k}},$$ where $J_\eta$ is the homogeneous ideal of $U(\mathfrak{g}_\mathds{k})$ generated by all $\{\bar x^p-\bar x^{[p]}-\eta(\bar x)^p|\bar x\in(\mathfrak{g}_\mathds{k})_{\bar0}\}$. Evidently $Q_{\chi}^\eta$ is a $\mathfrak{g}_\mathds{k}$-module with $p$-character $\eta$, and there exists a $\mathfrak{g}_\mathds{k}$-module isomorphism $$Q_{\chi}^\eta\cong U_\eta(\mathfrak{g}_\mathds{k})\otimes_{U_\eta(\mathfrak{m}_\mathds{k})}\bar 1_\chi.$$

\begin{defn}\label{reduced W}
Define the reduced enveloping algebra $U_\eta(\mathfrak{g}_\mathds{k},e)$ of finite $W$-superalgebras associated to $p$-character $\eta\in\chi+(\mathfrak{m}_\mathds{k}^\bot)_{\bar0}$ by $$U_\eta(\mathfrak{g}_\mathds{k},e):=(\text{End}_{\mathfrak{g}_\mathds{k}}Q_{\chi}^\eta)^{\text{op}}.$$
\end{defn}

In this paper, we will call $U_\eta(\mathfrak{g}_\mathds{k},e)$ {\bf the reduced $W$-superalgebra} associated to the even nilpotent element $e$ and $p$-character $\eta\in\chi+(\mathfrak{m}_\mathds{k}^\bot)_{\bar0}$. It is immediate that the restriction of $\eta$ coincides with that of $\chi$ on $(\mathfrak{m}_\mathds{k})_{\bar{0}}$. If we let $\eta((\mathfrak{m}_\mathds{k})_{\bar{1}})=0$, then the ideal of $U(\mathfrak{m}_\mathds{k})$ generated by all $\{\bar x-\eta(\bar x)|\bar x\in(\mathfrak{m}_\mathds{k})_i,~i\in\mathbb{Z}_2\}$ equals $N_{\chi,\mathds{k}}=N_{\chi,A}\otimes_A\mathds{k}$, and $\mathds{k}_\chi=\mathds{k}_\eta$ as $\mathfrak{m}_\mathds{k}$-modules. Let $N_{\mathfrak{m}_\mathds{k}}$ denote the Jacobson radical of $U_\eta(\mathfrak{m}_\mathds{k})$, i.e. the ideal of codimensional one in $U_\eta(\mathfrak{m}_\mathds{k})$ generated by all $\langle x-\eta(x)\rangle$ with $x\in\mathfrak{m}_\mathds{k}$, and define $I_{\mathfrak{m}_\mathds{k}}:=U_\eta(\mathfrak{g}_\mathds{k})N_{\mathfrak{m}_\mathds{k}}$ be the ideal of $U_\eta(\mathfrak{g}_\mathds{k})$.

\begin{rem}\label{notreduced}
It is notable that Definition~\ref{reduced W} was first introduced by Wang and Zhao in ([\cite{WZ}], Theorem 4.4). However, what they have defined is the case $\eta=\chi$, i.e. the reduced $W$-superalgebra $U_\chi(\mathfrak{g}_\mathds{k},e)$ associated to $p$-character $\chi$ (which was called a finite $W$-superalgebra in [\cite{WZ}]). Notably, they did not introduce the definition in the same way as we do here (i.e. first introduce the finite $W$-superalgebras over $\mathbb{C}$, then define the reduced $W$-superalgebras over positive characteristic field $\mathds{k}=\overline{\mathbb{F}}_p$ by ``modular $p$ reduction''), but constructed the reduced $W$-superalgebras over $\mathds{k}$ directly by the same way as Premet's treatment to the reduced $W$-algebras over $\mathds{k}$ (see [\cite{P2}]). Therefore, the restriction on the character of field $\mathds{k}$ given by Wang-Zhao is much weaker than the requirement in this paper, see ([\cite{WZ}], Section 2.2). The reason why we do not follow their treatment is that we also want to study on the case over $\mathbb{C}$.

Wang and Zhao pointed out that the $\mathbb{Z}$-grading of the vector space $\mathfrak{g}_\mathds{k}$ defined in ([\cite{WZ}], Remark 3.2) can be thought of coming from the $\mathfrak{sl}_2$-theory in characteristic zero, i.e. the vector space decomposition of $\mathfrak{g}$ with respect to $\text{ad}h$. Hence when the characteristic $p\gg0$, the grading defined by Wang and Zhao is the same as the Dynkin grading of $\mathfrak{g}$ applied in this paper.
\end{rem}

\begin{prop}$^{[\cite{W}]}$\label{invariant}
There exists an isomorphism between $\mathds{k}$-algebras:
\[\begin{array}{lcll}\varphi:&U_\chi(\mathfrak{g}_\mathds{k},e)&\longrightarrow&(Q_\chi^\chi)^{\text{ad}\mathfrak{m}_\mathds{k}}.
\end{array}\]
\end{prop}

This result was firstly referred by Wang in ([\cite{W}], Remark 70) but without a proof. Premet gave a proof for the finite $W$-algebra case in ([\cite{P2}], Theorem 2.3(iv)) with the help of the support variety machinery. It is notable that Wang and Zhao bypassed the support variety machinery completely when they introduced the reduced $W$-superalgebra $U_\chi(\mathfrak{g}_\mathds{k},e)$ over $\mathds{k}$ in [\cite{WZ}]. Hence some modification is needed in the proof. First notice that

\begin{lemma}\label{free}
$Q_{\chi}^\chi$ is a free $U_\chi(\mathfrak{m}_\mathds{k})$-module under the action of ad\,$\mathfrak{m}_\mathds{k}$.
\end{lemma}

\begin{proof} First of all, for every $\bar x=\bar x_{\bar{0}}+\bar x_{\bar{1}}\in\mathfrak{m}_\mathds{k}$ and $\bar u=\bar u_{\bar{0}}+\bar u_{\bar{1}}\in U_\chi(\mathfrak{g}_\mathds{k})$, one has
\begin{eqnarray}\label{adfree}
[\bar x,\bar u]&=&[\bar x-\chi(\bar x),\bar u]=[\bar x_{\bar{0}}+\bar x_{\bar{1}}-\chi(\bar x),\bar u_{\bar{0}}+\bar u_{\bar{1}}]\nonumber\\
&=&(\bar x_{\bar{0}}+\bar x_{\bar{1}}-\chi(\bar x))(\bar u_{\bar{0}}+\bar u_{\bar{1}})
-\bar u_{\bar{0}}(\bar x_{\bar{0}}+\bar x_{\bar{1}}-\chi(\bar x))-\bar u_{\bar{1}}(\bar x_{\bar{0}}-\chi(\bar x_{\bar{0}}))\nonumber\\
&&+\bar u_{\bar{1}}\bar x_{\bar{1}}\nonumber\\
&=&(\bar x-\chi(\bar x))\bar u-\bar u_{\bar{0}}(\bar x-\chi(\bar x))
-\bar u_{\bar{1}}(\bar x_{\bar{0}}-\chi(\bar x_{\bar{0}}))+\bar u_{\bar{1}}(\bar x_{\bar{1}}-\chi(\bar x_{\bar{1}}))
\end{eqnarray}
since $\chi((\mathfrak{g}_\mathds{k})_{\bar{1}})=0$. By the definition of $I_{\mathfrak{m}_\mathds{k}}$ one knows that the last three terms in \eqref{adfree} are in $I_{\mathfrak{m}_\mathds{k}}$, then it is immediate that
\begin{equation}\label{[]tomult}
[\bar x,\bar u]\equiv(\bar x-\chi(\bar x))\bar u\quad(\text{mod}I_{\mathfrak{m}_\mathds{k}})
\end{equation}
for all $\bar x\in\mathfrak{m}_\mathds{k}$ and $\bar u\in Q_{\chi}^\chi$.

By ([\cite{WZ}], Proposition 4.2) we know that every $U_\chi(\mathfrak{g}_\mathds{k})$-module is $U_\chi(\mathfrak{m}_\mathds{k})$-free under the action of left-multiplication. It is immediate from \eqref{[]tomult} that $U_\chi(\mathfrak{g}_\mathds{k})$-module $Q_\chi^\chi$ is $U_\chi(\mathfrak{m}_\mathds{k})$-free under the action of ad$\mathfrak{m}_\mathds{k}$.
\end{proof}

Now we are in a position to prove Proposition~\ref{invariant}. In fact, the proof of Proposition~\ref{invariant} is the same as the finite $W$-algebra case after one establishing Lemma~\ref{free}, see ([\cite{P2}], Theorem 2.3(iv)). As some consequences in the proof are needed later on, we will prove this proposition in detail.

\begin{proof} First introduce the $\mathds{k}$-algebra$$B:=\{\bar u\in U_\chi(\mathfrak{g}_\mathds{k})|I_{\mathfrak{m}_\mathds{k}} \bar u\subseteq I_{\mathfrak{m}_\mathds{k}}\}=\{\bar u\in U_\chi(\mathfrak{g}_\mathds{k})|[I_{\mathfrak{m}_\mathds{k}},\bar u]\subseteq I_{\mathfrak{m}_\mathds{k}}\}.$$

The proposition can be proved in two steps:

(1) We claim that the mapping:
\begin{equation}\label{psi}
\begin{array}{lcll}\psi:&(\text{End}_{\mathfrak{g}_\mathds{k}}Q_{\chi}^\chi)^{\text{op}}&\rightarrow&B/I_{\mathfrak{m}_\mathds{k}}\\ &\phi&\mapsto&\phi(\bar1_\chi)
\end{array}
\end{equation}
is an isomorphism of $\mathds{k}$-algebras.

It is easy to verify that this mapping is well-defined, both injective and surjective, and keeps the $\mathbb{Z}_2$-graded structure. Along the same discussion as Theorem~\ref{W-C2}, we can obtain that $\psi$ is a homomorphism of $\mathds{k}$-algebras.

(2) The two mappings $$\psi_1: B/I_{\mathfrak{m}_\mathds{k}}\longrightarrow U_\chi(\mathfrak{g}_\mathds{k})^{\text{ad}\mathfrak{m}_\mathds{k}}/I_{\mathfrak{m}_\mathds{k}}\cap U_\chi(\mathfrak{g}_\mathds{k})^{\text{ad}\mathfrak{m}_\mathds{k}}$$ and $$\psi_2: U_\chi(\mathfrak{g}_\mathds{k})^{\text{ad}\mathfrak{m}_\mathds{k}}/I_{\mathfrak{m}_\mathds{k}}\cap U_\chi(\mathfrak{g}_\mathds{k})^{\text{ad}\mathfrak{m}_\mathds{k}}\longrightarrow (U_\chi(\mathfrak{g}_\mathds{k})/I_{\mathfrak{m}_\mathds{k}})^{\text{ad}\mathfrak{m}_\mathds{k}}=(Q_{\chi}^\chi)^{\text{ad}\mathfrak{m}_\mathds{k}}$$ are isomorphisms of $\mathds{k}$-algebras.

Since $Q_{\chi}^\chi$ is a free $U_\chi(\mathfrak{m}_\mathds{k})$-module under the action of $\text{ad}\mathfrak{m}_\mathds{k}$ by Lemma~\ref{free}, the short exact sequence of $\text{ad}\mathfrak{m}_\mathds{k}$-modules
$$0\longrightarrow I_{\mathfrak{m}_\mathds{k}}\longrightarrow U_\chi(\mathfrak{g}_\mathds{k})\longrightarrow Q_\chi^\chi\longrightarrow0$$splits. In other words, there is a $\mathbb{Z}_2$-graded subspace
$V\subseteq U_\chi(\mathfrak{g}_\mathds{k})$ such that $[\mathfrak{m}_\mathds{k},V]\subseteq V$ and that
\begin{equation}\label{VI}
U_\chi(\mathfrak{g}_\mathds{k})\cong V\oplus I_{\mathfrak{m}_\mathds{k}}.
\end{equation}
as $\text{ad}\mathfrak{m}_\mathds{k}$-modules. From the definition of $B$ we know that $I_{\mathfrak{m}_\mathds{k}}\subseteq B$, thus we have \begin{equation}\label{BVI}
B=V^{\text{ad}\mathfrak{m}_\mathds{k}}\oplus I_{\mathfrak{m}_\mathds{k}}.
\end{equation}
Then it follows from \eqref{VI} and \eqref{BVI} that
\begin{equation}\label{BI}
B/I_{\mathfrak{m}_\mathds{k}}\cong U_\chi(\mathfrak{g}_\mathds{k})^{\text{ad}\mathfrak{m}_\mathds{k}}/I_{\mathfrak{m}_\mathds{k}}\cap U_\chi(\mathfrak{g}_\mathds{k})^{\text{ad}\mathfrak{m}_\mathds{k}}\cong(U_\chi(\mathfrak{g}_\mathds{k})/I_{\mathfrak{m}_\mathds{k}})^{\text{ad}\mathfrak{m}_\mathds{k}}.
\end{equation}

Now we can deduce from \eqref{psi} and \eqref{BI} that Proposition~\ref{invariant} is true.
\end{proof}

\begin{rem}\label{a r}
We get another equivalent definition of reduced $W$-superalgebra $U_\chi(\mathfrak{g}_\mathds{k},e)$ with $p$-character $\chi$ by Proposition~\ref{invariant}. In fact, for any reduced $W$-superalgebra $U_\eta(\mathfrak{g}_\mathds{k},e)$ with $p$-character $\eta\in\chi+(\mathfrak{m}_\mathds{k}^\perp)_{\bar{0}}$, Proposition~\ref{invariant} still establishes (see Theorem~\ref{sumresult}(2)).
\end{rem}

\begin{rem}\label{Kazh}
For the algebras $U(\mathfrak{g}_\mathds{k}),~U(\mathfrak{\widetilde{p}}_\mathds{k}),~
Q_{\chi,\mathds{k}},~Q_\chi^\chi$ over positive characteristic field $\mathds{k}$, we can also define their Kazhdan filtration algebras and graded algebras in the same way as those in Section 3.2 when the characteristic $p$ is sufficient large.
\end{rem}

\subsection{The Morita equivalence theorem}

In this part we will introduce the Morita equivalence theorem between the reduced enveloping algebra of a basic classical Lie superalgebra and the reduced $W$-superalgebra. All these provide a new perspective toward the representation theory of Lie superalgebras.

First recall the following theorem formulated by Wang-Zhao in ([\cite{WZ}], Theorem 4.4),

\begin{theorem}$^{[\cite{WZ}]}$\label{matrix}
Set $\delta=\text{dim}~U_{\chi}(\mathfrak{m}_\mathds{k})$. Then $Q_\chi^\chi$ is a projective $U_{\chi}(\mathfrak{m}_\mathds{k})$-module and
$$U_\chi(\mathfrak{g}_\mathds{k})\cong\text{Mat}_\delta(U_\chi(\mathfrak{g}_\mathds{k},e)),$$
where $\text{Mat}_\delta(U_\chi(\mathfrak{g}_\mathds{k},e))$ denotes the matrix algebra of $U_\chi(\mathfrak{g}_\mathds{k},e)$.
\end{theorem}

It is notable this theorem not only establishes the foundation for Theorem~\ref{reducedfunctors}  where the Morita equivalence theorem between $\mathds{k}$-algebras $U_{\chi}(\mathfrak{g}_\mathds{k})$ and $U_\chi(\mathfrak{g}_\mathds{k},e)$ is introduced, but also provides an effective tool to settle the problem on the existence of the minimal dimensional representation in the Super Kac-Weisfeiler Property which we will deal with in the final section. First note that

\begin{lemma}\label{Endto}
There exists an isomorphism of $\mathds{k}$-algebras:
\[\begin{array}{lcll}\varphi:&(\text{End}_{(U_\chi(\mathfrak{g}_\mathds{k}),U_\chi(\mathfrak{m}_\mathds{k}))}
U_\chi(\mathfrak{g}_\mathds{k}))^{\text{op}}&\longrightarrow&U_\chi(\mathfrak{g}_\mathds{k})^{\text{ad}\mathfrak{m}_\mathds{k}}\\ &\theta&\mapsto&\theta(\bar 1)
\end{array}\]
where $(\text{End}_{(U_\chi(\mathfrak{g}_\mathds{k}),U_\chi(\mathfrak{m}_\mathds{k}))}
U_\chi(\mathfrak{g}_\mathds{k}))^{\text{op}}$ denotes the opposite algebra of the endomorphism algebra of $(U_\chi(\mathfrak{g}_\mathds{k}),U_\chi(\mathfrak{m}_\mathds{k}))$-bimodule $U_\chi(\mathfrak{g}_\mathds{k})$.
\end{lemma}

\begin{proof} We claim that $\varphi$ is well-defined. Since $(\text{End}_{(U_\chi(\mathfrak{g}_\mathds{k}),U_\chi(\mathfrak{m}_\mathds{k}))}
U_\chi(\mathfrak{g}_\mathds{k}))^{\text{op}}$ is the opposite algebra of the endomorphism algebra of $(U_\chi(\mathfrak{g}_\mathds{k}),U_\chi(\mathfrak{m}_\mathds{k}))$-bimodule $U_\chi(\mathfrak{g}_\mathds{k})$, then
$$\theta(m)=\theta( m.\bar 1)=(-1)^{|\theta||m|}m\theta(\bar 1),\quad\theta( m)= \theta(\bar 1.m)= \theta(\bar 1)m$$for any homogeneous elements $m\in\mathfrak{m}_\mathds{k}$ and $\theta\in U_\chi(\mathfrak{g}_\mathds{k})$.
Hence $[m,\theta(\bar 1)]=m\theta(\bar 1)-(-1)^{|\theta||m|}\theta(\bar 1)m=0$, i.e. $\theta(\bar 1)\in U_\chi(\mathfrak{g}_\mathds{k})^{\text{ad}\mathfrak{m}_\mathds{k}}$. Then $\varphi$ is well-defined.

It is easy to verify that the even mapping $\varphi$ is both injective and surjective, and keeps the $\mathbb{Z}_2$-graded structure by the same discussion as Theorem~\ref{W-C2}. Hence $\varphi$ is an isomorphism of $\mathds{k}$-algebras.
\end{proof}

Given a $\mathds{k}$-algebra $\mathscr{A}$ we denote by $\mathscr{A}$-mod the category of all finite-dimensional left $\mathscr{A}$-modules. Given a left $U_{\chi}(\mathfrak{g}_\mathds{k})$-module $M$ define
$$M^{\mathfrak{m}_\mathds{k}}:=\{v\in M|I_{\mathfrak{m}_\mathds{k}}.v=0\}.$$

It follows from the proof of Proposition~\ref{invariant} that $U_{\chi}(\mathfrak{g}_\mathds{k},e)$ can be identified with $U_\chi(\mathfrak{g}_\mathds{k})^{\text{ad}\mathfrak{m}_\mathds{k}}/ U_\chi(\mathfrak{g}_\mathds{k})^{\text{ad}\mathfrak{m}_\mathds{k}}\cap U_\chi(\mathfrak{g}_\mathds{k})N_{\mathfrak{m}_\mathds{k}}$. Therefore, any left $U_\chi(\mathfrak{g}_\mathds{k})^{\text{ad}\mathfrak{m}_\mathds{k}}$-module can be considered as a $U_{\chi}(\mathfrak{g}_\mathds{k},e)$-module with the trivial action of the ideal $U_\chi(\mathfrak{g}_\mathds{k})^{\text{ad}\mathfrak{m}_\mathds{k}}\cap U_\chi(\mathfrak{g}_\mathds{k})N_{\mathfrak{m}_\mathds{k}}$.

\begin{theorem}\label{reducedfunctors}
The functors$$U_\chi(\mathfrak{g}_\mathds{k})\text{-mod}\longrightarrow U_{\chi}(\mathfrak{g}_\mathds{k},e)\text{-mod},\qquad M\mapsto M^{\mathfrak{m}_\mathds{k}}$$and $$U_{\chi}(\mathfrak{g}_\mathds{k},e)\text{-mod}\longrightarrow U_\chi(\mathfrak{g}_\mathds{k})\text{-mod},\qquad V\mapsto U_\chi(\mathfrak{g}_\mathds{k})\otimes_{U_\chi(\mathfrak{g}_\mathds{k})^
{\text{ad}\mathfrak{m}_\mathds{k}}}V$$
are mutually inverse category equivalences.
\end{theorem}

\begin{proof}
It follows from ([\cite{WZ}], Proposition 4.2) that every $U_\chi(\mathfrak{g}_\mathds{k})$-module is $U_\chi(\mathfrak{m}_\mathds{k})$-free under the action of left-multiplication. The theorem can be proved in the same way as ([\cite{P2}], Theorem 2.4) for the Lie algebra case after substituting the discussion in ([\cite{P2}], Section 2.2) for ([\cite{WZ}], Proposition 4.2), thus will be omitted here.
\end{proof}

\section{The structure of reduced $W$-superalgebras in positive characteristic}

Following Premet's treatment of finite $W$-algebras in ([\cite{P2}], Section 3), in this part we will study the construction theory of reduced $W$-superalgebra $U_\chi(\mathfrak{g}_\mathds{k},e)$ associated to the basic classical Lie superalgebra $\mathfrak{g}_\mathds{k}$ over positive characteristic field $\mathds{k}=\overline{\mathbb{F}}_p$.

Recall in Section 4.1 the elements in the Lie superalgebra $\mathfrak{g}_\mathds{k}$ is obtained by ``modular $p$ reduction'' from the ones in the $A$-algebra $\mathfrak{g}_A$, and we denote by
$\bar{x}=x\otimes1\in\mathfrak{g}_\mathds{k}$ for each $x\in\mathfrak{g}_A$.

By the discussion preceding Definition~\ref{reduced W} we know that there is an isomorphism of $\mathfrak{g}_\mathds{k}$-modules $Q_\chi^\chi\cong U_\chi(\mathfrak{g}_\mathds{k})\otimes_{U_\chi(\mathfrak{m}_\mathds{k})}\bar 1_\chi$. It follows from the PBW theorem that $U_\chi(\mathfrak{\widetilde{p}}_\mathds{k})$ and $Q_\chi^\chi$ are isomorphism as $\mathds{k}$-vector spaces. Therefore, the basis of $\mathfrak{\widetilde{p}_\mathds{k}}$ can be considered as a basis of $Q_\chi^\chi$ and this will cause no confusion. For the sake of clarity, we will relist the basis of $\mathfrak{\widetilde{p}_\mathds{k}}$ as follows:
\[\begin{array}{ll}
\bar x_1,\cdots,\bar x_l\in(\mathfrak{g}^e_\mathds{k})_{\bar{0}},&\qquad \bar x_{l+1},\cdots,\bar x_m\in\bigoplus\limits_{j\geqslant 2}[f,(\mathfrak{g}_\mathds{k}(j))_{\bar0}];\\
\bar y_1,\cdots,\bar y_q\in(\mathfrak{g}^e_\mathds{k})_{\bar{1}},&\qquad \bar y_{q+1},\cdots,\bar y_n\in\bigoplus\limits_{j\geqslant 2}[f,(\mathfrak{g}_\mathds{k}(j))_{\bar1}];\\
\bar u_1,\cdots,\bar u_s\in\mathfrak{g}_\mathds{k}(-1)_{\bar{0}}\cap(\mathfrak{g}_\mathds{k}(-1)_{\bar{0}}^{\prime})^{\bot},&\qquad \bar u_{s+1},\cdots,\bar u_{2s}\in\mathfrak{g}_\mathds{k}(-1)_{\bar{0}}';\\
\bar v_1,\cdots,\bar v_t\in\mathfrak{g}_\mathds{k}(-1)_{\bar{1}}\cap(\mathfrak{g}_\mathds{k}(-1)_{\bar{1}}^{\prime})^{\bot},&\qquad \bar v_{t+1},\cdots,\bar v_{r}\in\mathfrak{g}_\mathds{k}(-1)_{\bar{1}}'
\end{array}
\]
where $t=\lfloor\frac{r}{2}\rfloor$ and ${\bot}$ is respect to the bilinear $\langle\cdot,\cdot\rangle$.

Given an element $\bar x\in\mathfrak{g}_\mathds{k}(i)$, we denote whose weight (with the action of $\text{ad}h$) by $\text{wt}(\bar x)=i$. For $k\in\mathbb{Z}_+$, define$$\Lambda_k:=\{(i_1,\cdots,i_k)|i_j\in\mathbb{Z}_+,~0\leqslant  i_j\leqslant  p-1\},~\Lambda'_k:=\{(i_1,\cdots,i_k)|i_j\in\{0,1\}\}$$ with $1\leqslant j\leqslant k$.
Set $\mathbf{e}_i=(\delta_{i1},\cdots,\delta_{ik})$. For $\mathbf{i}=(i_1,\cdots,i_k)$ in $\Lambda_k\,$ or $\Lambda'_k$, set $|\mathbf{i}|=i_1+\cdots+i_k$.

Given  $\mathbf{a}=(a_1,\cdots,a_m)\in\Lambda_m,~\mathbf{b}=(b_1,\cdots,b_n)\in\Lambda'_n,~\mathbf{c}=(c_1,\cdots,c_s)\in\Lambda_s,~\mathbf{d}=(d_1,\cdots,d_t)\in\Lambda'_t$ (recall that $t=\lfloor\frac{r}{2}\rfloor$), define$$\bar x^{\mathbf{a}}\bar y^\mathbf{b}\bar u^\mathbf{c}\bar v^\mathbf{d}:=\bar x_1^{a_1}\cdots \bar x_m^{a_m}\bar y_1^{b_1}\cdots \bar y_n^{b_n}\bar u_1^{c_1}\cdots \bar u_s^{c_s}\bar v_1^{d_1}\cdots \bar v_t^{d_t}.$$It is obvious that the $\mathds{k}$-span of monomials $\bar x^{\mathbf{a}}\bar y^\mathbf{b}\bar u^\mathbf{c}\bar v^\mathbf{d}\otimes\bar 1_\chi$ form a basis of $Q_\chi^\chi$.

We can assume that the basis of $\widetilde{\mathfrak{p}}_\mathds{k}$ is homogeneous under the action of ad$h$, i.e. $\bar x_1\in\mathfrak{g}_\mathds{k}(k_1)_{\bar{0}},~\cdots,~\bar x_m\in\mathfrak{g}_\mathds{k}(k_m)_{\bar{0}},~\bar y_1\in\mathfrak{g}_\mathds{k}(k'_1)_{\bar{1}}, \cdots,~\bar y_n\in\mathfrak{g}_\mathds{k}(k'_n)_{\bar{1}}$. Define
$$|(\mathbf{a},\mathbf{b},\mathbf{c},\mathbf{d})|_e=\sum_{i=1}^ma_i(k_i+2)+\sum_{i=1}^nb_i(k'_i+2)+\sum_{i=1}^sc_i+\sum_{i=1}^td_i.$$
Say that $\bar x^{\mathbf{a}}\bar y^\mathbf{b}\bar u^\mathbf{c}\bar v^\mathbf{d}$ has $e$-degree $|(\mathbf{a},\mathbf{b},\mathbf{c},\mathbf{d})|_e$ and write $\text{deg}_e(\bar x^{\mathbf{a}}\bar y^\mathbf{b}\bar u^\mathbf{c}\bar v^\mathbf{d})=|(\mathbf{a},\mathbf{b},\mathbf{c},\mathbf{d})|_e$. It is notable that the $e$-degree defined above is the same as Kazhdan degree in Section 3.2. Note that
\begin{equation}\label{dege}
\text{deg}_e(\bar x^{\mathbf{a}}\bar y^\mathbf{b}\bar u^\mathbf{c}\bar v^\mathbf{d})=\text{wt}(\bar x^{\mathbf{a}}\bar y^\mathbf{b}\bar u^\mathbf{c}\bar v^\mathbf{d})+2\text{deg}(\bar x^{\mathbf{a}}\bar y^\mathbf{b}\bar u^\mathbf{c}\bar v^\mathbf{d}),
\end{equation}
where $\text{wt}(\bar x^{\mathbf{a}}\bar y^\mathbf{b}\bar u^\mathbf{c}\bar v^\mathbf{d})=(\sum\limits_{i=1}^mk_ia_i)+(\sum\limits_{i=1}^nk'_ib_i)-|\mathbf{c}|-|\mathbf{d}|$\, and $\text{deg}(\bar x^{\mathbf{a}}\bar y^\mathbf{b}\bar u^\mathbf{c}\bar v^\mathbf{d})=|\mathbf{a}|+|\mathbf{b}|+|\mathbf{c}|+|\mathbf{d}|$ are the weight and the standard degree of $\bar x^{\mathbf{a}}\bar y^\mathbf{b}\bar u^\mathbf{c}\bar v^\mathbf{d}$, respectively.

\subsection{Some Lemmas}
Some Lemmas will be formulated in this part, which play the key role in the study of the construction theory of reduced $W$-superalgebra $U_\chi(\mathfrak{g}_\mathds{k},e)$. Firstly, some commutative relations for the elements in the basis of $U_\chi(\mathfrak{g}_\mathds{k})$ are introduced in Lemma~\ref{commutative relations k1} and Lemma~\ref{commutative relations k2}.

\begin{lemma}\label{commutative relations k1}
Let $\bar w\in U_\chi(\mathfrak{g}_\mathds{k})_{i}$ ($i\in\mathbb{Z}_2$) be a $\mathbb{Z}_2$-homogeneous element, then we have
$$\bar w\cdot \bar x^{\mathbf{a}}\bar y^\mathbf{b}\bar u^\mathbf{c}\bar v^\mathbf{d}=\sum_{\mathbf{i}\in\Lambda_{m}}\sum_{j_1=0}^{b_1}\cdots\sum_{j_n=0}^{b_n}\left(\begin{array}{@{\hspace{0pt}}c@{\hspace{0pt}}} \mathbf{a}\\ \mathbf{i}\end{array}\right)\bar x^{\mathbf{a}-\mathbf{i}}\bar y^{\mathbf{b}-\mathbf{j}}\cdot[\bar w\bar x^{\mathbf{i}}\bar y^{\mathbf{j}}]\cdot \bar u^\mathbf{c}\bar v^\mathbf{d},$$
where $\mathbf{a}\choose\mathbf{i}$$=\prod\limits_{l'=1}^m$$a_{l'}\choose i_{l'}$ and $$[\bar w\bar x^{\mathbf{i}}\bar y^{\mathbf{j}}]=k_{1,b_1,j_1}\cdots k_{n,b_n,j_n}(-1)^{|\mathbf{i}|}(\text{ad}\bar y_n)^{j_n}\cdots(\text{ad}\bar y_1)^{j_1}(\text{ad}\bar x_m)^{i_m}\cdots(\text{ad}\bar x_1)^{i_1}(\bar w),$$
in which the coefficients $k_{1,b_1,j_1},\cdots,k_{n,b_n,j_n}\in \mathds{k}$ (recall that $\mathbf{b}=(b_1,\cdots,b_n)\in\Lambda'_n$) and the indices $j_1,\cdots,j_n\in\{0,1\}$. If we write $j_0=0$, then
$$k_{t',0,0}=1, k_{t',0,1}=0, k_{t',1,0}=(-1)^{|\bar w|+j_1+\cdots+j_{{t'}-1}}, k_{t',1,1}=(-1)^{|\bar w|+1+j_1+\cdots+j_{{t'}-1}},$$
where $1\leqslant {t'}\leqslant n$.
\end{lemma}

\begin{proof}
Let $R_{\bar y_j}^{i}$ denote the $i$-th right multiplication by $\bar y_j(1\leqslant j\leqslant n)$, i.e. $R_{\bar y_j}^{i}(\bar u)=\bar u\bar y_j^{i}$ for any $\bar u\in U_\chi(\mathfrak{g}_\mathds{k})$. The Lemma can be proved by induction.

Let $\bar w$ be any $\mathbb{Z}_2$-homogeneous element in $U_\chi(\mathfrak{g}_\mathds{k})$ and denote its $\mathbb{Z}_2$-degree by $|\bar w|$. Recall that all the $\bar y_i's~(1\leqslant i\leqslant n)$ are in $(\mathfrak{g}_\mathds{k})_{\bar1}$. For each $0\leqslant  s'\leqslant  n-1$, since
\[
\begin{array}{ll}
&[\bar y_{s'+1},(\text{ad}\bar y_{s'})^{k_{s'}}(\text{ad}\bar y_{s'-1})^{k_{s'-1}}\cdots(\text{ad}\bar y_{1})^{k_1}(\bar w)]\\
=&\bar y_{s'+1}(\text{ad}\bar y_{s'})^{k_{s'}}(\text{ad}\bar y_{s'-1})^{k_{s'-1}}\cdots(\text{ad}\bar y_{1})^{k_1}(\bar w)-(-1)^{|\bar w|+k_1+\cdots+k_{s'}}\\
&(\text{ad}\bar y_{s'})^{k_{s'}}(\text{ad}\bar y_{s'-1})^{k_{s'-1}}\cdots(\text{ad}\bar y_{1})^{k_1}(\bar w)\bar y_{s'+1},
\end{array}
\]
then
\[
\begin{array}{ll}
&R_{\bar y_{s'+1}}((\text{ad}\bar y_{s'})^{k_{s'}}(\text{ad}\bar y_{s'-1})^{k_{s'-1}}\cdots(\text{ad}\bar y_{1})^{k_1}(\bar w))\\
=&(-1)^{|\bar w|+1+k_1+\cdots+k_{s'}}\bar y_{s'+1}^0(\text{ad}\bar y_{s'+1})^1(\text{ad}\bar y_{s'})^{k_{s'}}(\text{ad}\bar y_{s'-1})^{k_{s'-1}}\cdots(\text{ad}\bar y_{1})^{k_1}(\bar w)\\
+&(-1)^{|\bar w|+k_1+\cdots+k_{s'}}\bar y_{s'+1}^1(\text{ad}\bar y_{s'+1})^0(\text{ad}\bar y_{s'})^{k_{s'}}(\text{ad}\bar y_{s'-1})^{k_{s'-1}}\cdots(\text{ad}\bar y_{1})^{k_1}(\bar w).
\end{array}
\]

For any monomial $\bar x^{\mathbf{a}}\bar y^\mathbf{b}\bar u^\mathbf{c}\bar v^\mathbf{d}$ in the basis of $U_\chi(\mathfrak{g}_\mathds{k})$, recall that all the indices of the odd elements in $\mathfrak{g}_\mathds{k}$ (i.e. the indices of $\bar y_i$'s and $\bar v_i$'s) are in the set $\{0,1\}$ by the PBW theorem. Let $0\leqslant j_1,\cdots,j_n\leqslant 1$ be positive integers, and define$$k_{j_i,0,0}:=1,~k_{j_i,0,1}:=0, k_{j_i,1,0}:=(-1)^{|w|+k_1+\cdots+k_{j_i-1}},~k_{j_i,1,1}:=(-1)^{|w|+1+ k_1+\cdots+k_{j_i-1}}$$ for $1\leqslant i\leqslant n$, and $j_0$ is interpreted as $0$.  Then we have
\begin{equation}\label{wdot}
\begin{split}
\bar w\cdot \bar y_1^{j_1}\cdots \bar y_n^{j_n}=&R_{\bar y_1}^{j_1}(\bar w)\cdot \bar y_2^{j_2}\cdots \bar y_n^{j_n}\\
=&(\sum\limits_{i_1=0}^{j_1}k_{1,j_1,i_1}\bar y_1^{j_1-i_1}(\text{ad}\bar y_1)^{i_1}(\bar w))\cdot \bar y_2^{j_2}\cdots \bar y_n^{j_n}\\
=&R_{\bar y_2}^{j_2}(\sum\limits_{i_1=0}^{j_1}k_{1,j_1,i_1}\bar y_1^{j_1-i_1}(\text{ad}\bar y_1)^{i_1}(\bar w))\cdot \bar y_3^{j_3}\cdots \bar y_n^{j_n}\\
=&(\sum\limits_{i_1=0}^{j_1}\sum\limits_{i_2=0}^{j_2}k_{1,j_1,i_1}k_{2,j_2,i_2}\bar y_1^{j_1-i_1}\bar y_2^{j_2-i_2}(\text{ad}\bar y_2)^{i_2}(\text{ad}\bar y_1)^{i_1}(\bar w))\cdot\bar  y_3^{j_3}\\&\cdots\bar  y_n^{j_n}\\=&\cdots\cdots\\
=&\sum\limits_{i_1=0}^{j_1}\sum\limits_{i_2=0}^{j_2}\cdots\sum\limits_{i_n=0}^{j_n}k_{1,j_1,i_1}k_{2,j_2,i_2}\cdots k_{n,j_n,i_n}\bar y_1^{j_1-i_1}\bar y_2^{j_2-i_2}\cdots\bar  y_n^{j_n-i_n}\\
&(\text{ad}\bar y_n)^{i_n}\cdot(\text{ad}\bar y_{n-1})^{i_{n-1}}\cdots(\text{ad}\bar y_{1})^{i_{1}}(\bar w)
\end{split}
\end{equation}
by induction.
For any $\mathbb{Z}_2$-homogeneous elements $\bar u,\bar v$ in $\mathfrak{g}_\mathds{k}$, we have that $\bar u\bar v=[\bar u,\bar v]+\bar v\bar u$ if at least one of them is in $(\mathfrak{g}_\mathds{k})_{\bar0}$, i.e. the commutative operation between $\bar u$ and $\bar v$ is the same as the Lie algebra case. Since all the $\bar x_i's$ for $1\leqslant i \leqslant m$ are even elements in $\mathfrak{g}_\mathds{k}$, then
\begin{equation}\label{wdot-1}
\bar w\cdot \bar x_1^{a_1}\cdots \bar x_m^{a_m}=\sum_{\mathbf{i}\in\Lambda_{m}}(-1)^{|\mathbf{i}|}\left(\begin{array}{@{\hspace{0pt}}c@{\hspace{0pt}}} \mathbf{a}\\ \mathbf{i}\end{array}\right)\bar x^{\mathbf{a}-\mathbf{i}}\cdot(\text{ad}\bar x_m)^{i_m}\cdots (\text{ad}\bar x_1)^{i_1}(\bar w)
\end{equation}
by ([\cite{P2}], Section 3.1(2)), where $\mathbf{a}\choose\mathbf{i}$$=\prod\limits_{l'=1}^m$$a_{l'}\choose i_{l'}$.

Write $[\bar w\bar x^{\mathbf{i}}]=(-1)^{|\mathbf{i}|}(\text{ad}\bar x_m)^{i_m}\cdots (\text{ad}\bar x_1)^{i_1}(\bar w)$.
Since all the elements $\bar x_1,\cdots,\bar x_m$ are even, $[\bar w\bar x^{\mathbf{i}}]$ is also a $\mathbb{Z}_2$-homogeneous element with the same parity as $\bar w$. It can be inferred from \eqref{wdot} and \eqref{wdot-1} that
\begin{equation}\label{wxyuv}
\begin{split}
\bar w\cdot \bar x^{\mathbf{a}}\bar y^\mathbf{b}\bar u^\mathbf{c}\bar v^\mathbf{d}=&\sum\limits_{\mathbf{i}\in\Lambda_{m}}\left(\begin{array}{@{\hspace{0pt}}c@{\hspace{0pt}}} \mathbf{a}\\ \mathbf{i}\end{array}\right)\bar x^{\mathbf{a}-\mathbf{i}}\cdot[\bar w\bar x^{\mathbf{i}}]\cdot
\bar y^{\mathbf{b}}\cdot \bar u^\mathbf{c}\bar v^\mathbf{d}\\
=&\sum\limits_{\mathbf{i}\in\Lambda_{m}}\sum\limits_{j_1=0}^{b_1}\cdots\sum\limits_{j_n=0}^{b_n}
\left(\begin{array}{@{\hspace{0pt}}c@{\hspace{0pt}}} \mathbf{a}\\ \mathbf{i}\end{array}\right)\bar x^{\mathbf{a}-\mathbf{i}}k_{1,b_1,j_1}\cdots k_{n,b_n,j_n}\bar y_1^{b_1-j_1}\bar y_2^{b_2-j_2}\cdots\\
&\bar y_n^{b_n-j_n}(\text{ad}\bar y_n)^{j_n}(\text{ad}\bar y_{n-1})^{j_{n-1}}\cdots(\text{ad}\bar y_{1})^{j_{1}}([\bar w\bar x^{\mathbf{i}}])\cdot\bar  u^\mathbf{c}\bar v^\mathbf{d}\\
=&\sum\limits_{\mathbf{i}\in\Lambda_{m}}\sum\limits_{j_1=0}^{b_1}\cdots\sum\limits_{j_n=0}^{b_n}(-1)^{|\mathbf{i}|}\left(\begin{array}{@{\hspace{0pt}}c@{\hspace{0pt}}} \mathbf{a}\\ \mathbf{i}\end{array}\right)k_{1,b_1,j_1}\cdots k_{n,b_n,j_n}\bar x^{\mathbf{a}-\mathbf{i}}\bar y^{\mathbf{b}-\mathbf{j}}\\
&(\text{ad}\bar y_n)^{j_n}\cdots(\text{ad}\bar y_{1})^{j_{1}}(\text{ad}\bar x_m)^{i_m}\cdots(\text{ad}\bar x_1)^{i_1}(\bar w)\cdot \bar u^\mathbf{c}\bar v^\mathbf{d},
\end{split}
\end{equation}
where the coefficients $k_{1,b_1,j_1},\cdots, k_{n,b_n,j_n}\in\mathds{k}$ in \eqref{wxyuv} are defined by:
$$k_{t',0,0}=1, k_{t',0,1}=0, k_{t',1,0}=(-1)^{|\bar w|+j_1+\cdots+j_{{t'}-1}}, k_{t',1,1}=(-1)^{|\bar w|+1+j_1+\cdots+j_{{t'}-1}}$$for  $1\leqslant {t'}\leqslant n$, and $j_0$ is interpreted as $0$.

If we set$$[\bar w\bar x^{\mathbf{i}}\bar y^{\mathbf{j}}]=k_{1,b_1,j_1}\cdots k_{n,b_n,j_n}(-1)^{|\mathbf{i}|}(\text{ad}\bar y_n)^{j_n}\cdots(\text{ad}\bar y_1)^{j_1}(\text{ad}\bar x_m)^{i_m}\cdots(\text{ad}\bar x_1)^{i_1}(\bar w),$$
then \eqref{wxyuv} can be written as
$$\bar w\cdot \bar x^{\mathbf{a}}\bar y^\mathbf{b}\bar u^\mathbf{c}\bar v^\mathbf{d}=\sum_{\mathbf{i}\in\Lambda_{m}}\sum_{j_1=0}^{b_1}\cdots\sum_{j_n=0}^{b_n}\left(\begin{array}{@{\hspace{0pt}}c@{\hspace{0pt}}} \mathbf{a}\\ \mathbf{i}\end{array}\right)\bar x^{\mathbf{a}-\mathbf{i}}\bar y^{\mathbf{b}-\mathbf{j}}\cdot[\bar w\bar x^{\mathbf{i}}\bar y^{\mathbf{j}}]\cdot \bar u^\mathbf{c}\bar v^\mathbf{d}.$$
\end{proof}

Let $\rho_\chi$ denote the natural representation of $U_\chi(\mathfrak{g}_\mathds{k})$ in $\text{End}_\mathds{k}Q_\chi^\chi$. We can get the following result:

\begin{lemma}\label{commutative relations k2}
Let $(\mathbf{a},\mathbf{b},\mathbf{c},\mathbf{d}),~(\mathbf{a}',\mathbf{b}',\mathbf{c}',\mathbf{d}')\in\Lambda_m\times
\Lambda'_n\times\Lambda_s\times\Lambda'_t$ be such that $|(\mathbf{a},\mathbf{b},\mathbf{c},\mathbf{d})|_e=A,~|(\mathbf{a}',\mathbf{b}',\mathbf{c}',\mathbf{d}')|_e=B$, then
\[\begin{array}{ccl}(\rho_\chi(\bar x^{\mathbf{a}}\bar y^\mathbf{b}\bar u^\mathbf{c}\bar v^\mathbf{d}))(\bar x^{\mathbf{a}'}\bar y^{\mathbf{b}'}\bar u^{\mathbf{c}'}\bar v^{\mathbf{d}'}\otimes\bar 1_\chi)&=&
(K\bar x^{\mathbf{a}+\mathbf{a}'}\bar y^{\mathbf{b}+\mathbf{b}'}\bar u^{\mathbf{c}+\mathbf{c}'}\bar v^{\mathbf{d}+\mathbf{d}'}+\text{terms of}~e\text{-degree}\\ &&\leqslant A+B-2)\otimes\bar 1_\chi,
\end{array}\]
where the coefficient $K\in\mathds{k}$ is defined by:

(1) if $(\mathbf{a}+\mathbf{a}',\mathbf{b}+\mathbf{b}',\mathbf{c}+\mathbf{c}',\mathbf{d}+\mathbf{d}')\notin\Lambda_m\times\Lambda'_n\times\Lambda_s\times\Lambda'_t$, then $K=0$.

(2) if $K\neq0$, then each entry in $(\mathbf{b}+\mathbf{b}',\mathbf{d}+\mathbf{d}')$ is taken from the set $\{0,1\}$. Delete all the zero terms in $(\mathbf{b},\mathbf{d},\mathbf{b}',\mathbf{d}')$, then we can get a new sequence $(b_1,b'_1,b_2,b'_2,\cdots,$\\$b_n,b'_n,d_1,d'_1,d_2,d'_2\cdots,d_t,d'_t)$ from the old one by transpositions. Let $\tau(\mathbf{b},\mathbf{d},\mathbf{b}',\mathbf{d}')$ denote the times for which we do the transpositions, then $K=(-1)^{\tau(\mathbf{b},\mathbf{d},\mathbf{b}',\mathbf{d}')}$.
\end{lemma}

\begin{proof}
(1) First suppose that $(\mathbf{a},\mathbf{b},\mathbf{c})=\mathbf{0}$ and $|\mathbf{d}|=1$, so that $A=1$. Then $\bar v^\mathbf{d}=\bar v_S$ for some $1\leqslant  S\leqslant  t$. Applying Lemma~\ref{commutative relations k1} one obtains
\begin{equation}\label{vs}
\begin{split}
(\rho_\chi(\bar v_S))(\bar x^{\mathbf{a}'}\bar y^{\mathbf{b}'}\bar u^{\mathbf{c}'}\bar v^{\mathbf{d}'}\otimes\bar 1_\chi)
=&((-1)^{|\mathbf{b}'|}\bar x^{\mathbf{a}'}\bar y^{\mathbf{b}'}
\cdot\rho_\chi(\bar v_S)\bar u^{\mathbf{c}'}\bar v^{\mathbf{d}'}+\sum\limits_{(\mathbf{i},\mathbf{j})\neq\mathbf{0}}\alpha_{\mathbf{i}\mathbf{j}}
\bar x^{\mathbf{a}'-\mathbf{i}}\\
&\cdot \bar y^{\mathbf{b}'-\mathbf{j}}\cdot\rho_\chi([\bar v_S\bar x^{\mathbf{i}}\bar y^{\mathbf{j}}])\cdot \bar u^{\mathbf{c}'}\bar v^{\mathbf{d}'})\otimes\bar 1_\chi
\end{split}
\end{equation}
for some $\alpha_{\mathbf{i}\mathbf{j}}\in\mathds{k}$. Since $\rho_\chi(\mathfrak{m}_\mathds{k})$ stabilises the line $\mathds{k}\bar 1_\chi$, the first summand on the right equals $(-1)^{|\mathbf{b}'|+\sum\limits_{l=1}^{S-1}d'_l}\bar x^{\mathbf{a}'}\bar y^{\mathbf{b}'}\bar u^{\mathbf{c}'}\bar v^{\mathbf{d}'+\mathbf{e}_S}\otimes\bar 1_\chi$ (where $\mathbf{e}_{S-1}=\mathbf{e}_0$ is interpreted as $0$ for the case $S=1$) modulo terms of lower $e$-degree in \eqref{vs} (if $d'_S+1=2$, then $\bar v^{\mathbf{d}'+\mathbf{e}_S}$ is interpreted as $0$). For the second summand on the right of \eqref{vs}, we have:

(i) suppose $(\mathbf{i},\mathbf{j})\neq\mathbf{0}$ is such that $\text{wt}([\bar v_S\bar x^{\mathbf{i}}\bar y^{\mathbf{j}}])\leqslant -1$. Then $\rho_\chi([\bar v_S\bar x^{\mathbf{i}}\bar y^{\mathbf{j}}])\bar u^{\mathbf{c}'}\bar v^{\mathbf{d}'}\otimes\bar 1_\chi$ is a linear combination of $\bar u^\mathbf{f}\bar v^\mathbf{g}$ with $|\mathbf{f}|+|\mathbf{g}|\leqslant |\mathbf{c}'|+|\mathbf{d}'|+1$ as $\rho_\chi(\mathfrak{m_\mathds{k}})$ stabilises the line $\mathds{k}\bar1_\chi$. As a consequence, $\bar x^{\mathbf{a}'-\mathbf{i}}\bar y^{\mathbf{b}'-\mathbf{j}}\cdot\rho_\chi([\bar v_S\bar x^{\mathbf{i}}y^{\mathbf{j}}])\bar u^{\mathbf{c}'}\bar v^{\mathbf{d}'}$ is a linear combination of $
\bar x^{\mathbf{a}'-\mathbf{i}}\bar y^{\mathbf{b}'-\mathbf{j}}\bar u^\mathbf{f}\bar v^\mathbf{g}$. Since $(\mathbf{i},\mathbf{j})\neq\mathbf{0}$, and the weights $k_{s'}$ and $k'_{t'}$ of elements $\bar x_{s'}$ and $\bar y_{t'}$ for each $1\leqslant s'\leqslant m$ and $1\leqslant t'\leqslant n$ are all non-negative integers, then
\[\begin{array}{ccl}
\text{deg}_e(\bar x^{\mathbf{a}'-\mathbf{i}}\bar y^{\mathbf{b}'-\mathbf{j}}\bar u^\mathbf{f}\bar v^\mathbf{g})&=&\sum\limits_{s'=1}^m(a'_{s'}-i'_{s'})(k_{s'}+2)+\sum\limits_{t'=1}^n(b'_{t'}-j'_{t'})(k'_{t'}+2)
+|\mathbf{f}|+|\mathbf{g}|\\
&\leqslant& \sum\limits_{s'=1}^ma'_{s'}(k_{s'}+2)+\sum\limits_{t'=1}^nb'_{t'}(k'_{t'}+2)+(|\mathbf{f}|+|\mathbf{g}|-2|\mathbf{i}|-2|\mathbf{j}|)
\\&\leqslant& \sum\limits_{s'=1}^ma'_{s'}(k_{s'}+2)+\sum\limits_{t'=1}^nb'_{t'}(k'_{t'}+2)+(|\mathbf{c}'|+|\mathbf{d}'|+1-2|\mathbf{i}|-2|\mathbf{j}|)
\\&\leqslant &A+B-2.
\end{array}\]

(ii) suppose $(\mathbf{i},\mathbf{j})\neq\mathbf{0}$ is such that $\text{wt}([\bar v_S\bar x^{\mathbf{i}}y^{\mathbf{j}}])\geqslant 0$. Since $\mathfrak{g}_\mathds{k}=\bigoplus\limits_{i\in\mathbb{Z}}\mathfrak{g}_\mathds{k}(i)$ is the Dynkin grading of $\mathfrak{g}_\mathds{k}$, the image of $\mathfrak{p}_\mathds{k}$ is still in $\mathfrak{p}_\mathds{k}$ under the action of $\text{ad}h$. This implies that $\bar x^{\mathbf{a}'-\mathbf{i}}\bar y^{\mathbf{b}'-\mathbf{j}}\cdot[\bar v_S\bar x^{\mathbf{i}}\bar y^{\mathbf{j}}]$ is a linear combination of $\bar x^\mathbf{f}\bar y^\mathbf{g}$ with $\text{wt}(\bar x^\mathbf{f}\bar y^\mathbf{g})=\text{wt}(\bar x^{\mathbf{a}'}\bar y^{\mathbf{b}'})-1$ and $|\mathbf{f}|+|\mathbf{g}|\leqslant
|\mathbf{a}'|+|\mathbf{b}'|-|\mathbf{i}|-|\mathbf{j}|+1$. Therefore, $\bar x^{\mathbf{a}'-\mathbf{i}}\bar y^{\mathbf{b}'-\mathbf{j}}\cdot\rho_\chi([\bar v_S\bar x^{\mathbf{i}}\bar y^{\mathbf{j}}])\bar u^{\mathbf{c}'}\bar v^{\mathbf{d}'}$ is a linear combination of $\bar x^{\mathbf{f}}\bar y^{\mathbf{g}}\bar u^{\mathbf{c}'}\bar v^{\mathbf{d}'}$ with
\[\begin{array}{ccl}
\text{deg}_e(\bar x^{\mathbf{f}}\bar y^{\mathbf{g}}\bar u^{\mathbf{c}'}\bar v^{\mathbf{d}'})&=&\text{wt}(\bar x^{\mathbf{f}}\bar y^{\mathbf{g}}\bar u^{\mathbf{c}'}\bar v^{\mathbf{d}'})+
2\text{deg}(\bar x^{\mathbf{f}}\bar y^{\mathbf{g}}\bar u^{\mathbf{c}'}\bar v^{\mathbf{d}'})\\
&=&\text{wt}(\bar x^{\mathbf{f}}\bar y^{\mathbf{g}})-(|\mathbf{c}'|+|\mathbf{d}'|)
+2(|\mathbf{f}|+|\mathbf{g}|+|\mathbf{c}'|+|\mathbf{d}'|)\\
&=&\text{wt}(\bar x^{\mathbf{a}'}\bar y^{\mathbf{b}'})-1+2(|\mathbf{f}|+|\mathbf{g}|)+(|\mathbf{c}'|+|\mathbf{d}'|)\\
&\leqslant& \text{wt}(\bar x^{\mathbf{a}'}\bar y^{\mathbf{b}'})-1+2(|\mathbf{a}'|+|\mathbf{b}'|-|\mathbf{i}|-|\mathbf{j}|+1)+(|\mathbf{c}'|+|\mathbf{d}'|)\\
&=&\text{wt}(\bar x^{\mathbf{a}'}\bar y^{\mathbf{b}'}\bar u^{\mathbf{c}'}\bar v^{\mathbf{d}'})+2\text{deg}(\bar x^{\mathbf{a}'}\bar y^{\mathbf{b}'}\bar u^{\mathbf{c}'}\bar v^{\mathbf{d}'})-2(|\mathbf{i}|+|\mathbf{j}|)+1\\
&\leqslant &A+B-2.
\end{array}\]

By (i) and (ii) we have
\[\begin{array}{ccl}
(\rho_\chi(\bar v_S))(\bar x^{\mathbf{a}'}\bar y^{\mathbf{b}'}\bar u^{\mathbf{c}'}\bar v^{\mathbf{d}'}\otimes\bar 1_\chi)&=&((-1)^{|\mathbf{b}'|+\sum\limits_{l=1}^{S-1}d'_l}\bar x^{\mathbf{a}'}\bar y^{\mathbf{b}'}\bar u^{\mathbf{c}'}\bar v^{\mathbf{d}'+\mathbf{e}_S}+\text{terms of}~e\text{-degree}\\
&&\leqslant A+B-2)\otimes\bar 1_\chi.
\end{array}\]

(2) Induction on $|\mathbf{d}|=|(\mathbf{0},\mathbf{0},\mathbf{0},\mathbf{d})|_e=A$ now shows that
$$(\rho_\chi(\bar v^\mathbf{d}))(\bar x^{\mathbf{a}'}\bar y^{\mathbf{b}'}\bar u^{\mathbf{c}'}\bar v^{\mathbf{d}'}\otimes\bar 1_\chi)=(K^{''}\bar x^{\mathbf{a}'}\bar y^{\mathbf{b}'}\bar u^{\mathbf{c}'}\bar v^{\mathbf{d}+\mathbf{d}'}+\text{terms of}~e\text{-degree}\leqslant A+B-2)\otimes\bar 1_\chi,$$
where the coefficient $K^{''}\in\mathds{k}$ is a power of $-1$. If $\mathbf{d}+\mathbf{d}'\notin\Lambda'_t$, then set $K^{''}=0$.

(3) Notice that $\bar u^\mathbf{c}$ is a product of even elements in $\mathfrak{g}_\mathds{k}$. Combining the formula displayed in step (2) and discussing in the same way as (1) and (2), it is now easy to derive that
\[\begin{array}{ccl}(\rho_\chi(\bar u^\mathbf{c}\bar v^\mathbf{d}))(\bar x^{\mathbf{a}'}\bar y^{\mathbf{b}'}\bar u^{\mathbf{c}'}\bar v^{\mathbf{d}'}\otimes\bar 1_\chi)&=&(K^{''}\bar x^{\mathbf{a}'}\bar y^{\mathbf{b}'}\bar u^{\mathbf{c}+\mathbf{c}'}\bar v^{\mathbf{d}+\mathbf{d}'}+\text{terms
of}~e\text{-degree}\\
&&\leqslant A+B-2)\otimes\bar 1_\chi.
\end{array}\]
If $(\mathbf{c}+\mathbf{c}',\mathbf{d}+\mathbf{d}')\notin \Lambda_s\times\Lambda'_t$, then the first summand on the right hand is interpreted as $0$.

(4) Since the image of $\mathfrak{p}_\mathds{k}$ is still in $\mathfrak{p}_\mathds{k}$ under the action of $\text{ad}h$, the PBW theorem for $U_\chi(\mathfrak{p}_\mathds{k})$ implies that
$$\bar x^{\mathbf{a}}\bar y^\mathbf{b}\cdot \bar x^{\mathbf{a}'}\bar y^{\mathbf{b}'}=K^{'''}\bar x^{\mathbf{a}+\mathbf{a}'}\bar y^{\mathbf{b}+\mathbf{b}'}+\sum\limits_{|\mathbf{i}|+|\mathbf{j}|<
|\mathbf{a}|+|\mathbf{a}'|+|\mathbf{b}|+|\mathbf{b}'|}\beta_{\mathbf{i},\mathbf{j}}\bar x^{\mathbf{i}}\bar y^{\mathbf{j}},$$
where $K^{'''}\in\mathds{k}$ is a power of $-1$. If $(\mathbf{a}+\mathbf{a}',\mathbf{b}+\mathbf{b}')\notin \Lambda_m\times\Lambda'_n$, then set $K^{'''}=0$, and $\beta_{\mathbf{i},\mathbf{j}}=0$ unless $\text{wt}(\bar x^{\mathbf{i}}\bar y^{\mathbf{j}})=\text{wt}(\bar x^{\mathbf{a}}\bar y^\mathbf{b})+\text{wt}(\bar x^{\mathbf{a}'}\bar y^{\mathbf{b}'})$.

(5) It can be inferred from (3) and (4) that
\[\begin{array}{ccl}(\rho_\chi(\bar x^{\mathbf{a}}\bar y^\mathbf{b}\bar u^\mathbf{c}\bar v^\mathbf{d}))(\bar x^{\mathbf{a}'}\bar y^{\mathbf{b}'}\bar u^{\mathbf{c}'}\bar v^{\mathbf{d}'}\otimes\bar 1_\chi)&=&
(K'\bar x^{\mathbf{a}+\mathbf{a}'}\bar y^{\mathbf{b}+\mathbf{b}'}\bar u^{\mathbf{c}+\mathbf{c}'}\bar v^{\mathbf{d}+\mathbf{d}'}+\text{terms of}~e\text{-degree}\\&&\leqslant A+B-2)\otimes\bar1_\chi,
\end{array}\]
where $K'\in\mathds{k}$ is a power of $-1$. If $(\mathbf{a}+\mathbf{a}',\mathbf{b}+\mathbf{b}',\mathbf{c}+\mathbf{c}',\mathbf{d}+\mathbf{d}')\notin\Lambda_m\times\Lambda'_n\times\Lambda_s\times\Lambda'_t$, set $K'=0$.

(6) Finally we will discuss the value of $K'$ in (5).

Given any homogeneous elements $\bar u,\bar v\in\mathfrak{g}_\mathds{k}$, it can be deduced from the definition of $e$-degree that
\[\bar u\bar v\equiv\left\{\begin{array}{ll}\bar v\bar u&\text{if at least one of}~\bar u, \bar v~\text{is even;}\\-\bar v\bar u&\text{if}~\bar u~\text{and}~\bar v~\text{are both odd}\end{array}\right.\]
modolo terms of lower $e$-degree in $U_\chi(\mathfrak{g}_\mathds{k})$. Therefore, in order to determine the value of $K'$, one just needs to deal with the odd elements. For each case we will consider separately:

(i) if $(\mathbf{a}+\mathbf{a}',\mathbf{b}+\mathbf{b}',\mathbf{c}+\mathbf{c}',\mathbf{d}+\mathbf{d}')\notin\Lambda_m\times\Lambda'_n\times\Lambda_s\times\Lambda'_t$, then $K'=0$ by (1)---(5).

(ii) if $(\mathbf{a}+\mathbf{a}',\mathbf{b}+\mathbf{b}',\mathbf{c}+\mathbf{c}',\mathbf{d}+\mathbf{d}')\in\Lambda_m\times\Lambda'_n\times\Lambda_s\times\Lambda'_t$, it follows from the definition of $\Lambda'_n$ and $\Lambda'_t$ that each entry in the sequence $(\mathbf{b},\mathbf{d},\mathbf{b}',\mathbf{d}')$ is in the set $\{0,1\}$. From above one knows that if two odd elements exchange their positions in the product of $U_\chi(\mathfrak{g}_\mathds{k})$, there is a sign change modulo terms of lower $e$-degree. Since the position exchange in $U_\chi(\mathfrak{g}_\mathds{k})$ corresponds to the transposition of sequence $(b_1,b_2,\cdots,b_n,d_1,d_2,\cdots,d_t,b'_1,b'_2,\cdots,b'_n,d'_1,d'_2,\cdots,d'_t)$, it follows that the constant $K'$ in step (5) coincides with the constant $K$ which is defined in the Lemma.
\end{proof}

\begin{rem}\label{reverse}
It is immediate from the knowledge of linear algebra that $\tau(\mathbf{b},\mathbf{d},\mathbf{b}',\mathbf{d}')$ in Lemma~\ref{commutative relations k2}(2) is just the reverse number of $$
(b_1,b_2,\cdots,b_n,d_1,d_2,\cdots,d_t,b'_1,b'_2,\cdots,b'_n,d'_1,d'_2,\cdots,d'_t)$$ with
respect to the sequence $$(b_1,b'_1,b_2,b'_2,\cdots,b_n,b'_n,d_1,d'_1,d_2,d'_2\cdots,d_t,d'_t).$$
\end{rem}

Recall that any non-zero element $\bar h\in U_\chi(\mathfrak{g}_\mathds{k},e)$ is uniquely determined by its value on $\bar h(\bar 1_\chi)\in Q_\chi^\chi$. Write$$\bar h(\bar 1_\chi)=(\sum\limits_{|(\mathbf{a},\mathbf{b},\mathbf{c},\mathbf{d})|_e\leqslant  n(\bar h)}\lambda_{\mathbf{a},\mathbf{b},\mathbf{c},\mathbf{d}}
\bar x^{\mathbf{a}}\bar y^\mathbf{b}\bar u^\mathbf{c}\bar v^\mathbf{d})\otimes\bar 1_\chi,$$ where $n=n(\bar h)$ is the highest $e$-degree of the terms in the linear expansion of $\bar h(\bar 1 _\chi)$, and $\lambda_{\mathbf{a},\mathbf{b},\mathbf{c},\mathbf{d}}\neq0$ for at least one $(\mathbf{a},\mathbf{b},\mathbf{c},\mathbf{d})$ with $|(\mathbf{a},\mathbf{b},\mathbf{c},\mathbf{d})|_e=n(\bar h)$.

For $k\in\mathbb{Z}_+$, Put $\Lambda^{k}_{\bar h}=\{(\mathbf{a},\mathbf{b},\mathbf{c},\mathbf{d})|\lambda_{\mathbf{a},\mathbf{b},\mathbf{c},\mathbf{d}}\neq0
\,\&\,|(\mathbf{a},\mathbf{b},\mathbf{c},\mathbf{d})|_e=k\}$ and let $\Lambda^{\text{max}}_{\bar h}$ denote the set of all $(\mathbf{a},\mathbf{b},\mathbf{c},\mathbf{d})\in\Lambda^{n(\bar h)}_{\bar h}$ for which the quantity $\text{wt}(\bar x^{\mathbf{a}}\bar y^\mathbf{b}\bar u^\mathbf{c}\bar v^\mathbf{d})$ assumes its maximum value. This maximum value will be denoted by $N=N(\bar h)$.

Some result about the leading term of $\bar h(\bar 1_\chi)\in Q_\chi^\chi$ for each element $\bar h\in U_\chi(\mathfrak{g}_\mathds{k},e)$ will be given in Lemma~\ref{hw}. It can be proved by the same treatment as Premet to the finite $W$-algebra case in ([\cite{P2}], Lemma 3.2). Compared with the Lie algebra case, the appearance of odd elements in the Lie superalgebra $\mathfrak{g}_\mathds{k}$ makes the situation different. For each monomial in the basis of $U_\chi(\mathfrak{g}_\mathds{k})$, it follows from the PBW theorem that the indices of odd elements in $\mathfrak{g}_\mathds{k}$ can not exceed $1$. In fact, we have that $\bar a^2\otimes\bar1_\chi=\frac{[\bar a,\bar a]}{2}\otimes\bar1_\chi\in Q_\chi^\chi$ for each $\bar a\in(\mathfrak{g}_\mathds{k})_{\bar1}$. It is obvious that when we write $\bar a^2\otimes\bar1_\chi$ as $\frac{[\bar a,\bar a]}{2}\otimes\bar1_\chi\in Q_\chi^\chi$, whose weight remains unchanged with standard degree decreasing. It follows from \eqref{dege} that the $e$-degree of $\bar a^2\otimes\bar1_\chi$ is lower than what it seems to be when we put it as a linear combination of the canonical basis of $Q_\chi^\chi$. Therefore, if the index of some odd element in $\mathfrak{g}_\mathds{k}$ exceeds $1$ in some monomial of $Q_\chi^\chi$, the $e$-degree of this monomial decreases when we put it as a linear combination of the canonical basis of $Q_\chi^\chi$. Obviously, this can not occur in the Lie algebra case. Now we prove the lemma in detail.

\begin{lemma}\label{hw}
Let $\bar h\in U_\chi(\mathfrak{g}_\mathds{k},e)\backslash\{0\}$ and $(\mathbf{a},\mathbf{b},\mathbf{c},\mathbf{d})\in \Lambda^{\text{max}}_{\bar h}$. Then $\mathbf{c}=\mathbf{0}$ and $\mathbf{a}\in\Lambda_{l}\times\{\mathbf{0}\},~\mathbf{b}\in\Lambda'_{q}\times\{\mathbf{0}\}$. Moreover, the sequence $\mathbf{d}$ satisfies

(1) $\mathbf{d}=\mathbf{0}$ when dim~$\mathfrak{g}_\mathds{k}(-1)_{\bar{1}}$ is even;

(2) $\mathbf{d}\in\{\mathbf{0}\}_{\frac{r-1}{2}}\times\Lambda'_{1}\times\{\mathbf{0}\}_{\frac{r-1}{2}}$ when dim~$\mathfrak{g}_\mathds{k}(-1)_{\bar{1}}$ (recall that which equals to $r$) is odd.
\end{lemma}

\begin{proof}
(1) Suppose the contrary, i.e.

(I) if dim~$\mathfrak{g}_\mathds{k}(-1)_{\bar{1}}$ is even, then $$(a_{l+1},\cdots,a_m,b_{q+1},\cdots,b_n,c_1,\cdots,c_s,d_1\cdots,d_{\frac{r}{2}})\neq\{\mathbf{0}\};$$

(II) if dim~$\mathfrak{g}_\mathds{k}(-1)_{\bar{1}}$ is odd, then $$(a_{l+1},\cdots,a_m,b_{q+1},\cdots,b_n,c_1,\cdots,c_s,d_1\cdots,d_{\frac{r-1}{2}})\neq\{\mathbf{0}\}.$$

(2) First assume that

(i) if $a_k\neq0$ for some $k>l$ set $\bar x_k\in\mathfrak{g}_\mathds{k}(n_k)_{\bar0}$. Since $\bar x_k\notin(\mathfrak{g}_\mathds{k}^e)_{\bar{0}}$ and the bilinear form $(\cdot,\cdot)$ is non-degenerated, there is $\bar w=\bar w_k\in\mathfrak{g}_\mathds{k}(-n_k-2)_{\bar{0}}$ such that $\chi([\bar w_k,\bar x_i])=\delta_{ki}$ for all $i>l$.

(ii) if all $a_i's$ are zero for $i>l$, and there is $b_k\neq0$ for some $k>q$, then set $\bar y_k\in\mathfrak{g}_\mathds{k}(n'_k)_{\bar1}$. Since $\bar y_k\notin(\mathfrak{g}_\mathds{k}^e)_{\bar{1}}$ and the bilinear form $(\cdot,\cdot)$ is non-degenerated, there is $\bar w=\bar w'_k\in\mathfrak{g}_\mathds{k}(-n'_k-2)_{\bar{1}}$ such that $\chi([\bar w'_k,\bar y_i])=\delta_{ki}$ for all $i>q$.

(iii) if all $a_i's$ and $b_j's$ are zero for $i>l$ and $j>q$, and there is $\bar u_k\neq0$ for some $1\leqslant  k\leqslant  s$, choose $\bar w=\bar z_k\in\mathfrak{g}_\mathds{k}(-1)'_{\bar{0}}$ such that $\langle \bar z_k,\bar u_i\rangle=\delta_{ki}$ for all $1\leqslant  k\leqslant  s$.

(iv) if all $a_i's$, $b_j's$ and $c_k's$ are zero for $i>l$, $j>q$ and $1\leqslant k\leqslant s$, respectively, then

(a) when dim~$\mathfrak{g}_\mathds{k}(-1)_{\bar{1}}$ is even, there is $\bar w=\bar z'_k\in\mathfrak{g}_\mathds{k}(-1)'_{\bar{1}}$ such that $\langle \bar z'_k,\bar v_i\rangle=\delta_{ki}$ for all $1\leqslant k\leqslant \frac{r}{2}$;

(b) when dim~$\mathfrak{g}_\mathds{k}(-1)_{\bar{1}}$ is odd, there is $\bar w=\bar z^{''}_k\in\mathfrak{g}_\mathds{k}(-1)'_{\bar{1}}$ such that $\langle \bar z^{''}_k,\bar v_i\rangle=\delta_{ki}$ for all $1\leqslant  k\leqslant \frac{r-1}{2}$.

Under above assumptions, we write $\nu:=\text{wt}(\bar w)$.

(3) Let $(\mathbf{a},\mathbf{b},\mathbf{c},\mathbf{d})\in\Lambda^{d'}_{\bar h}$ where $d'\in\mathbb{Z}$. By the assumptions in (2) one knows that $\bar w$ is $\mathbb{Z}_2$-homogeneous. It is immediate from Lemma~\ref{commutative relations k1} and the definition of $Q_\chi^\chi$ that
\begin{equation}\label{rhow}
(\rho_\chi(\bar w)) (\bar x^{\mathbf{a}}\bar y^\mathbf{b}\bar u^\mathbf{c}\bar v^\mathbf{d})\otimes\bar 1_\chi=\sum_{\mathbf{i}\in\Lambda_{m}}\sum_{\mathbf{j}\in\Lambda'_{n}}\left(\begin{array}{@{\hspace{0pt}}c@{\hspace{0pt}}} \mathbf{a}\\ \mathbf{i}\end{array}\right)\bar x^{\mathbf{a}-\mathbf{i}}\bar y^{\mathbf{b}-\mathbf{j}}\cdot\rho_\chi([\bar w\bar x^{\mathbf{i}}\bar y^{\mathbf{j}}])\cdot \bar u^\mathbf{c}\bar v^\mathbf{d}\otimes\bar 1_\chi,
\end{equation}
where the summation in \eqref{rhow} runs over all $(\mathbf{i},\mathbf{j})\in\Lambda_m\times\Lambda'_n$ such that $[\bar w\bar x^{\mathbf{i}}\bar y^{\mathbf{j}}]$ is nonzero and $\text{wt}([\bar w\bar x^{\mathbf{i}}\bar y^{\mathbf{j}}])\geqslant -2$.

Based on the value of $\text{wt}([\bar w\bar x^{\mathbf{i}}y^{\mathbf{j}}])$, for each case we will consider separately.

(i) Suppose $(\mathbf{i},\mathbf{j})\in\Lambda_m\times\Lambda'_n$ is such that wt$([\bar w\bar x^{\mathbf{i}}\bar y^{\mathbf{j}}])\geqslant 0$. Then $|\mathbf{i}|+|\mathbf{j}|\geqslant 1$. Recall that the decomposition  $\mathfrak{g}_\mathds{k}=\bigoplus\limits_{i\in\mathbb{Z}}\mathfrak{g}_\mathds{k}(i)$ is the Dynkin grading of $\mathfrak{g}_\mathds{k}$, and the action of ad$h$ keeps $\mathfrak{p}_\mathds{k}$ invariant. This implies that $\bar x^{\mathbf{a}-\mathbf{i}}\bar y^{\mathbf{b}-\mathbf{j}}\cdot\rho_\chi([\bar w\bar x^{\mathbf{i}}\bar y^{\mathbf{j}}]) \bar u^\mathbf{c}\bar v^\mathbf{d}\otimes\bar 1_\chi$ is a linear combination of $\bar x^{\mathbf{i'}}\bar y^{\mathbf{j'}}\bar u^\mathbf{c}\bar v^\mathbf{d}\otimes\bar 1_\chi$ with
$$\text{wt}(\bar x^{\mathbf{i'}}\bar y^{\mathbf{j'}}\bar u^\mathbf{c}\bar v^\mathbf{d})=\nu+\text{wt}(\bar x^{\mathbf{a}}\bar y^{\mathbf{b}}\bar u^\mathbf{c}\bar v^\mathbf{d}),$$
and\[\begin{array}{ccl}
\text{deg}_e(\bar x^{\mathbf{i'}}\bar y^{\mathbf{j'}} \bar u^\mathbf{c}\bar v^\mathbf{d})&=&\text{wt}(\bar x^{\mathbf{i'}}\bar y^{\mathbf{j'}} \bar u^\mathbf{c}\bar v^\mathbf{d})+2\text{deg}(\bar x^{\mathbf{i'}}\bar y^{\mathbf{j'}} \bar u^\mathbf{c}\bar v^\mathbf{d})\\
&=&\nu+\text{wt}(\bar x^{\mathbf{a}}\bar y^{\mathbf{b}}\bar u^\mathbf{c}\bar v^\mathbf{d})
+2(|\mathbf{i'}|+|\mathbf{j'}|+|\mathbf{c}|+|\mathbf{d}|).
\end{array}\]

It follows from the remark preceding this lemma that when we put $\bar x^{\mathbf{a}-\mathbf{i}}\bar y^{\mathbf{b}-\mathbf{j}}\cdot\rho_\chi([\bar w\bar x^{\mathbf{i}}y^{\mathbf{j}}])\cdot \bar u^\mathbf{c}\bar v^\mathbf{d}$ as a linear combination of the canonical basis of $Q_\chi^\chi$, whose $e$-degree maybe lower than what it seems to be since the power of odd elements must $\leqslant1$. As
$$|\mathbf{i'}|+|\mathbf{j'}|\leqslant |\mathbf{a}|+|\mathbf{b}|-|\mathbf{i}|-|\mathbf{j}|+1,$$
then\[\begin{array}{cl}
&\text{deg}_e(\bar x^{\mathbf{i'}}\bar y^{\mathbf{j'}}\bar u^\mathbf{c}\bar v^\mathbf{d})=\nu+\text{wt}(\bar x^{\mathbf{a}}\bar y^{\mathbf{b}}\bar u^\mathbf{c}\bar v^\mathbf{d})
+2(|\mathbf{i}'|+|\mathbf{j}'|+|\mathbf{c}|+|\mathbf{d}|)\\
\leqslant&\text{wt}(\bar x^{\mathbf{a}}\bar y^{\mathbf{b}}\bar u^\mathbf{c}\bar v^\mathbf{d})+2(|\mathbf{a}|+|\mathbf{b}|+|\mathbf{c}|+|\mathbf{d}|)+
\nu-2(|\mathbf{i}|+|\mathbf{j}|)+2\\
=&2+\nu+d'-2(|\mathbf{i}|+|\mathbf{j}|).
\end{array}\]

(ii) Now suppose $(\mathbf{i},\mathbf{j})\in\Lambda_m\times\Lambda'_n$ is such that wt$([\bar w\bar x^{\mathbf{i}}\bar y^{\mathbf{j}}])=-1$. For $k,g\in\mathbb{Z}_+$, set $\bar x_k\in\mathfrak{g}_\mathds{k}(n_k)_{\bar{0}},~\bar y_g\in\mathfrak{g}_\mathds{k}(n'_g)_{\bar{1}}$, then $\sum\limits_{1\leqslant k\leqslant  m}i_kn_k+\sum\limits_{1\leqslant g\leqslant  n}j_gn'_g=-\nu-1$. Since $\rho_\chi(\mathfrak{m}_\mathds{k}\cap\mathfrak{g}_\mathds{k}(-1))$ annihilates $\bar 1_\chi$, the vector $\bar x^{\mathbf{a}-\mathbf{i}}\bar y^{\mathbf{b}-\mathbf{j}}\cdot\rho_\chi([\bar w\bar x^{\mathbf{i}}\bar y^{\mathbf{j}}])\cdot \bar u^\mathbf{c}\bar v^\mathbf{d}\otimes\bar 1_\chi$ is a linear combination of $\bar x^{\mathbf{a}-\mathbf{i}}\bar y^{\mathbf{b}-\mathbf{j}}\bar u^\mathbf{i'}\bar v^\mathbf{j'}\otimes\bar 1_\chi$ with $|\mathbf{i'}|=|\mathbf{c}|\pm1, \mathbf{j'}=\mathbf{d}$, or $\mathbf{i'}=\mathbf{c}, |\mathbf{j'}|=|\mathbf{d}|\pm1$.

(a) If $|\mathbf{i'}|=|\mathbf{c}|+1, \mathbf{j'}=\mathbf{d}$, or $\mathbf{i'}=\mathbf{c}, |\mathbf{j'}|=|\mathbf{d}|+1$, then $|\mathbf{i}|+|\mathbf{j}|\geqslant 1$,
$$\text{wt}(\bar x^{\mathbf{a}-\mathbf{i}}\bar y^{\mathbf{b}-\mathbf{j}}\bar u^\mathbf{i'}\bar v^\mathbf{j'})= \text{wt}(\bar x^{\mathbf{a}-\mathbf{i}}\bar y^{\mathbf{b}-\mathbf{j}}\bar u^\mathbf{c}\bar v^\mathbf{d})-1
=\nu+\text{wt}(\bar x^{\mathbf{a}}\bar y^{\mathbf{b}}\bar u^\mathbf{c}\bar v^\mathbf{d}),$$
and\[\begin{array}{ccl}
\text{deg}_e(\bar x^{\mathbf{a}-\mathbf{i}}\bar y^{\mathbf{b}-\mathbf{j}}\bar u^\mathbf{i'}\bar v^\mathbf{j'})&=&\text{wt}(\bar x^{\mathbf{a}-\mathbf{i}}\bar y^{\mathbf{b}-\mathbf{j}}\bar u^\mathbf{i'}\bar v^\mathbf{j'})+
2(|\mathbf{a}|-|\mathbf{i}|+|\mathbf{b}|-|\mathbf{j}|+|\mathbf{i'}|+|\mathbf{j'}|)\\
&=&  d'+\nu+2(-|\mathbf{i}|-|\mathbf{j}|+|\mathbf{i'}|+|\mathbf{j'}|-|\mathbf{c}|-|\mathbf{d}|)\\
&=&2+d'+\nu-2(|\mathbf{i}|+|\mathbf{j}|).
\end{array}\]

(b) If $|\mathbf{i'}|=|\mathbf{c}|-1, \mathbf{j'}=\mathbf{d}$, or $\mathbf{i'}=\mathbf{c}, |\mathbf{j'}|=|\mathbf{d}|-1$, then
$$\text{wt}(\bar x^{\mathbf{a}-\mathbf{i}}\bar y^{\mathbf{b}-\mathbf{j}}\bar u^\mathbf{i'}\bar v^\mathbf{j'})
=2+\nu+\text{wt}(\bar x^{\mathbf{a}}\bar y^{\mathbf{b}}\bar u^\mathbf{c}\bar v^\mathbf{d}),$$
and\[\begin{array}{ccl}\text{deg}_e(\bar x^{\mathbf{a}-\mathbf{i}}\bar y^{\mathbf{b}-\mathbf{j}}\bar u^\mathbf{i'}\bar v^\mathbf{j'})&=&\text{wt}(\bar x^{\mathbf{a}-\mathbf{i}}\bar y^{\mathbf{b}-\mathbf{j}}\bar u^\mathbf{i'}\bar v^\mathbf{j'})+
2(|\mathbf{a}|-|\mathbf{i}|+|\mathbf{b}|-|\mathbf{j}|+|\mathbf{i'}|+|\mathbf{j'}|)\\
&=&\nu+\text{wt}(\bar x^{\mathbf{a}}\bar y^{\mathbf{b}}\bar u^\mathbf{c}\bar v^\mathbf{d})+2+
2(|\mathbf{a}|+|\mathbf{b}|+|\mathbf{c}|+|\mathbf{d}|)\\
&&+2(-|\mathbf{c}|-|\mathbf{d}|-|\mathbf{i}|-|\mathbf{j}|+|\mathbf{i'}|+|\mathbf{j'}|)\\
&=&\nu+d'-2(|\mathbf{i}|+|\mathbf{j}|).
\end{array}\]

By the assumption one knows that wt$([\bar w\bar x^{\mathbf{i}}\bar y^{\mathbf{j}}])=-1$. When we put $\bar x^{\mathbf{a}-\mathbf{i}}\bar y^{\mathbf{b}-\mathbf{j}}\cdot\rho_\chi([\bar w\bar x^{\mathbf{i}}\bar y^{\mathbf{j}}])\cdot \bar u^\mathbf{c}\bar v^\mathbf{d}\otimes1_\chi$ as a linear combination of the canonical basis of $Q_\chi^\chi$, if the index of odd element of $\mathfrak{g}_\mathds{k}$ reaches $2$ in some monomial of $Q_\chi^\chi$, e.g. set the element as $\bar y\in\mathfrak{g}_\mathds{k}(-1)_{\bar1}$, then it follows from $\bar y^2\otimes \bar 1_\chi=\frac{1}{2}[\bar y,\bar y]\otimes \bar 1_\chi=\frac{1}{2}\chi([\bar y,\bar y])\otimes \bar 1_\chi$ that the weight of the monomial rises by $2$, while the standard degree decreases by $2$. Then it is immediate from \eqref{dege} that the $e$-degree of the monomial decreases by $2$ compared with what it seems to be. As the index of some odd element decreases in this situation, it can be deduced to case (ii)(b).

(iii) Finally, suppose $(\mathbf{i},\mathbf{j})\in\Lambda_m\times\Lambda'_n$ is such that wt$([\bar w\bar x^{\mathbf{i}}\bar y^{\mathbf{j}}])=-2$. Then $$\bar x^{\mathbf{a}-\mathbf{i}}\bar y^{\mathbf{b}-\mathbf{j}}\cdot\rho_\chi([\bar w\bar x^{\mathbf{i}}\bar y^{\mathbf{j}}])\cdot \bar u^\mathbf{c}\bar v^\mathbf{d}\otimes\bar 1_\chi=\chi([\bar w\bar x^{\mathbf{i}}\bar y^{\mathbf{j}}])\bar x^{\mathbf{a}-\mathbf{i}}\bar y^{\mathbf{b}-\mathbf{j}}\bar u^\mathbf{c}\bar v^\mathbf{d}\otimes\bar 1_\chi.$$
As $\sum\limits_{1\leqslant k\leqslant  m}i_kn_k+\sum\limits_{1\leqslant g\leqslant  n}j_gn'_g=-\nu-2$, one has
$$\text{wt}(\bar x^{\mathbf{a}-\mathbf{i}}\bar y^{\mathbf{b}-\mathbf{j}}\bar u^\mathbf{c}\bar v^\mathbf{d})=2+\nu+\text{wt}(\bar x^{\mathbf{a}}\bar y^{\mathbf{b}}\bar u^\mathbf{c}\bar v^\mathbf{d}),\,\text{deg}_e(\bar x^{\mathbf{a}-\mathbf{i}}\bar y^{\mathbf{b}-\mathbf{j}}\bar u^\mathbf{c}\bar v^\mathbf{d})=2+\nu+d'-2(|\mathbf{i}|+|\mathbf{j}|).$$

(4) Now based on the results of (2) and (3), we will discuss the validity of the assumptions in (1).

For $i,j\in\mathbb{Z}$ let $\pi_{ij}$ denote the endomorphism of $Q_\chi^\chi$ such that

(a) if $(\mathbf{a},\mathbf{b},\mathbf{c},\mathbf{d})$ satisfies $\text{deg}_e(\bar x^{\mathbf{a}}\bar y^{\mathbf{b}}\bar u^\mathbf{c}\bar v^\mathbf{d})=i$ and $\text{wt}(\bar x^{\mathbf{a}}\bar y^{\mathbf{b}}\bar u^\mathbf{c}\bar v^\mathbf{d})=j$, then define
\begin{equation}\label{pi}
\pi_{ij}(\bar x^{\mathbf{a}}\bar y^{\mathbf{b}}\bar u^\mathbf{c}\bar v^\mathbf{d}\otimes\bar1_\chi)=\bar x^{\mathbf{a}}\bar y^{\mathbf{b}}\bar u^\mathbf{c}\bar v^\mathbf{d}\otimes\bar 1_\chi;
\end{equation}

(b) if $(\mathbf{a},\mathbf{b},\mathbf{c},\mathbf{d})$ does not satisfy the condition in (a), then define $$\pi_{ij}(\bar x^{\mathbf{a}}\bar y^{\mathbf{b}}\bar u^\mathbf{c}\bar v^\mathbf{d}\otimes\bar 1_\chi)=0.$$

(i) If assumption (i) or (ii) in (2) is true, then $\nu\leqslant -2$ and $w\in\mathfrak{m}_\mathds{k}$. Set $\bar h=\bar h_{\bar{0}}+\bar h_{\bar{1}}\in U_\chi({\mathfrak{g}_\mathds{k},e})$. As $\bar w$ is $\mathbb{Z}_2$-homogeneous and $\chi((\mathfrak{g}_\mathds{k})_{\bar1})=0$, then $\chi(\bar w)=0$ if $\bar w\in(\mathfrak{g}_\mathds{k})_{\bar1}$. It follows from the definition of $U_\chi({\mathfrak{g}_\mathds{k},e})$ that
\begin{equation}\label{rhochi}
\begin{split}
(\rho_\chi(\bar w)-\chi(\bar w)\text{id}).\bar h(1_\chi)=&(\rho_\chi(\bar w)-\chi(\bar w)\text{id}).
(\bar h_{\bar{0}}+\bar h_{\bar{1}})(\bar 1_\chi)\\
=&\bar h_{\bar{0}}((\bar w-\chi(\bar w)).\bar 1_\chi)
+(-1)^{|\bar w|}\bar h_{\bar{1}}(\bar w.\bar 1_\chi)-\chi(\bar w)\bar h_{\bar{1}}(\bar 1_\chi)\\
=&\bar h_{\bar{0}}((\bar w-\chi(\bar w)).\bar 1_\chi)
+(-1)^{|\bar w|}\cdot \bar h_{\bar{1}}((\bar w-\chi(\bar w)).\bar 1_\chi)\\
=&0.
\end{split}
\end{equation}

For $a\in\mathbb{Z}$ we let $\bar{a}$ denote the residue of $a$ in $\mathbb{F}_p\subset\overline{\mathbb{F}}_p=\mathds{k}$. By (3) one knows that the terms with $e$-degree $n(\bar h)+\nu$ and weight $N(\bar h)+\nu+2$ only occur in (3)(iii) when $(\rho_\chi(\bar w)-\chi(\bar w)\text{id}).\bar h(\bar 1_\chi)$ is written as a linear combination of the canonical basis of $Q_\chi^\chi$. If write $b_{n+1}=0$, it follows from \eqref{rhochi} that
\begin{equation}\label{0pi}
\begin{split}
0=&\pi_{n(\bar h)+\nu,N(\bar h)+\nu+2}((\rho_\chi(\bar w)-\chi(\bar w)\text{id}). \bar h(\bar 1_\chi))\\
=&(\sum\limits_{(\mathbf{a},\mathbf{b},\mathbf{c},\mathbf{d})\in\Lambda_{\bar h}^\text{max}}\lambda_{\mathbf{a},\mathbf{b},\mathbf{c},\mathbf{d}}
\sum\limits_{i=1}^m(-1)^{|\bar w|(\sum\limits_{j=1}^nb_j)}\bar{a}_i \bar x^{\mathbf{a}-\mathbf{e}_i}\bar y^{\mathbf{b}}\cdot\chi([\bar w,\bar x_i]) \bar u^\mathbf{c}\bar v^\mathbf{d})\otimes\bar 1_\chi+\\
&(\sum\limits_{(\mathbf{a},\mathbf{b},\mathbf{c},\mathbf{d})\in\Lambda_{\bar h}^\text{max}}\lambda_{\mathbf{a},\mathbf{b},\mathbf{c},\mathbf{d}}
\sum\limits_{i=1}^n(-1)^{|\bar w|(1+\sum\limits_{j=1}^nb_j)+\sum\limits_{j=i+1}^nb_{j}} \bar x^{\mathbf{a}}\bar y^{\mathbf{b}-\mathbf{e}_i}\cdot\chi([\bar w,\bar y_i])\\
&\bar u^\mathbf{c}\bar v^\mathbf{d})\otimes\bar 1_\chi.
\end{split}
\end{equation}
Since $\chi((\mathfrak{g}_\mathds{k})_{\bar1})=0$, then

(a) if $\bar w=\bar w_k$, then \eqref{0pi} equals $\sum\limits_{(\mathbf{a},\mathbf{b},\mathbf{c},\mathbf{d})\in\Lambda_h^\text{max}}\lambda_{\mathbf{a},\mathbf{b},\mathbf{c},\mathbf{d}}
\bar{a}_k\bar x^{\mathbf{a}-\mathbf{e}_k}\bar y^{\mathbf{b}}\bar u^\mathbf{c}\bar v^\mathbf{d}\otimes\bar 1_\chi\neq0$;

(b) if $\bar w=\bar w'_k$, then \eqref{0pi} equals $$\sum\limits_{(\mathbf{a},\mathbf{b},\mathbf{c},\mathbf{d})\in\Lambda_{\bar h}^\text{max}}\lambda_{\mathbf{a},\mathbf{b},\mathbf{c},\mathbf{d}}
(-1)^{1+\sum\limits_{j=1}^nb_j+\sum\limits_{j=k+1}^nb_{j}}\bar x^{\mathbf{a}}\bar y^{\mathbf{b}-\mathbf{e}_k}\cdot \bar u^\mathbf{c}\bar v^\mathbf{d}\otimes\bar 1_\chi\neq0,$$
a contradiction, i.e. the assumptions (i) and (ii) in (2) are invalid.

(ii) If assumption (iii) or (iv) in (2) is true, then $\nu=-1$ and $\chi(\bar w)=0$. It follows from the definition of $U_\chi({\mathfrak{g}_\mathds{k},e})$ that
\begin{equation}\label{rhowchi}
\rho_\chi(\bar w).\bar h(\bar 1_\chi)=\rho_\chi(\bar w).(\bar h_{\bar{0}}+\bar h_{\bar{1}})(\bar 1_\chi)=\bar h_{\bar{0}}(\bar w.\bar 1_\chi)
+(-1)^{|\bar w|}\bar h_{\bar{1}}(\bar w.\bar 1_\chi)=0.
\end{equation}
Set $b_{n+1}=0, d_{0}=0$, then
\begin{equation}\label{exten}
\begin{split}
\rho_\chi(\bar w).\bar h(\bar 1_\chi)=&(\sum\limits_{(\mathbf{a},\mathbf{b},\mathbf{c},\mathbf{d})\in\Lambda_{\bar h}^{n(\bar h)}}\lambda_{\mathbf{a},\mathbf{b},\mathbf{c},\mathbf{d}}(
\sum\limits_{i=1}^m(-1)^{|\bar w|(\sum\limits_{j=1}^nb_j)}\bar{a}_i\bar x^{\mathbf{a}-\mathbf{e}_i}\bar y^{\mathbf{b}}\cdot\rho_\chi([\bar w,\bar x_i])\\
&\cdot \bar u^\mathbf{c}\bar v^\mathbf{d}+\sum\limits_{i=1}^n(-1)^{|\bar w|(1+\sum\limits_{j=1}^nb_j)+\sum\limits_{j=i+1}^nb_{j}}\bar x^{\mathbf{a}}\bar y^{\mathbf{b}-\mathbf{e}_i}\cdot\rho_\chi([\bar w,\bar y_i])\cdot \bar u^\mathbf{c}\bar v^\mathbf{d}+\\
&\sum\limits_{i=1}^s(-1)^{|\bar w|(\sum\limits_{j=1}^nb_j)}\bar{c}_i\bar x^{\mathbf{a}}\bar y^{\mathbf{b}}\cdot\chi([\bar w,\bar u_i])\bar u^{\mathbf{c}-\mathbf{e}_i}\bar v^{\mathbf{d}}+\\&\sum\limits_{i=1}^{\lceil \frac{r}{2}\rceil}(-1)^{|\bar w|(\sum\limits_{j=1}^nb_j)+\sum\limits_{j=0}^{i-1}d_{j}}\cdot \bar x^{\mathbf{a}}\bar y^{\mathbf{b}}\cdot\chi([\bar w,\bar v_i])\bar u^{\mathbf{c}}\bar v^{\mathbf{d}-\mathbf{e}_i})+\\
&\sum\limits_{|(\mathbf{i},\mathbf{j},\mathbf{k},\mathbf{l})|_e\leqslant  n(\bar h)-2}\beta_{\mathbf{i},\mathbf{j},\mathbf{k},\mathbf{l}}\cdot \bar x^{\mathbf{i}}\bar y^{\mathbf{j}}\bar u^\mathbf{k}\bar v^\mathbf{l})\otimes\bar1_\chi,
\end{split}
\end{equation}

When $(\rho_\chi(\bar w)-\chi(\bar w)\text{id}).\bar h(\bar 1_\chi)$ is written as a linear combination of the canonical basis of $Q_\chi^\chi$, it is immediate from (3) that the terms with $e$-degree $n(\bar h)-1$ and weight $N(\bar h)+1$ only occur in (3)(ii), and by \eqref{exten} we have

(a) if $\bar w=\bar z'_k$, then
$$\pi_{n(\bar h)-1,N(\bar h)+1}(\rho_\chi(\bar w).\bar h(\bar 1_\chi))=\sum
\limits_{(\mathbf{a},\mathbf{b},\mathbf{c},\mathbf{d})\in\Lambda_{\bar h}^{n(\bar h)}}\lambda_{\mathbf{a},\mathbf{b},\mathbf{c},\mathbf{d}}
\bar{c}_k\bar x^{\mathbf{a}}\bar y^{\mathbf{b}}\bar u^{\mathbf{c}-\mathbf{e}_k}\bar v^{\mathbf{d}}\otimes\bar 1_\chi\neq0;$$

(b) if $\bar w=\bar z^{''}_k$, then
\[\begin{array}{ll}
&\pi_{n(\bar h)-1,N(\bar h)+1}(\rho_\chi(\bar w).\bar h(\bar 1_\chi))\\
=&\sum
\limits_{(\mathbf{a},\mathbf{b},\mathbf{c},\mathbf{d})\in\Lambda_{\bar h}^{n(\bar h)}}\lambda_{\mathbf{a},\mathbf{b},\mathbf{c},\mathbf{d}}
(-1)^{\sum\limits_{j=1}^nb_j+\sum\limits_{j=0}^{k-1}d_{j}}\bar x^{\mathbf{a}}\bar y^{\mathbf{b}}\bar u^{\mathbf{c}}\bar v^{\mathbf{d}-\mathbf{e}_k}\otimes\bar1_\chi
\neq0,
\end{array}\] a contradiction, i.e. the assumptions (iii) and (iv) in (2) are invalid.

All the contradictions in (4) complete the proof of the lemma.
\end{proof}

\subsection{The construction theory of reduced $W$-superalgebras in positive characteristic}

In this part we will study the construction theory of the reduced $W$-superalgebra $U_\chi(\mathfrak{g}_\mathds{k},e)$.

For $k\in\mathbb{Z}_+$ let $H^k$ denote the $\mathds{k}$-linear span of all $0\neq \bar h\in U_\chi(\mathfrak{g}_\mathds{k},e)$ with $n(\bar h)\leqslant  k$ in $U_\chi(\mathfrak{g}_\mathds{k},e)$. It follows readily from Lemma~\ref{commutative relations k2} that $H^i\cdot H^j\subseteq H^{i+j}$ for all $i,j\in\mathbb{Z}_+$. In other words, $\{H^i|i\in\mathbb{Z}_+\}$ is a filtration of the algebra $U_\chi(\mathfrak{g}_\mathds{k},e)$ and obviously $U_\chi(\mathfrak{g}_\mathds{k},e)=H^k$ for all $k\gg0$. We set $H^{-1}=0$ and let $\text{gr}(U_\chi(\mathfrak{g}_\mathds{k},e))=\sum\limits_{i\geqslant 0}H^i/H^{i-1}$ denote the corresponding graded algebra. Lemma~\ref{commutative relations k2} implies that the $\mathds{k}$-algebra $\text{gr}(U_\chi(\mathfrak{g}_\mathds{k},e))$ is supercommutative.

\begin{prop}\label{reduced basis}
Let $U_\chi(\mathfrak{g}_\mathds{k},e)$ be a reduced $W$-superalgebra,

(1) if dim~$\mathfrak{g}_\mathds{k}(-1)_{\bar{1}}$ is even, then for any $(\mathbf{a},\mathbf{b})\in\Lambda_l\times\Lambda'_q$ there is $\bar h_{\mathbf{a},\mathbf{b}}\in U_\chi(\mathfrak{g}_\mathds{k},e)$ such that $\Lambda_{\bar h_{\mathbf{a},\mathbf{b}}}^{\text{max}}=\{(\mathbf{a},\mathbf{b})\}$. The vectors $\{\bar h_{\mathbf{a},\mathbf{b}}|(\mathbf{a},\mathbf{b})\in\Lambda_l\times\Lambda'_q\}$ form a basis of $U_\chi(\mathfrak{g}_\mathds{k},e)$ over $\mathds{k}$.

(2) if dim~$\mathfrak{g}_\mathds{k}(-1)_{\bar{1}}$ is odd, then for any $(\mathbf{a},\mathbf{b},c)\in\Lambda_l\times\Lambda'_q\times\Lambda'_1$ there is $\bar h_{\mathbf{a},\mathbf{b},c}\in U_\chi(\mathfrak{g}_\mathds{k},e)$ such that $\Lambda_{\bar h_{\mathbf{a},\mathbf{b},c}}^{\text{max}}=\{(\mathbf{a},\mathbf{b},c)\}$. The vectors $\{\bar h_{\mathbf{a},\mathbf{b},c}|(\mathbf{a},\mathbf{b},c)\in\Lambda_l\times\Lambda'_q\times\Lambda'_1\}$ form a basis of $U_\chi(\mathfrak{g}_\mathds{k},e)$ over $\mathds{k}$.
\end{prop}

\begin{proof}

Given $(a,b)\in\mathbb{Z}_+^2$, let $H^{a,b}$ denote the subspace of $U_\chi(\mathfrak{g}_\mathds{k},e)$ spanned by $H^{a-1}$ and all $\bar h\in U_\chi(\mathfrak{g}_\mathds{k},e)$ such that $n(\bar h)=a,~N(\bar h)\leqslant  b$. Order the elements in $\mathbb{Z}_+^2$ lexicographically. By construction, $H^{a,b}\subseteq H^{c,d}$ whenever $(a,b)\prec(c,d)$. Applying the Basis Extension Theorem to the finite chain of subspaces just defined we obtain that $U_\chi(\mathfrak{g}_\mathds{k},e)$ has basis $B:=\bigsqcup_{(i,j)}B_{i,j}$ such that $n(\mu)=i,~N(\mu)=j$ whenever $\mu\in B_{i,j}$.

Recall in \eqref{pi} the mapping $\pi_{ij}:~U_\chi(\mathfrak{g}_\mathds{k},e)\rightarrow Q_\chi^\chi$ is defined by setting $$\pi_{ij}(\bar x^{\mathbf{a}}\bar y^{\mathbf{b}}\bar u^\mathbf{c}\bar v^\mathbf{d}\otimes\bar1_\chi)=\bar x^{\mathbf{a}}\bar y^{\mathbf{b}}\bar u^\mathbf{c}\bar v^\mathbf{d}\otimes\bar 1_\chi$$ for any $(\mathbf{a},\mathbf{b},\mathbf{c},\mathbf{d})$ which satisfies $\text{deg}_e(\bar x^{\mathbf{a}}\bar y^{\mathbf{b}}\bar u^\mathbf{c}\bar v^\mathbf{d})=i$ and $\text{wt}(\bar x^{\mathbf{a}}\bar y^{\mathbf{b}}\bar u^\mathbf{c}\bar v^\mathbf{d})=j$; and $0$ otherwise.

Define the linear map $\pi_B: U_\chi(\mathfrak{g}_\mathds{k},e)\longrightarrow Q_\chi^\chi$ by setting $\mu\in B_{i,j}$ for any $\pi_B(\mu)=\pi_{i,j}(\mu(\bar1_\chi))$ and extending to $U_\chi(\mathfrak{g}_\mathds{k},e)$ by linearity.

Based on the parity of dim~$\mathfrak{g}_\mathds{k}(-1)_{\bar{1}}$, for each case we will consider separately.

(1) When dim~$\mathfrak{g}_\mathds{k}(-1)_{\bar{1}}$ is even, it can be inferred from Lemma~\ref{hw} that $\pi_B$ maps $U_\chi(\mathfrak{g}_\mathds{k},e)$ into the subspace $U_\chi(\mathfrak{g}_\mathds{k}^e)\otimes\bar1_\chi$ of $Q_\chi^\chi$. By construction, $\pi_B$ is injective. On the other hand,\[\begin{array}{ccl}
\text{\underline{dim}}~U_\chi(\mathfrak{g}_\mathds{k},e)&=&(\frac{p^{\text{dim}(\mathfrak{g}_\mathds{k})_{\bar{0}}}}
{(p^{\frac{\text{dim}~(\mathfrak{g}_\mathds{k})_{\bar{0}}-\text{dim}~(\mathfrak{g}_\mathds{k}^e)_{\bar{0}}}{2}})^2},~
\frac{2^{\text{dim}~(\mathfrak{g}_\mathds{k})_{\bar{1}}}}{(2^{\frac{\text{dim}~(\mathfrak{g}_\mathds{k})_{\bar{1}}-\text{dim}~(\mathfrak{g}_\mathds{k}^e)_{\bar{1}}}{2}})^2})\\
&=&(p^{\text{dim}~(\mathfrak{g}_\mathds{k}^e)_{\bar{0}}},~2^{\text{dim}~(\mathfrak{g}_\mathds{k}^e)_{\bar{1}}})\\
&=&\text{\underline{dim}}~
U_\chi(\mathfrak{g}_\mathds{k}^e)\otimes\bar1_\chi
\end{array}\]due to Remark~\ref{centralizer} and Theorem~\ref{matrix}. Thus $\pi_B: U_\chi(\mathfrak{g}_\mathds{k},e)\longrightarrow U_\chi(\mathfrak{g}_\mathds{k}^e)\otimes1_\chi$ is a linear isomorphism. For $(\mathbf{a},\mathbf{b})=(a_1,\cdots,a_l,b_1,\cdots,b_q)\in\Lambda_l\times\Lambda'_q$ set $\bar h_{\mathbf{a},\mathbf{b}}:=\pi_B^{-1}(\bar x_1^{a_1}\cdots \bar x_l^{a_l}\bar y_1^{b_1}\cdots \bar y_q^{b_q}\otimes\bar 1_\chi)$. By the bijectivity of $\pi_B$ and the PBW theorem (applied to $U_\chi(\mathfrak{g}_\mathds{k}^e)$), the vectors $\bar h_{\mathbf{a},\mathbf{b}}$ with $(\mathbf{a},\mathbf{b})\in\Lambda_l\times\Lambda'_q$ form a basis of $U_\chi(\mathfrak{g}_\mathds{k},e)$ over $\mathds{k}$, while from the definition of $\pi_B$ it follows that $\Lambda_{\bar h_{\mathbf{a},\mathbf{b}}}^{\text{max}}=\{(\mathbf{a},\mathbf{b})\}$ for any $(\mathbf{a},\mathbf{b})\in\Lambda_l\times\Lambda'_q$.

(2) When dim~$\mathfrak{g}_\mathds{k}(-1)_{\bar{1}}$ is odd, it can be inferred from Lemma~\ref{hw} that $\pi_B$ maps $U_\chi(\mathfrak{g}_\mathds{k},e)$ into the subspace of $(U_\chi(\mathfrak{g}_\mathds{k}^e)\otimes_\mathds{k} U_\chi(\mathds{k}\bar v_{\frac{r+1}{2}}))\otimes\bar 1_\chi$ of $Q_\chi^\chi$. By construction, $\pi_B$ is injective. On the other hand,
\[\begin{array}{ccl}\text{\underline{dim}}~U_\chi(\mathfrak{g}_\mathds{k},e)&=&(\frac{p^{\text{dim}~(\mathfrak{g}_\mathds{k})_{\bar{0}}}}
{(p^{\frac{\text{dim}~(\mathfrak{g}_\mathds{k})_{\bar{0}}-\text{dim}~(\mathfrak{g}_\mathds{k}^e)_{\bar{0}}}{2}})^2},~
\frac{2^{\text{dim}~(\mathfrak{g}_\mathds{k})_{\bar{1}}}}{(2^{\frac{\text{dim}~(\mathfrak{g}_\mathds{k})_{\bar{1}}-\text{dim}~(\mathfrak{g}_\mathds{k}^e)_{\bar{1}}-1}{2}})^2})\\
&=&(p^{\text{dim}~(\mathfrak{g}_\mathds{k}^e)_{\bar{0}}},~2^{\text{dim}~(\mathfrak{g}_\mathds{k}^e)_{\bar{1}}+1})\\
&=&\text{\underline{dim}}~
(U_\chi(\mathfrak{g}_\mathds{k}^e)\otimes_\mathds{k} U_\chi(\mathds{k}\bar v_{\frac{r+1}{2}}))\otimes\bar1_\chi,
\end{array}\]due to Remark~\ref{centralizer} and Theorem~\ref{matrix}.  Thus $\pi_B: U_\chi(\mathfrak{g}_\mathds{k},e)\longrightarrow (U_\chi(\mathfrak{g}_\mathds{k}^e)\otimes_\mathds{k} U_\chi(\mathds{k}\bar v_{\frac{r+1}{2}}))\otimes1_\chi$ is a linear isomorphism. For $(\mathbf{a},\mathbf{b},c)=(a_1,\cdots,a_l,b_1,\cdots,b_q,c)\in\Lambda_l\times\Lambda'_q\times\Lambda'_1$ set $$h_{\mathbf{a},\mathbf{b},c}=\pi_B^{-1}(\bar x_1^{a_1}\cdots \bar x_l^{a_l}\bar y_1^{b_1}\cdots \bar y_q^{b_q}\bar v_{\frac{r+1}{2}}^c\otimes\bar1_\chi).$$ By the bijectivity of $\pi_B$ and the PBW theorem (applied to $U_\chi(\mathfrak{g}_\mathds{k}^e)\otimes_\mathds{k} U_\chi(\mathds{k}\bar v_{\frac{r+1}{2}})$), the vectors $h_{\mathbf{a},\mathbf{b},c}$ with $(\mathbf{a},\mathbf{b},c)\in\Lambda_l\times\Lambda'_q\times\Lambda'_1$ form a basis of $U_\chi(\mathfrak{g}_\mathds{k},e)$, while from the definition of $\pi_B$ it follows that $\Lambda_{\bar h_{\mathbf{a},\mathbf{b},c}}^{\text{max}}=\{(\mathbf{a},\mathbf{b},c)\}$ for any $(\mathbf{a},\mathbf{b},c)\in\Lambda_l\times\Lambda'_q\times\Lambda'_1$.
\end{proof}

\begin{corollary}\label{rg}
There exist even elements $\theta_1,\cdots,\theta_l\in U_\chi(\mathfrak{g}_\mathds{k},e)_{\bar0}$ and odd elements $ \theta_{l+1},\cdots,\theta_{l+q}\in U_\chi(\mathfrak{g}_\mathds{k},e)_{\bar1}$ such that

(1)$$\theta_k(\bar 1_\chi)=(\bar x_k+\sum\limits_{\mbox{\tiny $\begin{array}{c}|\mathbf{a},\mathbf{b},\mathbf{c},\mathbf{d}|_e=m_k+2,\\|\mathbf{a}|
+|\mathbf{b}|+|\mathbf{c}|+|\mathbf{d}|\geqslant 2\end{array}$}}\lambda^k_{\mathbf{a},\mathbf{b},\mathbf{c},\mathbf{d}}\bar x^{\mathbf{a}}
\bar y^{\mathbf{b}}\bar u^{\mathbf{c}}\bar v^{\mathbf{d}}+\sum\limits_{|\mathbf{a},\mathbf{b},\mathbf{c},\mathbf{d}|_e<m_k+2}\lambda^k_{\mathbf{a},\mathbf{b},\mathbf{c},\mathbf{d}}\bar x^{\mathbf{a}}
\bar y^{\mathbf{b}}\bar u^{\mathbf{c}}\bar v^{\mathbf{d}})\otimes\bar 1_\chi,$$
where $\bar x_k\in\mathfrak{g}^e_\mathds{k}(m_k)_{\bar0}$ for $1\leqslant  k\leqslant  l$.

(2)$$\theta_{l+k}(\bar 1_\chi)=(\bar y_k+\sum\limits_{\mbox{\tiny $\begin{array}{c}|\mathbf{a},\mathbf{b},\mathbf{c},\mathbf{d}|_e=n_k+2,\\|\mathbf{a}|
+|\mathbf{b}|+|\mathbf{c}|+|\mathbf{d}|\geqslant 2\end{array}$}}\lambda^k_{\mathbf{a},\mathbf{b},\mathbf{c},\mathbf{d}}\bar x^{\mathbf{a}}
\bar y^{\mathbf{b}}\bar u^{\mathbf{c}}\bar v^{\mathbf{d}}+\sum\limits_{|\mathbf{a},\mathbf{b},\mathbf{c},\mathbf{d}|_e<n_k+2}\lambda^k_{\mathbf{a},\mathbf{b},\mathbf{c},\mathbf{d}}\bar x^{\mathbf{a}}
\bar y^{\mathbf{b}}\bar u^{\mathbf{c}}\bar v^{\mathbf{d}})\otimes\bar 1_\chi,$$
where $\bar y_k\in\mathfrak{g}^e_\mathds{k}(n_k)_{\bar1}$ for $1\leqslant  k\leqslant  q$.

(3) When dim~$\mathfrak{g}_\mathds{k}(-1)_{\bar{1}}$ is odd, there is an odd element $\theta_{l+q+1}\in U_\chi(\mathfrak{g}_\mathds{k},e)_{\bar1}$ such that$$\theta_{l+q+1}(\bar 1_\chi)=\bar v_{\frac{r+1}{2}}\otimes\bar 1_\chi.$$

All the coefficients $\lambda^k_{\mathbf{a},\mathbf{b},\mathbf{c},\mathbf{d}}\in\mathds{k}$. Moreover, $\lambda^k_{\mathbf{a},\mathbf{b},\mathbf{c},\mathbf{d}}=0$ if $(\mathbf{a},\mathbf{b},\mathbf{c},\mathbf{d})$ is such that $a_{l+1}=\cdots=a_m=b_{q+1}=\cdots=b_n=c_1=\cdots=c_s=d_1=\cdots=d_{\lceil\frac{r}{2}\rceil}=0$.
\end{corollary}

\begin{proof}
The existence of all the elements in the corollary is an immediate consequence of Proposition~\ref{reduced basis}, and what remains to prove is the $\mathbb{Z}_2$-homogeneity for the elements in (1) and (2). In fact, it can be obtained by applying Proposition~\ref{invariant} directly. First note that each element $\theta\in U_\chi(\mathfrak{g}_\mathds{k},e)$ can be written as $\theta=\theta_{\bar0}+\theta_{\bar1}$ where $\theta_{\bar0}(\bar1_\chi)\in (Q_\chi^\chi)_{\bar0}$ and $\theta_{\bar1}(\bar1_\chi)\in (Q_\chi^\chi)_{\bar1}$. Since both the $\mathds{k}$-algebras $U_\chi(\mathfrak{g}_\mathds{k},e)$ and $(Q_\chi^\chi)^{\text{ad}\mathfrak{m}_\mathds{k}}$ are $\mathbb{Z}_2$-graded and the mapping $$\varphi:~U_\chi(\mathfrak{g}_\mathds{k},e)\overset{\sim}{\longrightarrow}(Q_\chi^\chi)^{\text{ad}\mathfrak{m}_\mathds{k}}$$ in Proposition~\ref{invariant} is even, it follows from $$\varphi(\theta)=\varphi(\theta_{\bar0})+\varphi(\theta_{\bar1})= \theta_{\bar0}(\bar 1_\chi)+\theta_{\bar1}(\bar 1_\chi)\in(Q_\chi^\chi)^{\text{ad}\mathfrak{m}_\mathds{k}}$$
that $\varphi(\theta_{\bar0})\in(Q_\chi^\chi)^{\text{ad}\mathfrak{m}_\mathds{k}}_{\bar0}$ and $\varphi( \theta_{\bar1})\in(Q_\chi^\chi)^{\text{ad}\mathfrak{m}_\mathds{k}}_{\bar1}$, i.e. $\theta_{\bar0}\in U_\chi(\mathfrak{g}_\mathds{k},e)_{\bar0}$ and $\theta_{\bar1}\in U_\chi(\mathfrak{g}_\mathds{k},e)_{\bar1}$.

Therefore, for any given  element $\theta\in U_\chi(\mathfrak{g}_\mathds{k},e)$, if the leading term of $\theta(\bar 1_\chi)$ (i.e. the term with the highest $e$-degree and whose weight also assumes its maximum value) is $\bar x_i$ (which is an element in $(\mathfrak{g}_\mathds{k})_{\bar0}$) for $1\leqslant  i\leqslant  l$, one can choose $\theta_{\bar0}$ as the desired element in $U_\chi(\mathfrak{g}_\mathds{k},e)_{\bar0}$. Similarly, if the leading term of $\theta(\bar 1_\chi)$ is $\bar y_i$ (which is an element in $(\mathfrak{g}_\mathds{k})_{\bar1}$) for $1\leqslant  i\leqslant  q$, one can choose $\theta_{ \bar1}$ as the desired element in $U_\chi(\mathfrak{g}_\mathds{k},e)_{\bar1}$.
\end{proof}

Recall that $\{\bar x_1,\cdots,\bar x_l\}$ and $\{\bar y_1,\cdots,\bar y_q\}$ are $\mathds{k}$-basis of $(\mathfrak{g}^e_\mathds{k})_{\bar{0}}$ and $(\mathfrak{g}^e_\mathds{k})_{\bar{1}}$, respectively. When dim~$\mathfrak{g}_\mathds{k}(-1)_{\bar{1}}$ is odd, one will see that the element  $\bar v_{\frac{r+1}{2}}\in\mathfrak{g}_\mathds{k}(-1)_{\bar{1}}\cap(\mathfrak{g}_\mathds{k}(-1)_{\bar{1}}^{\prime})^{\bot }$ plays the key role for the construction theory of $U_\chi(\mathfrak{g}_\mathds{k},e)$. Set
\[\bar Y_i:=\left\{
\begin{array}{ll}
\bar x_i&\text{if}~1\leqslant  i\leqslant  l;\\
\bar y_{i-l}&\text{if}~l+1\leqslant  i\leqslant  l+q;\\
\bar v_{\frac{r+1}{2}}&\text{if}~i=l+q+1.
\end{array}
\right.
\]

For any $1\leqslant  i\leqslant  l+q+1$, assume that $\bar Y_i\in\mathfrak{g}_\mathds{k}(m_i)$ is homogeneous, and the term $\bar Y_{l+q+1}$ occurs only in the case when dim~$\mathfrak{g}_\mathds{k}(-1)_{\bar{1}}$ is odd. Note that $\bar Y_i\in\mathfrak{g}_\mathds{k}^e$ for any $1\leqslant  i\leqslant  l+q$. Let $\bar{\theta}_i$ denote the image of $\theta_i\in U_\chi(\mathfrak{g}_\mathds{k},e)$ in $\text{gr}(U_\chi(\mathfrak{g}_\mathds{k},e))$.

The following theorem introduces the generators and their relations for the reduce $W$-superalgebra $U_\chi(\mathfrak{g}_\mathds{k},e)$ over positive characteristic field $\mathds{k}$, and also the PBW theorem.

\begin{theorem}\label{reduced Wg}
For any reduced $W$-superalgebra $U_\chi(\mathfrak{g}_\mathds{k},e)$, all the elements given in Corollary~\ref{commutative relations k1} can be chosen as a set of generators, and

(1) when dim~$\mathfrak{g}_\mathds{k}(-1)_{\bar{1}}$ is even,

(i) let $1\leqslant  i,j\leqslant  l+q$. Then
$$[\theta_i,\theta_j]=\theta_i\cdot\theta_j-(-1)^{|\theta_i||\theta_j|}\theta_j\cdot\theta_i=(-1)^{|\theta_i|\cdot|\theta_j|}\theta_j\circ\theta_i-\theta_i\circ\theta_j
\in H^{m_i+m_j+2}.$$

(ii) if the elements $\bar Y_i,~\bar Y_j\in \mathfrak{g}_\mathds{k}^e$ for $1\leqslant  i,j\leqslant  l+q$ satisfy $[\bar Y_i,\bar Y_j]=\sum\limits_{k=1}^{l+q}\alpha_{ij}^k\bar Y_k$ in $\mathfrak{g}_\mathds{k}^e$, then
\begin{equation}\label{thetacom}
[\theta_i,\theta_j]\equiv\sum\limits_{k=1}^{l+q}\alpha_{ij}^k\theta_k+q_{ij}(\theta_1,\cdots,\theta_{l+q})\qquad(\text{mod}~H^{m_i+m_j+1}),
\end{equation}where $q_{ij}$ is a truncated polynomial in $l+q$ variables whose constant term and linear part are both zero.

(iii) the monomials $\bar{\theta}_1^{a_1}\cdots\bar{\theta}_l^{a_l}\bar{\theta}_{l+1}^{b_1}\cdots\bar{\theta}_{l+q}^{b_q}$ and $\theta_1^{a_1}\cdots\theta_l^{a_l}\theta_{l+1}^{b_1}\cdots\theta_{l+q}^{b_q}$ form bases of $\text{gr}(U_\chi(\mathfrak{g}_\mathds{k},e))$ and $U_\chi(\mathfrak{g}_\mathds{k},e)$ respectively, where $0\leqslant  a_i\leqslant  p-1,~b_i\in\{0,1\}$.

(2) when dim~$\mathfrak{g}_\mathds{k}(-1)_{\bar{1}}$ is odd,

(i) let $1\leqslant  i,j\leqslant  l+q+1$. Then $$[\theta_i,\theta_j]=\theta_i\cdot\theta_j-(-1)^{|\theta_i||\theta_j|}\theta_j\cdot\theta_i=(-1)^{|\theta_i|\cdot|\theta_j|}\theta_j\circ\theta_i-\theta_i\circ\theta_j
\in H^{m_i+m_j+2}.$$

(ii) if the elements $\bar Y_i,~\bar Y_j\in \mathfrak{g}_\mathds{k}^e$ for $1\leqslant  i,j\leqslant  l+q$ satisfy $[\bar Y_i,\bar Y_j]=\sum\limits_{k=1}^{l+q}\alpha_{ij}^k\bar Y_k$ in $\mathfrak{g}_\mathds{k}^e$, then
\begin{equation}\label{thetacom2}
[\theta_i,\theta_j]\equiv\sum\limits_{k=1}^{l+q}\alpha_{ij}^k\theta_k+q_{ij}(\theta_1,\cdots,\theta_{l+q+1})\qquad(\text{mod}~H^{m_i+m_j+1}),
\end{equation}where $q_{ij}$ is a truncated polynomial in $l+q+1$ variables whose constant term and linear part are both zero. For the case $i=j=l+q+1$, we have $[\theta_{l+q+1},\theta_{l+q+1}]=\text{id}$.

(iii) the monomials $\bar{\theta}_1^{a_1}\cdots\bar{\theta}_l^{a_l}\bar{\theta}_{l+1}^{b_1}\cdots\bar{\theta}_{l+q}^{b_q}\bar{\theta}_{l+q+1}^{c}$ and $\theta_1^{a_1}\cdots\theta_l^{a_l}\theta_{l+1}^{b_1}\cdots\theta_{l+q}^{b_q}\theta_{l+q+1}^{c}$ form bases of $\text{gr}(U_\chi(\mathfrak{g}_\mathds{k},e))$ and $U_\chi(\mathfrak{g}_\mathds{k},e)$ respectively, where $0\leqslant  a_i\leqslant  p-1,~b_i,c\in\{0,1\}$.
\end{theorem}

\begin{proof}

Since the proof for both cases is similar, we will just consider part (2) when dim~$\mathfrak{g}_\mathds{k}(-1)_{\bar{1}}$ is odd.

(1) Recall that $\theta_1,\cdots,\theta_l\in U_\chi(\mathfrak{g}_\mathds{k},e)_{\bar0}$ and $\theta_{l+1},\cdots,\theta_{l+q+1}\in U_\chi(\mathfrak{g}_\mathds{k},e)_{\bar1}$ by Corollary~\ref{commutative relations k1}. As $b_i,c\in\{0,1\}$ for $1\leqslant i\leqslant q$, it follows from the definition of opposite algebra that
\[\begin{array}{cl}
&\theta_1^{a_1}\cdots\theta_l^{a_l}\theta_{l+1}^{b_1}\cdots\theta_{l+q}^{b_q}\theta_{l+q+1}^{c}\\
=&
(-1)^{c(b_1+\cdots+b_q)}\theta_{l+q+1}^{c}\circ(\theta_1^{a_1}\cdots\theta_l^{a_l}\theta_{l+1}^{b_1}\cdots\theta_{l+q}^{b_q})=\cdots\cdots\\
=&(-1)^{\sum\limits_{1\leqslant  i<j\leqslant  q}b_ib_j+c\sum\limits_{i=1}^{q}b_i}\theta_{l+q+1}^{c}\circ\theta_{l+q}^{b_q}\circ\cdots\circ\theta_{l+1}^{b_1}\circ\theta_{l}^{a_l}\circ\cdots\circ\theta_{1}^{a_1},
\end{array}\]
and\[\begin{array}{cl}
&\theta_{l+q+1}^{c}\circ\theta_{l+q}^{b_q}\circ\cdots\circ\theta_{l+1}^{b_1}\circ\theta_{l}^{a_l}\circ\cdots\circ\theta_{2}^{a_2}(\theta_{1}^{a_1}(\bar 1_\chi))\\
=&\theta_{1}^{a_1}(\bar 1_\chi)\cdot(\theta_{l+q+1}^{c}\circ\theta_{l+q}^{b_q}\circ\cdots\circ\theta_{l+1}^{b_1}\circ\theta_{l}^{a_l}\circ\cdots\circ\theta_{2}^{a_2}(\bar 1_\chi))\\
=&\theta_{1}^{a_1}(\bar 1_\chi)\cdots\theta_{l}^{a_l}(\bar 1_\chi)\cdot(\theta_{l+q+1}^{c}\circ\theta_{l+q}^{b_q}\circ\cdots\circ\theta_{l+1}^{b_1}(\bar 1_\chi))\\
=&(-1)^{b_1(b_2+\cdots+b_n+c)}\theta_{1}^{a_1}(\bar 1_\chi)\cdots\theta_{l}^{a_l}(\bar 1_\chi)\theta_{l+1}^{b_1}(\bar 1_\chi)\cdot(\theta_{l+q+1}^{c}\circ\theta_{l+q}^{b_q}\circ\cdots\circ\theta_{l+2}^{b_2}(\bar 1_\chi))\\
=&\cdots\cdots\\
=&(-1)^{\sum\limits_{1\leqslant  i<j\leqslant  q}b_ib_j+c\sum\limits_{i=1}^{q}b_i}\theta_{1}^{a_1}(\bar 1_\chi)\cdots\theta_{l}^{a_l}(\bar 1_\chi)\theta_{l+1}^{b_1}(\bar 1_\chi)\cdots\theta_{l+q}^{b_q}(\bar 1_\chi)\theta_{l+q+1}^{c}(\bar 1_\chi),
\end{array}\]
so \begin{equation}\label{1in}
\begin{split}
\theta_1^{a_1}\cdots\theta_l^{a_l}\theta_{l+1}^{b_1}\cdots\theta_{l+q}^{b_q}\theta_{l+q+1}^{c}(\bar 1_\chi)=&\theta_{1}^{a_1}(\bar 1_\chi)\cdots\theta_{l}^{a_l}(\bar 1_\chi)\theta_{l+1}^{b_1}(\bar 1_\chi)\cdots\\
&\theta_{l+q}^{b_q}(\bar 1_\chi)\theta_{l+q+1}^{c}(\bar 1_\chi).
\end{split}\end{equation}

Induction on $|\mathbf{a}|+|\mathbf{b}|+|c|$ shows that
\[\begin{array}{cl}
&\theta_{1}^{a_1}(\bar 1_\chi)\cdots\theta_{l}^{a_l}(\bar 1_\chi)\theta_{l+1}^{b_1}(\bar 1_\chi)\cdots\theta_{l+q}^{b_q}(\bar 1_\chi)\theta_{l+q+1}^{c}(\bar 1_\chi)\\
=&(\bar Y_1^{a_1}\cdots \bar Y_l^{a_l}\bar Y_{l+1}^{b_1}\cdots \bar Y_{l+q}^{b_q}\bar Y_{l+q+1}^c+\sum\limits_{\mbox{\tiny $\begin{array}{c}|(\mathbf{i},\mathbf{j},\mathbf{f},\mathbf{g})|_e=|(\mathbf{a},\mathbf{b},\mathbf{0},\mathbf{e}_{\frac{r+1}{2}})|_e,\\|\mathbf{i}|+|\mathbf{j}|+|\mathbf{f}|+|\mathbf{g}|>|\mathbf{a}|+|\mathbf{b}|+1\end{array}$}}
\lambda^{\mathbf{a},\mathbf{b},\mathbf{0},\mathbf{e}_{\frac{r+1}{2}}}_{\mathbf{i},\mathbf{j},\mathbf{f},\mathbf{g}}\bar x^{\mathbf{i}}
\bar y^{\mathbf{j}}\bar u^{\mathbf{f}}\bar v^{\mathbf{g}}\\
&+\text{terms of lower}~e\text{-degree})\otimes\bar 1_\chi
\end{array}\]for any $(\mathbf{a},\mathbf{b},c)=(a_1,\cdots,a_l,b_1,\cdots,b_q,c)\in\Lambda_l\times\Lambda'_q\times\Lambda'_1$ (the induction step is based on Corollary~\ref{rg} and Lemma~\ref{commutative relations k2}). Due to Proposition~\ref{reduced basis} we have that
\begin{equation}\label{thetamu}
\theta_1^{a_1}\cdots\theta_{l+q+1}^c=\mu_{\mathbf{a},\mathbf{b},c}\bar h_{\mathbf{a},\mathbf{b},c}+\sum
\limits_{(\mathbf{i},\mathbf{j},k)\in\Lambda_l\times\Lambda'_q\times\Lambda'_1}\mu_{\mathbf{i},\mathbf{j},k}
\bar h_{\mathbf{i},\mathbf{j},k},\qquad\mu_{\mathbf{a},\mathbf{b},c}\neq0,
\end{equation}where $\mu_{\mathbf{i},\mathbf{j},k}=0$ unless $(n(\bar h_{\mathbf{i},\mathbf{j},k}),N(\bar h_{\mathbf{i},\mathbf{j},k}))\prec (n(\bar h_{\mathbf{a},\mathbf{b},c}),N(\bar h_{\mathbf{a},\mathbf{b},c}))$.

Since this establishes for any $(\mathbf{a},\mathbf{b},c)\in\Lambda_l\times\Lambda'_q\times\Lambda'_1$, the monomials $\theta_1^{a_1}\cdots\theta_{l+q+1}^c$ with $(\mathbf{a},\mathbf{b},c)=(a_1,\cdots,a_l,b_1,\cdots,b_q,c)\in\Lambda_l\times\Lambda'_q\times\Lambda'_1$ form a basis of $U_\chi(\mathfrak{g}_\mathds{k},e)$. It follows from \eqref{thetamu} and the proof of Proposition~\ref{reduced basis} that for any $M\geqslant 0$, the cosets $\bar h_{\mathbf{i},\mathbf{j},k}+H^{M-1}$ with $(\mathbf{i},\mathbf{j},k)\in\Lambda_l\times\Lambda'_q\times\Lambda'_1$ and $\sum\limits_{1\leqslant f\leqslant  l}i_f(m_f+2)+\sum\limits_{1\leqslant g\leqslant  q}j_g(m_{g+l}+2)+k=M$ (where $m_f's,~m_{g+l}'s$ are the weights of $\bar Y_f's$ for $1\leqslant f\leqslant l$ and  $\bar Y_{g+l}'s$ for $1\leqslant g\leqslant q$, respectively) form a basis of $\text{gr}(U_\chi(\mathfrak{g}_\mathds{k},e))$. Due to \eqref{thetamu} and Lemma~\ref{commutative relations k2} the cosets $\theta_1^{i_1}\cdots\theta_l^{i_l}\theta_{l+1}^{j_1}\cdots\theta_{l+q}^{j_q}\theta_{l+q+1}^{k}+H^{M-1}$ with $(\mathbf{i},\mathbf{j},k)\in\Lambda_l\times\Lambda'_q\times\Lambda'_1$ and $\sum\limits_{1\leqslant f\leqslant  l}i_f(m_f+2)+\sum\limits_{1\leqslant g\leqslant  q}j_g(m_{g+l}+2)+k=M$ have the same property. To complete the proof of part (2)(iii) it remains to note that
$\bar{\theta}_1^{i_1}\cdots\bar{\theta}_l^{i_l}\bar{\theta}_{l+1}^{j_1}\cdots\bar{\theta}_{l+q}^{j_q}\bar{\theta}_{l+q+1}^{k}= \theta_1^{i_1}\cdots\theta_l^{i_l}\theta_{l+1}^{j_1}\cdots\theta_{l+q}^{j_q}\theta_{l+q+1}^{k}+H^{M-1}$ for any $(\mathbf{i},\mathbf{j},k)\in\Lambda_l\times\Lambda'_q\times\Lambda'_1$  with $\sum\limits_{1\leqslant f\leqslant  l}i_f(m_f+2)+\sum\limits_{1\leqslant g\leqslant  q}j_g(m_{g+l}+2)+k=M$.

(2) Given any $\mathbb{Z}_2$-homogeneous elements $\theta_i,\theta_j\in U_\chi(\mathfrak{g}_\mathds{k},e)$ for $1\leqslant i,j\leqslant l+q$ (note that $i,j\neq l+q+1$), if either $\theta_i$ or $\theta_j$ is even, then $[\theta_i,\theta_j]=\theta_i\cdot\theta_j-\theta_j\cdot\theta_i$; if both $\theta_i$ and $\theta_j$ are odd, then $[\theta_i,\theta_j]=\theta_i\cdot\theta_j+\theta_j\cdot\theta_i$. By Corollary~\ref{rg} and Lemma~\ref{commutative relations k2} one can deduce that
$[\theta_i,\theta_j]\in H^{m_i+m_j+2}$. It is immediate from \eqref{1in} that $$[\theta_i,\theta_j](\bar 1_\chi)=\theta_i(\bar 1_\chi)\theta_j(\bar 1_\chi)-(-1)^{|\theta_i||\theta_j|}\theta_j(\bar 1_\chi)\theta_i(\bar 1_\chi).$$

As in the proof of Lemma~\ref{commutative relations k2}, induction on $|\mathbf{d}|$ yields
\[\begin{array}{ll}&(\rho_\chi(\bar v^\mathbf{d}))(\bar x^{\mathbf{a}'}\bar y^{\mathbf{b}'}\bar u^{\mathbf{c}'}\bar v^{\mathbf{d}'}\otimes\bar 1_\chi)\\
=&(K'\bar x^{\mathbf{a}'}\bar y^{\mathbf{b}'}\bar u^{\mathbf{c}'}\bar v^{\mathbf{d}+\mathbf{d}'}+
\sum\limits_{i,j}\omega_{i,j}\bar x^{\mathbf{a}'-\mathbf{e}_i}\bar y^{\mathbf{b}'}\rho_\chi([\bar v_j,\bar x_i])\bar u^{\mathbf{c}'}\bar v^{\mathbf{d}+\mathbf{d}'-\mathbf{e}_j}\\
&+\sum\limits_{i,j}\nu_{i,j}\bar x^{\mathbf{a}'}\bar y^{\mathbf{b}'-\mathbf{e}_i}\rho_\chi([\bar v_j,\bar y_i])\bar u^{\mathbf{c}'}\bar v^{\mathbf{d}+\mathbf{d}'-\mathbf{e}_j}+\text{terms of}~e\text{-degree}\leqslant |(\mathbf{a}',\mathbf{b}',\mathbf{c}',\mathbf{d}+\mathbf{d}')|_e\\
&-3)\otimes\bar 1_\chi.
\end{array}\] Similar induction on $|\mathbf{c}|$ yields
\[\begin{array}{ll}&(\rho_\chi(\bar u^{\mathbf{c}}\bar v^\mathbf{d}))(\bar x^{\mathbf{a}'}\bar y^{\mathbf{b}'}\bar u^{\mathbf{c}'}\bar v^{\mathbf{d}'}\otimes\bar 1_\chi)\\
=&(K'\bar x^{\mathbf{a}'}\bar y^{\mathbf{b}'}\bar u^{\mathbf{c}+\mathbf{c}'}\bar v^{\mathbf{d}+\mathbf{d}'}+
\sum\limits_{i,j}\kappa_{i,j}\bar x^{\mathbf{a}'-\mathbf{e}_i}\bar y^{\mathbf{b}'}\rho_\chi([\bar u_j,\bar x_i])\bar u^{\mathbf{c}+\mathbf{c}'-\mathbf{e}_j}\bar v^{\mathbf{d}+\mathbf{d}'}\\
&+\sum\limits_{i,j}\lambda_{i,j}\bar x^{\mathbf{a}'}\bar y^{\mathbf{b}'-\mathbf{e}_i}\rho_\chi([\bar u_j,\bar y_i])\bar u^{\mathbf{c}+\mathbf{c}'-\mathbf{e}_j}\bar v^{\mathbf{d}+\mathbf{d}'}+
\sum\limits_{i,j}\omega_{i,j}\bar x^{\mathbf{a}'-\mathbf{e}_i}\bar y^{\mathbf{b}'}\rho_\chi([\bar v_j,\bar x_i])\bar u^{\mathbf{c}+\mathbf{c}'}\bar v^{\mathbf{d}+\mathbf{d}'-\mathbf{e}_j}\\
&+\sum\limits_{i,j}\nu_{i,j}\bar x^{\mathbf{a}'}\bar y^{\mathbf{b}'-\mathbf{e}_i}\rho_\chi([\bar v_j,\bar y_i])\bar u^{\mathbf{c}+\mathbf{c}'}\bar v^{\mathbf{d}+\mathbf{d}'-\mathbf{e}_j}+\text{terms of}~e\text{-degree}\leqslant |(\mathbf{a}',\mathbf{b}',\mathbf{c}+\mathbf{c}',\\&\mathbf{d}+\mathbf{d}')|_e-3)\otimes\bar 1_\chi.
\end{array}\]
It follows that
\[\begin{array}{ll}&(\rho_\chi(\bar y^\mathbf{b}\bar u^\mathbf{c}\bar v^\mathbf{d}))(\bar x^{\mathbf{a}'}\bar y^{\mathbf{b}'}\bar u^{\mathbf{c}'}\bar v^{\mathbf{d}'}\otimes\bar 1_\chi)\\
=&(K\bar x^{\mathbf{a}'}\bar y^{\mathbf{b}+\mathbf{b}'}\bar u^{\mathbf{c}+\mathbf{c}'}\bar v^{\mathbf{d}+\mathbf{d}'}
+\sum\limits_{i,j}\gamma_{i,j}\bar x^{\mathbf{a}'-\mathbf{e}_i}\bar y^{\mathbf{b}+\mathbf{b}'-\mathbf{e}_j}\rho_\chi([\bar x_i,\bar y_j])\bar u^{\mathbf{c}+\mathbf{c}'}\bar v^{\mathbf{d}+\mathbf{d}'}\\
&+\sum\limits_{i<j}\iota_{i,j}\bar x^{\mathbf{a}'}\bar y^{\mathbf{b}+\mathbf{b}'-\mathbf{e}_i-\mathbf{e}_j}\rho_\chi([\bar y_i,\bar y_j])\bar u^{\mathbf{c}+\mathbf{c}'}\bar v^{\mathbf{d}+\mathbf{d}'}
+\sum\limits_{i,j}\kappa'_{i,j}\bar x^{\mathbf{a}'-\mathbf{e}_i}\bar y^{\mathbf{b}+\mathbf{b}'}\rho_\chi([\bar u_j,\bar x_i])\bar u^{\mathbf{c}+\mathbf{c}'-\mathbf{e}_j}\cdot\\&\bar v^{\mathbf{d}+\mathbf{d}'}
+\sum\limits_{i,j}\lambda'_{i,j}\bar x^{\mathbf{a}'}\bar y^{\mathbf{b}+\mathbf{b}'-\mathbf{e}_i}\rho_\chi([\bar u_j,\bar y_i])\bar u^{\mathbf{c}+\mathbf{c}'-\mathbf{e}_j}\bar v^{\mathbf{d}+\mathbf{d}'}
+\sum\limits_{i,j}\omega'_{i,j}\bar x^{\mathbf{a}'-\mathbf{e}_i}\bar y^{\mathbf{b}+\mathbf{b}'}\rho_\chi([\bar v_j,\bar x_i])\cdot\\&\bar u^{\mathbf{c}+\mathbf{c}'}\bar v^{\mathbf{d}+\mathbf{d}'-\mathbf{e}_j}+\sum\limits_{i,j}
\nu'_{i,j}\bar x^{\mathbf{a}'}\bar y^{\mathbf{b}+\mathbf{b}'-\mathbf{e}_i}\rho_\chi([\bar v_j,\bar y_i])\bar u^{\mathbf{c}+\mathbf{c}'}\bar v^{\mathbf{d}+\mathbf{d}'-\mathbf{e}_j}
+\text{terms of}~e\text{-degree}\leqslant\\&|(\mathbf{a}',\mathbf{b}+\mathbf{b}',\mathbf{c}+\mathbf{c}',\mathbf{d}+\mathbf{d}')|_e-3)\otimes\bar 1_\chi,
\end{array}\] and
\[\begin{array}{ll}&(\rho_\chi(\bar x^{\mathbf{a}}\bar y^\mathbf{b}\bar u^\mathbf{c}\bar v^\mathbf{d}))(\bar x^{\mathbf{a}'}\bar y^{\mathbf{b}'}\bar u^{\mathbf{c}'}\bar v^{\mathbf{d}'}\otimes\bar 1_\chi)\\
=&(K\bar x^{\mathbf{a}+\mathbf{a}'}\bar y^{\mathbf{b}+\mathbf{b}'}\bar u^{\mathbf{c}+\mathbf{c}'}\bar v^{\mathbf{d}+\mathbf{d}'}
+\sum\limits_{i<j}\alpha_{i,j}\bar x^{\mathbf{a}+\mathbf{a}'-\mathbf{e}_i-\mathbf{e}_j}\bar y^{\mathbf{b}+\mathbf{b}'}\rho_\chi([\bar x_i,\bar x_j])\bar u^{\mathbf{c}+\mathbf{c}'}\bar v^{\mathbf{d}+\mathbf{d}'}
+\sum\limits_{i,j}\gamma_{i,j}\cdot\\&\bar x^{\mathbf{a}+\mathbf{a}'-\mathbf{e}_i}\bar y^{\mathbf{b}+\mathbf{b}'-\mathbf{e}_j}\rho_\chi([\bar x_i,\bar y_j])\bar u^{\mathbf{c}+\mathbf{c}'}\bar v^{\mathbf{d}+\mathbf{d}'}
+\sum\limits_{i<j}\iota_{i,j}\bar x^{\mathbf{a}+\mathbf{a}'}\bar y^{\mathbf{b}+\mathbf{b}'-\mathbf{e}_i-\mathbf{e}_j}\rho_\chi([\bar y_i,\bar y_j]) \bar u^{\mathbf{c}+\mathbf{c}'}\cdot\\& \bar v^{\mathbf{d}+\mathbf{d}'}+\sum\limits_{i,j}\kappa'_{i,j}\bar x^{\mathbf{a}+\mathbf{a}'-\mathbf{e}_i}\bar y^{\mathbf{b}+\mathbf{b}'}\rho_\chi([\bar u_j,\bar x_i])\bar u^{\mathbf{c}+\mathbf{c}'-\mathbf{e}_j}\bar v^{\mathbf{d}+\mathbf{d}'}
+\sum\limits_{i,j}\lambda'_{i,j}\bar x^{\mathbf{a}+\mathbf{a}'}\bar y^{\mathbf{b}+\mathbf{b}'-\mathbf{e}_i}\cdot\\&\rho_\chi([\bar u_j,\bar y_i])\bar u^{\mathbf{c}+\mathbf{c}'-\mathbf{e}_j}\bar v^{\mathbf{d}+\mathbf{d}'}+\sum\limits_{i,j}\omega'_{i,j}\bar x^{\mathbf{a}+\mathbf{a}'-\mathbf{e}_i}\bar y^{\mathbf{b}+\mathbf{b}'}\rho_\chi([\bar v_j,\bar x_i])\bar u^{\mathbf{c}+\mathbf{c}'}\bar v^{\mathbf{d}+\mathbf{d}'-\mathbf{e}_j}
+\sum\limits_{i,j}\nu'_{i,j}\cdot\\&\bar x^{\mathbf{a}+\mathbf{a}'}\bar y^{\mathbf{b}+\mathbf{b}'-\mathbf{e}_i}\rho_\chi([\bar v_j,\bar y_i])\bar u^{\mathbf{c}+\mathbf{c}'}\bar v^{\mathbf{d}+\mathbf{d}'-\mathbf{e}_j}+\text{terms of}~e\text{-degree}\leqslant|(\mathbf{a}+\mathbf{a}',\mathbf{b}+\mathbf{b}',\mathbf{c}+\mathbf{c}',\\&\mathbf{d}+\mathbf{d}')|_e-3)\otimes\bar 1_\chi.
\end{array}\]
Together with Corollary~\ref{rg} this shows that
\[\begin{array}{cl}
&\theta_i(1_\chi)\theta_j(\bar 1_\chi)-(-1)^{|\theta_i||\theta_j|}\theta_j(\bar 1_\chi)\theta_i(\bar 1_\chi)\\
=&([\bar Y_i,\bar Y_j]+\sum\limits_{\mbox{\tiny $\begin{array}{c}|\mathbf{i},\mathbf{j},\mathbf{f},\mathbf{g}|_e=m_i+m_j+2,\\|\mathbf{i}|+|\mathbf{j}|+|\mathbf{f}|+|\mathbf{g}|\geqslant 2\end{array}$}}\mu_{\mathbf{i},\mathbf{j},\mathbf{f},\mathbf{g}}\bar x^{\mathbf{i}}
\bar y^{\mathbf{j}}\bar u^{\mathbf{f}}\bar v^{\mathbf{g}}+\sum\limits_{|(\mathbf{i},\mathbf{j},\mathbf{f},\mathbf{g})|_e<m_i+m_j+2}\mu_{\mathbf{i},\mathbf{j},\mathbf{f},\mathbf{g}}\bar x^{\mathbf{i}}
\bar y^{\mathbf{j}}\bar u^{\mathbf{f}}\bar v^{\mathbf{g}})\\ &\otimes\bar 1_\chi,
\end{array}\]
where $\mu_{\mathbf{i},\mathbf{j},\mathbf{f},\mathbf{g}}\in\mathds{k}$. As a consequence, $\pi_{m_i+m_j+2,m_i+m_j}([\theta_i,\theta_j]-\sum\limits_{k=1}^{l+q}\alpha_{ij}^k\theta_k)=0$. On the other hand, part (1) of this proof shows that there exists a unique truncated polynomial $\tilde{q}_{ij}$ in $\bar Y_1,\cdots,\bar Y_{l+q+1}$ such that $$[\theta_i,\theta_j]-\sum\limits_{k=1}^{l+q}\alpha_{ij}^k\theta_k
=\tilde{q}_{ij}(\theta_1,\cdots,\theta_{l+q+1}).$$Moreover, by the preceding remark it is obvious that the linear part of $\tilde{q}_{ij}$ involves only those $\bar Y_1,\cdots,\bar Y_{l+q+1}$ whose weights $<m_i+m_j$. Hence there exists a truncated polynomial $q_{ij}$ in $l+q+1$ variables with initial form of degree at least $2$ such that
$$[\theta_i,\theta_j]-\sum\limits_{k=1}^{l+q}\alpha_{ij}^k\theta_k-q_{ij}(\theta_1,\cdots,\theta_{l+q+1})\in H^{m_i+m_j+1}.$$

It is immediate from the assumption in Section 3.1 that $[\bar v_{\frac{r+1}{2}},\bar v_{\frac{r+1}{2}}]\otimes\bar 1_\chi= \langle\bar  v_{\frac{r+1}{2}},\bar v_{\frac{r+1}{2}}\rangle\otimes\bar 1_\chi=1\otimes\bar 1_\chi$. Then by \eqref{1in} we have $$[\theta_{l+q+1},\theta_{l+q+1}](\bar 1_\chi)=2(\theta_{l+q+1}(\bar 1_\chi))^2=2\bar v_{\frac{r+1}{2}}^2\otimes\bar 1_\chi=[\bar v_{\frac{r+1}{2}},\bar v_{\frac{r+1}{2}}]\otimes\bar 1_\chi=1\otimes\bar 1_\chi,$$
i.e. $[\theta_{l+q+1},\theta_{l+q+1}]=\text{id}$. All these complete the proof.
\end{proof}

Compared with the Lie algebra case, the results obtained in part (1) of Theorem~\ref{reduced Wg} are similar to the ones on the reduced $W$-algebras over positive characteristic field $\mathds{k}$ (see [\cite{P2}], Section 3). However, part (2) is a new case. Compared with the reduced $W$-algebra case, the structure of reduced $W$-superalgebras has greatly changed. Therefore, the parity of $\text{dim}~\mathfrak{g}_\mathds{k}(-1)_{\bar1}$ plays the key role for the construction of reduced $W$-superalgebras, which makes the structure and representation theory of reduced $W$-superalgebras distinguish from that of reduced $W$-algebras.

\begin{rem}\label{redun}
When dim~$\mathfrak{g}_\mathds{k}(-1)_{\bar{1}}$ is odd, it is notable that $[\theta_i,\theta_j]=-(-1)^{|\theta_i||\theta_j|}[\theta_j,\theta_i]$ for any $1\leqslant  i,j\leqslant  l+q+1$. In particular, we can deduce that $[\theta_i,\theta_i]=0$ for $1\leqslant  i\leqslant  l$ as $\theta_i$ is an even element in $U_\chi(\mathfrak{g}_\mathds{k},e)$. Therefore, after deleting all the redundant commutating relations which are equivalent to each other in Theorem~\ref{reduced Wg}, what left are the cases when $i,j$ satisfy $1\leqslant  i<j\leqslant  l+q+1$ and $l+1\leqslant  i=j\leqslant  l+q+1$. When dim~$\mathfrak{g}_\mathds{k}(-1)_{\bar{1}}$ is even, similar conclusion can also be reached.
\end{rem}

\section{The structure of finite $W$-superalgebras over the field of complex numbers}

In this part, we will generalize the results obtained in Section 5 to the field of complex numbers. Some consequences about the $A$-algebra $U(\mathfrak{g}_A,e)$ where $A$ is an admissible algebra (which was introduced in Section 3.1) are also included in this part.

We will adopt the first definition of finite $W$-superalgebras over $\mathbb{C}$ (i.e. $U(\mathfrak{g},e)=(\text{End}_\mathfrak{g}Q_\chi)^{\text{op}}$ by Definition~\ref{W-C}) here. Given $(\mathbf{a},\mathbf{b},\mathbf{c},\mathbf{d})\in\mathbb{Z}_+^m\times\mathbb{Z}_2^n\times\mathbb{Z}_+^s\times\mathbb{Z}_2^t$ (recall that $t=\lfloor\frac{\text{dim}\,\mathfrak{g}(-1)_{\bar1}}{2}\rfloor$) we let $x^\mathbf{a}y^\mathbf{b}u^\mathbf{c}v^\mathbf{d}$ denote the monomial$$x_1^{a_1}\cdots x_m^{a_m}y_1^{b_1}\cdots y_n^{b_n}u_1^{c_1}\cdots u_s^{c_s}v_1^{d_1}\cdots v_t^{d_t}$$in $U(\mathfrak{g})$. Moreover, the monomials $x^\mathbf{i}y^\mathbf{j}u^\mathbf{f}v^\mathbf{g}$ with $(\mathbf{i},
\mathbf{j},\mathbf{f},\mathbf{g})\in\mathbb{Z}_+^m\times\mathbb{Z}_2^n\times\mathbb{Z}_+^s\times\mathbb{Z}_2^t$ form a PBW basis of the vector space $Q_\chi$ over $\mathbb{C}$.

Assume that $\text{wt}(x_i)=k_i$, $\text{wt}(y_j)=k'_j$, i.e. $x_i\in\mathfrak{g}(k_i)_{\bar0}$ and $y_j\in\mathfrak{g}(k'_j)_{\bar1}$ where $1\leqslant  i\leqslant  m$ and $1\leqslant  j\leqslant  n$. Given $(\mathbf{a},\mathbf{b},\mathbf{c},\mathbf{d})\in\mathbb{Z}_+^m\times\mathbb{Z}_2^n\times\mathbb{Z}_+^s\times\mathbb{Z}_2^t$, set
$$|(\mathbf{a},\mathbf{b},\mathbf{c},\mathbf{d})|_e:=\sum_{i=1}^ma_i(k_i+2)+\sum_{i=1}^nb_i(k'_i+2)+\sum_{i=1}^sc_i+\sum_{i=1}^td_i,$$
and say $x^{\mathbf{a}}y^\mathbf{b}u^\mathbf{c}v^\mathbf{d}$ to have $e$-degree $|(\mathbf{a},\mathbf{b},\mathbf{c},\mathbf{d})|_e$, which is written as
$\text{deg}_e(x^{\mathbf{a}}y^\mathbf{b}u^\mathbf{c}v^\mathbf{d})=|(\mathbf{a},\mathbf{b},\mathbf{c},\mathbf{d})|_e$.

\subsection{Some Lemmas}

For the $\mathbb{C}$-algebra $U(\mathfrak{g})$, by the same discussion as the proof of Lemma~\ref{commutative relations k1} we can obtain that

\begin{lemma}\label{commutative}
For any homogeneous element $w\in U(\mathfrak{g})_{i}$ ($i\in\mathbb{Z}_2$), we have
$$w\cdot x^{\mathbf{a}}y^\mathbf{b}u^\mathbf{c}v^\mathbf{d}=\sum_{\mathbf{i}\in\mathbb{Z}_+^m}\sum_{j_1=0}^{b_1}\cdots\sum_{j_n=0}^{b_n}\left(\begin{array}{@{\hspace{0pt}}c@{\hspace{0pt}}} \mathbf{a}\\ \mathbf{i}\end{array}\right)x^{\mathbf{a}-\mathbf{i}}y^{\mathbf{b}-\mathbf{j}}\cdot[wx^{\mathbf{i}}y^{\mathbf{j}}]\cdot u^\mathbf{c}v^\mathbf{d},$$
where $\mathbf{a}\choose\mathbf{i}$$=\prod\limits_{l'=1}^m$$a_{l'}\choose i_{l'}$, and $$[wx^{\mathbf{i}}y^{\mathbf{j}}]=k_{1,b_1,j_1}\cdots k_{n,b_n,j_n}(-1)^{|\mathbf{i}|}(\text{ad}y_n)^{j_n}\cdots(\text{ad}y_1)^{j_1}(\text{ad}x_m)^{i_m}\cdots(\text{ad}x_1)^{i_1}(w),$$
in which the coefficients $k_{1,b_1,j_1},\cdots, k_{n,b_n,j_n}\in\mathbb{C}$ (recall that $\mathbf{b}=(b_1,\cdots,b_n)\in\mathbb{Z}_2^n$) and the indices $j_1,\cdots,j_n\in\{0,1\}$. If we write $j_0=0$, then
$$k_{t',0,0}=1, k_{t',0,1}=0, k_{t',1,0}=(-1)^{|w|+j_1+\cdots+j_{{t'}-1}}, k_{t',1,1}=(-1)^{|w|+1+j_1+\cdots+j_{{t'}-1}},$$
where $1\leqslant {t'}\leqslant n$.
\end{lemma}

\begin{rem}
When we put the canonical basis of $U(\mathfrak{g})$ as a product of the elements of $\mathfrak{g}$ over $\mathbb{C}$, the indices of all the even elements of $\mathfrak{g}$ are taken from $\mathbb{Z}_+$. As for the canonical basis of $U_\chi(\mathfrak{g}_\mathds{k})$ over $\mathds{k}$, the indices of even elements of $\mathfrak{g}_\mathds{k}$ can only be taken from the set $\{0,1,\cdots,p-1\}$. However, no matter what the base field is (i.e. the field of complex numbers or the positive characteristic field), the indices of odd elements are taken from the set $\{0,1\}$.
\end{rem}

Let $\tilde{\rho}_\chi$ denote the representation of $U(\mathfrak{g})$ in $\text{End}Q_\chi$.

\begin{lemma}\label{com c}
Let $(\mathbf{a},\mathbf{b},\mathbf{c},\mathbf{d}),~(\mathbf{a}',\mathbf{b}',\mathbf{c}',\mathbf{d}')\in\mathbb{Z}_+^m\times\mathbb{Z}_2^n\times\mathbb{Z}_+^s\times\mathbb{Z}_2^t$ be such that $|(\mathbf{a},\mathbf{b},\mathbf{c},\mathbf{d})|_e=A,~|(\mathbf{a}',\mathbf{b}',\mathbf{c}',\mathbf{d}')|_e=B$. Then
\[\begin{array}{ccl}(\tilde{\rho}_\chi(x^{\mathbf{a}}y^\mathbf{b}u^\mathbf{c}v^\mathbf{d}))(x^{\mathbf{a}'}y^{\mathbf{b}'}u^{\mathbf{c}'}v^{\mathbf{d}'}\otimes1_\chi)&=&
(Cx^{\mathbf{a}+\mathbf{a}'}y^{\mathbf{b}+\mathbf{b}'}u^{\mathbf{c}+\mathbf{c}'}v^{\mathbf{d}+\mathbf{d}'}+\text{terms of}~e\text{-degree}\\ &&\leqslant A+B-2)\otimes1_\chi,
\end{array}\]
where the coefficient $C\in\mathbb{C}$ is defined by:

(1) if $(\mathbf{b}+\mathbf{b}',\mathbf{d}+\mathbf{d}')\notin\mathbb{Z}_2^n\times\mathbb{Z}_2^t$, then $C=0$.

(2) if $C\neq0$, then each entry in $(b_1+b'_1,\cdots,b_n+b'_n,d_1+d'_1,\cdots,d_t+d'_t)$ is taken from the set $\{0,1\}$. Delete all the zero terms in $(\mathbf{b},\mathbf{d},\mathbf{b}',\mathbf{d}')$ and let $\tau(\mathbf{b},\mathbf{d},\mathbf{b}',\mathbf{d}')$ denote the inverse number of
$(\mathbf{b},\mathbf{d},\mathbf{b}',\mathbf{d}')$ with respect to the sequence $(b_1,b'_1,b_2,b'_2,\cdots,b_n,b'_n,$\\$d_1,d'_1,d_2,d'_2\cdots,d_t,d'_t)$, then
$C=(-1)^{\tau(\mathbf{b},\mathbf{d},\mathbf{b}',\mathbf{d}')}$.

\end{lemma}

\begin{proof}
Now repeat the proof of Lemma~\ref{commutative relations k2} applying Lemma~\ref{commutative} in place of Lemma~\ref{commutative relations k1}.
\end{proof}

Any $0\neq h\in U(\mathfrak{g},e)$ is uniquely determined by its value $h(1_\chi)\in Q_\chi$. For $h\neq0$ we let $n(h),~N(h)$ and $\Lambda_h^{\text{max}}$ have the same meaning as the assumptions preceding Lemma~\ref{hw}.

\begin{lemma}\label{hwc}
Let $h\in U(\mathfrak{g},e)\backslash\{0\}$ and $(\mathbf{a},\mathbf{b},\mathbf{c},\mathbf{d})\in \Lambda^{\text{max}}_h$. Then $\mathbf{c}=\mathbf{0}$, and $\mathbf{a}\in\mathbb{Z}_+^{l}\times\{\mathbf{0}\},~\mathbf{b}\in\mathbb{Z}_2^{q}\times\{\mathbf{0}\}$. Moreover, the sequence $\mathbf{d}$ satisfies

(1) $\mathbf{d}=\mathbf{0}$ when dim~$\mathfrak{g}(-1)_{\bar{1}}$ is even;

(2) $\mathbf{d}=\{\mathbf{0}\}_{\frac{r-1}{2}}\times\mathbb{Z}_2\times\{\mathbf{0}\}_{\frac{r-1}{2}}$ when dim~$\mathfrak{g}(-1)_{\bar{1}}$ (recall that which equals to $r$) is odd.
\end{lemma}

\begin{proof}
Repeat verbatim the proof of Lemma~\ref{hw} but apply Lemma~\ref{com c} in place of Lemma~\ref{commutative relations k2}.
\end{proof}

\subsection{The construction theory of finite $W$-superalgebras over the field of complex numbers}

For $k\in\mathbb{Z}_+$ we denote by $\tilde{H}^k$ the linear span of all $0\neq h\in U(\mathfrak{g},e)$ with $n(h)\leqslant  k$. Put $\tilde{H}^{-1}=0$. It follows from Lemma~\ref{com c} that the subspaces $\{\tilde{H}^i|i\in\mathbb{Z}_+\}$ form a filtration of the algebra $U(\mathfrak{g},e)$. Moreover, Lemma~\ref{com c} implies that the graded algebra $\text{gr}(U(\mathfrak{g},e))=\bigoplus\limits_{i\geqslant 0}\tilde{H}^i/\tilde{H}^{i-1}$ is supercommutative.

Recall that $\{x_1,\cdots,x_l\}$ and $\{y_1,\cdots,y_q\}$ are $\mathbb{C}$-basis of $\mathfrak{g}^e_{\bar{0}}$ and $\mathfrak{g}^e_{\bar{1}}$, respectively. When dim~$\mathfrak{g}(-1)_{\bar{1}}$ is odd, there is  $v_{\frac{r+1}{2}}\in\mathfrak{g}(-1)_{\bar{0}}\cap(\mathfrak{g}(-1)'_{\bar{0}})^{\bot }$. Set
\[Y_i:=\left\{
\begin{array}{ll}
x_i&\text{if}~1\leqslant  i\leqslant  l;\\
y_{i-l}&\text{if}~l+1\leqslant  i\leqslant  l+q;\\
v_{\frac{r+1}{2}}&\text{if}~i=l+q+1.
\end{array}
\right.
\]

For any $1\leqslant  i\leqslant  l+q+1$, assume that $Y_i\in\mathfrak{g}(m_i)$ where $m_i\in\mathbb{Z}$, and the term $Y_{l+q+1}$ occurs only when dim~$\mathfrak{g}(-1)_{\bar{1}}$ is odd. By the assumption it is immediate that $Y_i\in\mathfrak{g}^e$ for $1\leqslant  i\leqslant  l+q$. Let $\tilde{\Theta}_i$ denote the image of $\Theta_i\in U(\mathfrak{g},e)$ in $\text{gr}(U(\mathfrak{g},e))$.

\begin{theorem}\label{PBWC}
Let $U(\mathfrak{g},e)$ be a finite $W$-superalgebra over $\mathbb{C}$, then

(1) if $\text{dim}~\mathfrak{g}(-1)_{\bar1}$ is even,

(i) there exist homogeneous elements $\Theta_1,\cdots,\Theta_{l+q}\in U(\mathfrak{g},e)$, where $\Theta_1,\cdots,\Theta_{l}\in U(\mathfrak{g},e)_{\bar0}$ and $\Theta_{l+1},\cdots,\Theta_{l+q}\in U(\mathfrak{g},e)_{\bar1}$ such that
$$\Theta_k(1_\chi)=(Y_k+\sum\limits_{\mbox{\tiny $\begin{array}{c}|\mathbf{a},\mathbf{b},\mathbf{c},\mathbf{d}|_e=m_k+2,\\|\mathbf{a}|
+|\mathbf{b}|+|\mathbf{c}|+|\mathbf{d}|\geqslant 2\end{array}$}}\lambda^k_{\mathbf{a},\mathbf{b},\mathbf{c},\mathbf{d}}x^{\mathbf{a}}
y^{\mathbf{b}}u^{\mathbf{c}}v^{\mathbf{d}}+\sum\limits_{|\mathbf{a},\mathbf{b},\mathbf{c},\mathbf{d}|_e<m_k+2}\lambda^k_{\mathbf{a},\mathbf{b},\mathbf{c},\mathbf{d}}x^{\mathbf{a}}
y^{\mathbf{b}}u^{\mathbf{c}}v^{\mathbf{d}})\otimes1_\chi$$for $1\leqslant  k\leqslant  l+q$, where $\lambda^k_{\mathbf{a},\mathbf{b},\mathbf{c},\mathbf{d}}\in\mathbb{Q}$, and $\lambda^k_{\mathbf{a},\mathbf{b},\mathbf{c},\mathbf{d}}=0$ if $a_{l+1}=\cdots=a_m=b_{q+1}=\cdots=b_n=c_1=\cdots=c_s=
d_1=\cdots=d_{\frac{r}{2}}=0$.

(ii) the monomials $\Theta_1^{a_1}\cdots\Theta_l^{a_l}\Theta_{l+1}^{b_1}\cdots\Theta_{l+q}^{b_q}$ with $a_i\in\mathbb{Z}_+,~b_j\in\mathbb{Z}_2$ for $1\leqslant i\leqslant l$ and $1\leqslant j\leqslant q$ form a basis of $U(\mathfrak{g},e)$ over $\mathbb{C}$.

(iii) for $1\leqslant  i\leqslant  l+q$, the elements $\tilde{\Theta}_i=\Theta_i+\tilde{H}^{m_i+1}\in\text{gr}(U(\mathfrak{g},e))$ are algebraically independent and generate $\text{gr}(U(\mathfrak{g},e))$. In particular, $\text{gr}(U(\mathfrak{g},e))$ is a graded polynomial superalgebra with homogeneous generators of degrees $m_1+2,\cdots,m_{l+q}+2$.

(iv) let $1\leqslant  i,j\leqslant  l+q$. Then
$$[\Theta_i,\Theta_j]=\Theta_i\cdot\Theta_j-(-1)^{|\Theta_i||\Theta_j|}\Theta_j\cdot\Theta_i=(-1)^{|\Theta_i||\Theta_j|}\Theta_j\circ\Theta_i-\Theta_i\circ\Theta_j\in\tilde{H}^{m_i+m_j+2}.$$
Moreover, if the elements $Y_i,~Y_j\in \mathfrak{g}^e$ for $1\leqslant  i,j\leqslant  l+q$ satisfy $[Y_i,Y_j]=\sum\limits_{k=1}^{l+q}\alpha_{ij}^kY_k$ in $\mathfrak{g}^e$, then
\begin{equation}\label{Theta2}
[\Theta_i,\Theta_j]\equiv\sum\limits_{k=1}^{l+q}\alpha_{ij}^k\Theta_k+q_{ij}(\Theta_1,\cdots,\Theta_{l+q})\qquad(\text{mod}~\tilde{H}^{m_i+m_j+1}),
\end{equation}where $q_{ij}$ is a polynomial in $l+q$ variables in $\mathbb{Q}$ whose constant term and linear part are zero, and the modulo part is a polynomial in $\Theta_1,\cdots,\Theta_{l+q}$.

(2) if $\text{dim}~\mathfrak{g}(-1)_{\bar1}$ is odd,

(i) there exist homogeneous elements $\Theta_1,\cdots,\Theta_{l+q+1}\in U(\mathfrak{g},e)$, where $\Theta_1,\cdots,\Theta_{l}\in U(\mathfrak{g},e)_{\bar0}$ and $\Theta_{l+1},\cdots,\Theta_{l+q+1}\in U(\mathfrak{g},e)_{\bar1}$ such that

(a) let $1\leqslant  k\leqslant  l+q$. Then
$$\Theta_k(1_\chi)=(Y_k+\sum\limits_{\mbox{\tiny $\begin{array}{c}|\mathbf{a},\mathbf{b},\mathbf{c},\mathbf{d}|_e=m_k+2,\\|\mathbf{a}|
+|\mathbf{b}|+|\mathbf{c}|+|\mathbf{d}|\geqslant 2\end{array}$}}\lambda^k_{\mathbf{a},\mathbf{b},\mathbf{c},\mathbf{d}}x^{\mathbf{a}}
y^{\mathbf{b}}u^{\mathbf{c}}v^{\mathbf{d}}+\sum\limits_{|\mathbf{a},\mathbf{b},\mathbf{c},\mathbf{d}|_e<m_k+2}\lambda^k_{\mathbf{a},\mathbf{b},\mathbf{c},\mathbf{d}}x^{\mathbf{a}}
y^{\mathbf{b}}u^{\mathbf{c}}v^{\mathbf{d}})\otimes1_\chi;$$

(b) for the case $k=l+q+1$ we have
$$\Theta_{l+q+1}(1_\chi)=v_{\frac{r+1}{2}}\otimes1_\chi,$$
where $\lambda^k_{\mathbf{a},\mathbf{b},\mathbf{c},\mathbf{d}}\in\mathbb{Q}$, and $\lambda^k_{\mathbf{a},\mathbf{b},\mathbf{c},\mathbf{d}}=0$ if $a_{l+1}=\cdots=a_m=b_{q+1}=\cdots=b_n=c_1=\cdots=c_s=
d_1=\cdots=d_{\frac{r-1}{2}}=0$.

(ii) the monomials $\Theta_1^{a_1}\cdots\Theta_l^{a_l}\Theta_{l+1}^{b_1}\cdots\Theta_{l+q}^{b_q}\Theta_{l+q+1}^{c}$ with $a_i\in\mathbb{Z}_+,~b_j,c\in\mathbb{Z}_2$ for $1\leqslant i\leqslant l$ and $1\leqslant j\leqslant q$ form a basis of $U(\mathfrak{g},e)$.

(iii) for $1\leqslant  i\leqslant  l+q+1$, the elements $\tilde{\Theta}_i=\Theta_i+\tilde{H}^{m_i+1}\in\text{gr}(U(\mathfrak{g},e))$ are algebraically independent and generate $\text{gr}(U(\mathfrak{g},e))$. In particular, $\text{gr}(U(\mathfrak{g},e))$ is a graded polynomial superalgebra with homogeneous generators of degrees $m_1+2,\cdots,m_{l+q+1}+2$.

(iv) let $1\leqslant  i,j\leqslant  l+q+1$. Then $$[\Theta_i,\Theta_j]=(-1)^{|\Theta_i||\Theta_j|}\Theta_j\circ\Theta_i-\Theta_i\circ\Theta_j\in\tilde{H}^{m_i+m_j+2}.$$
Moreover, if the elements $Y_i,~Y_j\in \mathfrak{g}^e$ for $1\leqslant  i,j\leqslant  l+q$ satisfy $[Y_i,Y_j]=\sum\limits_{k=1}^{l+q}\alpha_{ij}^kY_k$ in $\mathfrak{g}^e$, then
\begin{equation}\label{Thetacom2}
[\Theta_i,\Theta_j]\equiv\sum\limits_{k=1}^{l+q}\alpha_{ij}^k\Theta_k+q_{ij}(\Theta_1,\cdots,\Theta_{l+q+1})\qquad(\text{mod}~\tilde{H}^{m_i+m_j+1}),
\end{equation}where $q_{ij}$ is a polynomial in $l+q+1$ variables in $\mathbb{Q}$ whose constant term and linear part are zero, and the modulo part is a polynomial in $\Theta_1,\cdots,\Theta_{l+q+1}$. For the case $i=j=l+q+1$, we have $[\Theta_{l+q+1},\Theta_{l+q+1}]=\text{id}$.
\end{theorem}

Given an admissible algebra $A$ and a prime $p\in\Pi(A)$, the ``modular $p$'' version of this theorem (the structure theory of $\mathds{k}$-algebra $U_\chi(\mathfrak{g}_\mathds{k},e)$) has been formulated in Theorem~\ref{reduced Wg} for $\mathds{k}=\overline{\mathbb{F}}_p$. In ([\cite{P2}], Section 4), Premet obtained the PBW theorem for the finite $W$-algebras over $\mathbb{C}$ through the procedure of ``admissible'' based on the knowledge of reduced $W$-algebras over $\mathds{k}$. Inspired by Premet's treatment of finite $W$-algebras, the knowledge of reduced $W$-superalgebras over $\mathds{k}$ obtained in Section 5 can be applied to prove the theorem. Since the choice of admissible algebra $A$ has nothing to do with the super property of Lie superalgebra $\mathfrak{g}$, we can take the same steps as the Lie algebra case.

In the language of ([\cite{P2}], Section 4.2), first define the admissible algebra $A=\mathbb{Z}[\frac{1}{N!}]$ by choosing a sufficient large integer $N$, then select a prime $p\gg N$ (this is exactly the example given in Section 3.1 when we introduce the admissible algebras). For part (i) in both parts of the theorem, one can translate the formulas to a system of linear equations over $\mathbb{Q}$, then discuss the existence of solution for these equations, respectively. It can be inferred from Corollary~\ref{rg} that the system of linear equations has a solution over $\mathbb{F}_p\subset\mathds{k}$. As this is true for almost all primes in the set $\Pi(A)$ (which contains infinite elements), we can conclude that the former system has a solution over $\mathbb{Q}$. Recall that $\tilde{H}^k$ for $k\in\mathbb{Z}_+$ forms Kazhdan filtration of $\mathbb{C}$-algebra $U(\mathfrak{g},e)$ (see the remark preceding Theorem~\ref{PBWC}). As for (ii), (iii) and (iv) in the theorem, these results can be obtained by induction based on the Kazhdan degree (i.e. the $e$-degree) of $\tilde{H}^k$.

Now we will sketch a proof here following ([\cite{P2}], Theorem 4.6).
\begin{proof}
The proof is based on the results obtained in Section 5 and Section 6. Repeat verbatim the proof of Theorem 4.6 in [\cite{P2}] but apply Theorem~\ref{reduced Wg}, Lemma~\ref{com c} and Lemma~\ref{hwc} in place of Theorem 3.4, Lemma 4.4 and Lemma 4.5 in [\cite{P2}] respectively. For more detail refer to ([\cite{P2}], Theorem 4.6).
\end{proof}

\begin{rem}
Notice that the reduced $W$-superalgebra $U_\chi(\mathfrak{g}_\mathds{k},e)$ over $\mathds{k}$ is finite-dimensional and we have to calculate the dimension of $U_\chi(\mathfrak{g}_\mathds{k},e)$ as a $\mathds{k}$-vector space in the proof of Proposition~\ref{reduced basis}. However, the dimension of reduced $W$-superalgebra $U(\mathfrak{g},e)$ over $\mathbb{C}$ is infinite. Then Theorem~\ref{PBWC} can not be established by the same means as Theorem~\ref{reduced Wg} since we can not get similar conclusion as Proposition~\ref{reduced basis} by comparing the dimension of corresponding algebras directly.
\end{rem}

It is notable that the theory of finite $W$-superalgbras associated to queer Lie superalgebra $\mathfrak{q}_n$ over $\mathbb{C}$ has been initiated and systemically developed by Zhao in [\cite{Z2}]. He pointed out that some of the constructions for $\mathfrak{q}_n$ admit natrual generalizations in basic classical Lie superalgebras when $\text{dim}~\mathfrak{g}(-1)_{\bar1}$ is even. Therefore, in this case we can introduce the cohomology definition for the finite $W$-superalgebra associated to a basic classical Lie superalgebra by $U(\mathfrak{g},e):=H^0(\mathfrak{m},Q_\chi)$ and call $\mathcal {S}:=\chi+\text{ker\,ad}^*f$ the Slodowy slice through $\chi$. Following Zhao's treatment to the queer Lie superalgebra in ([\cite{Z2}], Theorem 3.5), one can get an isomorphism of graded $\mathbb{C}$-superalgebras under the Kazhdan filtration, i.e.

\begin{lemma}$^{[\cite{Z2}]}$\label{eveniso}
When $\text{dim}~\mathfrak{g}(-1)_{\bar1}$ is even, the map$$\text{gr}~U(\mathfrak{g},e)\rightarrow \mathbb{C}[\mathcal {S}]$$is an isomorphism of graded superalgebras. Moreover,
$$H^i(\mathfrak{m},Q_\chi)=H^i(\mathfrak{m},\text{gr}Q_\chi)=0$$ for any $i>0$.
\end{lemma}

Apart from Premet's treatment, Brundan-Goodwin-Kleshchev also introduced the PBW theorem for the finite $W$-algebras in a more direct way in ([\cite{BGK}], Section 3.2). Moreover, it is remarkable that the cohomology definition of finite $W$-algebras plays the key role in their proof. In the case when $\text{dim}~\mathfrak{g}(-1)_{\bar1}$ is even, if we follow their treatment and apply Lemma~\ref{eveniso}, the PBW theorem of finite $W$-superalgebras over $\mathbb{C}$ (i.e. the conclusions (i)-(iv) of (1) in Theorem~\ref{PBWC}) can also be formulated with less effort. All these will greatly simplify the proof. The reason why we did not adopt that method is based on the following considerations:

(1) if the proof of Theorem~\ref{PBWC} is carried in that way, the coefficients of each monomial for the generators of finite $W$-superalgebras can only be guaranteed over $\mathbb{C}$, but not over $\mathbb{Q}$. In the study of related topics on finite $W$-superalgebras in positive characteristic, we need to construct an admissible algebra, and the most critical step is to ensure all the numbers occurred be in $\mathbb{Q}$. However, it is difficult to find an efficient way to achieve this by that means.

(2) another important reason is that when $\text{dim}~\mathfrak{g}(-1)_{\bar1}$ is odd, one can observe that $\mathfrak{g}^e\oplus\mathbb{C}v_{\frac{r+1}{2}}$ is not necessary a subalgebra of $\mathfrak{g}$. Therefore, the lack of the cohomology definition of finite $W$-superalgebras makes it difficult to reach the same conclusion as Lemma~\ref{eveniso}. Without this result, it is hard to obtain the PBW basis of $U(\mathfrak{g},e)$. In fact, Zhao noticed that a key lemma in the establishment of the cohomology definition of finite $W$-superalgebras may go wrong for the case when $\text{dim}~\mathfrak{g}(-1)_{\bar1}$ is odd, see ([\cite{Z2}], Remark 3.11).

Based on above considerations, we mainly follow Premet's treatment of the finite $W$-algebra case in the proof of Theorem~\ref{PBWC}.

Recall that in Section 3.4 we have introduced the restricted root decomposition for the basic classical Lie superalgebra $\mathfrak{g}$ associated to the even nilpotent element $e$. Based on which, the construction of the generators for the finite $W$-superalgebra $U(\mathfrak{g},e)$ given in Theorem~\ref{PBWC} can be refined, i.e.

\begin{lemma}\label{refi}
Let $U(\mathfrak{g},e)$ be a finite $W$-superalgebra over $\mathbb{C}$. The generators of $U(\mathfrak{g},e)$ introduced in Theorem~\ref{PBWC} can be chosen as $\mathfrak{t}^e$-weight vectors satisfying

(1) $\Theta_k$ has the same $\mathfrak{t}^e$-weight as $Y_k$ for $1\leqslant  k\leqslant  l+q$;

(2) $\Theta_{l+q+1}$ has the same $\mathfrak{t}^e$-weight as $v_{\frac{r+1}{2}}$ for the case when $d_1$ is odd.

\end{lemma}

\begin{proof}
The proof of part (1) is a straightforward generalization of ([\cite{P3}], Lemma 2.2), thus will be omitted. Part (2) can be easily observed by the definition of $\Theta_{l+q+1}$ in Theorem~\ref{PBWC}(2)(i)(b).
\end{proof}

Following Gan-Ginzburg's treatment for the finite $W$-algebra case in ([\cite{GG}], Section 2.1), we can define a linear action of $\mathbb{C}^*$ on $\mathfrak{g}$. First, consider the Lie algebra homomorphism $\mathfrak{sl}_2\longrightarrow\mathfrak{g}_{\bar{0}}$ defined by
$$\left(\begin{array}{cc} 0 & 1\\0 & 0\end{array}\right)\mapsto e,
\qquad \left(\begin{array}{@{\hspace{0pt}}c@{\hspace{8pt}}c@{\hspace{0pt}}} 1 & 0\\0 & -1\end{array}\right)\mapsto h,
\qquad \left(\begin{array}{cc} 0 & 0\\1 & 0\end{array}\right)\mapsto f.$$
This Lie algebra homomorphism exponentiates to a rational homomorphism $\tilde{\gamma}:~SL_2\longrightarrow G_{\text{ev}}\longrightarrow G$. Put
$$\gamma:\mathbb{C}^*\longrightarrow G,\qquad \gamma(t)=\tilde{\gamma}
\left(\begin{array}{@{\hspace{0pt}}c@{\hspace{8pt}}c@{\hspace{0pt}}} t & 0\\0 & t^{-1}\end{array}\right),
\qquad \forall~t\in\mathbb{C}^*.$$

Define $\sigma:=\text{Ad}\gamma(-1)$. For any $i\in\mathbb{Z}$, let $x\in\mathfrak{g}(i)$ be a homogeneous element in $\mathfrak{g}$. Since the Dynkin grading of the Lie superalgebra $\mathfrak{g}$ is obtained by the action of ad$h$, we have $\sigma(x)=(-1)^ix$, i.e. $\sigma$ is an element of order  $\leqslant2$ in $\text{Ad}G$. As $\sigma$ perserves the left ideal $I_\chi$ of $U(\mathfrak{g})$ and the subalgebra $\mathfrak{m}$ of $\mathfrak{g}$, the action of $\sigma$ on $Q_\chi^{\text{ad}\mathfrak{m}}$ is an even automorphism of the finite $W$-superalgebra $U(\mathfrak{g},e)\cong Q_\chi^{\text{ad}\mathfrak{m}}$. Given $(\mathbf{a},\mathbf{b},\mathbf{c},\mathbf{d})\in\mathbb{Z}_+^m\times\mathbb{Z}_2^n\times\mathbb{Z}_+^s\times\mathbb{Z}_2^t$, we have
\begin{equation}\label{sigmadef}
\sigma(x^\mathbf{a}y^\mathbf{b}u^\mathbf{c}v^\mathbf{d}\otimes1_\chi)=(-1)^{|(\mathbf{a},\mathbf{b},\mathbf{c},\mathbf{d})|_e}
x^\mathbf{a}y^\mathbf{b}u^\mathbf{c}v^\mathbf{d}\otimes1_\chi
\end{equation}
for any $x^\mathbf{a}y^\mathbf{b}u^\mathbf{c}v^\mathbf{d}\otimes1_\chi\in Q_\chi$.

\begin{prop}\label{fijeo}
Retain the notations in Theorem~\ref{PBWC}. The following are true:

(1) when $\text{dim}~\mathfrak{g}(-1)_{\bar1}$ is even, define the polynomials $F_{ij}\in\mathbb{Q}[\Theta_1,\cdots,\Theta_{l+q}]$ by
$$F_{ij}(\Theta_1,\cdots,\Theta_{l+q}):=[\Theta_i,\Theta_j]$$for $1\leqslant i,j\leqslant l+q$.
Moreover, if $[Y_i,Y_j]=\sum\limits_{k=1}^{l+q}\alpha_{ij}^kY_k$ in $\mathfrak{g}^e$ for $1\leqslant i,j\leqslant l+q$, we have
\begin{equation}\label{refine1}
F_{ij}(\Theta_1,\cdots,\Theta_{l+q})\equiv\sum\limits_{k=1}^{l+q}\alpha_{ij}^k\Theta_k+q_{ij}(\Theta_1,\cdots,\Theta_{l+q})\qquad(\text{mod}~\tilde{H}^{m_i+m_j}).
\end{equation}

(2) when $\text{dim}~\mathfrak{g}(-1)_{\bar1}$ is odd, define the polynomials $F_{ij}(\Theta_1,\cdots,\Theta_{l+q+1})\in\mathbb{Q}[\Theta_1,\cdots,$\\$\Theta_{l+q+1}]$ by
$$F_{ij}(\Theta_1,\cdots,\Theta_{l+q+1}):=[\Theta_i,\Theta_j]$$for $1\leqslant i,j\leqslant l+q+1$.
Moreover, if $[Y_i,Y_j]=\sum\limits_{k=1}^{l+q}\alpha_{ij}^kY_k$ in $\mathfrak{g}^e$ for $1\leqslant i,j\leqslant l+q$, we have \begin{equation}\label{refine2}
F_{ij}(\Theta_1,\cdots,\Theta_{l+q+1})\equiv\sum\limits_{k=1}^{l+q}\alpha_{ij}^k\Theta_k+q_{ij}(\Theta_1,\cdots,\Theta_{l+q+1})\qquad(\text{mod}~\tilde{H}^{m_i+m_j}).\end{equation}
For the case $i=j=l+q+1$ we have $F_{l+q+1,l+q+1}(\Theta_1,\cdots,\Theta_{l+q+1})=1\otimes1_\chi$.

\end{prop}

\begin{proof}
Compared with the results obtained in Theorem~\ref{PBWC}, it is notable that the modulo part in \eqref{Theta2} and \eqref{Thetacom2} are in $\tilde{H}^{m_i+m_j+1}$, while the modulo part in the corresponding place of \eqref{refine1} and \eqref{refine2} are in $\tilde{H}^{m_i+m_j}$. In fact, if choosing the generators of $U(\mathfrak{g},e)$ as what we have introduced in Lemma~\ref{refi}, it follows from \eqref{Theta2}, \eqref{Thetacom2} and \eqref{sigmadef} that \eqref{refine1} and \eqref{refine2} establish.

\end{proof}

\begin{rem}
When $\text{dim}~\mathfrak{g}(-1)_{\bar1}$ is odd, there are no obvious formulas for the leading term of $F_{i,l+q+1}(\Theta_1,\cdots,\Theta_{l+q+1})$ for $1\leqslant i\leqslant l+q$. For both cases in Proposition~\ref{fijeo}, we stress that although the $F_{ij}(\Theta_1,\cdots,\Theta_{l+q})'s$ (or $F_{ij}(\Theta_1,\cdots,\Theta_{l+q+1})'s$, respectively) live in an associative algebra which is, in general, non-supercommutative, the PBW theorem of $U(\mathfrak{g},e)$ in Theorem~\ref{PBWC} allows us to view the
$F_{ij}(\Theta_1,\cdots,\Theta_{l+q})'s$ (or $F_{ij}(\Theta_1,\cdots,\Theta_{l+q+1})'s$) as polynomials in $l+q$ (or $l+q+1$) variables with coefficients in $\mathbb{Q}$.
\end{rem}

Along the same discussion as Remark~\ref{redun}, if deleting the redundant ones, the number of defining relations for the generators of $U(\mathfrak{g},e)$ can be reduced. The following theorem completely characterizes the structure of finite $W$-superalgebra $U(\mathfrak{g},e)$ over the field of complex numbers.

\begin{theorem}\label{relationc}
The following are true:

(1) when $\text{dim}~\mathfrak{g}(-1)_{\bar1}$ is even, the finite $W$-superalgebra $U(\mathfrak{g},e)$ is generated by the homogeneous elements $\Theta_1,\cdots,\Theta_{l+q}$ (where $\Theta_1,\cdots,\Theta_{l}\in U(\mathfrak{g},e)_{\bar0}$ and $\Theta_{l+1},\cdots,\Theta_{l+q}\in U(\mathfrak{g},e)_{\bar1}$) subject to the relations
$$[\Theta_i,\Theta_j]=F_{ij}(\Theta_1,\cdots,\Theta_{l+q}),\qquad[\Theta_j,\Theta_i]=-(-1)^{|\Theta_i||\Theta_j|}[\Theta_i,\Theta_j],$$
where $1\leqslant i<j\leqslant l+q$ and $l+1\leqslant i=j\leqslant l+q$.

(2) when $\text{dim}~\mathfrak{g}(-1)_{\bar1}$ is odd, the finite $W$-superalgebra $U(\mathfrak{g},e)$ is generated by the homogeneous elements $\Theta_1,\cdots,\Theta_{l+q+1}$ (where $\Theta_1,\cdots,\Theta_{l}\in U(\mathfrak{g},e)_{\bar0}$ and $\Theta_{l+1},\cdots,\Theta_{l+q+1}$\\$\in U(\mathfrak{g},e)_{\bar1}$) subject to the relations
$$[\Theta_i,\Theta_j]=F_{ij}(\Theta_1,\cdots,\Theta_{l+q+1}),,\qquad[\Theta_j,\Theta_i]=-(-1)^{|\Theta_i||\Theta_j|}[\Theta_i,\Theta_j],$$
where $1\leqslant i<j\leqslant l+q+1$ and $l+1\leqslant i=j\leqslant l+q+1$.
\end{theorem}

\begin{proof}
Since the proof goes through for both cases, we will just consider part (2).

Let $I$ be the two-sided ideal of the free associative superalgebra $\mathbb{C}[T_1,\cdots,T_l;$\\$T_{l+1},\cdots,T_{l+q+1}]$ generated by all$$[T_i,T_j]-F_{ij}(T_1,\cdots,T_l;T_{l+1},\cdots,T_{l+q+1})
\quad(\text{where}~1\leqslant  i<j\leqslant  l+q+1)$$and$$[T_i,T_i]-F_{ii}(T_1,\cdots,T_l;T_{l+1},\cdots,T_{l+q+1})\quad(\text{where}~l+1\leqslant  i\leqslant  l+q+1).$$ Let $U:=\mathbb{C}[T_1,\cdots,T_l;T_{l+1},\cdots,T_{l+q+1}]/I$. With the same notation as Section 3.2, for $k\in\mathbb{Z}_+$ we let $\text{F}_kU(\mathfrak{g},e)$ denote the $\mathbb{C}$-span of all products $\Theta_{j_1}\cdots\Theta_{j_i}$ with $(m_{j_1}+2)+\cdots+(m_{j_i}+2)\leqslant k$ where $i\in\mathbb{Z}_+$ (recall that $m_{j_i}$ is the weight of $Y_{j_i}$). By the same discussion as the finite $W$-algebra case in ([\cite{P4}], Lemma 4.1) (i.e. argue by upward induction on the Kazhdan degree $k$ and downward induction on the number of elements $i$ for each product of the monomials based on \eqref{Thetacom2}), we have $U\cong U(\mathfrak{g},e)$ as $\mathbb{C}$-algebras.
\end{proof}

\subsection{Another definition of finite $W$-superalgebras}

In light of Gan-Ginzburg's definition of $W$-algebras over $\mathbb{C}$ in [\cite{GG}], Wang defined the reduced $W$-superalgebra over $\mathds{k}=\overline{\mathbb{F}}_p$ in a new way in ([\cite{W}], Remark 70), which he thought makes better sense (where it was called the modular $W$-superalgebra), i.e.

\begin{defn}$^{[\cite{W}]}$\label{rwc}
Define the reduced $W$-superalgebra over $\mathds{k}$ by
\[\begin{array}{ll}
&W'_{\chi,\mathds{k}}: =(Q^\chi_\chi)^{\text{ad}\,\mathfrak{m}'_{\mathds{k}}}.
\end{array}\]
\end{defn}

In light of Wang's definition, we can also define the corresponding finite $W$-superalgebra over the field of complex numbers.

\begin{defn}\label{rewcc}
Define the finite $W$-superalgebra over $\mathbb{C}$ by\[\begin{array}{ll}
&W'_\chi:=(U(\mathfrak{g})/I_\chi)^{\text{ad}\,\mathfrak{m}'}\cong Q_\chi^{\text{ad}\,\mathfrak{m}'}\\
\equiv&\{\bar{y}\in U(\mathfrak{g})/I_\chi|~[a,y]\in I_\chi, \forall a\in\mathfrak{m}'\},
\end{array}\]
where $\bar{y}_1\cdot\bar{y}_2:=\overline{y_1y_2}$ for $\bar{y}_1,\bar{y}_2\in W'_\chi$.
\end{defn}

\begin{rem}\label{QQ'}
When $\text{dim}~\mathfrak{g}(-1)_{\bar1}$ is even, it is immediate that $\mathfrak{m}'=\mathfrak{m}$ by definition. Thus we can obtain from Theorem~\ref{an iso} that $U(\mathfrak{g},e)\cong W'_\chi$ as $\mathbb{C}$-algebras. However, the situation changes in the case when $\text{dim}~\mathfrak{g}(-1)_{\bar1}$ is odd. Since $\mathfrak{m}$ is a proper subalgebra of $\mathfrak{m}'$, it follows that $W'_\chi$ is a subalgebra of $Q_\chi^{\text{ad}\mathfrak{m}}=U(\mathfrak{g},e)$. In fact, we have $$Q_\chi^{\text{ad}\mathfrak{m}'}=[v_{\frac{r+1}{2}},Q_\chi^{\text{ad}\mathfrak{m}}]$$as $\mathbb{C}$-algebras.
\end{rem}

\begin{proof}

Firstly, we claim that $Q_\chi^{\text{ad}\mathfrak{m}'}$ is strictly contained in $Q_\chi^{\text{ad}\mathfrak{m}}$. It is immediate from Theorem~\ref{PBWC}(2)(i) that $\Theta_{l+q+1}(1_\chi)=v_{\frac{r+1}{2}}\otimes1_\chi\in Q_\chi^{\text{ad}\mathfrak{m}}$. By definition we have $v_{\frac{r+1}{2}}\in\mathfrak{m}'$, and $[v_{\frac{r+1}{2}},v_{\frac{r+1}{2}}\otimes1_\chi]=[v_{\frac{r+1}{2}},v_{\frac{r+1}{2}}]\otimes1_\chi=\chi([v_{\frac{r+1}{2}},v_{\frac{r+1}{2}}])\otimes1_\chi=1\otimes1_\chi$, then $v_{\frac{r+1}{2}}\otimes1_\chi\in
Q_\chi^{\text{ad}\mathfrak{m}}$, but which is not in $Q_\chi^{\text{ad}\mathfrak{m}'}$.

(1) For any $\mathbb{Z}_2$-homogeneous element $x\in Q_\chi^{\text{ad}\mathfrak{m}}$, we claim that
$$[v_{\frac{r+1}{2}},x]\subseteq Q_\chi^{\text{ad}\mathfrak{m}'}.$$

(i) Recall that $\Theta_{l+q+1}(1_\chi)=v_{\frac{r+1}{2}}\otimes1_\chi\in Q_\chi^{\text{ad}\mathfrak{m}}$. Since $x\in Q_\chi^{\text{ad}\mathfrak{m}}$, it follows from the Jacobi identity that $[v_{\frac{r+1}{2}},x]\in Q_\chi^{\text{ad}\mathfrak{m}}$.

(ii) Since $v_{\frac{r+1}{2}}\in\mathfrak{g}(-1)_{\bar{1}}$, then $[v_{\frac{r+1}{2}},v_{\frac{r+1}{2}}]\in\mathfrak{g}(-2)_{\bar0}\subseteq\mathfrak{m}_{\bar0}$, and it follows from $x\in Q_\chi^{\text{ad}\mathfrak{m}}$ that $[[v_{\frac{r+1}{2}},v_{\frac{r+1}{2}}],x]=0$. On the other hand,
\[\begin{array}{ccl}
[[v_{\frac{r+1}{2}},v_{\frac{r+1}{2}}],x]&=&[v_{\frac{r+1}{2}},[v_{\frac{r+1}{2}},x]]+(-1)^{|x|}[[v_{\frac{r+1}{2}},x],v_{\frac{r+1}{2}}]\\
&=&[v_{\frac{r+1}{2}},[v_{\frac{r+1}{2}},x]]-(-1)^{|x|}\cdot(-1)^{|x|+1}[v_{\frac{r+1}{2}},[v_{\frac{r+1}{2}},x]]\\
&=&[v_{\frac{r+1}{2}},[v_{\frac{r+1}{2}},x]]+(-1)^{2|x|+2}[v_{\frac{r+1}{2}},[v_{\frac{r+1}{2}},x]]\\
&=&2[v_{\frac{r+1}{2}},[v_{\frac{r+1}{2}},x]],
\end{array}
\]
then $[v_{\frac{r+1}{2}},[v_{\frac{r+1}{2}},x]]=0$, i.e. $[v_{\frac{r+1}{2}},x]\in Q_\chi^{\text{ad}v_{\frac{r+1}{2}}}$.

Since $\mathfrak{m}'=\mathfrak{m}\oplus\mathbb{C}v_{\frac{r+1}{2}}$ as vector spaces, and it follows from (i) and (ii) that $[v_{\frac{r+1}{2}},x]\in Q_\chi^{\text{ad}\mathfrak{m}'}$, then $[v_{\frac{r+1}{2}},Q_\chi^{\text{ad}\mathfrak{m}}]\subseteq  Q_\chi^{\text{ad}\mathfrak{m}'}$ by the arbitrary of $x$.

(2) We claim that the reverse of (1) is also true, i.e. $[v_{\frac{r+1}{2}},Q_\chi^{\text{ad}\mathfrak{m}}]\supseteq Q_\chi^{\text{ad}\mathfrak{m}'}$.

As $v_{\frac{r+1}{2}}\in\mathfrak{m}$, for any $\mathbb{Z}_2$-homogeneous element $y\in Q_\chi^{\text{ad}\mathfrak{m}'}\subseteq Q_\chi^{\text{ad}\mathfrak{m}}$, we have $[v_{\frac{r+1}{2}},y]=0$, and $yv_{\frac{r+1}{2}}\otimes1_\chi\in Q_\chi^{\text{ad}\mathfrak{m}}$ by the Jacobi identity. Since$$[v_{\frac{r+1}{2}},yv_{\frac{r+1}{2}}\otimes1_\chi]=([v_{\frac{r+1}{2}},y]v_{\frac{r+1}{2}}+(-1)^{|y|}y[v_{\frac{r+1}{2}},v_{\frac{r+1}{2}}])\otimes1_\chi=(-1)^{|y|}y,$$i.e. $y=[v_{\frac{r+1}{2}},(-1)^{|y|}yv_{\frac{r+1}{2}}\otimes1_\chi]$, it can be concluded that $[v_{\frac{r+1}{2}},Q_\chi^{\text{ad}\mathfrak{m}}]\supseteq Q_\chi^{\text{ad}\mathfrak{m}'}$ by the arbitrary of $y$.

All the discussions in (1) and (2) complete the proof.
\end{proof}

\section{Finite $W$-superalgebras and their subalgebras in positive characteristic}

This section is a generalization of the finite $W$-algebras theory introduced by Premet in [\cite{P7}]. Recall that in Section 5 we have studied the structure of reduced $W$-superalgebras in positive characteristic, and finite $W$-superalgebras over the field of complex numbers in Section 6. From which we know that the parity of $\text{dim}~\mathfrak{g}(-1)_{\bar1}$ plays the key role for the construction of finite $W$-superalgebras. {\bf Recall that $\text{dim}~\mathfrak{g}(-1)_{\bar1}$ and $d_1=\text{dim}~\mathfrak{g}_{\bar1}-\text{dim}~\mathfrak{g}^e_{\bar1}$ have the same parity by Remark~\ref{centralizer}}. Based on the parity of $d_1$, we will study the construction of the finite $W$-superalgebras in positive characteristic and their subalgebras for each case respectively. First assume that

(1) When $d_1$ is even, let $U(\mathfrak{g}_A,e)$ denote the $A$-span of all monomials $\Theta_1^{a_1}\cdots\Theta_l^{a_l}\cdot$\\$\Theta_{l+1}^{b_1}\cdots\Theta_{l+q}^{b_{q}}$ with $(a_1,\cdots,a_l;b_1,\cdots,b_q)\in\mathbb{Z}_+^l\times\mathbb{Z}_2^q$;

(2) When $d_1$ is odd, let $U(\mathfrak{g}_A,e)$ denote the $A$-span of all monomials $\Theta_1^{a_1}\cdots\Theta_l^{a_l}\cdot$\\$\Theta_{l+1}^{b_1}\cdots\Theta_{l+q}^{b_{q}}\Theta_{l+q+1}^c$ with $(a_1,\cdots,a_l;b_1,\cdots,b_q;c)\in\mathbb{Z}_+^l\times\mathbb{Z}_2^q\times\mathbb{Z}_2^1$.

Our assumptions on $A$ (see Section 3.1) guarantee $U(\mathfrak{g}_A,e)$ is an $A$-subalgebra of $U(\mathfrak{g},e)$ contained in
$(\text{End}_{\mathfrak{g}_A}Q_{\chi,A})^{\text{op}}$. By the definition of $Q_{\chi,A}$ in Section 4.1 we know that $Q_{\chi,A}$ can be identified with the $\mathfrak{g}_A$-module $U(\mathfrak{g}_A)/U(\mathfrak{g}_A)N_{\chi,A}$. Hence $U(\mathfrak{g}_A,e)$ embeds into the $A$-algebra $(U(\mathfrak{g}_A)/U(\mathfrak{g}_A)N_{\chi,A})^{\text{ad}\mathfrak{m}_A}\cong (Q_{\chi,A})^{\text{ad}\mathfrak{m}_A}$. As $Q_{\chi,A}$ is a free $A$-module with basis $\{x^\mathbf{a}y^\mathbf{b}u^\mathbf{c}v^\mathbf{d}\otimes1_\chi|(\mathbf{a},\mathbf{b},\mathbf{c},\mathbf{d})\in\mathbb{Z}^m_+
\times\mathbb{Z}^n_2\times\mathbb{Z}^s_+\times\mathbb{Z}^t_2\}$, an easy induction on Kazhdan degree (based on Lemma~\ref{hwc} and the formulas displayed in Lemma~\ref{com c} \& Theorem~\ref{PBWC}) shows that
$$U(\mathfrak{g}_A,e)=(\text{End}_{\mathfrak{g}_A}Q_{\chi,A})^{\text{op}}\cong(U(\mathfrak{g}_A,e)/U(\mathfrak{g}_A,e)N_{\chi,A})^{\text{ad}\mathfrak{m}_A}.$$

\begin{defn}\label{reduced k}
Define the induced $\mathds{k}$-algebra
$U(\mathfrak{g}_\mathds{k},e):=U(\mathfrak{g}_A,e)\otimes_A\mathds{k}.$
\end{defn}

It is immediate by definition that $U(\mathfrak{g}_\mathds{k},e)$  can be identified with a subalgebra of the finite $W$-superalgebra $\widehat{U}(\mathfrak{g}_\mathds{k},e)$ (see Definition~\ref{W-k}) over $\mathds{k}$. The $\mathds{k}$-algebra $U(\mathfrak{g}_\mathds{k},e)$ will be called {\bf the transition subalgebra}. On the other hand,

(1) when $d_1$ is even, the algebra $U(\mathfrak{g}_\mathds{k},e)$ has a $\mathds{k}$-basis consisting of all monomials $\bar{\Theta}_1^{a_1}\cdots\bar{\Theta}_l^{a_l}\bar{\Theta}_{l+1}^{b_1}\cdots\bar{\Theta}_{l+q}^{b_{q}}$ with $(a_1,\cdots,a_l;b_1,\cdots,b_q)\in\mathbb{Z}_+^l\times\mathbb{Z}_2^q$, where $\bar{\Theta}_i:=\Theta_i\otimes1\in U(\mathfrak{g}_A,e)\otimes_A\mathds{k}$;

(2) when $d_1$ is odd, the algebra $U(\mathfrak{g}_\mathds{k},e)$ has a $\mathds{k}$-basis consisting of all monomials $\bar{\Theta}_1^{a_1}\cdots\bar{\Theta}_l^{a_l}\bar{\Theta}_{l+1}^{b_1}\cdots\bar{\Theta}_{l+q}^{b_{q}}\bar{\Theta}_{l+q+1}^{c}$ with $(a_1,\cdots,a_l;b_1,\cdots,b_q;c)\in\mathbb{Z}_+^l\times\mathbb{Z}_2^q\times\mathbb{Z}_2^1$, where $\bar{\Theta}_i:=\Theta_i\otimes1\in U(\mathfrak{g}_A,e)\otimes_A\mathds{k}$.

Recall that all the coefficients of polynomial $F_{ij}'s$ in Theorem~\ref{relationc} are in $\mathbb{Q}$, then one can also assume the $F_{ij}'s$ are in $A$ after enlarging $A$ if need be. Given a polynomial $g\in A[T_1,\cdots T_n]$, let $^pg$ denote the image of $g$ in the polynomial superalgebra $\mathds{k}[T_1,\cdots T_n]=A[T_1,\cdots T_n]\otimes_A\mathds{k}$. By the same discussion as Theorem~\ref{relationc}, we have

\begin{theorem}\label{translation}
For the $\mathds{k}$-algebra $U(\mathfrak{g}_\mathds{k},e)$, the following are true:

(1) when $d_1$ is even, we can choose $\bar{\Theta}_1,\cdots,\bar{\Theta}_{l+q}$ as the homogeneous generators of $U(\mathfrak{g}_\mathds{k},e)$ (where $\bar{\Theta}_1,\cdots,\bar{\Theta}_{l}\in U(\mathfrak{g}_\mathds{k},e)_{\bar0}$, $\bar{\Theta}_{l+1},\cdots,\bar{\Theta}_{l+q}\in U(\mathfrak{g}_\mathds{k},e)_{\bar1}$) subject to the relations $$[\bar{\Theta}_i,\bar{\Theta}_j]=^pF_{ij}(\bar{\Theta}_1,\cdots,\bar{\Theta}_{l+q}),\qquad[\bar{\Theta}_j,\bar{\Theta}_i]=-
(-1)^{|\bar{\Theta}_i||\bar{\Theta}_j|}[\bar{\Theta}_i,\bar{\Theta}_j]$$
where $1\leqslant i<j\leqslant l+q$ and $l+1\leqslant i=j\leqslant l+q$.

(2) when $d_1$ is odd, we can choose $\bar{\Theta}_1,\cdots,\bar{\Theta}_{l+q+1}$ as the homogeneous generators of $U(\mathfrak{g}_\mathds{k},e)$ (where $\bar{\Theta}_1,\cdots,\bar{\Theta}_{l}\in U(\mathfrak{g}_\mathds{k},e)_{\bar0}$, $\bar{\Theta}_{l+1},\cdots,\bar{\Theta}_{l+q+1}\in U(\mathfrak{g}_\mathds{k},e)_{\bar1}$) subject to the relations $$[\bar{\Theta}_i,\bar{\Theta}_j]=^pF_{ij}(\bar{\Theta}_1,\cdots,\bar{\Theta}_{l+q+1}),\qquad[\bar{\Theta}_j,\bar{\Theta}_i]=-
(-1)^{|\bar{\Theta}_i||\bar{\Theta}_j|}[\bar{\Theta}_i,\bar{\Theta}_j]$$
where $1\leqslant i<j\leqslant l+q+1$ and $l+1\leqslant i=j\leqslant l+q+1$.

Moreover, (1) and (2) completely determine the structure of the $\mathds{k}$-algebra $U(\mathfrak{g}_\mathds{k},e)$.
\end{theorem}

Recall that we have studied the construction of reduced $W$-superalgebra $U_\chi(\mathfrak{g}_\mathds{k},e)$ in Section 4. In fact, all the results obtained there can be generalized to the cases with $p$-character $\eta\in\chi+(\mathfrak{m}_\mathds{k}^\bot)_{\bar{0}}$. First note that

\begin{lemma}\label{m_k free}
Let $\mathfrak{g}_\mathds{k}$ be one of the basic classical Lie superalgebras over $\mathds{k}$. For any $\eta\in\chi+(\mathfrak{m}_\mathds{k}^\bot)_{\bar{0}}\subseteq(\mathfrak{g}_\mathds{k})^*_{\bar0}$, every $U_\eta(\mathfrak{g}_\mathds{k})$-module is $U_\eta(\mathfrak{m}_\mathds{k})$-free.
\end{lemma}

\begin{proof}
Since $\chi|_{\mathfrak{m}_\mathds{k}}=\eta|_{\mathfrak{m}_\mathds{k}}$, this Lemma can be proved by the same way as ([\cite{WZ}], Proposition 4.2), thus will be omitted here.
\end{proof}

\begin{theorem}\label{sumresult}
The following are true:

(1) $Q_\chi^\eta\cong U_\eta(\mathfrak{g}_\mathds{k})\otimes_{U_\eta(\mathfrak{m}_\mathds{k})}\mathds{k}_\chi$ as $\mathfrak{g}_\mathds{k}$-modules;

(2) $U_\eta(\mathfrak{g}_\mathds{k},e)\cong (U_\eta(\mathfrak{g}_\mathds{k})/U_\eta(\mathfrak{g}_\mathds{k})N_{\mathfrak{m}_\mathds{k}})^{\text{ad}\mathfrak{m}_\mathds{k}}$;

(3) $Q_\chi^\eta$ is a projective generator for $U_\eta(\mathfrak{g}_\mathds{k})$ and $\delta=\text{dim}~U_\eta(\mathfrak{m}_\mathds{k})=p^{\frac{d_0}{2}}2^{\lceil\frac{d_1}{2}\rceil}$. Moreover, $$U_\eta(\mathfrak{g}_\mathds{k})\cong\text{Mat}_\delta(U_\eta(\mathfrak{g}_\mathds{k},e));$$

(4)(i) when $d_1$ is even, the monomials $\theta_1^{a_1}\cdots\theta_{l}^{a_l}\theta_{l+1}^{b_1}\cdots\theta_{l+q}^{b_q}$ with $0\leqslant  a_k\leqslant  p-1$ for $1\leqslant k\leqslant l$ and $0\leqslant b_k\leqslant 1$ for $1\leqslant k\leqslant q$ form a $\mathds{k}$-basis of $U_\eta(\mathfrak{g}_\mathds{k},e)$.

(ii) when $d_1$ is odd, the monomials $\theta_1^{a_1}\cdots\theta_{l}^{a_l}\theta_{l+1}^{b_1}\cdots\theta_{l+q}^{b_q}\theta_{l+q+1}^c$ with $0\leqslant  a_k\leqslant  p-1$ for $1\leqslant k\leqslant l$ and $0\leqslant b_k,c\leqslant 1$ for $1\leqslant k\leqslant q$ form a $\mathds{k}$-basis of $U_\eta(\mathfrak{g}_\mathds{k},e)$.
\end{theorem}

\begin{proof}
By the same discussion as the finite $W$-algebra case (see [\cite{P7}], Lemma 2.2(i)), we can get (1). Repeat verbatim the proof of Proposition~\ref{invariant} we can get (2). Apply Lemma~\ref{m_k free} and (3) follows from the same treatment as for the algebra $U_\chi(\mathfrak{g}_\mathds{k},e)$ in Proposition~\ref{invariant}. Since the dimension of $\mathds{k}$-algebra $U_\eta(\mathfrak{g}_\mathds{k},e)$ (as a vector space) can be computed by (3), repeat verbatim the proof of Theorem~\ref{reduced Wg}(1)(iii) and (2)(iii) respectively one obtains (4).
\end{proof}

At the beginning of Section 5.1 we have assumed that $\{x_1,\cdots,x_m,y_1,\cdots,y_n\}$ is an $A$-basis of $\mathfrak{p}_A=\bigoplus\limits_{i\geqslant0}\mathfrak{g}_A(i)$. Set
\[X_i:=\left\{\begin{array}{ll}
x_{i+l}&\text{if}~1\leqslant  i\leqslant  m-l;\\
y_{l+q-m+i}&\text{if}~m-l+1\leqslant  i\leqslant  m+n-l-q;\\
u_{l+q-m-n+i}&\text{if}~m+n-l-q+1\leqslant  i\leqslant  m+n-l-q+s;\\
v_{l+q-m-n-s+i}&\text{if}~m+n-l-q+s+1\leqslant  i\leqslant  m+n-l-q+s+t',
\end{array}\right.
\]
where $t'=\lceil\frac{r}{2}\rceil$.

\begin{rem}
In the following we will {\bf denote $\lceil\frac{r}{2}\rceil$ by $t'$} once for all. Recall that we have defined $d_i=\text{dim}~(\mathfrak{g}_\mathds{k})_i-\text{dim}~(\mathfrak{g}_\mathds{k}^e)_i$ for $i\in\mathbb{Z}_2$ in Remark~\ref{centralizer}. It follows from ([\cite{WZ}], Theorem 4.3) that $\text{dim}~U_\chi(\mathfrak{m}_\mathds{k})=p^{\frac{d_0}{2}}2^{\lceil \frac{d_1}{2}\rceil}$ and denote it by $\delta$. By the assumption of the notations we can obtain that $\frac{d_0}{2}+\lceil\frac{d_1}{2}\rceil=m+n-l-q+s+t'$.
\end{rem}

For $(\mathbf{a},\mathbf{b},\mathbf{c},\mathbf{d})\in\mathbb{Z}_+^{m-l}\times\mathbb{Z}_2^{n-q}\times\mathbb{Z}_+^s\times\mathbb{Z}_2^{t'}$, define
\[\begin{array}{ccl}
X^{\mathbf{a},\mathbf{b},\mathbf{c},\mathbf{d}}:&=&X_1^{a_1}\cdots X_{m-l}^{a_{m-l}}X_{m-l+1}^{b_{1}}\cdots X_{m+n-l-q}^{b_{n-q}}X_{m+n-l-q+1}^{c_1}\cdots \\
&&\cdot X_{m+n-l-q+s}^{c_{s}}X_{m+n-l-q+s+1}^{d_{1}}\cdots X_{m+n-l-q+s+t'}^{d_{t'}}
\end{array}\]
and\[\begin{array}{ccl}
\bar{X}^{\mathbf{a},\mathbf{b},\mathbf{c},\mathbf{d}}:&=&\bar{X}_1^{a_1}\cdots \bar{X}_{m-l}^{a_{m-l}}\bar{X}_{m-l+1}^{b_{1}}\cdots\bar{X}_{m+n-l-q}^{b_{n-q}}\bar{X}_{m+n-l-q+1}^{c_1}\cdots\\ &&\cdot\bar{X}_{m+n-l-q+s}^{c_{s}}\bar{X}_{m+n-l-q+s+1}^{d_{1}}\cdots\bar{X}_{m+n-l-q+s+t'}^{d_{t'}},
\end{array}\] elements of $U(\mathfrak{g}_A)$ and $U(\mathfrak{g}_\mathds{k})$, respectively. Denote by $\bar{1}_\chi$ the image of $1_\chi\in Q_{\chi,\mathds{k}}$ in $Q^\eta_{\chi}$.

\begin{lemma}\label{right}
For any $\eta\in\chi+(\mathfrak{m}_\mathds{k}^\bot)_{\bar{0}}$, the right modules $Q_{\chi,A}$ and $Q_\chi^\eta$ are free over $U(\mathfrak{g}_A,e)$ and $U_\eta(\mathfrak{g}_\mathds{k},e)$, respectively. More precisely,

(1) the set $\{X^{\mathbf{a},\mathbf{b},\mathbf{c},\mathbf{d}}\otimes1_\chi|(\mathbf{a},\mathbf{b},\mathbf{c},\mathbf{d})\in\mathbb{Z}_+^{m-l}\times\mathbb{Z}_2^{n-q}\times\mathbb{Z}_+^s\times\mathbb{Z}_2^{t'}\}$\, is a free basis of the $U(\mathfrak{g}_A,e)$-module $Q_{\chi,A}$;

(2) the set $\{\bar{X}^{\mathbf{a},\mathbf{b},\mathbf{c},\mathbf{d}}\otimes\bar{1}_\chi|(\mathbf{a},\mathbf{b},\mathbf{c},\mathbf{d})\in\Lambda_{m-l}\times\Lambda'_{n-q}\times\Lambda_s\times\Lambda'_{t'}\}$\, is a free basis of the $U_\eta(\mathfrak{g}_\mathds{k},e)$-module $Q_{\chi}^\eta$.
\end{lemma}

\begin{proof}
The proof is the same as the finite $W$-algebra case, thus will be omitted here (see [\cite{P4}], Lemma 4.2).
\end{proof}

Let $\mathfrak{a}_\mathds{k}$ be the $\mathds{k}$-span of $\bar{X}_1,\cdots,\bar{X}_{m+n-l-q+s+t'}$ in $\mathfrak{g}_\mathds{k}$. By the definition of $\widetilde{\mathfrak{p}}_\mathds{k}$ we have

(1) $\widetilde{\mathfrak{p}}_\mathds{k}=\mathfrak{a}_\mathds{k}\oplus\mathfrak{g}_\mathds{k}^e$ when $d_1$ is even;

(2) $\widetilde{\mathfrak{p}}_\mathds{k}=\mathfrak{a}_\mathds{k}\oplus\mathfrak{g}_\mathds{k}^e\oplus\mathds{k}v_{\frac{r+1}{2}}$ when $d_1$ is odd.

Recall the notations (b) and (c) preceding Definition~\ref{Gelfand-Graev}. By the inclusion $\mathfrak{g}_\mathds{k}^f\subseteq\bigoplus\limits_{i\leqslant 0}\mathfrak{g}_\mathds{k}(i)$ we have

(1) $\mathfrak{a}_\mathds{k}=\{\bar x\in\widetilde{\mathfrak{p}}_\mathds{k}|(\bar x,\mathfrak{g}_\mathds{k}^f)=0\}$ when $d_1$ is even;

(2)
$\mathfrak{a}_\mathds{k}\oplus\mathds{k}v_{\frac{r+1}{2}}=\{\bar x\in\widetilde{\mathfrak{p}}_\mathds{k}|(\bar x,\mathfrak{g}_\mathds{k}^f)=0\}$ when $d_1$ is odd.

Let $\rho_\mathds{k}$ denote the representation of $U(\mathfrak{g}_\mathds{k})$ in $\text{End}_\mathds{k}Q_{\chi,\mathds{k}}$. Given a subspace $V$ in $\mathfrak{g}_\mathds{k}$ we denote by $Z_p(V)$ the subalgebra of $p$-center $Z_p(\mathfrak{g}_\mathds{k})$ generated by all $\bar x^p-\bar x^{[p]}$ with $\bar x\in V_{\bar{0}}$. Clearly, $Z_p(V)$ is isomorphic to a polynomial algebra (in the usual sense, not super) in $\text{dim}~V_{\bar{0}}$ variables. We will denote $Z_p(\mathfrak{g}_\mathds{k})$ by $Z_p$ for short.

\begin{theorem}\label{keyisotheorem}
For any even nilpotent element $e\in(\mathfrak{g}_\mathds{k})_{\bar{0}}$, we have

(1) the algebra $\widehat{U}(\mathfrak{g}_\mathds{k},e)$ is generated by its subalgebras $U(\mathfrak{g}_\mathds{k},e)$ and $\rho_\mathds{k}(Z_p)$;

(2) $\rho_\mathds{k}(Z_p)\cong Z_p(\widetilde{\mathfrak{p}}_\mathds{k})$ as $\mathds{k}$-algebras. Moreover, if $d_1$ is even, then $\widehat{U}(\mathfrak{g}_\mathds{k},e)$ is a free $\rho_\mathds{k}(Z_p)$-module of rank $p^l2^q$; if $d_1$ is odd, $\widehat{U}(\mathfrak{g}_\mathds{k},e)$ is a free $\rho_\mathds{k}(Z_p)$-module of rank $p^l2^{q+1}$;

(3) $\widehat{U}(\mathfrak{g}_\mathds{k},e)\cong U(\mathfrak{g}_\mathds{k},e)\otimes_\mathds{k}Z_p(\mathfrak{a}_\mathds{k})$ as $\mathds{k}$-algebras.
\end{theorem}
This theorem is a generalization of the Lie algebra case in ([\cite{P7}], Theorem 2.1). Compared with the finite $W$-algebras, the construction of the finite $W$-superalgebras is much more complicated. Moreover, if $d_1$ is odd, it is a new case which never occurs under the background of Lie algebra. Now we will prove it in detail.

\begin{proof}
(i) Since $\mathfrak{g}_\mathds{k}=\mathfrak{m}_\mathds{k}\oplus\widetilde{\mathfrak{p}}_\mathds{k}$ by definition, then $Z_p(\mathfrak{g}_\mathds{k})\cong Z_p(\mathfrak{m}_\mathds{k})\otimes_\mathds{k}Z_p(\widetilde{\mathfrak{p}}_\mathds{k})$ as $\mathds{k}$-algebras, and $Z_p(\mathfrak{m}_\mathds{k})\cap \text{Ker}\rho_\mathds{k}$ is an ideal of codimension $1$ in $Z_p(\mathfrak{m}_\mathds{k})$. Hence $\rho_\mathds{k}(Z_p)=\rho_\mathds{k}(Z_p(\widetilde{\mathfrak{p}}_\mathds{k}))$. As the monomials $\bar{x}^\mathbf{a}\bar{y}^\mathbf{b}\bar{u}^\mathbf{c}\bar{v}^\mathbf{d}\otimes1_\chi$ with $(\mathbf{a},\mathbf{b},\mathbf{c},\mathbf{d})\in\mathbb{Z}_+^m\times\mathbb{Z}_2^n\times\mathbb{Z}_+^s\times\mathbb{Z}_2^t$ (recall that $t=\lfloor\frac{\text{dim}\mathfrak{g}_\mathds{k}(-1)_{\bar1}}{2}\rfloor$) form a basis of $Q_{\chi,\mathds{k}}$, and $Z_p(\widetilde{\mathfrak{p}}_\mathds{k})$ is the polynomial algebra in $\bar{x}_i^p-\bar{x}_i^{[p]}(1\leqslant  i\leqslant  m)$ \& $\bar{u}_j^p-\bar{u}_j^{[p]}(1\leqslant  j\leqslant  s)$, we have $Z_p(\widetilde{\mathfrak{p}}_\mathds{k})\cap\text{Ker}\rho_\mathds{k}=\{0\}$. It follows that $\rho_\mathds{k}(Z_p)\cong Z_p(\widetilde{\mathfrak{p}}_\mathds{k})$ as $\mathds{k}$-algebras. By the discussion in Section 3.2 we know that $S((\widetilde{\mathfrak{p}}_\mathds{k})_{\bar0})\cong\mathds{k}[\chi+(\mathfrak{m}_\mathds{k}^\bot)_{\bar{0}}]$, hence $Z_p(\widetilde{\mathfrak{p}}_\mathds{k})\cong \mathds{k}[(\chi+(\mathfrak{m}_\mathds{k}^\bot)_{\bar{0}})^{(1)}]$, where $(\chi+(\mathfrak{m}_\mathds{k}^\bot)_{\bar{0}})^{(1)}\subseteq(\mathfrak{g}_\mathds{k}^*)^{(1)}$ is the Frobenius twist of $\chi+(\mathfrak{m}_\mathds{k}^\bot)_{\bar{0}}$.

(ii) Since the proof is similar for both cases, we will formulate a detailed proof for the case when $d_1$ is odd as which is more complicated.

Denote by $\mathbf{I}_j$ the set of all tuples with $j$ components and let $\mathbf{e}_i$ denote the tuple in $\mathbf{I}_j$ whose only nonzero component equals $1$ and occupies the $i$th position. As an immediate consequence of Theorem~\ref{PBWC}(2)(i), by induction we have that
\begin{equation}\label{dede}
\bar{\Theta}_k^p(1_\chi)-(\bar{x}_k^p+\sum\limits_{|(\mathbf{a},\mathbf{0},\mathbf{c},\mathbf{0})|_e=m_k+2}\mu^k_{\mathbf{a},\mathbf{0},\mathbf{c},\mathbf{0}}\bar{x}^{p\mathbf{a}}
\bar{u}^{p\mathbf{c}})\otimes1_\chi\in(Q_{\chi,\mathds{k}})_{p(m_k+2)-1}
\end{equation}for $1\leqslant  k\leqslant  l$, where $\mu^k_{\mathbf{a},\mathbf{0},\mathbf{c},\mathbf{0}}\in\mathbb{F}_p$.
Discussing in the graded algebra $\text{gr}(U(\mathfrak{g}_\mathds{k}))$ under the Kazhdan filtration (notice that $\bar x^{[p]}\in\mathfrak{g}_\mathds{k}(pi)$ whenever $\bar x\in\mathfrak{g}_\mathds{k}(i)$ for all $i\in\mathbb{Z}$), we can obtain that
\begin{equation}\label{grad}
\text{gr}(\bar{x}_i^p-\bar{x}_i^{[p]})=\text{gr}(\bar{x}_i)^p,~~~~\text{and}~~~~\text{gr}(\bar{u}_j^p-\bar{u}_j^{[p]})=\text{gr}(\bar{u}_j)^p~~~~(1\leqslant  i\leqslant m;~~ 1\leqslant  j\leqslant  s).
\end{equation}

On the other hand, Lemma~\ref{right}(1) implies that the vectors $\bar{X}^{(\mathbf{a},\mathbf{b},\mathbf{c},\mathbf{d})}\otimes1_\chi$ with\[\begin{array}{ccl}
(\mathbf{a},\mathbf{b},\mathbf{c},\mathbf{d})
&=&(a_1,\cdots,a_{m-l};b_1,\cdots,b_{n-q};c_1,\cdots,c_{s};d_1,\cdots,d_{t'})\\
&\in &\mathbb{Z}_+^{m-l}\times\mathbb{Z}_2^{n-q}\times\mathbb{Z}_+^{s}\times\mathbb{Z}_2^{t'},
\end{array}\]
(recall that $t'=\lceil\frac{\text{dim}\mathfrak{g}_\mathds{k}(-1)_{\bar1}}{2}\rceil$) form a free basis of the right $U(\mathfrak{g}_\mathds{k},e)$-module $Q_{\chi,\mathds{k}}$. As $Q_{\chi,\mathds{k}}$ is a Kazhdan-filtrated $U(\mathfrak{g}_\mathds{k},e)$-module, straightforward induction on filtration degree based on \eqref{dede} and \eqref{grad} shows that $Q_{\chi,\mathds{k}}$ is generated as a $Z_p(\widetilde{\mathfrak{p}}_\mathds{k})$-module by the set
$$\{\bar{X}^{(\mathbf{a},\mathbf{b},\mathbf{c},\mathbf{d})}\bar{\Theta}^{(\mathbf{i},\mathbf{j})}\otimes1_\chi| (\mathbf{a},\mathbf{b},\mathbf{c},\mathbf{d},\mathbf{i},\mathbf{j})\in\Lambda_{m-l}\times\Lambda'_{n-q}\times\Lambda_{s}\times\Lambda'_{t'}\times\Lambda_{l}\times\Lambda'_{q+1}\},$$
where $\Lambda_i$ and $\Lambda'_j$ are defined at the beginning of Section 5.

Let $h$ be an arbitrary element of $\widehat{U}(\mathfrak{g}_\mathds{k},e)$. Then we can assume
\[\begin{array}{cccl}
h(\bar{1}_\chi)&=&\sum f_{\mathbf{a},\mathbf{b},\mathbf{c},\mathbf{d},\mathbf{i},\mathbf{j}}&
\bar{X}_1^{a_1}\cdots\bar{X}_{m-l}^{a_{m-l}}\bar{X}_{m-l+1}^{b_1}\cdots\bar{X}_{m+n-l-q}^{b_{n-q}}
\bar{X}_{m+n-l-q+1}^{c_1}\cdots\\
&&&\bar{X}_{m+n-l-q+s}^{c_{s}}\bar{X}_{m+n-l-q+s+1}^{d_1}\cdots
\bar{X}_{m+n-l-q+s+t'}^{d_{t'}}\cdot\bar{\Theta}_1^{i_1}\cdots\bar{\Theta}_l^{i_l}\\
&&&\bar{\Theta}_{l+1}^{j_1}\cdots\bar{\Theta}_{l+q}^{j_q}
\bar{\Theta}_{l+q+1}^{j_{q+1}}(1_\chi)
\end{array}\]by above discussion, where
$f_{\mathbf{a},\mathbf{b},\mathbf{c},\mathbf{d},\mathbf{i},\mathbf{j}}\in Z_p(\widetilde{\mathfrak{p}}_\mathds{k})$ with $(\mathbf{a},\mathbf{b},\mathbf{c},\mathbf{d},\mathbf{i},\mathbf{j})$ in the set $\Lambda_{m-l}\times\Lambda'_{n-q}\times\Lambda_{s}\times\Lambda'_{t'}\times\Lambda_{l}\times\Lambda'_{q+1}$. For every $\xi\in\chi+(\mathfrak{m}_\mathds{k}^\bot)_{\bar{0}}$ the image of $f_{\mathbf{a},\mathbf{b},\mathbf{c},\mathbf{d},\mathbf{i},\mathbf{j}}$ in $U_\xi(\mathfrak{g}_\mathds{k})$ is a scalar in $\mathds{k}$ which shall be denoted by $\xi(\mathbf{a},\mathbf{b},\mathbf{c},\mathbf{d},\mathbf{i},\mathbf{j})$.

Suppose $f_{\mathbf{a},\mathbf{b},\mathbf{c},\mathbf{d},\mathbf{i},\mathbf{j}}\neq0$ for a nonzero $(\mathbf{a},\mathbf{b},\mathbf{c},\mathbf{d})\in\Lambda_{m-l}\times\Lambda'_{n-q}\times\Lambda_{s}\times\Lambda'_{t'}$ and some $(\mathbf{i},\mathbf{j})
\in\Lambda_{l}\times\Lambda'_{q+1}$. Then there exists $\eta\in\chi+(\mathfrak{m}_\mathds{k}^\bot)_{\bar{0}}$ such that $\eta(\mathbf{a},\mathbf{b},\mathbf{c},\mathbf{d},\mathbf{i},\mathbf{j})\neq0$. Let $h(\eta)$ be the image of $h\in\widehat{U}(\mathfrak{g}_\mathds{k},e)$ in $U_\eta(\mathfrak{g}_\mathds{k},e)=(\text{End}_{\mathfrak{g}_\mathds{k}}Q_\chi^\eta)^{\text{op}}$. Theorem~\ref{sumresult}(4)(ii) implies that $h(\eta)(\bar{1}_\chi)$ is a $\mathds{k}$-linear combination of $\theta_1^{i_1}\cdots\theta_l^{i_l}\theta_{l+1}^{j_1}\cdots\theta_{l+q}^{j_{q}}\theta_{l+q+1}^{j_{q+1}}(\bar{1}_\chi)$ with$$(i_1,\cdots,i_l;j_1,\cdots,j_q;j_{q+1})\in\Lambda_l\times\Lambda'_q\times\Lambda'_1.$$ By Lemma~\ref{right}(2), the set $$\{\bar{X}^{(\mathbf{a},\mathbf{b},\mathbf{c},\mathbf{d})}\otimes\bar{1}_\chi|(\mathbf{a},\mathbf{b},\mathbf{c},\mathbf{d})\in\Lambda_{m-l}\times\Lambda'_{n-q}\times\Lambda_s\times\Lambda'_{t'}\}$$ is a free basis of the right $U_\eta(\mathfrak{g}_\mathds{k},e)$-module $Q_{\chi}^\eta$. Since $\eta(\mathbf{a},\mathbf{b},\mathbf{c},\mathbf{d},\mathbf{i},\mathbf{j})\neq0$ and $\theta_1^{i_1}\cdots\theta_l^{i_l}\theta_{l+1}^{j_1}\cdots\theta_{l+q}^{j_q}\theta_{l+q+1}^{j_{q+1}}$ is the image of $\bar{\Theta}_1^{i_1}\cdots\bar{\Theta}_l^{i_l}\bar{\Theta}_{l+1}^{j_1}\cdots\bar{\Theta}_{l+q}^{j_q}\bar{\Theta}_{l+q+1}^{j_{q+1}}$ in $U_\eta(\mathfrak{g}_\mathds{k},e)$, it is now evident that $h(\eta)(\bar{1}_\chi)$ cannot be a $\mathds{k}$-linear combination of $\theta_1^{i_1}\cdots\theta_l^{i_l}\theta_{l+1}^{j_1}\cdots\theta_{l+q}^{j_{q}}\cdot$\\$\theta_{l+q+1}^{j_{q+1}}(\bar{1}_\chi)$ with $$(i_1,\cdots,i_l;j_1,\cdots,j_q;j_{q+1})\in\Lambda_l\times\Lambda'_q\times\Lambda'_1.$$ This contradiction shows that $f_{\mathbf{a},\mathbf{b},\mathbf{c},\mathbf{d},\mathbf{i},\mathbf{j}}=0$ unless $(\mathbf{a},\mathbf{b},\mathbf{c},\mathbf{d})=\mathbf{0}$. As a consequence, $$\{\bar{\Theta}_1^{i_1}\cdots\bar{\Theta}_l^{i_l}\bar{\Theta}_{l+1}^{j_1}\cdots\bar{\Theta}_{l+q}^{j_q}\bar{\Theta}_{l+q+1}^{j_{q+1}}|
(\mathbf{i},\mathbf{j})\in\Lambda_{l}\times\Lambda'_{q+1}\}$$generates $\widehat{U}(\mathfrak{g}_\mathds{k},e)$ as a $Z_p(\widetilde{\mathfrak{p}}_\mathds{k})$-module. Specialising at a suitable $\eta\in\chi+(\mathfrak{m}_\mathds{k}^\bot)_{\bar{0}}$ and applying Theorem~\ref{sumresult}(4)(ii) we deduce that the set $\{\bar{\Theta}_1^{i_1}\cdots\bar{\Theta}_l^{i_l}\bar{\Theta}_{l+1}^{j_1}\cdots\bar{\Theta}_{l+q}^{j_q}\bar{\Theta}_{l+q+1}^{j_{q+1}}|
(\mathbf{i},\mathbf{j})\in\Lambda_{l}\times\Lambda'_{q+1}\}$ is a free basis of the $Z_p(\widetilde{\mathfrak{p}}_\mathds{k})$-module $\widehat{U}(\mathfrak{g}_\mathds{k},e)$.

(iii) Our next goal is to show that $\widehat{U}(\mathfrak{g}_\mathds{k},e)=U(\mathfrak{g}_\mathds{k},e)\cdot Z_p(\mathfrak{a}_\mathds{k})$, the subalgebra of $\widehat{U}(\mathfrak{g}_\mathds{k},e)$ generated by $U(\mathfrak{g}_\mathds{k},e)$ and $Z_p(\mathfrak{a}_\mathds{k})$.

Firstly, every $\mathfrak{g}_\mathds{k}$-endomorphism of $Q_{\chi,\mathds{k}}$ is uniquely determined by its value at $1_\chi$. For a nonzero $u\in\widehat{U}(\mathfrak{g}_\mathds{k},e)$ write $$u(1_\chi)=\sum\limits_{|(\mathbf{a},\mathbf{b},\mathbf{c},\mathbf{d})|_e\leqslant  n(u)}\lambda_{\mathbf{a},\mathbf{b},\mathbf{c},\mathbf{d}}\bar{x}^\mathbf{a}\bar{y}^\mathbf{b}\bar{u}^\mathbf{c}\bar{v}^\mathbf{d}\otimes1_\chi,$$ where $\lambda_{\mathbf{a},\mathbf{b},\mathbf{c},\mathbf{d}}\neq0$ for at least one $(\mathbf{a},\mathbf{b},\mathbf{c},\mathbf{d})$ with $|(\mathbf{a},\mathbf{b},\mathbf{c},\mathbf{d})|_e=n(u)$. With the same notation preceding Lemma~\ref{hw}, for $k\in\mathbb{Z}_+$ put $$\Lambda^k(u):=\{(\mathbf{a},\mathbf{b},\mathbf{c},\mathbf{d})\in\mathbb{Z}_+^m\times\mathbb{Z}_2^n\times\mathbb{Z}_+^s\times\mathbb{Z}_2^t|
\lambda_{\mathbf{a},\mathbf{b},\mathbf{c},\mathbf{d}}\neq0~\&~|(\mathbf{a},\mathbf{b},\mathbf{c},\mathbf{d})|_e=k\},$$
and denote by $\Lambda^{\text{max}}(u)$ the set of all $
(\mathbf{a},\mathbf{b},\mathbf{c},\mathbf{d})\in\Lambda^{n(u)}(u)$ for which the quantity $n(u)-|\mathbf{a}|-|\mathbf{b}|-|\mathbf{c}|-|\mathbf{d}|$ assumes its maximum value. This maximum value will be denoted by $N(u)$ (recall that $\Lambda^{\text{max}}(u)=\Lambda^{\text{max}}_u+|\mathbf{a}|+|\mathbf{b}|+|\mathbf{c}|+|\mathbf{d}|$ by \eqref{dege} and the definition of $\Lambda_h^{\text{max}}$ preceding Lemma~\ref{hw}). For each $(\mathbf{a},\mathbf{b},\mathbf{c},\mathbf{d})\in\Lambda^{\text{max}}$, let $\bar x_i\in\mathfrak{g}_\mathds{k}(k_i)_{\bar0},~\bar y_j\in\mathfrak{g}_\mathds{k}(k'_j)_{\bar1}$ for $1\leqslant i\leqslant m$ and $1\leqslant j\leqslant n$ where $k_i, k_j'\in\mathbb{Z}_+$, then we have that
\[
\begin{array}{cl}
&|(\mathbf{a},\mathbf{b},\mathbf{c},\mathbf{d})|_e-|\mathbf{a}|-|\mathbf{b}|-|\mathbf{c}|-|\mathbf{d}|\\
=&\sum\limits_{i=1}^{m}(k_i+2)a_i
+\sum\limits_{i=1}^{n}(k'_i+2)b_i+\sum\limits_{i=1}^{s}c_i+\sum\limits_{i=1}^{t}d_i\\
&-|\mathbf{a}|-|\mathbf{b}|-|\mathbf{c}|-|\mathbf{d}|\geqslant 0.
\end{array}\]
Consequently, $n(u),~N(u)\in\mathbb{Z}_+$ and $n(u)\geqslant  N(u)$.

Put $\Omega:=\{(a,b)\in\mathbb{Z}_+^2|a\geqslant  b\}$. By the preceding remark we have $(n(u),N(u))\in\Omega$ for all nonzero $u\in\widehat{U}(\mathfrak{g}_\mathds{k},e)$. Theorem~\ref{PBWC}(2)(i) and the discussion in part (ii) show that
\[
\begin{array}{ll}
\Lambda^{\text{max}}(\bar{\Theta}_i)=\{(\mathbf{e}_i,\mathbf{0},\mathbf{0},\mathbf{0})\}&\text{for}~1\leqslant  i\leqslant l;\\
\Lambda^{\text{max}}(\rho_\mathds{k}(\bar{x}_i^p-\bar{x}_i^{[p]}))=\{(p\mathbf{e}_i,\mathbf{0},\mathbf{0},\mathbf{0})\}&\text{for}~1\leqslant  i\leqslant m;\\
\Lambda^{\text{max}}(\rho_\mathds{k}(\bar{u}_j^p-\bar{u}_j^{[p]}))=\{(\mathbf{0},\mathbf{0},p\mathbf{e}_j,\mathbf{0})\}& \text{for}~1\leqslant  j\leqslant  s;\\
\Lambda^{\text{max}}(\bar{\Theta}_k)=\{(\mathbf{0},\mathbf{e}_{k-l},\mathbf{0},\mathbf{0})\}& \text{for}~l+1\leqslant  k\leqslant  l+q;\\
\Lambda^{\text{max}}(\bar{\Theta}_{l+q+1})=\{(\mathbf{0},\mathbf{0},\mathbf{0},\mathbf{e}_{t})\}.&
\end{array}\]Since $Q_{\chi,\mathds{k}}$ is a Kazhdan filtrated $U(\mathfrak{g}_\mathds{k})$-module, this implies that
\[
\begin{array}{ll}
&\Lambda^{\text{max}}(\prod\limits_{i=1}^m\rho_\mathds{k}(\bar{x}_i^p-\bar{x}_i^{[p]})^{a_i}\cdot
\prod\limits_{i=1}^s\rho_\mathds{k}(\bar{u}_i^p-\bar{u}_i^{[p]})^{b_i}\cdot\bar{\Theta}_1^{c_1}\cdots\bar{\Theta}_l^{c_l}\bar{\Theta}_{l+1}^{d_1}\cdots\bar{\Theta}_{l+q}^{d_q}\bar{\Theta}_{l+q+1}^{d_{q+1}})\\
=&\{\sum\limits_{i=1}^{m}pa_i\mathbf{e}_i+\sum\limits_{j=1}^{l}c_j\mathbf{e}_j,\sum\limits_{i=1}^{q}d_i\mathbf{e}_i,\sum\limits_{i=1}^{s}pb_i\mathbf{e}_i,d_{q+1}\mathbf{e}_t\},
\end{array}\]
for all $(a_1,\cdots,a_m;b_1,\cdots,b_s;c_1,\cdots,c_l;d_1,\cdots,d_{q+1})\in\mathbb{Z}_+^m\times
\mathbb{Z}_+^s\times\mathbb{Z}_{p}^l\times\mathbb{Z}_2^{q+1}$. Since $\widehat{U}(\mathfrak{g}_\mathds{k},e)$ is generated as a $Z_p(\widetilde{\mathfrak{p}}_\mathds{k})$-module by the set $$\{\bar{\Theta}_1^{i_1}\cdots\bar{\Theta}_l^{i_l}\bar{\Theta}_{l+1}^{j_1}\cdots\bar{\Theta}_{l+q}^{j_q}\bar{\Theta}_{l+q+1}^{j_{q+1}}|
(\mathbf{i},\mathbf{j})\in\Lambda_{l}\times\Lambda'_{q+1}\},$$ it follows that for every $u\in\widehat{U}(\mathfrak{g}_\mathds{k},e)$ with $(n(u),N(u))=(d,d')$ there exists a $\mathds{k}$-linear combination $u'$ of the endomorphism
\begin{equation*}
u(\mathbf{a},\mathbf{b},\mathbf{c},\mathbf{d}):=\prod\limits_{i=1}^m\rho_\mathds{k}(\bar{x}_i^p-\bar{x}_i^{[p]})^{a_i}\cdot
\prod\limits_{i=1}^s\rho_\mathds{k}(\bar{u}_i^p-\bar{u}_i^{[p]})^{b_i}\cdot\bar{\Theta}_1^{c_1}\cdots\bar{\Theta}_l^{c_l}\bar{\Theta}_{l+1}^{d_1}\cdots\bar{\Theta}_{l+q}^{d_q}\bar{\Theta}_{l+q+1}^{d_{q+1}},
\end{equation*}for all $(a_1,\cdots,a_m;b_1,\cdots,b_s;c_1,\cdots,c_l;d_1,\cdots,d_{q+1})\in\mathbb{Z}_+^m\times
\mathbb{Z}_+^s\times\mathbb{Z}_{p}^l\times\mathbb{Z}_2^{q+1}$ (where $\mathbb{Z}_{p}^l=\Lambda_l$) with $\Lambda^{\text{max}}(u(\mathbf{a},\mathbf{b},\mathbf{c},
\mathbf{d}))\subseteq\Lambda^{\text{max}}(u)$ such that either $n(u-u')<d$ or $n(u-u')=d$ and $N(u-u')<d'$.

Order the tuples in $\Omega$ lexicographically and assume that $u\in U(\mathfrak{g}_\mathds{k},e)\cdot Z_p(\mathfrak{a}_\mathds{k})$ for all nonzero $u\in\widehat{U}(\mathfrak{g}_\mathds{k},e)$ with $(n(u),N(u))\prec(d,d')$ (when $(n(u),N(u))=(0,0)$ this is a valid assumption). Now let $u\in\widehat{U}(\mathfrak{g}_\mathds{k},e)$ be such that $(n(u),N(u))=(d,d')$. By the preceding remark we know that there exists $u'=\sum\limits_{(\mathbf{a},\mathbf{b},\mathbf{c},\mathbf{d})}\lambda_{\mathbf{a},\mathbf{b},\mathbf{c},\mathbf{d}}u(\mathbf{a},\mathbf{b},\mathbf{c},\mathbf{d})
$ with $\Lambda^{\text{max}}(u(\mathbf{a},\mathbf{b},\mathbf{c},
\mathbf{d}))\subseteq\Lambda^{\text{max}}(u)$ for all $(\mathbf{a},\mathbf{b},\mathbf{c},\mathbf{d})$ with $\lambda_{\mathbf{a},\mathbf{b},\mathbf{c},\mathbf{d}}\neq0$ such that $(n(u-u'),N(u-u'))\prec(d,d')$. Set
\[
\begin{array}{lll}
v(\mathbf{a},\mathbf{b},\mathbf{c},\mathbf{d})&:=&u((0,\cdots,0,a_{l+1},\cdots,a_m),\mathbf{b},\mathbf{0},\mathbf{0})\cdot\\
&&\prod\limits_{i=1}^{l}\bar{\Theta}^{pa_i}\cdot
(\bar{\Theta}_1^{c_1}\cdots\bar{\Theta}_l^{c_l}\bar{\Theta}_{l+1}^{d_1}\cdots\bar{\Theta}_{l+q}^{d_q}\bar{\Theta}_{l+q+1}^{d_{q+1}}).
\end{array}\]
Using \eqref{dede} it is easy to observe that $\Lambda^{\text{max}}(u(\mathbf{a},\mathbf{b},\mathbf{c},\mathbf{d}))=\Lambda^{\text{max}}(v(\mathbf{a},\mathbf{b},\mathbf{c},\mathbf{d}))$ and
\[
\begin{array}{lll}
&(n(u(\mathbf{a},\mathbf{b},\mathbf{c},\mathbf{d})-v(\mathbf{a},\mathbf{b},\mathbf{c},\mathbf{d})),N(u(\mathbf{a},\mathbf{b},\mathbf{c},\mathbf{d})-v(\mathbf{a},\mathbf{b},\mathbf{c},\mathbf{d})))\\
\prec&(n(u(\mathbf{a},\mathbf{b},\mathbf{c},\mathbf{d})),N(u(\mathbf{a},\mathbf{b},\mathbf{c},\mathbf{d}))).
\end{array}\]
We now put $u^{''}:=\sum\limits_{(\mathbf{a},\mathbf{b},\mathbf{c},\mathbf{d})}\lambda_{\mathbf{a},\mathbf{b},\mathbf{c},
\mathbf{d}}v(\mathbf{a},\mathbf{b},\mathbf{c},\mathbf{d})$, an element of $U(\mathfrak{g}_\mathds{k},e)\cdot Z_p(\mathfrak{a}_\mathds{k})$. Because $(n(u-u^{''}),N(u-u^{''}))\prec(n(u),N(u))$, the equality $\widehat{U}(\mathfrak{g}_\mathds{k},e)\cong U(\mathfrak{g}_\mathds{k},e)\cdot Z_p(\mathfrak{a}_\mathds{k})$ follows by induction on the length $(d,d')$ in the linearly ordered set $(\Omega,\prec)$.

(iv) It is immediate from Lemma~\ref{right}(1) and the procedure of ``modular $p$ induction'' that the vectors $\bar{X}^{(\mathbf{a},\mathbf{b},\mathbf{c},\mathbf{d})}\otimes1_\chi$ with
\[\begin{array}{lll}
(\mathbf{a},\mathbf{b},\mathbf{c},\mathbf{d})&=&(a_1,\cdots,a_{m-l};b_1,\cdots,b_{n-q};c_1,\cdots,c_{s};d_1,\cdots,d_{t'})\\
&\in &\mathbb{Z}_+^{m-l}\times\mathbb{Z}_2^{n-q}\times\mathbb{Z}_+^{s}\times\mathbb{Z}_2^{t'},
\end{array}\]
form a free basis of the right $U(\mathfrak{g}_\mathds{k},e)$-module $Q_{\chi,\mathds{k}}$. Since \eqref{grad} shows that $\bar X_i^{p}$ and $\bar X_i^{p}-\bar X_i^{[p]}$ have the same Kazhdan degree in $U(\mathfrak{g}_\mathds{k})$ for $1\leqslant i\leqslant m-l$ and $m+n-l-q+1\leqslant i\leqslant m+n-l-q+s$, respectively, and $Q_{\chi,\mathds{k}}$ is a Kazhdan filtered $U(\mathfrak{g}_\mathds{k})$-module, it follows that the vectors
$$\prod\limits_{i=1}^{m-l}\prod\limits_{j=m+n-l-q+1}^{m+n-l-q+s}\rho_\mathds{k}(\bar X_i^{p}-\bar X_i^{[p]})^{a_i}
\rho_\mathds{k}(\bar X_j^p-\bar X_j^{[p]})^{b_j}\cdot\bar{\Theta}_1^{c_1}\cdots\bar{\Theta}_l^{c_l}
\bar{\Theta}_{l+1}^{d_1}\cdots\bar{\Theta}_{l+q}^{d_q}\bar{\Theta}_{l+q+1}^{d_{q+1}}$$
are linearly independent, where $(\mathbf{a},\mathbf{b},\mathbf{c},\mathbf{d})\in\mathbb{Z}_+^{m-l}\times\mathbb{Z}_+^s\times\mathbb{Z}_+^l\times\mathbb{Z}_2^{q+1}$.

(iii) and (iv) yield that there is an isomorphism between $\mathds{k}$-algebras$$\widehat{U}(\mathfrak{g}_\mathds{k},e)\cong U(\mathfrak{g}_\mathds{k},e)\otimes_\mathds{k}Z_p(\mathfrak{a}_\mathds{k})$$when $d_1$ is odd.

(v) For the case when $d_1$ is even, the same discussion as (ii)---(iv) shows that the vectors $\{\bar{\Theta}_1^{i_1}\cdots\bar{\Theta}_l^{i_l}\bar{\Theta}_{l+1}^{j_1}\cdots\bar{\Theta}_{l+q}^{j_q}\}$ with $(\mathbf{i},\mathbf{j})\in\Lambda_{l}\times\Lambda'_{q}$ form a basis of $Z_p(\widetilde{\mathfrak{p}}_\mathds{k})$-module $\widehat{U}(\mathfrak{g}_\mathds{k},e)$. Then $\widehat{U}(\mathfrak{g}_\mathds{k},e)$ is a free $\rho_\mathds{k}(Z_p)$-module of rank $p^l2^q$. The vectors
$$\prod\limits_{i=1}^{m-l}\prod\limits_{j=m+n-l-q+1}^{m+n-l-q+s}\rho_\mathds{k}(\bar X_i^{p}-\bar X_i^{[p]})^{a_i}
\rho_\mathds{k}(\bar X_j^p-\bar X_j^{[p]})^{b_j}\cdot\bar{\Theta}_1^{c_1}\cdots\bar{\Theta}_l^{c_l}
\bar{\Theta}_{l+1}^{d_1}\cdots\bar{\Theta}_{l+q}^{d_q}$$
form a $\mathds{k}$-basis of the algebra $\widehat{U}(\mathfrak{g}_\mathds{k},e)$, where $(\mathbf{a},\mathbf{b},\mathbf{c},\mathbf{d})\in\mathbb{Z}_+^{m-l}\times\mathbb{Z}_+^s\times\mathbb{Z}_+^l\times\mathbb{Z}_2^{q}$.
This implies that $$\widehat{U}(\mathfrak{g}_\mathds{k},e)\cong U(\mathfrak{g}_\mathds{k},e)\otimes_\mathds{k}Z_p(\mathfrak{a}_\mathds{k})$$ as $\mathds{k}$-algebras, completing the proof.
\end{proof}

\section{On the minimal dimensional representations of finite $W$-superalgebras}

Let $\mathfrak{g}_\mathds{k}$ be a basic classical Lie superalgebra over positive characteristic field $\mathds{k}=\overline{\mathbb{F}}_p$. In ([\cite{WZ}], Theorem 4.3 and Theorem 5.6), Wang and Zhao introduced the Super Kac-Weisfeiler Property for $\mathfrak{g}_\mathds{k}$ (with some restrictions on $p$). In section 8 and section 9, we will discuss the existence of the minimal dimensional representations in the Super Kac-Weisfeiler Property with character $p\gg0$. Following Premet's treatment to the finite $W$-algebras in [\cite{P7}], we first need to estimate the minimal dimension of the representations for the finite $W$-superalgebra $U(\mathfrak{g},e)$ over $\mathbb{C}$.

The existence of $1$-dimensional representations for the finite $W$-algebra associated to a classical Lie algebra over $\mathbb{C}$ was obtained by Losev in ([\cite{L3}], Theorem 1.2.3(1)), which also can be found in his ICM talk in 2010 (see [\cite{L1}], Section 6). The main technique he adopted is the symplectic geometry machinery (i.e. Fedosov deformation quantization) and the knowledge of nilpotent orbits and primitive ideals for the Lie algebras. Using computational methods, Goodwin-R\"{o}hrle-Ubly$^{[\cite{GRU}]}$ proved that the $W$-algebras associated to exceptional Lie algebras $E_6,E_7,F_4,G_2$, or $E_8$ with $e$ not rigid, admit $1$-dimensional representations (see also [\cite{P7}]). Since the knowledge of finite $W$-superalgebras is very limited so far and there is no effective method to deal with them, the dimension of minimal representations for the $W$-superalgebras over $\mathbb{C}$ can only be reasonable estimated.

In this section, the basic classical Lie superalgebra $\mathfrak{g}$ over $\mathbb{C}$ will be referred to all except for type $D(2,1;a)(a\notin\mathbb{Q})$, whereas the basic classical Lie superalgebra $\mathfrak{g}_\mathds{k}$ over $\mathds{k}=\overline{\mathbb{F}}_p$ will be referred to all types including the case $D(2,1;\bar a)(\bar a\in\mathds{k}\backslash\{\bar0,\overline {-1}\})$ (recall that the translation $\mathds{k}$-algebra $U(\mathfrak{g}_\mathds{k},e)$ with $\mathfrak{g}_\mathds{k}\in D(2,1;\bar a)(\bar a\in\mathds{k}\backslash\{\bar0,\overline {-1}\})$ can be induced from the $\mathbb{C}$-algebra $U(\mathfrak{g},e)$ associated to Lie superalgebra $D(2,1;a)(a\in\mathbb{Q}\backslash\{0,1\})$ by ``modular $p$ reduction''). For more detail we refer to Remark~\ref{D(2,1)}.

Since the parity of $d_1$ ({\bf recall that $\text{dim}~\mathfrak{g}(-1)_{\bar1}$ and $d_1=\text{dim}~\mathfrak{g}_{\bar1}-\text{dim}~\mathfrak{g}^e_{\bar1}$ have the same parity by Remark~\ref{centralizer}}) plays the key role for the construction of finite $W$-superalgebra $U(\mathfrak{g},e)$, for each case we will consider separately.

Recall that in Section 4.1 we have identified $e,\,h,\,f\in(\mathfrak{g}_A)_{\bar0}$ with the nilpotent elements $\bar{e}=e\otimes1,\bar{h}=h\otimes1,\bar{f}=f\otimes1$ in $(\mathfrak{g}_\mathds{k})_{\bar0}\cong(\mathfrak{g}_A)_{\bar0}\otimes_A\mathds{k}$  respectively, and $\chi\in\mathfrak{g}_A^*$ with the linear function $(e,\cdot)$ on $\mathfrak{g}_\mathds{k}$. In this section we will continue to adopt these notations.

\subsection{The characterization of $1$-dimensional representations for finite $W$-superalgebras when $d_1$ is even}

In this part, we will consider the case when $d_1$ is even. In fact, a large number of examples can be given in this situation, e.g. $\mathfrak{g}=\mathfrak{gl}(M|N)$ (note that which is not simple, but shares a lot of common properties with $\mathfrak{sl}(M|N)$). Recall that in Remark~\ref{notreduced} we have mentioned that Wang-Zhao obtained the explicit description for the Dynkin grading of basic classical Lie superalgebras of all types over positive characteristic field $\mathds{k}$ in [\cite{WZ}]. In the case when the character of $\mathds{k}$ satisfies $p\gg0$, which can also be regarded as the grading over $\mathbb{C}$. In ([\cite{WZ}], Section 3.2) Wang-Zhao computed the dimension of $\mathfrak{gl}(M|N)^e_{\bar1}$ (as a vector space) for any even nilpotent element $e\in\mathfrak{gl}(M|N)_{\bar0}$ and showed that which is an even number. As the dimension of $\mathfrak{gl}(M|N)_{\bar1}$ is always even, it is immediate from Remark~\ref{centralizer} that $d_1$ is also an even number.

Now we first formulate a {\bf conjecture} that each finite $W$-superalgebra $U(\mathfrak{g},e)$ over $\mathbb{C}$ admits a $1$-dimensional representation, and denote it by $V:=\mathbb{C}v$.

By Theorem~\ref{PBWC}(1) it follows that the elements $\Theta_1,\cdots,\Theta_l,\Theta_{l+1},\cdots,\Theta_{l+q}$ generate the algebra $U(\mathfrak{g},e)$ with $\Theta_1,\cdots,\Theta_l\in U(\mathfrak{g},e)_{\bar0}$ \& $\Theta_{l+1},\cdots,\Theta_{l+q}\in U(\mathfrak{g},e)_{\bar1}$. Keep in mind that all the modules considered are $\mathbb{Z}_2$-graded. Let $M$ be a $U(\mathfrak{g},e)$-module, then pick an odd element $u\in U(\mathfrak{g},e)_{\bar1}$ and a $\mathbb{Z}_2$-homogeneous element $m\in M$. It is obvious that the parity of $m$ changes when $u$ acts on it, i.e. the elements $m$ and $u.m$ have different parity. Since the vector space $V$ is $1$-dimensional \& $\mathbb{Z}_2$-graded, and $\Theta_{l+1},\cdots,\Theta_{l+q}\in U(\mathfrak{g},e)_{\bar1}$, then $\Theta_i.v=0$ for $l+1\leqslant  i\leqslant  l+q$. For $1\leqslant  i\leqslant  l$, set $\Theta_i.v=c_iv$ with the constants $c_i\in\mathbb{C}$ not all zero. Recall Theorem~\ref{relationc}(1) shows that the algebra $U(\mathfrak{g},e)$ is completely determined by the commuting relations of $\Theta_1,\cdots,\Theta_{l+q}$. We have

(i) for $1\leqslant  i<j\leqslant  l$, the elements $[\Theta_i,\Theta_j]$ are even since $\Theta_i,\Theta_j\in U(\mathfrak{g},e)_{\bar0}$. It is immediate from $[\Theta_i,\Theta_j].v=(\Theta_i\cdot\Theta_j-\Theta_j\cdot\Theta_i).v=(c_ic_j-c_jc_i).v=0$ that $V$ is a $1$-dimensional representation of $U(\mathfrak{g},e)$ iff $F_{ij}(\Theta_1,\cdots,\Theta_{l+q})$ (see Theorem~\ref{relationc}) acts on $V$ trivially. Recall that each $F_{ij}(\Theta_1,\cdots,\Theta_{l+q})$ is a polynomial superalgebra in $l+q$ variables, and $\Theta_i.v=0$ for $l+1\leqslant  i\leqslant  l+q$ by preceding remark. Since each polynomial $F_{ij}(\Theta_1,\cdots,\Theta_{l+q})$ can be written as a $\mathbb{C}$-linear combination of $\Theta_1^{a_1}\cdots\Theta_l^{a_l}\Theta_{l+1}^{b_1}\cdots\Theta_{l+q}^{b_q}$ with $a_i's\in\mathbb{Z}_+$ \& $b_i's\in\mathbb{Z}_2$, after deleting all the terms in which any of the odd elements $\Theta_{l+1},\cdots,\Theta_{l+q}$ occurs, we can obtain a polynomial (in the usual sense, not super) in $l$ variables, and denote it by $F'_{ij}(\Theta_1,\cdots,\Theta_{l})$. Then the formulas in Theorem~\ref{relationc}(1) shows that $F'_{ij}(\Theta_1,\cdots,\Theta_{l}).v=0$ for $1\leqslant  i<j\leqslant  l$.

(ii) for $l+1\leqslant  i\leqslant  j\leqslant  l+q$, the elements $[\Theta_i,\Theta_j]$ are still even since $\Theta_i,\Theta_j$ are both odd. As $\Theta_i.v=0$ for $l+1\leqslant  i\leqslant  l+q$, we have $[\Theta_i,\Theta_j].v=(\Theta_i\cdot\Theta_j+\Theta_j\cdot\Theta_i).v=0$. By the same discuss as (i) we can also get polynomials $F'_{ij}(\Theta_1,\cdots,\Theta_{l})$ for $l+1\leqslant  i\leqslant  j\leqslant  l+q$, and $V$ is $1$-dimensional iff $F'_{ij}(\Theta_1,\cdots,\Theta_{l}).v=0$.

(iii) for $1\leqslant  i\leqslant  l<j\leqslant  l+q$, the elements $[\Theta_i,\Theta_j]$ are odd since $\Theta_i\in U(\mathfrak{g},e)_{\bar0}$ and $\Theta_j\in U(\mathfrak{g},e)_{\bar1}$. As $\Theta_i.v=0$ for $l+1\leqslant  i\leqslant  l+q$, we have that $[\Theta_i,\Theta_j].v=(\Theta_i\cdot\Theta_j-\Theta_j\cdot\Theta_i).v=0$. Hence $V$ is $1$-dimensional iff $F_{ij}(\Theta_1,\cdots,\Theta_{l+q})$ acts on $V$ trivially. As Theorem~\ref{relationc}(1) shows that $[\Theta_i,\Theta_j]=F_{ij}(\Theta_1,\cdots,\Theta_{l+q})$, it is immediate that all $F_{ij}(\Theta_1,\cdots,\Theta_{l+q})'s$ are odd elements. Therefore, when we put each polynomial $F_{ij}(\Theta_1,\cdots,\Theta_{l+q})$ as a $\mathbb{C}$-linear combination of monomials $\Theta_1^{a_1}\cdots\Theta_l^{a_l}\Theta_{l+1}^{b_1}\cdots\Theta_{l+q}^{b_q}$ with $a_i's\in\mathbb{Z}_+$ \& $b_i's\in\mathbb{Z}_2$, some odd element $\Theta_k$ (for $l+1\leqslant  k\leqslant  l+q$) will occur at least once in each given monomial. Since $\Theta_k.v=0$ for $l+1\leqslant  k\leqslant  l+q$, the equations $F_{ij}(\Theta_1,\cdots,\Theta_{l+q}).v=0$ are trivial for $1\leqslant  i\leqslant  l<j\leqslant  l+q$. In this case we do not get any new equations.

Let $U(\mathfrak{g},e)^{\text{ab}}$ denote the factor-algebra $U(\mathfrak{g},e)/R$, where $R$ is the ideal of $U(\mathfrak{g},e)$ generated by all the odd generators $\Theta_{l+1},\cdots,\Theta_{l+q}$ and all commutators $[a, b]$ with $a,b\in U(\mathfrak{g},e)$. It is obvious that the algebra $U(\mathfrak{g},e)^{\text{ab}}$ is isomorphic to the algebra $\mathbb{C}[X_1,\cdots,X_l]/\Lambda$, where $\mathbb{C}[X_1,\cdots,X_l]$ is a polynomial algebra (in the usual sense) in $l$ variables, and $\Lambda$  the ideal of $\mathbb{C}[X_1,\cdots,X_l]$ generated by all $F'_{ij}(X_1,\cdots,X_l)$ for $1\leqslant  i<j\leqslant  l$ and $l+1\leqslant  i\leqslant  j\leqslant  l+q$.  By Hilbert's Nullstellensatz, the maximal spectrum $\mathscr{E}:=\text{Specm}~U(\mathfrak{g},e)^{\text{ab}}$ parametrises the $1$-dimensional representations of $U(\mathfrak{g},e)$.

We now wish to describe the algebra $U(\mathfrak{g},e)^{\text{ab}}$ for $\mathfrak{g}=\mathfrak{gl}(M|N)$ with the Jordan type of $e$ satisfying certain condition. We are going to rely on the explicit presentation of $U(\mathfrak{g},e)$ obtained by Peng in [\cite{Peng3}]. Let $\lambda=(p_{n+1}\geqslant\cdots\geqslant p_1)$ be a partition of $(M|N)$ with $n+1$ parts. As in [\cite{Peng3}], we associate with $\lambda$ an even nilpotent element $e=e_\lambda\in\mathfrak{gl}(M|N)_{\bar0}$ of Jordan type $(p_1,p_2,\cdots,p_{n+1})$ satisfying that $e_\lambda=e_M\oplus e_N$, where $e_M$ is principal nilpotent in $\mathfrak{gl}(M|0)$ and the sizes of the Jordan blocks of $e_N$ are all greater or equal to $M$. By ([\cite{Peng3}], Theorem 9.1), the finite $W$-superalgebra $U(\mathfrak{g},e)$ is isomorphic to the truncated shifted super Yangian $Y_{1|n}^l(\sigma)$ of level $l:=p_{n+1}$. Here $\sigma$ is an upper triangular matrix of order $n+1$ with nonnegative integral entries; see ([\cite{Peng3}], Section 7) for more detail. It follows from the main results of [\cite{Peng3}] that $U(\mathfrak{g},e)$ is generated by elements
\begin{eqnarray*}
&\{D_i^{(r)}~|~1\leqslant i\leqslant  n + 1, r \geqslant 1\},&\\
&\{E_i^{(r)}~|~1\leqslant i\leqslant  n, r>p_{i+1}-p_i\},&\\
&\{F_i^{(r)}~|~1\leqslant i\leqslant  n, r \geqslant 1\}&
\end{eqnarray*}subject to certain relations (see ([\cite{Peng3}], (2.3)---(2.15)); also ([\cite{Gow}], (38)---(50))), where
$\{E^{(r)}_1~|~r>p_2-p_1\}\cup\{F^{(r)}_1~|~r \geqslant 1\}$ are the only odd generators, and $D_1^{(r)}=0$ for $r>p_1$. As noted by Peng in ([\cite{Peng3}], Section 7.1), dim~$\mathfrak{g}(-1)_{\bar{1}}$ is even in this case.

\begin{prop}\label{glmn}
Let $\mathfrak{g}=\mathfrak{gl}(M|N)$ and $e=e_\lambda\in\mathfrak{gl}(M|N)_{\bar0}$ be of Jordan type $(p_1,p_2,\cdots,p_{n+1})$ satisfying that $e_\lambda=e_M\oplus e_N$, where $e_M$ is principal nilpotent in $\mathfrak{gl}(M|0)$ and the sizes of the Jordan blocks of $e_N$ are all greater or equal to $M$. Then $U(\mathfrak{g},e)^{\text{ab}}$ is isomorphic to a polynomial algebra in $l=p_{n+1}$ variables.
\end{prop}

\begin{proof}
First delete all the odd generators $\{E^{(r)}_1~|~r>p_2-p_1\}\cup\{F^{(r)}_1~|~r \geqslant 1\}$ and denote by $d^{(r)}_i, e^{(r)}_i, f^{(r)}_i$ the images of $D^{(r)}_i, E^{(r)}_i, F^{(r)}_i$ in $U(\mathfrak{g},e)^{\text{ab}}$. Applying ([\cite{Peng3}], (2.6) and (2.7)) with $r=1$ we see that $e^{(s)}_i=f^{(s)}_i=0$ for all $2\leqslant i\leqslant n$ and $s\geqslant1$. By ([\cite{Gow}], Theorem 3) the elements $D^{(r)}_i$ and $D^{(s)}_j$ commute for all $i,j\leqslant n$ and all $r,s$ (it can be calculated that the right hand of (2.5) in [\cite{Peng3}] also equals to zero). Note that in the defining relations of truncated shifted super Yangian $Y_{1|n}^l(\sigma)$, there are other elements $D'^{(r)}_i$ for $1\leqslant i\leqslant n+1$ \& $r\geqslant0$ occur, which are given by the equations $\sum\limits_{t=0}\limits^{r}D_i^{(t)}D'^{(r-t)}_i=\delta_{r0}$. In the definition of truncated shifted Yangian $Y_{n+1,l}(\sigma)$ ([\cite{BK2}], Section 2), similar elements are introduced as follows:

let $$D_i(u):=\sum\limits_{r\geqslant0}D_i^{(r)}u^{-r}\in Y_{n+1,l}(\sigma)[[u^{-1}]],$$ where $D_i^{(0)}:=1$, then define $D'^{(r)}_i$ of $Y_{n+1,l}(\sigma)$ by $$D'_i(u)=\sum\limits_{r\geqslant0}D'^{(r)}_i u^{-r}:=-D_i(u)^{-1}.$$

In fact, if one changes all $D'^{(r)}_i$ to their negative counterparts $-D'^{(r)}_i$ for $1\leqslant i\leqslant n+1$ \& $r\geqslant0$ in $Y_{1|n}^l(\sigma)$, it can be easily verified that they are of the same definition as the Yangian case $Y_{n+1,l}(\sigma)$ by comparing the coefficients of $u^{-s}$ for $s\geqslant0$. From above discussion, it is immediate that $U(\mathfrak{g},e)^{\text{ab}}$ is isomorphic to the algebra generated by $d^{(r)}_i$ for $1\leqslant i\leqslant  n + 1, r \geqslant 1$ subject to the following defining relations:
\begin{eqnarray*}
[d_i^{(r)},d_j^{(s)}]~&=&0\\
\sum\limits_{t=0}\limits^{r+s-1}d'^{(t)}_id_{i+1}^{(r+s-1-t)}&=&0
\end{eqnarray*}
with $d_1^{(r)}=0$ for $r>p_1$ by ([\cite{Gow}], Theorem 3) and ([\cite{Peng3}], Definition 2.1) (note that the formulas (40) in [\cite{Gow}] are missing in Peng's paper, and Peng will update the missing relations in the next revision of [\cite{Peng3}]).

If we change all $d'^{(r)}_i$ to their negative counterparts $-d'^{(r)}_i$ for $1\leqslant i\leqslant n+1$ \& $r\geqslant0$, it is immediate that the algebra $U(\mathfrak{g},e)^{\text{ab}}$ is isomorphic to the commutative quotient $U(\mathfrak{g}',e)^{\text{ab}}$ of finite $W$-algebra $U(\mathfrak{g}',e)$ associated to Lie algebra $\mathfrak{g}'=\mathfrak{gl}(M+N)$ as their generators and defining relations are of the same name, see ([\cite{BK2}], (2.2)---(2.15)) and ([\cite{P7}], Section 3.8). Recall that $U(\mathfrak{g}',e)$ is isomorphic to the shifted truncated Yangian $Y_{n+1,l}(\sigma)$ of level $l=p_{n+1}$ by ([\cite{BK2}], Theorem 10.1), and ([\cite{P7}], Theorem 3.3) shows that $U(\mathfrak{g}',e)^{\text{ab}}$ is isomorphic to a polynomial algebra in $l=p_{n+1}$ variables. Therefore, $U(\mathfrak{g},e)^{\text{ab}}$ is isomorphic to a polynomial algebra in $l=p_{n+1}$ variables, completing the proof.
\end{proof}

\begin{rem}
It is notable that the results obtained by Premet$^{[\cite{P7}]}$ for the finite $W$-algebra case are applied in the proof. In fact, same discussion can also be carried to show that the algebra $U(\mathfrak{g},e)^{\text{ab}}$ given in Proposition~\ref{glmn} is generated by $l$ elements. For the finite $W$-algebra case, it is remarkable that the independence of these elements is obtained in virtue of the knowledge of sheets of Lie algebras. Since the related theory for Lie superalgebras has not yet been developed, Proposition~\ref{glmn} can not be proved by the same means as the finite $W$-algebra case in ([\cite{P7}], Theorem 3.3).
\end{rem}

\begin{corollary}\label{slmn}
Let $\mathfrak{g}=\mathfrak{sl}(M|N)$ and $e=e_\lambda\in\mathfrak{sl}(M|N)_{\bar0}$ be of Jordan type $(p_1,p_2,\cdots,$\\$p_{n+1})$ satisfying that $e_\lambda=e_M\oplus e_N$, where $e_M$ is principal nilpotent in $\mathfrak{gl}(M|0)$ and the sizes of the Jordan blocks of $e_N$ are all greater or equal to $M$. Then $U(\mathfrak{g},e)^{\text{ab}}\cong\mathbb{C}[X_1,\cdots,X_{l-1}]$, $l=p_{n+1}$.
\end{corollary}

\begin{proof}
The proof is based on Proposition~\ref{glmn}. Repeat verbatim the proof of ([\cite{P7}], Corollary 3.2) but apply ([\cite{Peng3}], (8.5), (8.6) \& Corollary 8.3) in place of ([\cite{BK2}], Example 9.1 \& Corollary 9.4).
\end{proof}

Recall that all the coefficients of $F_{ij}$ for $1\leqslant  i,j\leqslant  l$ are over the admissible ring $A$ by the remark preceding Theorem~\ref{translation}. Then all the coefficients of $F'_{ij}$ for $1\leqslant  i,j\leqslant  l$ are also over $A$. Let $\mathscr{E}(\mathbb{C})$ denote the set of all common zeros of the polynomials $F'_{ij}$ in the affine space $\mathbb{A}^l_\mathbb{C}$ where $1\leqslant  i<j\leqslant  l$, or $l+1\leqslant  i\leqslant  j\leqslant  l+q$. It is immediate that

\begin{lemma}\label{d_1even}
When $d_1$ is even, the $A$-defined Zariski closed set $\mathscr{E}(\mathbb{C})$ parametrises the $1$-dimensional representations of the finite $W$-superalgebra $U(\mathfrak{g},e)$ associated to basic classical Lie superalgebra $\mathfrak{g}$ over $\mathbb{C}$.
\end{lemma}

Set $^pF'_{ij}:=F'_{ij}\otimes_A\mathds{k}$, the polynomials over $\mathds{k}$, and denote by $\mathscr{E}(\mathds{k})$ the set of all common zeros of the polynomials $^pF_{ij}'$ in the affine space $\mathbb{A}^l_\mathds{k}$ where $1\leqslant  i<j\leqslant  l$, or $l+1\leqslant  i\leqslant  j\leqslant  l+q$. Applying Theorem~\ref{translation}(1) it follows that the Zariski closed set $\mathscr{E}(\mathds{k})$ parametrises the $1$-dimensional representations of the finite $W$-superalgebra $U(\mathfrak{g}_\mathds{k},e)$ over $\mathds{k}$.

Let $U(\mathfrak{g},e)$ be a finite $W$-superalgebra over $\mathbb{C}$. The discussion in Section 7 shows that the transition subalgebra $U(\mathfrak{g}_\mathds{k},e)$ over $\mathds{k}$ is induced from the algebra $U(\mathfrak{g},e)$ by ``modular $p$ reduction'', thus we have

\begin{lemma}\label{trans1}
When $d_1$ is even, if the algebra $U(\mathfrak{g},e)$ affords $1$-dimensional representations, then the transition subalgebra $U(\mathfrak{g}_\mathds{k},e)$ over $\mathds{k}=\overline{\mathbb{F}}_p$ also admits $1$-dimensional representations.
\end{lemma}

\begin{proof}
Lemma~\ref{d_1even} and the discussion thereafter translate the consideration for the $1$-dimensional representations of the algebras $U(\mathfrak{g},e)$ and $U(\mathfrak{g}_\mathds{k},e)$ into the discussion of the Zariski closed sets $\mathscr{E}(\mathbb{C})$ and $\mathscr{E}(\mathds{k})$, respectively. Since $F'_{ij}$ and $^pF_{ij}'$ for $1\leqslant  i<j\leqslant  l$, or $l+1\leqslant  i\leqslant  j\leqslant  l+q$ are all polynomials in the usual sense (not super), the Lemma can be proved in the same way as the finite $W$-algebra case, thus will be omitted here (see [\cite{P7}], Theorem 2.2(a)).
\end{proof}

\begin{lemma}\label{mindim1}
When $d_1$ is even, if the finite $W$-superalgebra $U(\mathfrak{g},e)$ over $\mathbb{C}$ affords a $1$-dimensional representation, then for $p\gg0$ there exists $\eta\in\chi+(\mathfrak{m}_\mathds{k}^\bot)_{\bar{0}}$  associated to which the reduced enveloping algebra $U_\eta(\mathfrak{g}_\mathds{k})$ admits irreducible representations of dimension $p^{\frac{d_0}{2}}2^{\frac{d_1}{2}}$.
\end{lemma}

\begin{proof}
Since $U(\mathfrak{g},e)$ affords a $1$-dimensional representation, it is immediate from Lemma~\ref{trans1} that the $\mathds{k}$-algebra $U(\mathfrak{g}_\mathds{k},e)$ affords a $1$-dimensional representation, either. Recall that Theorem~\ref{keyisotheorem}(3) shows  $\widehat{U}(\mathfrak{g}_\mathds{k},e)\cong U(\mathfrak{g}_\mathds{k},e)\otimes_\mathds{k}Z_p(\mathfrak{a}_\mathds{k})$ as  $\mathds{k}$-algebras. This yields that the $\mathds{k}$-algebra $\widehat{U}(\mathfrak{g}_\mathds{k},e)$ affords a $1$-dimensional representation too; we call it $\nu$.
By the proof of Theorem~\ref{keyisotheorem}, $\rho_\mathds{k}(Z_p)\cap\text{Ker}~\nu$ is a maximal ideal of the algebra $\rho_\mathds{k}(Z_p)\cong Z_p(\widetilde{\mathfrak{p}}_\mathds{k})\cong\mathds{k}[(\chi+(\mathfrak{m}_\mathds{k}^\bot)_{\bar{0}})^{(1)}]$ (where $(\chi+(\mathfrak{m}_\mathds{k}^\bot)_{\bar{0}})^{(1)}$ is the Frobenius twist of $\chi+(\mathfrak{m}_\mathds{k}^\bot)_{\bar{0}}$). So there exists $\eta\in\chi+(\mathfrak{m}_\mathds{k}^\bot)_{\bar{0}}$ such that $\rho_\mathds{k}(\bar x^p-\bar x^{[p]}-\eta(\bar x)^p)\in\text{Ker}~\nu$ for all $\bar x\in
(\mathfrak{g}_\mathds{k})_{\bar{0}}$. Our choice of $\eta$ ensures that the $\mathds{k}$-algebra $\widehat{U}_\eta(\mathfrak{g}_\mathds{k},e):=\widehat{U}(\mathfrak{g}_\mathds{k},e)
\otimes_{Z_p(\widetilde{\mathfrak{p}}_\mathds{k})}\mathds{k}_\eta$ affords a $1$-dimensional representation. On the other hand, the canonical projection $Q_{\chi,\mathds{k}}\twoheadrightarrow Q_{\chi,\mathds{k}}/J_\eta Q_{\chi,\mathds{k}}=Q_\chi^\eta$ gives rise to an algebra homomorphism
\begin{equation*}
\rho_\eta:\widehat{U}_\eta(\mathfrak{g}_\mathds{k},e)\rightarrow
(\text{End}_{\mathfrak{g}_\mathds{k}}Q_\chi^\eta)^{\text{op}}=U_\eta(\mathfrak{g}_\mathds{k},e).
\end{equation*} As $\text{dim}~\widehat{U}_\eta(\mathfrak{g}_\mathds{k},e)\leqslant  p^l2^q$ by Theorem~\ref{keyisotheorem}(2), applying Theorem~\ref{sumresult}(4)(i) yields that $\rho_\eta$ is an algebra isomorphism, i.e. $U_\eta(\mathfrak{g}_\mathds{k},e)$ admits a $1$-dimensional representation. As$$U_\eta(\mathfrak{g}_\mathds{k})\cong\text{Mat}_{p^{\frac{d_0}{2}}2^{\frac{d_1}{2}}}(U_\eta(\mathfrak{g}_\mathds{k},e))$$ by Theorem~\ref{sumresult}(3), it follows that the algebra $U_\eta(\mathfrak{g}_\mathds{k})$ has an irreducible representation of dimension $p^{\frac{d_0}{2}}2^{\frac{d_1}{2}}$.
\end{proof}

\subsection{The characterization of $2$-dimensional representations for finite $W$-superalgebras when $d_1$ is odd}

In this part we will consider the case when $d_1$ is odd. First note that

\begin{prop}\label{no1}
When $d_1$ is odd, the finite $W$-superalgebra $U(\mathfrak{g},e)$ over $\mathbb{C}$ can not afford a $1$-dimensional representation.
\end{prop}

\begin{proof}
Recall that Theorem~\ref{PBWC}(2)(i) shows $\Theta_{l+q+1}\in U(\mathfrak{g},e)$, and $$\Theta_{l+q+1}^2(1_\chi)=v_{\frac{r+1}{2}}^2\otimes1_\chi=\frac{1}{2}[v_{\frac{r+1}{2}},v_{\frac{r+1}{2}}]\otimes1_\chi=\frac{1}{2}\chi([v_{\frac{r+1}{2}},v_{\frac{r+1}{2}}])\otimes1_\chi=\frac{1}{2}\otimes1_\chi,$$
thus$$\Theta_{l+q+1}^2=\frac{1}{2}\text{id}.$$

Keep in mind that all the modules considered are $\mathbb{Z}_2$-graded. For any $U(\mathfrak{g},e)$-module $M$, let $0\neq v\in M$ be a $\mathbb{Z}_2$-homogeneous element. It is easy to verify that $0\neq\Theta_{l+q+1}.v\in M$. If not, i.e. $\Theta_{l+q+1}.v=0$, then $\Theta_{l+q+1}^2.v=0$. By the preceding remark we have $\Theta_{l+q+1}^2.v=\frac{1}{2}v$, then $v=0$, a contradiction. Therefore, the dimension of any $U(\mathfrak{g},e)$-module (as a vector space) is at least $2$, and the algebra $U(\mathfrak{g},e)$ can not afford a $1$-dimensional  representation.
\end{proof}

Recall that in Theorem~\ref{W-C2} we have\[\begin{array}{lcll}
\phi:&(\text{End}_\mathfrak{g}Q_{\chi})^{\text{op}}&\cong&Q_{\chi}^{\text{ad}\mathfrak{m}}\\ &\Theta&\mapsto&\Theta(1_\chi),
\end{array}
\]as $\mathbb{C}$-algebras. In the following we will identify $(\text{End}_\mathfrak{g}Q_{\chi})^{\text{op}}$ with $Q_{\chi}^{\text{ad}\mathfrak{m}}$ as $\mathbb{C}$-algebras, and which will cause no confusion. For any $\mathbb{Z}_2$-homogeneous elements $\Theta_1,\Theta_2\in Q_{\chi}^{\text{ad}\mathfrak{m}}$, since $\phi(\Theta_1\cdot\Theta_2)=\phi(\Theta_1)\phi(\Theta_2)$,
we have $\Theta_1\cdot\Theta_2:=\Theta_1(1_\chi)\cdot\Theta_2(1_\chi)$. When $d_1$ is odd, the element $\Theta_{l+q+1}$ in $(\text{End}_\mathfrak{g}Q_{\chi})^{\text{op}}$ can be considered as the element $v_{\frac{r+1}{2}}\otimes1_\chi$ in $Q_{\chi}^{\text{ad}\mathfrak{m}}$ by Theorem~\ref{PBWC}(2)(i), and $\Theta_{l+q+1}^2.v=\frac{1}{2}v$ for any $v\in Q_{\chi}$.

Recall that when $d_1$ is odd, the algebra $U(\mathfrak{g},e)$ can not afford $1$-dimensional representations. In this case a {\bf conjecture} can also be formulated, i.e. each finite $W$-superalgebra $U(\mathfrak{g},e)$ over $\mathbb{C}$ admits a $2$-dimensional representation, and denote it by $V$. In Definition~\ref{rewcc} we have introduced the $\mathbb{C}$-algebra $W'_\chi$, a proper subalgebra of $U(\mathfrak{g},e)$. In virtue of this algebra, we can formulate a more precise description for the $2$-dimensional representations of the $\mathbb{C}$-algebra $U(\mathfrak{g},e)$.

\begin{prop}\label{typeq}
As a $W'_\chi$-module, $V$ is of type $Q$.
\end{prop}

\begin{proof}
Define the $\mathbb{C}$-mapping
\[\begin{array}{lcll}
\tau:&V&\rightarrow&V\\ &v&\mapsto&\sqrt{2}\Theta_{l+q+1}.v.
\end{array}
\]

It is easy to verify the mapping $\tau$ is odd and surjective. In fact, $\tau$ is also injective since $\tau^2(v)=2\Theta_{l+q+1}^2.v=v$. We claim that $\tau$ is a homomorphism of $W'_\chi$-module $V$.

For any $\mathbb{Z}_2$-homogeneous elements $\Theta\in W'_\chi$ and $v\in V$, we have $\tau(\Theta.v)=\sqrt{2}\Theta_{l+q+1}.$\\$\Theta.v$ by definition. Since $\Theta_{l+q+1}\in\mathfrak{m}'$ and $\Theta\in Q_\chi^{\text{ad}\mathfrak{m}'}$, then $[\Theta_{l+q+1},\Theta]=0$. Moreover, $[\Theta_{l+q+1},\Theta]=\Theta_{l+q+1}\cdot\Theta-(-1)^{|\Theta|}\Theta\cdot\Theta_{l+q+1}$. It is immediate that $\Theta_{l+q+1}\Theta=(-1)^{|\Theta|}\Theta\cdot\Theta_{l+q+1}$, then $\Theta_{l+q+1}.\Theta.v=(-1)^{|\Theta|}\Theta\cdot\Theta_{l+q+1}.v$, i.e. $\tau(\Theta.v)=(-1)^{|\Theta|}\Theta.\tau(v)$.

From preceding remark it can be concluded that $\tau$ is an odd homomorphism of $W'_\chi$-module $V$. As $V$ is also irreducible as a $W'_\chi$-module, all the discussions above imply that $V$ is of type $Q$, completing the proof.
\end{proof}

Recall that in Theorem~\ref{relationc}(2) we have chosen the homogeneous elements $\Theta_1,\cdots,$\\$\Theta_{l+q+1}$ as a set of generators for the $\mathbb{C}$-algebra $U(\mathfrak{g},e)$ subject to the relations
$$[\Theta_i,\Theta_j]=F_{ij}(\Theta_1,\cdots,\Theta_{l+q+1}),\qquad[\Theta_j,\Theta_i]=-(-1)^{|\Theta_i||\Theta_j|}[\Theta_i,\Theta_j]$$
where $1\leqslant i<j\leqslant l ~\&~l+1\leqslant i\leqslant j\leqslant l+q+1$~ \& $1\leqslant i\leqslant l<j\leqslant l+q+1$.
To ease notation, we will denote $F_{ij}(\Theta_1,\cdots,\Theta_{l+q+1})$ by $F_{ij}$ for short.

If the $\mathbb{C}$-algebra $U(\mathfrak{g},e)$ affords $2$-dimensional representation $V$, Proposition~\ref{typeq} yields that $V$ is $\mathbb{C}$-spanned by an even element $v\in V_{\bar0}$ and the odd element $\Theta_{l+q+1}.v\in V_{\bar1}$. Hence we can get $4(l+q+1)$ variables $k_i^0,k_i^1,K_i^0,K_i^1\in\mathbb{C}$ such that
\begin{equation}\label{kKdef}
\Theta_{i}.v=k_i^0v+k_i^1\Theta_{l+q+1}.v,\qquad\Theta_{i}.\Theta_{l+q+1}.v=K_i^0v+K_i^1\Theta_{l+q+1}.v,
\end{equation}where $1\leqslant i\leqslant l+q+1$.

Similarly, for $1\leqslant i<j\leqslant l~\&~l+1\leqslant i\leqslant j\leqslant l+q+1$ \& $1\leqslant i\leqslant l<j\leqslant l+q+1$, there exist $2((l+q)^2+l+3q+2)$ variables
$(F_{ij})^0_{\bar{0}},(F_{ij})^0_{\bar{1}},(F_{ij})^1_{\bar{0}},(F_{ij})^1_{\bar{1}}\in\mathbb{C}$ such that
\begin{equation}\label{F_ijdefn}
F_{ij}.v=(F_{ij})^0_{\bar{0}}v+(F_{ij})^0_{\bar{1}}\Theta_{l+q+1}.v,\qquad F_{ij}.\Theta_{l+q+1}.v=(F_{ij})^1_{\bar{0}}v+(F_{ij})^1_{\bar{1}}\Theta_{l+q+1}.v.
\end{equation}

It is worth noting that each polynomial $F_{ij}$ is generated by the $\mathbb{Z}_2$-homogeneous elements $\Theta_1,\cdots,\Theta_{l+q+1}$ of $U(\mathfrak{g},e)$ over $\mathbb{Q}$. Therefore, the action of $F_{ij}$ on $v$ and $\Theta_{l+q+1}.v$ is completely determined by the constants in \eqref{kKdef}. Then $(F_{ij})^0_{\bar{0}},(F_{ij})^0_{\bar{1}},(F_{ij})^1_{\bar{0}},(F_{ij})^1_{\bar{1}}$ can be written as a $\mathbb{Q}$-linear combination of the products of the elements $k_i^0,k_i^1,K_i^0,K_i^1$, thus there are no new variables appear in \eqref{F_ijdefn}.

Since each $\Theta_i$ with $1\leqslant i\leqslant l+q+1$ is $\mathbb{Z}_2$-homogeneous, it follows that $2(l+q+1)$ variables in \eqref{kKdef} are zero. More precisely, $k_i^1=K_i^0=0$ for $1\leqslant i\leqslant l$  (in this case all $\Theta_i's$ are even) and $k_i^0=K_i^1=0$ for $l+1\leqslant i\leqslant l+q+1$ (in this case all $\Theta_i's$ are odd). By definition it is obvious that all $F_{ij}'s$ are $\mathbb{Z}_2$-graded and have the same parity with $[\Theta_i,\Theta_j]$. Therefore, $(l+q)^2+l+3q+2$ variables are zero in \eqref{F_ijdefn}. More precisely, $(F_{ij})^0_{\bar{1}}=(F_{ij})^1_{\bar{0}}=0$ if $F_{ij}$ is even, and $(F_{ij})^0_{\bar{0}}=(F_{ij})^1_{\bar{1}}=0$ if $F_{ij}$ is odd.

Based on \eqref{kKdef}, simple calculation shows that
\begin{equation}\label{Thetaijcom}
\begin{split}
\Theta_i.\Theta_j.v&=(k_i^0 k_j^0+K_i^0k_j^1)v+(k_i^1k_j^0+K_i^1k_j^1)\Theta_{l+q+1}.v,\\
\Theta_i.\Theta_j.\Theta_{l+q+1}.v&=(k_i^0K_j^0+K_i^0K_j^1)v+(k_i^1K_j^0+K_i^1K_j^1)\Theta_{l+q+1}.v.
\end{split}
\end{equation}

Changing the position of the indices $i$ and $j$ in \eqref{Thetaijcom}, we have
\begin{equation}\label{Thetaijcom2}
\begin{split}
\Theta_j.\Theta_i.v&=(k_i^0k_j^0+k_i^1K_j^0)v+(k_i^0k_j^1+k_i^1K_j^1)\Theta_{l+q+1}.v,\\
\Theta_j.\Theta_i.\Theta_{l+q+1}.v&=(K_i^0k_j^0+K_i^1K_j^0)v+(K_i^0k_j^1+K_i^1K_j^1)\Theta_{l+q+1}.v.
\end{split}
\end{equation}

Recall that the structure of $\mathbb{C}$-algebra $U(\mathfrak{g},e)$ is completely determined by commutating relations of the generators given in Theorem~\ref{relationc}(2). Therefore, $V$ is a $2$-dimensional representation of the $\mathbb{C}$-algebra $U(\mathfrak{g},e)$ if and only if
\begin{equation}\label{Thetaijcomre}
\begin{split}
(\Theta_i.\Theta_j-(-1)^{|\Theta_i||\Theta_j|}\Theta_j\Theta_i-F_{ij}).v&=0, \\ (\Theta_i.\Theta_j-(-1)^{|\Theta_i||\Theta_j|}\Theta_j\Theta_i-F_{ij}).\Theta_{l+q+1}.v&=0.
\end{split}
\end{equation}

Since all the vectors considered are $\mathbb{Z}_2$-graded, simple calculation based on \eqref{Thetaijcom}, \eqref{Thetaijcom2} and \eqref{Thetaijcomre} shows that the variables $k_i^0,k_i^1,K_i^0,K_i^1$ and $(F_{ij})^0_{\bar{0}},(F_{ij})^0_{\bar{1}},(F_{ij})^1_{\bar{0}},(F_{ij})^1_{\bar{1}}$ should satisfy the following system of linear equations:
\[\begin{array}{rcl}
(k_i^0k_j^0+K_i^0k_j^1)-(-1)^{|\Theta_i||\Theta_j|}(k_i^0k_j^0+k_i^1K_j^0)-(F_{ij})^0_{\bar{0}}&=&0;\\
(k_i^1k_j^0+K_i^1k_j^1)-(-1)^{|\Theta_i||\Theta_j|}(k_i^0k_j^1+k_i^1K_j^1)-(F_{ij})^0_{\bar{1}}&=&0;\\
(k_i^0K_j^0+K_i^0K_j^1)-(-1)^{|\Theta_i||\Theta_j|}(K_i^0k_j^0+K_i^1K_j^0)-(F_{ij})^1_{\bar{0}}&=&0;\\
(k_i^1K_j^0+K_i^1K_j^1)-(-1)^{|\Theta_i||\Theta_j|}(K_i^0k_j^1+K_i^1K_j^1)-(F_{ij})^1_{\bar{1}}&=&0
\end{array}\]for $1\leqslant i,j\leqslant l+q+1$.

It is notable that $4(l+q+1)$ variables are involved in above equations (recall that the $F_{ij}'s$ for $1\leqslant i,j\leqslant l+q+1$ can be written as a $\mathbb{Q}$-linear combination of the products of the elements $k_i^0,k_i^1,K_i^0,K_i^1$ for $1\leqslant i\leqslant l+q+1$). Since all the $\Theta_i's$ and $F_{ij}'s$ are $\mathbb{Z}_2$-graded, by the remark preceding \eqref{Thetaijcom} it is immediate that $2(l+q+1)$ variables are zero. Recall that $\Theta_1,\cdots,\Theta_{l}\in U(\mathfrak{g},e)_{\bar0}$ and $\Theta_{l+1},\cdots,\Theta_{l+q+1}\in U(\mathfrak{g},e)_{\bar1}$. Set
$$\mathbb{C}[X_1^0,X_1^1,\cdots,X_{l+q+1}^0,X_{l+q+1}^1,Y_1^0,Y_1^1,\cdots,Y_{l+q+1}^0,Y_{l+q+1}^1]$$ be a polynomial algebra in $4(l+q+1)$ variables over $\mathbb{C}$ (in the usual sense, not super). Suppose the variables in the polynomial satisfy

(1) $X_i^1=Y_i^0=0$ for $1\leqslant i\leqslant l$ (i.e. $\Theta_i's$ are even);

(2) $X_i^0=Y_i^1=0$ for $l+1\leqslant i\leqslant l+q+1$ (i.e. $\Theta_i's$ are odd).

After deleting all the zero variables, we can get a polynomial algebra in $2(l+q+1)$ indeterminate $$\mathbb{C}[X_1^0,X_2^0\cdots,X_{l}^0,X_{l+1}^1,\cdots,X_{l+q+1}^1,Y_1^1,Y_2^1,\cdots,Y_{l}^1,Y_{l+1}^0\cdots,Y_{l+q+1}^0].$$
Define the polynomials (in the usual sense)
\[\begin{array}{rcl}
A_{ij}&:=&(X_i^0X_j^0+X_j^1Y_i^0)-(-1)^{|\Theta_i||\Theta_j|}(X_i^0X_j^0+X_i^1Y_j^0)-S_{ij}^0;\\
B_{ij}&:=&(X_i^1X_j^0+X_j^1Y_i^1)-(-1)^{|\Theta_i||\Theta_j|}(X_i^0X_j^1+X_i^1Y_j^1)-S_{ij}^1;\\
C_{ij}&:=&(X_i^0Y_j^0+Y_i^0Y_j^1)-(-1)^{|\Theta_i||\Theta_j|}(X_j^0Y_i^0+Y_i^1Y_j^0)-T_{ij}^0;\\
D_{ij}&:=&(X_i^1Y_j^0+Y_i^1Y_j^1)-(-1)^{|\Theta_i||\Theta_j|}(X_j^1Y_i^0+Y_i^1Y_j^1)-T_{ij}^1
\end{array}\]for $1\leqslant i<j\leqslant l~\&~l+1\leqslant i\leqslant j\leqslant l+q+1~\&~1\leqslant i\leqslant l<j\leqslant l+q+1$,
where the notations $S_{ij}^0,S_{ij}^1,T_{ij}^0,T_{ij}^1$ stand for the polynomials over $A$ obtained by substituting the variables $k_i^0,k_i^1,K_i^0,K_i^1$ in the polynomials $(F_{ij})^0_{\bar{0}},(F_{ij})^0_{\bar{1}},(F_{ij})^1_{\bar{0}},(F_{ij})^1_{\bar{1}}$ for the indeterminate $X_i^0,X_i^1,Y_i^0,Y_i^1$, respectively. By (1) and (2) we have

(3) $S_{ij}^1=T_{ij}^0=0$ when $1\leqslant i<j\leqslant l$, or $l+1\leqslant i\leqslant j\leqslant l+q+1$ (i.e. $F_{ij}'s$ are even);

(4) $S_{ij}^0=T_{ij}^1=0$ when $1\leqslant i\leqslant l<j\leqslant l+q+1$ (i.e. $F_{ij}'s$ are odd).

It follows from \eqref{Thetaijcom}, \eqref{Thetaijcom2} and \eqref{Thetaijcomre} that there is a $1$-$1$ correspondence between the $2$-dimensional representations of $\mathbb{C}$-algebra $U(\mathfrak{g},e)$ and the set of all common zeros of the polynomials $A_{ij}, B_{ij}, C_{ij}, D_{ij}$ (satisfying conditions (1)-(4)) in $2(l+q+1)$ variables.

Given a subfield $K$ of $\mathbb{C}$ containing $A$ we denote by $\mathscr{E}(K)$ the set of all common zeros of the polynomials
$A_{ij},B_{ij},C_{ij},D_{ij}$ (satisfying conditions (1)-(4)) in the affine space $\mathbb{A}^{2(l+q+1)}_K$. Clearly, the $A$-defined Zariski closed set $\mathscr{E}(\mathbb{C})$ parametries the $2$-dimensional representations of $\mathbb{C}$-algebra $U(\mathfrak{g},e)$. More precisely,

\begin{lemma}\label{defnodd}
When $d_1$ is odd, the $2$-dimensional representations of $\mathbb{C}$-algebra $U(\mathfrak{g},e)$ are uniquely determined by all common zeros of the polynomials $A_{ij},B_{ij},C_{ij},D_{ij}$ (satisfying conditions (1)-(4)) in the affine space $\mathbb{A}^{2(l+q+1)}_\mathbb{C}$ for $1\leqslant i<j\leqslant l~\&~ l+1\leqslant i\leqslant j\leqslant l+q+1~\&~1\leqslant i\leqslant l<j\leqslant l+q+1$.
\end{lemma}

Similarly, let $\mathscr{E}(\mathds{k})$ be the set of common zeros of the polynomials $^pA_{ij},^pB_{ij},^pC_{ij},$\\$^pD_{ij}$ (satisfying the ``modular $p$'' version of the conditions (1)-(4)) in the affine space $\mathbb{A}^{2(l+q+1)}_\mathds{k}$ with $1\leqslant i<j\leqslant l~\&~l+1\leqslant i\leqslant j\leqslant l+q+1~\&~1\leqslant i\leqslant l<j\leqslant l+q+1$, where $^pA_{ij},^pB_{ij},^pC_{ij},^pD_{ij}$ stand for the polynomials over $\mathds{k}$ obtained from $A_{ij},B_{ij},C_{ij},D_{ij}$ by ``modular $p$ reduction'', i.e.
\[\begin{array}{rl}
&\mathds{k}[X_1^0,X_2^0\cdots,X_{l}^0,X_{l+1}^1,\cdots,X_{l+q+1}^1,Y_1^1,Y_2^1,\cdots,Y_{l}^1,Y_{l+1}^0\cdots,Y_{l+q+1}^0]\\
=&A[X_1^0,X_2^0\cdots,X_{l}^0,X_{l+1}^1,\cdots,X_{l+q+1}^1,Y_1^1,Y_2^1,\cdots,Y_{l}^1,Y_{l+1}^0\cdots,Y_{l+q+1}^0]\otimes_A\mathds{k}. \end{array}\]
It follows from Theorem~\ref{translation}(2) that the Zariski closed set $\mathscr{E}(\mathds{k})$ parametrises the $2$-dimensional representations of the $\mathds{k}$-algebra $U(\mathfrak{g}_\mathds{k},e)$.

Following Premet's treatment to finite $W$-algebras in ([\cite{P7}], Theorem 2.2(a)), we have that

\begin{lemma}\label{ck2}
When $d_1$ is odd, if the $\mathbb{C}$-algebra $U(\mathfrak{g},e)$ affords $2$-dimensional representations, then the transition subalgebra $U(\mathfrak{g}_\mathds{k},e)$ also admits $2$-dimensional representations.
\end{lemma}

\begin{proof}
Firstly,  Lemma~\ref{defnodd} and the discussion thereafter translate the existence of $2$-dimensional representations for both algebras into the solution to the same system of linear equations over the corresponding field, respectively. Then the lemma can be dealt with in the same way as the finite $W$-algebra case, thus will be omitted here (see [\cite{P7}], Theorem 2.2(a)).
\end{proof}

The following result is similar to the consequence introduced in Lemma~\ref{mindim1}. As $V$ is $2$-dimensional, some modifications are needed in the proof. Hence we will prove it in detail.

\begin{lemma}\label{red2}
When $d_1$ is odd, if the finite $W$-superalgebra $U(\mathfrak{g}_\mathds{k},e)$ affords $2$-dimensional representations, then for $p\gg0$ there exists $\eta\in\chi+(\mathfrak{m}_\mathds{k}^\bot)_{\bar{0}}$ associated to which the reduced enveloping algebra $U_\eta(\mathfrak{g}_\mathds{k})$ admits irreducible representations of dimension $p^{\frac{d_0}{2}}2^{\frac{d_1+1}{2}}$.
\end{lemma}

\begin{proof}
Recall that there is a $\mathds{k}$-algebras isomorphism $\widehat{U}(\mathfrak{g}_\mathds{k},e)\cong U(\mathfrak{g}_\mathds{k},e)\otimes_\mathds{k}Z_p(\mathfrak{a}_\mathds{k})$ by Theorem~\ref{keyisotheorem}(3). This yields that the $\mathds{k}$-algebra $\widehat{U}(\mathfrak{g}_\mathds{k},e)$ affords a $2$-dimensional representation too; we call it $\nu$, and denote the representation vector by $V$.

Let $0\neq v_{\bar{0}}\in V_{\bar{0}}$ be an even element in $V$. The proof of Proposition~\ref{no1} shows that $0\neq\Theta_{l+q+1}.v_{\bar{0}}\in V_{\bar{1}}$ and denote it by $v_{\bar{1}}$. Then the vector space $V$ is $\mathds{k}$-spanned by $v_{\bar{0}}$ and $v_{\bar{1}}$. For any $\bar x\in(\mathfrak{g}_\mathds{k})_{\bar{0}}$, since $\bar x^p-\bar x^{[p]}\in Z_p(\mathfrak{g}_\mathds{k})$ is central in the algebra $U(\mathfrak{g}_\mathds{k})$, we have $[\bar x^p-\bar x^{[p]},\Theta_{l+q+1}]=0$. Let $\lambda_{\bar x}\in\mathds{k}$ such that $(\bar x^p-\bar x^{[p]}).v_{\bar{0}}=\lambda_{\bar x}v_{\bar{0}}$, then $$(\bar x^p-\bar x^{[p]}).\Theta_{l+q+1}.v_{\bar{0}}=\Theta_{l+q+1}.(\bar x^p-\bar x^{[p]}).v_{\bar{0}}=\lambda_{\bar x}\Theta_{l+q+1}.v_{\bar{0}}=\lambda_{\bar x}v_{\bar{1}}.$$

For the  $\widehat{U}(\mathfrak{g}_\mathds{k},e)$-module $V$,  it follows from $\rho_\mathds{k}(Z_p)\subseteq\widehat{U}(\mathfrak{g}_\mathds{k},e)_{\bar0}$ that $\mathds{k}v_{\bar{0}}$ is a $1$-dimensional representation of the algebra $\rho_\mathds{k}(Z_p)$. From preceding remark it is obvious that $\mathds{k}v_{\bar{1}}$ is also a $1$-dimensional representation of the algebra $\rho_\mathds{k}(Z_p)$ with the same action on $\mathds{k}v_{\bar{0}}$. By Theorem~\ref{keyisotheorem}(2), $\rho_\mathds{k}(Z_p)\cap\text{ker}\nu=\rho_\mathds{k}(Z_p)\cap\text{ker}\nu_{\bar0}$ is a maximal ideal of the algebra $\rho_\mathds{k}(Z_p)\cong Z_p(\widetilde{\mathfrak{p}}_\mathds{k})\cong \mathds{k}[(\chi+(\mathfrak{m}_\mathds{k}^\bot)_{\bar{0}})^{(1)}]$. So there exists $\eta\in\chi+(\mathfrak{m}_\mathds{k}^\bot)_{\bar{0}}$  such that $\rho_\mathds{k}(\bar x^p-\bar x^{[p]}-\eta(\bar x)^p)\in\text{Ker}\nu$ for all
$\bar x\in(\mathfrak{g}_\mathds{k})_{\bar{0}}$.  Our choice of $\eta$ ensures that the $\mathds{k}$-algebra $\widehat{U}_\eta(\mathfrak{g}_\mathds{k},e):=\widehat{U}(\mathfrak{g}_\mathds{k},e)\otimes_{Z_p(\widetilde{\mathfrak{p}}_\mathds{k})}\mathds{k}_\eta$ affords a $2$-dimensional representation. On the other hand, the canonical projection $Q_{\chi,\mathds{k}}\twoheadrightarrow Q_{\chi,\mathds{k}}/I_\eta Q_{\chi,\mathds{k}}$ gives rise to an algebra homomorphism $\rho_\eta:\widehat{U}_\eta(\mathfrak{g}_\mathds{k},e)\rightarrow
(\text{End}_{\mathfrak{g}_\mathds{k}}Q_\chi^\eta)^{\text{op}}=U_\eta(\mathfrak{g}_\mathds{k},e)$. As $\text{dim}~\widehat{U}_\eta(\mathfrak{g}_\mathds{k},e)\leqslant  p^l2^{q+1}$ by Theorem~\ref{keyisotheorem}(2), applying Theorem~\ref{sumresult}(4)(ii) yields that $\rho_\eta$ is an algebra isomorphism. Since $$U_\eta(\mathfrak{g}_\mathds{k})\cong \text{Mat}_{p^{\frac{d_0}{2}}2^{\frac{d_1-1}{2}}}(U_\eta(\mathfrak{g}_\mathds{k},e))$$ by Theorem~\ref{sumresult}(3), and the algebra $U_\eta(\mathfrak{g}_\mathds{k},e)$ admits a $2$-dimensional representation, it follows that the algebra $U_\eta(\mathfrak{g}_\mathds{k})$ has an irreducible representation of dimension $p^\frac{d_0}{2}2^{\frac{d_1+1}{2}}$.
\end{proof}

\begin{lemma}\label{cdim2}
When $d_1$ is odd, if the finite $W$-superalgebra  $U(\mathfrak{g},e)$ over $\mathbb{C}$ affords a $2$-dimensional representation, then for $p\gg0$ there exists $\eta\in\chi+(\mathfrak{m}_\mathds{k}^\bot)_{\bar{0}}$  associated to which the reduced enveloping algebra $U_\eta(\mathfrak{g}_\mathds{k})$ admits irreducible representations of dimension $p^{\frac{d_0}{2}}2^{\frac{d_1+1}{2}}$.
\end{lemma}

\begin{proof}
The Lemma follows from Lemma~\ref{ck2} and Lemma~\ref{red2}.
\end{proof}

\section{The realization of minimal dimensional representations for reduced enveloping algebra $U_\chi(\mathfrak{g}_\mathds{k})$}

In ([\cite{WZ}], Theorem 4.3), Wang and Zhao introduced the Super Kac-Weisfeiler Property with nilpotent $p$-characters for the basic classical Lie superalgebra $\mathfrak{g}_\mathds{k}$ over positive characteristic field $\mathds{k}=\overline{\mathbb{F}}_p$ (with some restrictions on $p$), i.e.

\begin{prop}$^{[\cite{WZ}]}$\label{wzd}
Let $\mathfrak{g}_\mathds{k}$ be one of the basic classical Lie superalgebras, and let $\chi\in(\mathfrak{g}^*_\mathds{k})_{\bar0}$ be nilpotent. Then the dimension of every $U_\chi(\mathfrak{g}_\mathds{k})$-module $M$ is divisible by $p^{\frac{d_0}{2}}2^{\lfloor\frac{d_1}{2}\rfloor}$.
\end{prop}

In virtue of this result, the Super Kac-Weisfeiler Property with any $p$-characters for the basic classical Lie superalgebra $\mathfrak{g}_\mathds{k}$ over positive characteristic field $\mathds{k}=\overline{\mathbb{F}}_p$ was also formulated by them, i.e.

\begin{prop}$^{[\cite{WZ}]}$
Let $\mathfrak{g}_\mathds{k}$ be a  basic classical Lie superalgebra, and let $\xi$ be arbitrary $p$-character in $(\mathfrak{g}^*_\mathds{k})_{\bar0}$. Then the dimension of every $U_\xi(\mathfrak{g}_\mathds{k})$-module $M$ is divisible by $p^{\frac{d_0}{2}}2^{\lfloor\frac{d_1}{2}\rfloor}$.
\end{prop}

In this part we will try to formulate a realization of these modules. Based on the discussion in section 8, a conjecture about the minimal dimensional representations of finite $W$-superalgebra $U(\mathfrak{g},e)$ over $\mathbb{C}$ should be first formulated.

\subsection{A conjecture on the minimal dimensional representations of finite $W$-superalgebras}

\begin{conj}\label{con}
Let $\mathfrak{g}$ be a basic classical Lie superalgebra over $\mathbb{C}$, then the following are true:

(i) when $d_1$ is even, the finite $W$-superalgebra $U(\mathfrak{g},e)$ affords a $1$-dimensional representation;

(ii) when $d_1$ is odd, the finite $W$-superalgebra $U(\mathfrak{g},e)$ affords a $2$-dimensional representation.
\end{conj}

As mentioned at the beginning of Section 8, Premet formulated a conjecture that every finite $W$-algebra over $\mathbb{C}$ admits a $1$-dimensional representation$^{[\cite{P3}]}$. His conjecture was proved by Losev for the classical Lie algebra case using the method of symplectic geometry in [\cite{L3}]. Goodwin-R\"{o}hrle-Ubly$^{[\cite{GRU}]}$ obtained that all finite $W$-algebras associated to exceptional Lie algebras $E_6,E_7,F_4,G_2$, or $E_8$ with $e$ not rigid admit $1$-dimensional representations by computing methods. The formulation of  Conjecture~\ref{con} is based on the related results of finite $W$-algebras over $\mathbb{C}$, but there is no effective way to deal with it so far. However, it can be verified that Conjecture~\ref{con} establishes for some special cases. First note that

\begin{lemma}\label{transtoc}
Let $\mathfrak{g}$ be a basic classical Lie superalgebra over $\mathbb{C}$. The following are true:

(1) when $d_1$ is even, if the translation algebra $U(\mathfrak{g}_\mathds{k},e)$  (where $\mathds{k}=\overline{\mathbb{F}}_p$) affords $1$-dimensional representations for infinitely many $p\in\Pi(A)$, then the finite $W$-superalgebra $U(\mathfrak{g},e)$ over $\mathbb{C}$ has a $1$-dimensional representation;

(2) when $d_1$ is odd,  if the translation algebra $U(\mathfrak{g}_\mathds{k},e)$ (where $\mathds{k}=\overline{\mathbb{F}}_p$) affords $2$-dimensional representations for infinitely many $p\in\Pi(A)$, then the finite $W$-superalgebra $U(\mathfrak{g},e)$ over $\mathbb{C}$ has a $2$-dimensional representation.
\end{lemma}

\begin{proof}
Since the proof is similar for both cases, we will just consider the situation when $d_1$ is odd.

When $d_1$ is odd, Lemma~\ref{defnodd} shows that the $2$-dimensional representations of finite $W$-superalgebras over $\mathbb{C}$ can be parametrised by the Zariski closed set $\mathscr{E}(\mathbb{C})$. Here we will follow Premet's treatment to finite $W$-algebras in ([\cite{P7}], Corollary 2.1). The proof is sketched as follows.

Suppose for a contradiction that $U(\mathfrak{g}_\mathds{k},e)$ has no $2$-dimensional representations. Then $\mathscr{E}(\overline{\mathbb{Q}})=\emptyset$, where $\overline{\mathbb{Q}}$ denotes the algebraic closure of $\mathbb{Q}$ in $\mathbb{C}$. By the method of ``modular $p$ induction'' and Galois theory we can get that $\mathscr{E}(\mathds{k})=\emptyset$ for almost all $p\in \Pi(A)$, where $\mathds{k}=\overline{\mathbb{F}}_p$. This implies that the algebra $U(\mathfrak{g}_\mathds{k},e)$ has no $2$-dimensional representations for almost all $p\in \Pi(A)$. Since this contradicts our assumption, the lemma follows.
\end{proof}

Based on the discussion in Section 8.1, now we will establish a proof of Conjecture~\ref{con}(i) for some special cases.

\begin{prop}
Let $e=e_\lambda\in\mathfrak{sl}(M|N)_{\bar0}$ be an even nilpotent element in the Lie superalgebra $\mathfrak{g}=\mathfrak{sl}(M|N)$ over $\mathbb{C}$ with Jordan type $(p_1,p_2,\cdots,p_{n+1})$ satisfying that $e_\lambda=e_M\oplus e_N$, where $e_M$ is principal nilpotent in $\mathfrak{gl}(M|0)$ and the sizes of the Jordan blocks of $e_N$ are all greater or equal to $M$. Then the finite $W$-superalgebra $U(\mathfrak{g},e)$ affords a $1$-dimensional representation.
\end{prop}

\begin{proof}
Recall that in Corollary~\ref{slmn} we have proved that the algebra $U(\mathfrak{g},e)^{\text{ab}}$ associated to the finite $W$-superalgebra $U(\mathfrak{g},e)$ in the proposition is isomorphic to a polynomial algebra in $p_{n+1}-1$ variables. So this proposition is an immediate corollary of Hilbert's Nullstellensatz.
\end{proof}

Recall that for any  left $U_{\chi}(\mathfrak{g}_\mathds{k})$-module $M$, we have defined the $U_{\chi}(\mathfrak{g}_\mathds{k},e)$-module $M^{\mathfrak{m}_\mathds{k}}$ by
$$M^{\mathfrak{m}_\mathds{k}}:=\{v\in M|I_{\mathfrak{m}_\mathds{k}}.v=0\}.$$
Theorem~\ref{reducedfunctors} shows that there exists a category equivalence between the $U_\chi(\mathfrak{g}_\mathds{k})$-module $M$ and the $U_\chi(\mathfrak{g}_\mathds{k},e)$-module $M^{\mathfrak{m}_\mathds{k}}$.

As another example, Proposition~\ref{ii} shows that Conjecture~\ref{con}(ii) establishes for the finite $W$-superalgebra $U(\mathfrak{osp}(1|2n),e)$ associated to the basic classical Lie superalgebra of type $B(0,n)$ with the regular nilpotent element $e$. First note that

\begin{lemma}\label{bon}
Let $\mathfrak{g}_\mathds{k}=\mathfrak{osp}(1|2n)_\mathds{k}$ be the basic classical Lie superalgebra of type $B(0,n)$ over $\mathds{k}=\overline{\mathbb{F}}_p$. For any regular nilpotent element $e\in(\mathfrak{g}_{\mathds{k}})_{\bar0}$, let $\chi\in(\mathfrak{g}_{\mathds{k}})_{\bar0}^*$ be such that $\chi(y)=(e,y)$ for any $y\in\mathfrak{g}_{\mathds{k}}$ with respect to the bilinear form $(\cdot,\cdot)$. Then the translation subalgebra $U(\mathfrak{g}_\mathds{k},e)$ affords a $2$-dimensional representation.
\end{lemma}

\begin{proof}
First note that $d_1$ is odd in this case by ([\cite{PS2}], Corollary 2.10). Let $\mathfrak{g}_\mathds{k}=\mathfrak{n}_\mathds{k}^+\oplus \mathfrak{h}_\mathds{k}\oplus \mathfrak{n}_\mathds{k}^-$ denote the triangular decomposition of Lie superalgebra $\mathfrak{osp}(1|2n)_\mathds{k}$. It follows from ([\cite{WZ}], Corollary 5.8) that $$\text{\underline{dim}}\,\mathfrak{n}_\mathds{k}^-=\text{\underline{dim}}\,\mathfrak{m}'_\mathds{k}=
(\text{dim}\,(\mathfrak{m}_\mathds{k})_{\bar0},\text{dim}\,(\mathfrak{m}_\mathds{k})_{\bar1}+1).$$ Moreover, the baby Verma module $Z_\chi(\lambda)$ of reduced enveloping algebra $U_\chi(\mathfrak{osp}(1|2n)_\mathds{k})$ associated to the regular $p$-character $\chi$ is irreducible, which has the same dimension as dim\,$U_\chi(\mathfrak{m}_\mathds{k}')$. Recall that there is a category equivalence between the $U_\chi(\mathfrak{osp}(1|2n)_\mathds{k})$-modules and the $U_\chi(\mathfrak{osp}(1|2n)_\mathds{k},e)$-modules by Theorem~\ref{reducedfunctors}, then it follows that $Z_\chi(\lambda)^{\mathfrak{m}_\mathds{k}}$ is a $U_\chi(\mathfrak{osp}(1|2n)_\mathds{k},e)$-module. For every $U_\chi(\mathfrak{g}_\mathds{k})$-module $M$, since $M\cong U_\chi(\mathfrak{m}_\mathds{k})^*\otimes_\mathds{k}M^{\mathfrak{m}_\mathds{k}}$ is an isomorphism of $U_\chi(\mathfrak{m}_\mathds{k})$-modules by the proof of ([\cite{WZ}], Proposition 4.2), it is immediate that
$$\text{dim}\,Z_\chi(\lambda)^{\mathfrak{m}_\mathds{k}}=\frac{\text{dim}\,Z_\chi(\lambda)}
{\text{dim}\,U_\chi(\mathfrak{m}_\mathds{k})}=\frac{\text{dim}\,U_\chi(\mathfrak{m}_\mathds{k}')}
{\text{dim}\,U_\chi(\mathfrak{m}_\mathds{k})}=2.$$Therefore, the algebra $U_\chi(\mathfrak{g}_\mathds{k},e)$ admits a $2$-dimensional representation.

Let $\mathcal{M}_\mathds{k}$ be the connected unipotent subgroup of $G_\mathds{k}$ such that Ad~$\mathcal{M}_\mathds{k}$ is generated by all linear operators exp\,ad\,$\bar x$ with $\bar x\in\mathfrak{m}_\mathds{k}$. By the same discussion as ([\cite{P8}], Remark 2.1) and the remark following which, we can conclude that one embeds $Z_p(\widetilde{\mathfrak{p}}_\mathds{k})^{\mathcal{M}_\mathds{k}}$ into $U(\mathfrak{g}_\mathds{k},e)$ as an analogue of the $p$-center $Z_p(\mathfrak{g}_\mathds{k})$ (so that $U(\mathfrak{g}_\mathds{k},e)$ is a free $Z_p(\widetilde{\mathfrak{p}}_\mathds{k})^{\mathcal{M}_\mathds{k}}$-module of rank $p^l2^{q+1}$) and then obtains $U_\chi(\mathfrak{g}_\mathds{k},e)$ from $U(\mathfrak{g}_\mathds{k},e)$ by tensoring the latter over $Z_p(\widetilde{\mathfrak{p}}_\mathds{k})^{\mathcal{M}_\mathds{k}}$ by a suitable one-dimensional representation of $Z_p(\widetilde{\mathfrak{p}}_\mathds{k})^{\mathcal{M}_\mathds{k}}$, i.e. $U_\chi(\mathfrak{g}_\mathds{k},e)$ can be considered as the factor-algebra of $U(\mathfrak{g}_\mathds{k},e)$. Combining this with the discussion in the preceding paragraph we can conclude that the algebra $U(\mathfrak{g}_\mathds{k},e)$ admits a $2$-dimensional representation, either.

\end{proof}

\begin{prop}\label{ii}
Let $e\in\mathfrak{g}_{\bar0}$ be a regular nilpotent element in the Lie superalgebra $\mathfrak{g}=\mathfrak{osp}(1|2n)$ over $\mathbb{C}$, then the finite $W$-superalgebra $U(\mathfrak{g},e)$ affords a $2$-dimensional representation.
\end{prop}

\begin{proof}
The proposition follows from Lemma~\ref{transtoc}(2) and Lemma~\ref{bon}.
\end{proof}

\subsection{On the lower bound of the Super KW Property with nilpotent $p$-characters}

Recall that in Section 8.1 and Section 8.2 we have discussed the probable dimension for the ``small representations'' of finite $W$-superalgebras. Based on these results, the representations of minimal dimension for the reduced enveloping algebra $U_\eta(\mathfrak{g}_\mathds{k})$ associated to $p$-character $\eta\in\chi+(\mathfrak{m}_\mathds{k}^\bot)_{\bar0}$ were considered in Lemma~\ref{mindim1} and Lemma~\ref{cdim2} based on the parity of $d_1$, respectively. It is notable that $\eta$ can only be guaranteed in $\chi+(\mathfrak{m}_\mathds{k}^\bot)_{\bar0}$, but there is no further information. Following Premet's treatment to finite $W$-algebras in ([\cite{P7}], Theorem 2.2), Lemma~\ref{etachi} translates the conclusion about the $\mathds{k}$-algebra $U_\eta(\mathfrak{g}_\mathds{k})$ associated to $p$-character $\eta$ to the $\mathds{k}$-algebra $U_\chi(\mathfrak{g}_\mathds{k})$ with $p$-character $\chi$. In virtue of this result, the main result will be formulated in Theorem~\ref{main2}.

\begin{lemma}\label{etachi}
Let $\mathfrak{g}_\mathds{k}$ be a basic classical Lie superalgebra over positive characteristic field $\mathds{k}$. The following are true:

(1) when $d_1$ is even, if the algebra $U_\eta(\mathfrak{g}_\mathds{k})$ affords a representation of dimensional $p^{\frac{d_0}{2}}2^{\frac{d_1}{2}}$ for some $\eta\in\chi+(\mathfrak{m}_\mathds{k}^\bot)_{\bar{0}}$, then the algebra $U_\chi(\mathfrak{g}_\mathds{k})$ also admits a representation of dimension $p^{\frac{d_0}{2}}2^{\frac{d_1}{2}}$;

(2) when $d_1$ is odd, if the algebra $U_\eta(\mathfrak{g}_\mathds{k})$ affords a representation of dimensional $p^{\frac{d_0}{2}}2^{\frac{d_1+1}{2}}$ for some $\eta\in\chi+(\mathfrak{m}_\mathds{k}^\bot)_{\bar{0}}$,  then the algebra $U_\chi(\mathfrak{g}_\mathds{k})$ also admits a representation of dimension $p^{\frac{d_0}{2}}2^{\frac{d_1+1}{2}}$.
\end{lemma}

\begin{proof}
Since the proof is similar for both cases, we will just consider the situation when $d_1$ is odd.

Let $(G_\mathds{k})_{\text{ev}}$ be the reductive algebraic group associated to even part $(\mathfrak{g}_\mathds{k})_{\bar{0}}$ of Lie superalgebra $\mathfrak{g}_\mathds{k}$.
For any $\xi\in(\mathfrak{g}_\mathds{k}^*)_{\bar0}$, it is well known that the construction of the algebra $U_\xi(\mathfrak{g}_\mathds{k})$ only depends on the orbit of $\xi$ under the coadjoint action of $(G_\mathds{k})_{\text{ev}}$ up to isomorphism. Therefore, if $\xi':=(\text{Ad}^*g)(\xi)$ for some $g\in(G_\mathds{k})_{\text{ev}}$, then $U_\xi(\mathfrak{g}_\mathds{k})\cong U_{\xi'}(\mathfrak{g}_\mathds{k})$ as $\mathds{k}$-algebras.

Let $\Xi$ denote the set of all $\xi\in(\mathfrak{g}_\mathds{k}^*)_{\bar{0}}$ for which the algebra $U_\xi(\mathfrak{g}_\mathds{k})$ contains a two-sided ideal of codimension $p^{d_0}2^{d_1+1}$. It is immediate from ([\cite{Z1}], Lemma 2.2) that the set $\Xi$ is Zariski closed in $(\mathfrak{g}_\mathds{k}^*)_{\bar{0}}$. From preceding remark it is easy to verify that the set $\Xi$ is stable under the coadjoint action of $(G_\mathds{k})_{\text{ev}}$.

(i) We claim that $\bar t\cdot\xi\in\Xi$ for all $\bar t\in\mathds{k}^\times(=\mathds{k}\backslash\{0\})$ and $\xi\in\Xi$.

For any $\xi\in(\mathfrak{g}_\mathds{k}^*)_{\bar{0}}$,  we can regard $\xi\in\mathfrak{g}_\mathds{k}^*$ by letting $\xi((\mathfrak{g}_\mathds{k})_{\bar{1}})=0$. Recall that $\xi=(\bar x,\cdot)$ for some $\bar x\in(\mathfrak{g}_\mathds{k})_{\bar{0}}$. Let
$\bar x=\bar s+\bar n$ be the Jordan-Chevalley decomposition of $\bar x$ in the restricted Lie algebra $(\mathfrak{g}_\mathds{k})_{\bar{0}}$ and put
$\xi_{\bar s}:=(\bar s,\cdot)$ and $\xi_{\bar n}:=(\bar n,\cdot)$. Take a Cartan subalgebra $\mathfrak{h}_\mathds{k}$ of $\mathfrak{g}_\mathds{k}$ which contains $\bar s$, and let $\mathfrak{g}_\mathds{k}^{\bar s}$ denote the centralizer of $\bar s$ in $\mathfrak{g}_\mathds{k}$. Then it follows that $\mathfrak{g}_\mathds{k}^{\bar s}:=\mathfrak{l}_\mathds{k}=(\mathfrak{l}_\mathds{k})_{\bar{0}}+(\mathfrak{l}_\mathds{k})_{\bar{1}}$ also has a root space decomposition $\mathfrak{l}_\mathds{k}=\mathfrak{h}_\mathds{k}\oplus\bigoplus\limits_{\alpha\in\Phi(\mathfrak{l}_\mathds{k})}(\mathfrak{g}_\mathds{k})_\alpha$ where $\Phi(\mathfrak{l}_\mathds{k}):=\{\alpha\in\Phi|\alpha(\bar s)=0\}$. From ([\cite{WZ}], Proposition 5.1) we know that there exists a system $\Pi$ of simple roots of $\mathfrak{g}_\mathds{k}$ such that $\Pi\cap\Phi(\mathfrak{l}_\mathds{k})$ is a system of simple roots for $\Phi(\mathfrak{l}_\mathds{k})$. In particular $\mathfrak{l}_\mathds{k}$ is always a direct sum of basic classical Lie superalgebras (note that a toral subalgebra of $\mathfrak{g}_\mathds{k}$ may also appear in the summand). Let $\mathfrak{b}_\mathds{k}=\mathfrak{h}_\mathds{k}\oplus\mathfrak{n}_\mathds{k}$ be the Borel subalgebra associated to $\Pi$. Then we can define a parabolic subalgebra $\mathfrak{p}_\mathds{k}=\mathfrak{l}_\mathds{k}+\mathfrak{b}_\mathds{k}=\mathfrak{l}_\mathds{k}\oplus\mathfrak{u}_\mathds{k}$, where $\mathfrak{u}_\mathds{k}$ denotes the nilradical of $\mathfrak{p}_\mathds{k}$. Note that $\xi(\mathfrak{u}_\mathds{k})=0$ and $\xi|_{\mathfrak{l}_\mathds{k}}=\xi_n|_{\mathfrak{l}_\mathds{k}}$ is nilpotent.

Recall that $\bar x=\bar s+\bar n$ is the Jordan-Chevalley decomposition of $\bar x$. If $\bar t\in\mathds{k}^\times$, then $\bar t\bar x=\bar t\bar s+\bar t\bar n$ is the Jordan-Chevalley decomposition of $\bar t\bar x$. Obviously, $\mathfrak{g}_\mathds{k}^{\bar t\bar s}=\mathfrak{l}_\mathds{k}$.

It follows from ([\cite{WZ}], Theorem 5.3) that every irreducible $U_\xi(\mathfrak{g}_\mathds{k})$-module is $U_\xi(\mathfrak{u}_\mathds{k})$-projective. Since $\mathfrak{u}_\mathds{k}$ is nilpotent in $\mathfrak{g}_\mathds{k}$ and $\xi|_{\mathfrak{u}_\mathds{k}}=0$, it follows from ([\cite{WZ}], Proposition 2.6) that the $\mathds{k}$-algebra $U_\xi(\mathfrak{u}_\mathds{k})$ is local with trivial module as the unique simple module. Then every irreducible $U_\xi(\mathfrak{g}_\mathds{k})$-module is $U_\xi(\mathfrak{u}_\mathds{k})$-free, and the unique maximal ideal $N_{\mathfrak{u}_\mathds{k}}$ of $U_\xi(\mathfrak{u}_\mathds{k})$ is generated by the image of $\mathfrak{u}_\mathds{k}$ in $U_\xi(\mathfrak{u}_\mathds{k})$. This yields $[\mathfrak{p}_\mathds{k},N_{\mathfrak{u}_\mathds{k}}]\subseteq N_{\mathfrak{u}_\mathds{k}}$. Let $U_\xi(\mathfrak{p}_\mathds{k})$ denote the unital subalgebra of $U_\xi(\mathfrak{g}_\mathds{k})$ generated by $\mathfrak{p}_\mathds{k}$ (it is canonically isomorphic to the reduced enveloping algebra of $\mathfrak{p}_\mathds{k}$ associated with $\xi|_{\mathfrak{p}_\mathds{k}}$). Since $U_\xi(\mathfrak{p}_\mathds{k})N_{\mathfrak{u}_\mathds{k}}$ is a two-sided ideal of $U_\xi(\mathfrak{p}_\mathds{k})$, it follows from the PBW Theorem that $$U_\xi(\mathfrak{p}_\mathds{k})/U_\xi(\mathfrak{p}_\mathds{k})N_{\mathfrak{u}_\mathds{k}}\cong U_\xi(\mathfrak{p}_\mathds{k}/\mathfrak{u}_\mathds{k})$$ as $\mathds{k}$-algebras. Let $\bar1_\xi=1+N_{\mathfrak{u}_\mathds{k}}$, the image of $1$ in $\mathds{k}_\xi$, and $Q^0_{\mathfrak{u}_\mathds{k}}:=U_\xi(\mathfrak{p}_\mathds{k})\cdot\bar1_\xi$. It follows from the definition of $Q_\xi^\xi$ preceding Definition~\ref{reduced W} that $Q_\xi^\xi\cong U_\xi(\mathfrak{g}_\mathds{k})/U_\xi(\mathfrak{g}_\mathds{k})N_{\mathfrak{u}_\mathds{k}}$ as left $U_\xi(\mathfrak{g}_\mathds{k})$-modules. The PBW theorem and the discussion above imply that $\text{dim}~Q_{\mathfrak{u}_\mathds{k}}^0=\text{dim}~U_\xi
(\mathfrak{p}_\mathds{k}/\mathfrak{u}_\mathds{k})$. Given $\bar q\in Q_{\mathfrak{u}_\mathds{k}}^0$ there is $\bar u\in U_\xi(\mathfrak{p}_\mathds{k})$ such that $\bar q=\bar u\cdot \bar1_\xi$. Since $[\mathfrak{p}_{\mathds{k}},N_{\mathfrak{u}_\mathds{k}}]\subseteq N_{\mathfrak{u}_\mathds{k}}$, it follows from Jacobi identity that $[\bar n,\bar u]\in U_\xi(\mathfrak{p}_\mathds{k})N_{\mathfrak{u}_\mathds{k}}$ for any $\bar n\in N_{\mathfrak{u}_\mathds{k}}$. Then
$$\bar n\cdot \bar q=([\bar n,\bar q]+(-1)^{|\bar n||\bar q|}\bar q\cdot \bar n)\cdot \bar1_\xi\subseteq U_\xi(\mathfrak{p}_\mathds{k})N_{\mathfrak{u}_\mathds{k}}$$ for any $\mathbb{Z}_2$-homogeneous elements $\bar n\in N_{\mathfrak{u}_\mathds{k}}$ and $\bar q\in Q_{\mathfrak{u}_\mathds{k}}^0$. The universal property of induced modules implies that for any $\bar q\in Q_{\mathfrak{u}_\mathds{k}}^0$ there is a unique $h_{\bar q}\in\text{End}_{\mathfrak{g}_\mathds{k}}Q_\xi^\xi$ such that $h_{\bar q}(\bar1_\xi)=\bar q$.
Put $$d'_0:=2\text{dim}~(\mathfrak{u}_\mathds{k})_{\bar{0}}=\text{dim}~(\mathfrak{g}_\mathds{k})_{\bar{0}}-\text{dim}~(\mathfrak{l}_\mathds{k})_{\bar{0}},~
d'_1:=2\text{dim}~(\mathfrak{u}_\mathds{k})_{\bar{1}}=\text{dim}~(\mathfrak{g}_\mathds{k})_{\bar{1}}-\text{dim}~(\mathfrak{l}_\mathds{k})_{\bar{1}}.$$ Since every irreducible $U_\xi(\mathfrak{g}_\mathds{k})$-module is $U_\xi(\mathfrak{u}_\mathds{k})$-free, by the same discussion as ([\cite{WZ}], Theorem 4.4) we can obtain a $\mathds{k}$-algebras isomorphism:
\begin{equation}\label{isorr}
U_\xi(\mathfrak{g}_\mathds{k})\cong \text{Mat}_{p^{\frac{d'_0}{2}}2^{\frac{d'_1}{2}}}((\text{End}_{\mathfrak{g}_\mathds{k}}Q_\xi^\xi)^{\text{op}}).
\end{equation}
Therefore,$$\text{dim}~(\text{End}_{\mathfrak{g}_\mathds{k}}Q_\xi^\xi)=p^{\text{dim}~
(\mathfrak{g}_\mathds{k})_{\bar{0}}-d'_0}2^{\text{dim}~
(\mathfrak{g}_\mathds{k})_{\bar{1}}-d'_1}=p^{\text{dim}~
(\mathfrak{l}_\mathds{k})_{\bar{0}}}2^{\text{dim}~
(\mathfrak{l}_\mathds{k})_{\bar{1}}}$$and $$\text{dim}~U_\xi(\mathfrak{p}_\mathds{k}/\mathfrak{u}_\mathds{k})=p^{\text{dim}~
(\mathfrak{p}_\mathds{k})_{\bar{0}}-\text{dim}~(\mathfrak{u}_\mathds{k})_{\bar{0}}}2^{\text{dim}~
(\mathfrak{p}_\mathds{k})_{\bar{1}}-\text{dim}~
(\mathfrak{u}_\mathds{k})_{\bar{1}}}=p^{\text{dim}~
(\mathfrak{l}_\mathds{k})_{\bar{0}}}2^{\text{dim}~
(\mathfrak{l}_\mathds{k})_{\bar{1}}}.$$
Then it follows that $\text{End}_{\mathfrak{g}_\mathds{k}}Q_\xi^\xi=\{h_{\bar q}|\bar q\in Q_{\mathfrak{u}_\mathds{k}}^0\}$. Define the mapping
\[\begin{array}{lcll}
\tau:&\text{End}_{\mathfrak{g}_\mathds{k}}Q_{\xi}^\xi&\rightarrow&U_\xi(\mathfrak{p}_\mathds{k}/\mathfrak{u}_\mathds{k})^{\text{op}}\\ &\theta&\mapsto&\theta(\bar1_\chi).
\end{array}
\]
It is obvious that $\tau$ is a homomorphism of $\mathds{k}$-algebras. As both algebras have the same dimension (as vector spaces), one can deduce that $\tau$ is an isomorphism. Taking the opposite algebras for both sides, we have
\begin{equation}\label{iso2}
(\text{End}_{\mathfrak{g}_\mathds{k}}Q_\xi^\xi)^{\text{op}}\cong U_\xi(\mathfrak{p}_\mathds{k}/\mathfrak{u}_\mathds{k})\cong U_\xi(\mathfrak{l}_\mathds{k})
\end{equation} as $\mathds{k}$-algebras.

Let $N'_{\mathfrak{u}_\mathds{k}}$ be the unique maximal ideal of $U_{\bar t\xi}(\mathfrak{u}_\mathds{k})$. For the left $U_{\bar t\xi}(\mathfrak{g}_\mathds{k})$-module $Q_{\bar t\xi}^{\bar t\xi}\cong U_{\bar t\xi}(\mathfrak{g}_\mathds{k})/U_{\bar t\xi}(\mathfrak{g}_\mathds{k})
N'_{\mathfrak{u}_\mathds{k}}$, same consideration shows that
\begin{equation}\label{iso3}
(\text{End}_{\mathfrak{g}_\mathds{k}}Q_{\bar t\xi}^{\bar t\xi})^{\text{op}}\cong U_{\xi}(\mathfrak{l}_\mathds{k})
\end{equation} as $\mathds{k}$-algebras. Along the same discussion as ([\cite{WZ}], Theorem 4.4), one can get that
\begin{equation}\label{iso4}
U_{t\xi}(\mathfrak{g}_\mathds{k})\cong \text{Mat}_{p^{\frac{d'_0}{2}}2^{\frac{d'_1}{2}}}((\text{End}_{\mathfrak{g}_\mathds{k}}Q_{\bar t\xi}^{\bar t\xi})^{\text{op}})
\end{equation}
as $\mathds{k}$-algebras.

Recall that $\mathfrak{l}_\mathds{k}$ is a direct sum of basic classical Lie superalgebras and a toral subalgebra. Set $\mathfrak{l}_\mathds{k}=(\mathfrak{g}_\mathds{k})_1\oplus\cdots\oplus(\mathfrak{g}_\mathds{k})_r\oplus\mathfrak{t}'_\mathds{k}$, where $(\mathfrak{g}_\mathds{k})_i$ is a basic classical Lie superalgebra for each $1\leqslant  i\leqslant  r$, and $\mathfrak{t}'_\mathds{k}$ is a toral subalgebra of $\mathfrak{g}_\mathds{k}$. For each $1\leqslant  i\leqslant  r$, let $(G_\mathds{k})_i$ denote the algebraic supergroup associated to $(\mathfrak{g}_\mathds{k})_i$. It is well known that for each $1\leqslant  i\leqslant  r$ the even part of $(G_\mathds{k})_i$ is a reductive algebraic group, and denote it by $((G_\mathds{k})_i)_{\text{ev}}$. Since $\xi|_{\mathfrak{l}_\mathds{k}}=\xi_n|_{\mathfrak{l}_\mathds{k}}$ is nilpotent, it follows from ([\cite{J3}], Lemma 2.10) that $\mathds{k}^\times\cdot\xi|_{(\mathfrak{g}_\mathds{k})_i}\subseteq(\text{Ad}^*((G_\mathds{k})_i)_{\text{ev}})\xi|_{(\mathfrak{g}_\mathds{k})_i}$.
For each $1\leqslant  i\leqslant  r$, since the superalgebra $U_{\xi|_{(\mathfrak{g}_\mathds{k})_i}}((\mathfrak{g}_\mathds{k})_i)$ depends only on the orbit of $\xi|_{(\mathfrak{g}_\mathds{k})_i}$ under the coadjoint action of $((G_\mathds{k})_i)_{\text{ev}}$, then $U_{\xi|_{(\mathfrak{g}_\mathds{k})_i}}((\mathfrak{g}_\mathds{k})_i)\cong U_{\bar t\xi|_{(\mathfrak{g}_\mathds{k})_i}}((\mathfrak{g}_\mathds{k})_i)$ as $\mathds{k}$-algebras. By the arbitrary of $i$ it is immediate that $ \bigotimes\limits_{i=1}^r U_{\xi|_{(\mathfrak{g}_\mathds{k})_i}}((\mathfrak{g}_\mathds{k})_i)\cong \bigotimes\limits_{i=1}^r U_{\bar t\xi|_{(\mathfrak{g}_\mathds{k})_i}}((\mathfrak{g}_\mathds{k})_i)$.
As $\mathfrak{t}'_\mathds{k}$ is a toral subalgebra of $\mathfrak{g}_\mathds{k}$ (note that $\mathfrak{t}'_\mathds{k}\in(\mathfrak{g}_\mathds{k})_{\bar0}$), the reduced enveloping algebra $U_\psi(\mathfrak{t}'_\mathds{k})$ is commutative and semisimple for every $\psi\in(\mathfrak{t}'_\mathds{k})^*$ (Indeed, $\mathfrak{t}'_\mathds{k}$ has a $\mathds{k}$-basis $t_1,\cdots,t_d$ with $t_i^{[p]}=t_i$ for $1\leqslant i\leqslant d$. Therefore, $U_\psi(\mathfrak{t}'_\mathds{k})\cong A_1\otimes\cdots\otimes A_d$ where $A_i\cong\mathds{k}[X]/(X^p-X-\psi(t_i)^p)$ is a $p$-dimensional commutative semisimple $\mathds{k}$-algebra). From this it is immediate that $U_{\xi|_{\mathfrak{t}'_\mathds{k}}}(\mathfrak{t}'_\mathds{k})\cong U_{\bar t\xi|_{\mathfrak{t}'_\mathds{k}}}(\mathfrak{t}'_\mathds{k})$ as algebras. Since $\mathfrak{l}_\mathds{k}=\bigoplus\limits_{i=1}^r(\mathfrak{g}_\mathds{k})_i\oplus\mathfrak{t}'_\mathds{k}$, we have $U_{\xi}(\bigoplus\limits_{i=1}^r(\mathfrak{g}_\mathds{k})_i\oplus\mathfrak{t}'_\mathds{k})\cong \bigotimes\limits_{i=1}^r U_{\xi|_{(\mathfrak{g}_\mathds{k})_i}}((\mathfrak{g}_\mathds{k})_i)\otimes U_{\xi|_{\mathfrak{t}'_\mathds{k}}}(\mathfrak{t}'_\mathds{k})$ and $U_{\bar t\xi}(\bigoplus\limits_{i=1}^r(\mathfrak{g}_\mathds{k})_i\oplus\mathfrak{t}'_\mathds{k})\cong \bigotimes\limits_{i=1}^r U_{\bar t\xi|_{(\mathfrak{g}_\mathds{k})_i}}((\mathfrak{g}_\mathds{k})_i)\otimes U_{\bar t\xi|_{\mathfrak{t}'_\mathds{k}}}(\mathfrak{t}'_\mathds{k})$. Therefore, we can obtain that $U_{\xi}(\mathfrak{l}_\mathds{k})\cong U_{\bar t\xi}(\mathfrak{l}_\mathds{k})$ as algebras. It follows from \eqref{isorr}---\eqref{iso4} that $$U_{\xi}(\mathfrak{g}_\mathds{k})\cong U_{\bar t\xi}(\mathfrak{g}_\mathds{k})$$ for all $\bar t\in\mathds{k}^\times$. Our claim is an immediate consequence of the last isomorphism, i.e. if $\xi\in\Xi$, then $\bar t\cdot\xi\in\Xi$ for all $\bar t\in\mathds{k}^\times$. As the remark preceding part (i) shows that $\Xi$ is Zariski closed in $(\mathfrak{g}_\mathds{k}^*)_{\bar{0}}$, it is immediate that the set $\Xi$ is conical.

(ii) We claim that $\chi\in\Xi$.

By the assumption in the lemma we know that $U_\eta(\mathfrak{g}_\mathds{k})$ has a simple module of dimension $p^{\frac{d_0}{2}}2^{\frac{d_1+1}{2}}$, so we have $\eta\in\Xi$. As $\eta\in\chi+(\mathfrak{m}_\mathds{k}^\perp)_{\bar{0}}$ we can write $\eta=(e+\bar y,\cdot)$ for some $\bar y=\sum\limits_{i\leqslant 1}\bar y_i$ with $\bar y_i\in\mathfrak{g}_\mathds{k}(i)_{\bar{0}}$. Recall that there is a cocharacter $\lambda:\mathds{k}^\times\longrightarrow(G_\mathds{k})_{\text{ev}}$ such that $(\text{Ad}\lambda(\bar t))\bar x=\bar t^j\bar x$ for all $\bar x\in\mathfrak{g}_\mathds{k}(j)\,(j\in\mathbb{Z})$ and $\bar t\in\mathds{k}^\times$ (see the discussion preceding Proposition~\ref{fijeo}). For $i\leqslant 1$, set $\eta_i=(\bar y_i,\cdot)$, then $\eta=\chi+\sum\limits_{i\leqslant 1}\eta_i$. For any even element $\bar y'\in\mathfrak{g}_\mathds{k}(j)_{\bar{0}}$, it is straightforward that$$(\text{Ad}^*(\lambda(\bar t))(\eta))\bar y'=\eta(\text{Ad}\lambda(\bar t)^{-1}\bar  y')=\bar t^{-j}(e+\sum\limits_{i\leqslant 1}\bar y_i,\bar y')=\Bigg\{\begin{array}{ll}\bar t^{2}(e,\bar y')&(j=-2)\\
\delta_{i+j,0}\bar t^{i}(\bar y_i,\bar y')&(j\neq-2)\end{array}.$$ Since$$(\bar t^{2}\chi+\sum\limits_{i\leqslant 1}\bar t^{i}\eta_i)(\bar y')=\bar t^{2}(e,\bar y')+\sum\limits_{i\leqslant 1}\bar t^{i}(\bar y_i,\bar y'),$$ then
$(\text{Ad}^*\lambda(\bar t))\eta=\bar t^{2}\chi+\sum\limits_{i\leqslant 1}\bar t^{i}\eta_i$, and
$(\text{Ad}^*\lambda(\bar t))^{-1}\eta=\bar t^{-2}\chi+\sum\limits_{i\leqslant 1}\bar t^{-i}\eta_i$. As $\Xi$ is conical and $\text{Ad}^*(G_\mathds{k})_{\text{ev}}$-invariant by (i), this implies that $$\bar t^{2}\cdot(\text{Ad}^*\lambda(\bar t))^{-1}\eta=\chi+\sum\limits_{i\leqslant 1}\bar t^{2-i}(\bar y_i,\bar y')\in\Xi$$for all $\bar t\in\mathds{k}^\times$. Since $\Xi$ is Zariski closed, this yields $\chi\in\Xi$.

(iii) Now we are in a position to prove (2).

It follows from (ii) that the algebra $U_\chi(\mathfrak{g}_\mathds{k})$ admits a two-sided ideal of codimension $p^{d_0}2^{d_1+1}$ and denote it by $I$. Clearly, all simple modules of the factor algebra $U_\chi(\mathfrak{g}_\mathds{k})/I$ have dimension $\leqslant  p^{\frac{d_0}{2}}2^{\frac{d_1+1}{2}}$. On the other hand, Proposition~\ref{wzd} implies that all simple modules of $U_\chi(\mathfrak{g}_\mathds{k})/I$ have dimension divisible by $p^{\frac{d_0}{2}}2^{\frac{d_1+1}{2}}$. From this it is immediate that $U_\chi(\mathfrak{g}_\mathds{k})$ has a simple module of dimension $p^{\frac{d_0}{2}}2^{\frac{d_1+1}{2}}$.
\end{proof}

Let $\mathfrak{g}$ be a basic classical Lie superalgebra over $\mathbb{C}$ and $\mathfrak{g}_\mathds{k}$ the corresponding Lie superalgebra over positive characteristic field $\mathds{k}$. Let $\chi\in(\mathfrak{g}_\mathds{k}^*)_{\bar0}$ be a nilpotent $p$-character of $\mathfrak{g}_\mathds{k}$ such that $\chi(\bar y)=(e,\bar y)$ for any $\bar y\in\mathfrak{g}_\mathds{k}$. Under the assumption of Conjecture~\ref{con}, the following theorem shows that the minimal dimension for the representations of reduced enveloping algebra $U_\chi(\mathfrak{g}_\mathds{k})$ with nilpotent $p$-character $\chi$ given in the Super Kac-Weisfeiler Property is reachable.

\begin{theorem}\label{main2}
Retain the notations above. The following are true:

(1) when $d_1$ is even, if the finite $W$-superalgebra $U(\mathfrak{g},e)$ over $\mathbb{C}$ affords a $1$-dimensional representation, then for $p\gg0$ the reduced enveloping algebra $U_\chi(\mathfrak{g}_\mathds{k})$ over $\mathds{k}=\overline{\mathbb{F}}_p$ admits irreducible representations of dimension $p^{\frac{d_0}{2}}2^{\frac{d_1}{2}}$;

(2) when $d_1$ is odd, if the finite $W$-superalgebra $U(\mathfrak{g},e)$ over $\mathbb{C}$ affords a $2$-dimensional representation, then for $p\gg0$ the reduced enveloping algebra $U_\chi(\mathfrak{g}_\mathds{k})$ over $\mathds{k}=\overline{\mathbb{F}}_p$ admits irreducible representations of dimension $p^{\frac{d_0}{2}}2^{\frac{d_1+1}{2}}$.
\end{theorem}

\begin{proof}
(1) follows from Lemma~\ref{mindim1} and Lemma~\ref{etachi}(1). Lemma~\ref{cdim2} and Lemma~\ref{etachi}(2) show that (2) is also true.
\end{proof}

As a corollary of Theorem~\ref{main2}, some conclusions about the minimal dimension for the representations of reduced $W$-superalgebra $U_\chi(\mathfrak{g}_\mathds{k},e)$ associated to $p$-character $\chi$ over $\mathds{k}$ can also be reached, i.e.

\begin{corollary}\label{1122}
Let $\mathfrak{g}$ be a basic classical Lie superalgebra. The following are true:

(1) when $d_1$ is even, if the finite $W$-superalgebra $U(\mathfrak{g},e)$ over $\mathbb{C}$ affords a $1$-dimensional representation, then for $p\gg0$, the reduced $W$-superalgebra $U_\chi(\mathfrak{g}_\mathds{k},e)$ associated to $p$-character $\chi$ over $\mathds{k}$ also admits $1$-dimensional representations;

(2) when $d_1$ is odd, if the finite $W$-superalgebra $U(\mathfrak{g},e)$ over $\mathbb{C}$ affords a $2$-dimensional representation, then for $p\gg0$, the reduced $W$-superalgebra $U_\chi(\mathfrak{g}_\mathds{k},e)$ associated to $p$-character $\chi$ over $\mathds{k}$ also admits $2$-dimensional representations.
\end{corollary}

\begin{proof}
For any $U_\chi(\mathfrak{g}_\mathds{k})$-module $M$, it follows from Theorem~\ref{reducedfunctors} that $M^{\mathfrak{m}_\mathds{k}}$ is a $U_\chi(\mathfrak{g}_\mathds{k},e)$-module. As any $U_\chi(\mathfrak{g}_\mathds{k})$-module $M$ is $U_\chi(\mathfrak{m}_\mathds{k})$-free and there is an isomorphism $M\cong U_\chi(\mathfrak{m}_\mathds{k})^*\otimes_\mathds{k}M^{\mathfrak{m}_\mathds{k}}$ of $U_\chi(\mathfrak{m}_\mathds{k})$-modules by the proof of ([\cite{WZ}], Proposition 4.2), it is immediate that
$$\text{dim}~M^{\mathfrak{m}_\mathds{k}}=\frac{\text{dim}~M}{\text{dim}~U_{\chi}(\mathfrak{m}_\mathds{k})}=\frac{\text{dim}~M}{p^{\frac{d_0}{2}}2^{\lceil\frac{d_1}{2}\rceil}}.$$ Then the desired result follows from Theorem~\ref{main2}.
\end{proof}

\subsection{On the lower bound of the Super KW Property for a direct sum of basic classical Lie superalgebras with nilpotent $p$-characters}
In this part we will consider the lower bound of the Super KW Property for a direct sum of basic classical Lie superalgebras with nilpotent $p$-characters.

First recall some knowledge on finite dimensional superalgebras in (cf. [\cite{KL}], Section 12). Let $\mathbb{F}$ be an algebraically closed field of any characteristic. Given a finite dimensional superalgebra $A$ over $\mathbb{F}$, define the (left) parity change functor
$$\Pi:~A\text{-}mod\longrightarrow A\text{-}mod.$$
For an object $V$, $\Pi V$ is the same underlying vector space but with the opposite $\mathbb{Z}_2$-grading. The new action
of a $\mathbb{Z}_2$-homogeneous element $a\in A$ on $v\in V$ is defined in terms of the old action by $a\cdot v:=(-1)^{|a|}av$.
Given left modules $V$ and $W$ over superalgebras $A$ and $B$ respectively, the (outer) tensor product $V\boxtimes W$ is the space $V\otimes W$ considered as an $A\otimes B$-module via
$$(a\otimes b)(v\otimes w)=(-1)^{|b||v|}av\otimes bw\qquad(a\in A,\,b\in B,\,v\in V,\,w\in W).$$

For the irreducible representations of the algebra $A\otimes B$, the following result was obtained by Kleshchev in ([\cite{KL}], Lemma 12.2.13):

\begin{lemma}$^{[\cite{KL}]}$\label{AB}
Let $V$ be an irreducible $A$-module and $W$ be an irreducible $B$-module.

(i) If both $V$ and $W$ are of type $M$, then $V\boxtimes W$ is an irreducible $A\otimes B$-module of type $M$.

(ii) If one of $V$ or $W$ is of type $M$ and the other is of type $Q$, then $V\boxtimes W$ is an irreducible $A\otimes B$-module of type $Q$.

(iii) If both $V$ and $W$ are of type $Q$, then $V\boxtimes W\cong (V\circledast W)\oplus\Pi(V\circledast W)$ for a type $M$ irreducible $A\otimes B$-module $V\circledast W$.

Moreover, all irreducible $A\otimes B$-modules arise as constituents of $V\boxtimes W$ for some choice of irreducibles $V$, $W$.

\end{lemma}

Now we will recall some basics on superalgebras over $\mathbb{F}$ (c.f. [\cite{KL}], Section 12.1). Let $V$ be a superspace with $\text{\underline{dim}}\,V=(m,n)$ then $\mathcal{M}(V):=\text{End}_\mathbb{F}(V)$ is a superalgebra with $\text{\underline{dim}}\,\mathcal{M}(V)=(m^2+n^2,2mn)$. The algebra $\mathcal{M}(V)$ is defined uniquely up to an isomorphism by the superdimension $(m,n)$ of $V$. So we can speak of the superalgebra $\mathcal{M}_{m,n}$. We have an isomorphism of superalgebras
\begin{equation}\label{MM}
\mathcal{M}_{m,n}\otimes\mathcal{M}_{k,l}\cong\mathcal{M}_{mk+nl,ml+nk}.
\end{equation}

Let $V$ be a superspace with $\text{\underline{dim}}\,V=(n,n)$ and $J$ be a degree $\bar{1}$ involution in $\text{End}_\mathbb{F}(V)$. Consider the superalgebra $Q(V,J):=\{f\in\text{End}_\mathbb{F}(V)\,|\,fJ=(-1)^{|f|}Jf\}$. Note that all degree $\bar{1}$ involutions in $\text{End}_\mathbb{F}(V)$ are conjugate to each other by an invertible element in $\text{End}_\mathbb{F}(V)_{\bar 0}$. Hence another choice of $J$ will yield an isomorphism superalgebra. So we can speak of the superalgebra $Q(V)$, defined up to an isomorphism. Pick a basis $\{v_1,\cdots, v_n\}$ of $V_{\bar0}$, and set $v'_i=J(v_i)$ for $1\leqslant i\leqslant n$. Then $\{v'_1,\cdots, v'_n\}$ is a basis of $V_{\bar1}$. With respect to the basis $\{v_1,\cdots, v_n,v'_1,\cdots, v'_n\}$, the elements of $Q(V,J)$ have matrices of the form
\begin{equation}\label{QAB}
\left(
\begin{array}{cl}
A&B\\
-B&A
\end{array}
\right),
\end{equation}
where $A$ and $B$ are arbitrary $n\times n$ matrices, with $B=0$ for even endomorphisms and $A=0$ for odd ones. In particular, $\text{\underline{dim}}\,Q(V)=(n^2,n^2)$. The superalgebra $Q(V,J)$ can be identified with the superalgebra $Q_n$ of all matrices of the form \eqref{QAB}. Moreover, ([\cite{KL}], (12.6) \& (12.7)) show that
\begin{equation}\label{MQ}
\mathcal{M}_{m,n}\otimes Q_k\cong Q_{(m+n)k}
\end{equation}and
\begin{equation}\label{QQ}
Q_m\otimes Q_n\cong \mathcal{M}_{mn,mn}
\end{equation}
as $\mathbb{F}$-algebras.

Now we turn to the representations of reduced enveloping algebras for a direct sum of basic classical Lie superalgebras with nilpotent $p$-characters over $\mathds{k}$. In ([\cite{WZ}], Remark 4.6), Wang-Zhao showed that Proposition~\ref{wzd} still establishes for the case when $\mathfrak{l}_\mathds{k}$ is a direct sum of basic classical Lie superalgebras.

In fact, their result can be somewhat strengthened. Let $\mathfrak{l}_\mathds{k}=\bigoplus\limits_{i=1}^r(\mathfrak{l}_\mathds{k})_i$ be a direct sum of basic classical Lie superalgebras over $\mathds{k}=\overline{\mathbb{F}}_p$, where $(\mathfrak{l}_\mathds{k})_i$ is a basic classical Lie superalgebra for each $1\leqslant i\leqslant r$. Let $\chi=\chi_1+\cdots+\chi_r$ be the decomposition of nilpotent $p$-character $\chi$ in $\mathfrak{l}_\mathds{k}^*$ with $\chi_i\in(\mathfrak{l}_\mathds{k})_i^*$ (which can be viewed in $\mathfrak{l}_\mathds{k}^*$ by letting $\chi_i(\bar y)=0$ for all $\bar y\in\bigoplus\limits_{j\neq i}(\mathfrak{l}_\mathds{k})_j$) for $1\leqslant i\leqslant r$. Set $\bar e=\bar e_1+\cdots+\bar e_r$ be the corresponding decomposition of $\bar e\in (\mathfrak{l}_\mathds{k})_{\bar0}$ with respect to the non-degenerated bilinear form $(\cdot,\cdot)$ on $\mathfrak{l}_\mathds{k}$ such that $\chi_i(\cdot)=(\bar e_i,\cdot)$  for $1\leqslant i\leqslant r$.
Define
\begin{equation}\label{numsum}
\begin{array}{rcl}
d'_0&:=&\text{dim}\,(\mathfrak{l}_\mathds{k})_{\bar0}-\text{dim}\,(\mathfrak{l}_\mathds{k}^{\bar e})_{\bar0},\\
d'_1&:=&\text{dim}\,(\mathfrak{l}_\mathds{k})_{\bar1}-\text{dim}\,(\mathfrak{l}_\mathds{k}^{\bar e})_{\bar1},\\
(d_0)_i&:=&\text{dim}~((\mathfrak{l}_\mathds{k})_i)_{\bar0}-\text{dim}~((\mathfrak{l}_\mathds{k})^{\bar e_i}_i)_{\bar0},\\
(d_1)_i&:=&\text{dim}~((\mathfrak{l}_\mathds{k})_i)_{\bar1}-\text{dim}~((\mathfrak{l}_\mathds{k})^{\bar e_i}_i)_{\bar1},
\end{array}
\end{equation}where $\mathfrak{l}_\mathds{k}^{\bar e}$ denotes the centralizer of $\bar e$ in $\mathfrak{l}_\mathds{k}$, and $(\mathfrak{l}_\mathds{k})^{\bar e_i}_i$ the centralizer of $\bar e_i$ in $(\mathfrak{l}_\mathds{k})_i$ for each $i\in\{1,\cdots,r\}$. It is obvious that $d'_0=\sum\limits_{i=1}^r(d_0)_i$ and $d'_1=\sum\limits_{i=1}^r(d_1)_i$. Rearrange the summands of $\mathfrak{l}_\mathds{k}=\bigoplus\limits_{i=1}^r(\mathfrak{l}_\mathds{k})_i$ such that $(d_1)_i$ is odd for $1\leqslant i\leqslant l$ (if it occurs) and $(d_1)_i$ is even for $l+1\leqslant i\leqslant r$ (if it occurs). In particular,  $d'_1$ and $l$ have the same parity.

Note that all the definitions and results introduced in preceding sections remain valid for the direct sum of basic classical Lie superalgebras. Let $\mathfrak{m}_\mathds{k}$ and $\mathfrak{m}'_\mathds{k}$ be the subalgebras of $\mathfrak{l}_\mathds{k}$ as defined in Section 4.1. Let $\mathfrak{m}_\mathds{k}=\bigoplus\limits_{i=1}^r(\mathfrak{m}_\mathds{k})_i$ and $\mathfrak{m}'_\mathds{k}=\bigoplus\limits_{i=1}^r(\mathfrak{m}'_\mathds{k})_i$ be the decomposition of $\mathfrak{m}_\mathds{k}$ \& $\mathfrak{m}'_\mathds{k}$ in $\mathfrak{l}_\mathds{k}$ respectively, where $(\mathfrak{m}_\mathds{k})_i,\,(\mathfrak{m}'_\mathds{k})_i\in(\mathfrak{l}_\mathds{k})_i$ for $1\leqslant i\leqslant r$. As $\mathfrak{m}_\mathds{k}$ is $p$-nilpotent and the linear function $\chi$ vanishes on the $p$-closure of $[\mathfrak{m}_\mathds{k},\mathfrak{m}_\mathds{k}]$, it follows from ([\cite{WZ}], Proposition 2.6) that $U_{\chi}(\mathfrak{m}_\mathds{k})$ has a unique irreducible module and $U_{\chi}(\mathfrak{m}_\mathds{k})/N_{\mathfrak{m}_\mathds{k}}\cong \mathds{k}$, where $N_{\mathfrak{m}_\mathds{k}}$ is the Jacobson radical of $U_{\chi}(\mathfrak{m}_\mathds{k})$ which is generated by all the elements $\bar x-\chi(\bar x)$ with $\bar x\in\mathfrak{m}_\mathds{k}$.

First note that ([\cite{WZ}], Proposition 4.1) shows that every
$\mathds{k}$-algebra $U_{\chi_i}((\mathfrak{m}_\mathds{k})_i)$ ($1\leqslant i\leqslant r$) has a unique simple module (we find that there is a minor error in ([\cite{WZ}], Section 4.1) since which is not necessary a trivial module) which is $1$-dimensional and of type $M$. For the case when $(d_1)_i$ is odd (i.e. $1\leqslant i\leqslant l$), the $\mathds{k}$-algebra
$U_{\chi_i}((\mathfrak{m}'_\mathds{k})_i)$ also has a unique simple module; it is isomorphic to $V_i=U_{\chi_i}((\mathfrak{m}'_\mathds{k})_i)\otimes_{U_{\chi_i}((\mathfrak{m}_\mathds{k})_i)}\bar1_{\chi_i}$, which is $2$-dimensional and of type $Q$. Let $N_{(\mathfrak{m}_\mathds{k})_i}$ and $N_{(\mathfrak{m}'_\mathds{k})_i}$ denote the Jacobson radical of $U_{\chi_i}((\mathfrak{m}_\mathds{k})_i)$ and $U_{\chi_i}((\mathfrak{m}'_\mathds{k})_i)$ (which are the ideals of $U_{\chi_i}((\mathfrak{m}_\mathds{k})_i)$ and $U_{\chi_i}((\mathfrak{m}'_\mathds{k})_i)$ generated by all the elements $\bar x-\chi(\bar x)$ with $\bar x\in(\mathfrak{m}_\mathds{k})_i$, respectively). Then $U_{\chi_i}((\mathfrak{m}_\mathds{k})_i)/N_{(\mathfrak{m}_\mathds{k})_i}=\mathds{k}$, and $U_{\chi_i}((\mathfrak{m}'_\mathds{k})_i)/N_{(\mathfrak{m}'_\mathds{k})_i}$ is isomorphic to the simple superalgebra $Q_1$. For the case when $(d_1)_i$ is even (i.e. $l+1\leqslant i\leqslant r$), we have $U_{\chi_i}((\mathfrak{m}'_\mathds{k})_i)=U_{\chi_i}((\mathfrak{m}_\mathds{k})_i)$ and $N_{(\mathfrak{m}'_\mathds{k})_i}=N_{(\mathfrak{m}_\mathds{k})_i}$ since $(\mathfrak{m}'_\mathds{k})_i=(\mathfrak{m}_\mathds{k})_i$ by construction.

Since $\mathfrak{l}_\mathds{k}=\bigoplus\limits_{i=1}^r(\mathfrak{l}_\mathds{k})_i$, it is easy to verify that
\begin{equation}\label{NN'}
U_\chi(\mathfrak{m}_\mathds{k})\cong\bigotimes\limits_{i=1}^rU_{\chi_i}((\mathfrak{m}_\mathds{k})_i),\quad U_\chi(\mathfrak{m}'_\mathds{k})\cong\bigotimes\limits_{i=1}^rU_{\chi_i}((\mathfrak{m}'_\mathds{k})_i)
\end{equation}as $\mathds{k}$-algebras, respectively. For a $U_{\chi}(\mathfrak{l}_\mathds{k})$-module $M$ set $$M^{\mathfrak{m}_\mathds{k}}=\{v\in M\,|\,(\bar x-\chi(\bar x)).v=0~\text{for all}~\bar x\in\mathfrak{m}_\mathds{k}\}.$$As $\mathfrak{l}_\mathds{k}$ (which is a finite dimensional restricted Lie superalgebra) is a direct sum of basic classical Lie superalgebras, the same discussion as ([\cite{WZ}], Proposition 4.2) shows that every $U_{\chi}(\mathfrak{l}_\mathds{k})$-module $M$ is $U_{\chi}(\mathfrak{m}_\mathds{k})$-free and $M\cong U_{\chi}(\mathfrak{m}_\mathds{k})^*\otimes_{\mathds{k}}M^{\mathfrak{m}_\mathds{k}}$ as $U_{\chi}(\mathfrak{m}_\mathds{k})$-modules (which can also be inferred from ([\cite{WZ}], Remark 4.6)).

Let $N_{\mathfrak{m}'_\mathds{k}}$ denote the ideal of $U_\chi(\mathfrak{m}'_\mathds{k})$ generated by all the elements $\bar x-\chi(\bar x)$ with $\bar x\in\mathfrak{m}_\mathds{k}$, then $M^{\mathfrak{m}_\mathds{k}}$ is a $U_\chi(\mathfrak{m}'_\mathds{k})/N_{\mathfrak{m}'_\mathds{k}}$-module by definition.  Since $\mathfrak{m}'_\mathds{k}=\bigoplus\limits_{i=1}^r(\mathfrak{m}'_\mathds{k})_i$, one can conclude from \eqref{NN'} and its preceding remark that
\begin{equation*}
\begin{array}{rcccl}
U_\chi(\mathfrak{m}'_\mathds{k})/N_{\mathfrak{m}'_\mathds{k}}&\cong& U_\chi(\mathfrak{m}'_\mathds{k})\otimes_{U_\chi(\mathfrak{m}_\mathds{k})}\bar1_\chi
&\cong& U_\chi(\bigoplus\limits_{i=1}^r(\mathfrak{m}'_\mathds{k})_i)\otimes_{U_\chi(\bigoplus\limits_{i=1}^r
(\mathfrak{m}_\mathds{k})_i)}\bar1_\chi\\
&\cong&\bigotimes\limits_{i=1}^r(U_{\chi_i}((\mathfrak{m}'_\mathds{k})_i)
\otimes_{U_{\chi_i}((\mathfrak{m}_\mathds{k})_i)}\bar1_{\chi_i})
&\cong&\bigotimes\limits_{i=1}^rU_{\chi_i}((\mathfrak{m}'_\mathds{k})_i)/N_{(\mathfrak{m}'_\mathds{k})_i}\\
&\cong&\overbrace{Q_1\otimes\cdots\otimes Q_1}^{l}\otimes\overbrace{\mathds{k}\otimes\cdots\otimes\mathds{k}}^{r-l}
&\cong&\overbrace{Q_1\otimes\cdots\otimes Q_1}^{l}
\end{array}
\end{equation*}as $\mathds{k}$-algebras.

Now we will introduce the refined Super Kac-Weisfeiler Property for the direct sum of basic classical Lie superalgebras with nilpotent $p$-characters.

\begin{prop}\label{sumdivisible}
Let $\mathfrak{l}_\mathds{k}$ be a direct sum of basic classical Lie superalgebras over $\mathds{k}=\overline{\mathbb{F}}_p$  with $\chi$ a nilpotent $p$-character in $(\mathfrak{l}_\mathds{k}^*)_{\bar0}$. Retain the assumptions above. Then for the primes $p$ which satisfy the restrictions imposed in ([\cite{WZ}], Table 1), the dimension of every $U_\chi(\mathfrak{l}_\mathds{k})$-module $M$ is divisible by $p^{\frac{d'_0}{2}}2^{\frac{d'_1+l}{2}}$.
\end{prop}

\begin{proof}
For each $U_\chi(\mathfrak{l}_\mathds{k})$-module $M$, preceding remark shows that the $U_{\chi}(\mathfrak{m}_\mathds{k})$-module
\begin{equation}\label{MMm}
M\cong U_{\chi}(\mathfrak{m}_\mathds{k})^*\otimes_{\mathds{k}}M^{\mathfrak{m}_\mathds{k}}\end{equation} is free. Now we will consider the dimension of $M^{\mathfrak{m}_\mathds{k}}$ as a $\mathds{k}$-vector space. Recall that $M^{\mathfrak{m}_\mathds{k}}$ is a module over the superalgebra $U_\chi(\mathfrak{m}'_\mathds{k})/N_{\mathfrak{m}'_\mathds{k}}\cong\overbrace{Q_1\otimes\cdots\otimes Q_1}^{l}$. Based on the parity of $l$, for each case we will consider separately.

(i) When $l$ is odd, \eqref{MQ} and \eqref{QQ} imply that
$$U_\chi(\mathfrak{m}'_\mathds{k})/N_{\mathfrak{m}'_\mathds{k}}\cong\overbrace{Q_1\otimes\cdots\otimes Q_1}^{l}\cong Q_{2^{\frac{l-1}{2}}}.$$
Since $Q_{2^{\frac{l-1}{2}}}$ is a simple superalgebra whose unique simple module is $2\cdot2^{\frac{l-1}{2}}=2^{\frac{l+1}{2}}$-dimensional, it follows from Wedderburn's Theorem ([\cite{KL}], Theorem 12.2.9) that every $Q_{2^{\frac{l-1}{2}}}$-module has dimension divisible by $2^{\frac{l+1}{2}}$. In particular, the dimension of $M^{\mathfrak{m}_\mathds{k}}$ (as a vector space) is divisible by $2^{\frac{l+1}{2}}$. By the same discussion as ([\cite{WZ}], Theorem 4.3) we can conclude that $\text{\underline{dim}}\,\mathfrak{m}_\mathds{k}=(\frac{d'_0}{2},\frac{d'_1-1}{2})$, then dim\,$U_\chi(\mathfrak{m}_\mathds{k})=p^{\frac{d'_0}{2}}2^{\frac{d'_1-1}{2}}$. Together with \eqref{MMm} this implies that each $U_\chi(\mathfrak{l}_\mathds{k})$-module $M$ has dimension divisible by $p^{\frac{d'_0}{2}}2^{\frac{d'_1-1}{2}}\cdot2^{\frac{l+1}{2}}=p^{\frac{d'_0}{2}}2^{\frac{d'_1+l}{2}}$.

(ii) When $l$ is even, it follows from \eqref{MQ} and \eqref{QQ} that
$$U_\chi(\mathfrak{m}'_\mathds{k})/N_{\mathfrak{m}'_\mathds{k}}\cong\overbrace{Q_1\otimes\cdots\otimes Q_1}^{l}\cong
\mathcal{M}_{2^{\frac{l}{2}-1},2^{\frac{l}{2}-1}}.$$
Since $\mathcal{M}_{2^{\frac{l}{2}-1},2^{\frac{l}{2}-1}}$ is a simple superalgebra whose unique simple module is $2^{\frac{l}{2}}$-dimensional, it follows from Wedderburn's Theorem that every $\mathcal{M}_{2^{\frac{l}{2}-1},2^{\frac{l}{2}-1}}$-module has dimension divisible by $2^{\frac{l}{2}}$. In particular, the dimension of $M^{\mathfrak{m}_\mathds{k}}$ is divisible by $2^{\frac{l}{2}}$. The same discussion as ([\cite{WZ}], Theorem 4.3) shows that $\text{\underline{dim}}\,\mathfrak{m}_\mathds{k}=(\frac{d'_0}{2},\frac{d'_1}{2})$, then dim\,$U_\chi(\mathfrak{m}_\mathds{k})=p^{\frac{d'_0}{2}}2^{\frac{d'_1}{2}}$. Together with \eqref{MMm} this implies that each $U_\chi(\mathfrak{l}_\mathds{k})$-module $M$ has dimension divisible by $p^{\frac{d'_0}{2}}2^{\frac{d'_1}{2}}\cdot2^{\frac{l}{2}}=p^{\frac{d'_0}{2}}2^{\frac{d'_1+l}{2}}$.

All the discussions in (i) and (ii) complete the proof.

\end{proof}

\begin{rem}\label{refine}
Recall that $d'_1$ and $l$ have the same parity by the previous remark. It is obvious that Proposition~\ref{sumdivisible} coincides with the consequence obtained by Wang-Zhao in ([\cite{WZ}], Remark 4.6) if at most only one of the $(d_1)_i's$ is odd for $1\leqslant i\leqslant r$. But for the case when more than two $(d_1)_i's$ are odd for $1\leqslant i\leqslant r$, the boundary introduced in Proposition~\ref{sumdivisible} is much larger. Therefore, compared with the result obtained by Wang-Zhao, the characterization for the dimension of the $U_\chi(\mathfrak{l}_\mathds{k})$-modules in Proposition~\ref{sumdivisible} is optimal.
\end{rem}

Under the assumption of Conjecture~\ref{con}, the following theorem shows that the minimal dimension for the representations of reduced enveloping algebra $U_\chi(\mathfrak{l}_\mathds{k})$ with nilpotent $p$-character $\chi\in(\mathfrak{l}_\mathds{k}^*)_{\bar0}$ in Proposition~\ref{sumdivisible} is reachable.

\begin{theorem}\label{sumreachable}
Retain the assumptions as Proposition~\ref{sumdivisible}. For any basic classical Lie superalgebra $\mathfrak{g}$ over $\mathbb{C}$, assume that

(i) when dim\,$\mathfrak{g}(-1)_{\bar1}$ is even, the finite $W$-superalgebra $U(\mathfrak{g},e)$ affords a $1$-dimensional representation;

(ii) when dim\,$\mathfrak{g}(-1)_{\bar1}$ is odd, the finite $W$-superalgebra $U(\mathfrak{g},e)$ affords a $2$-dimensional representation.

Let $\mathfrak{l}_\mathds{k}$ be a direct sum of basic classical Lie superalgebras over $\mathds{k}=\overline{\mathbb{F}}_p$, and let $\chi$ be a nilpotent $p$-character in $(\mathfrak{l}_\mathds{k}^*)_{\bar0}$. Then for $p\gg0$ the reduced enveloping algebra $U_\chi(\mathfrak{l}_\mathds{k})$ admits irreducible representations of dimension $p^{\frac{d'_0}{2}}2^{\frac{d'_1+l}{2}}$.
\end{theorem}

\begin{proof}
For each $1\leqslant i\leqslant r$, let $Q_{\chi_i}^{\chi_i}$ be the $(\mathfrak{l}_\mathds{k})_i$-module as defined in Section 4.1, and denote by $U_{\chi_i}((\mathfrak{l}_\mathds{k})_i,\bar e_i)=(\text{End}_{(\mathfrak{l}_\mathds{k})_i}Q_{\chi_i}^{\chi_i})^{\text{op}}$ the reduced $W$-superalgebra of basic classical Lie superalgebra $(\mathfrak{l}_\mathds{k})_i$ associated with nilpotent element $\bar e_i$. Let $Q_\chi^\chi$ be the $\mathfrak{l}_\mathds{k}$-module with the same definition as Section 4.1, and $U_{\chi}(\mathfrak{l}_\mathds{k},\bar e)$ the reduced $W$-superalgebra of $\mathfrak{l}_\mathds{k}$ associated with nilpotent element $\bar e$. Then we have
\begin{equation}\label{dir w}
\begin{array}{lllll}
U_{\chi}(\mathfrak{l}_\mathds{k},\bar e)&=&(\text{End}_{\mathfrak{l}_\mathds{k}}Q_{\chi}^{\chi})^{\text{op}}&\cong& (\text{End}_{\bigoplus\limits_{i=1}^r(\mathfrak{l}_\mathds{k})_i}\bigoplus\limits_{i=1}^rQ_{\chi_i}^{\chi_i})^{\text{op}}\\
&\cong&\bigotimes\limits_{i=1}^r
(\text{End}_{(\mathfrak{l}_\mathds{k})_i}Q_{\chi_i}^{\chi_i})^{\text{op}}&=&\bigotimes\limits_{i=1}^r U_{\chi_i}((\mathfrak{l}_\mathds{k})_i,\bar e_i)
\end{array}
\end{equation}as $\mathds{k}$-algebras.

(1) First consider the case when $1\leqslant i\leqslant l$. Under the assumption of the theorem, Corollary~\ref{1122} shows that the $\mathds{k}$-algebra $U_{\chi_i}((\mathfrak{l}_\mathds{k})_i,\bar e_i)$ admits $2$-dimensional representations for $1\leqslant i\leqslant l$.  Denote by $V_{1}$ and $V_{2}$ the $2$-dimensional irreducible representations (of type $Q$) of the $\mathds{k}$-algebras $U_{\chi_{1}}((\mathfrak{l}_\mathds{k})_{1},\bar e_{1})$ and $U_{\chi_{2}}((\mathfrak{l}_\mathds{k})_{2},\bar e_{2})$ (if occurs), respectively. It follows from Lemma~\ref{AB}(iii) that $V_{1}\boxtimes V_{2}\cong (V_{1}\circledast V_{2}) \oplus\Pi(V_{1}\circledast V_{2})$ as $U_{\chi_{1}}((\mathfrak{l}_\mathds{k})_{1},\bar e_{1})\otimes U_{\chi_{2}}((\mathfrak{l}_\mathds{k})_{2},\bar e_{2})$-modules, where $V_{1}\circledast V_{2}$ is an irreducible $U_{\chi_{1}}((\mathfrak{l}_\mathds{k})_{1},\bar e_{1})\otimes U_{\chi_{2}}((\mathfrak{l}_\mathds{k})_{2},\bar e_{2})$-module of type $M$. Since the $U_{\chi_{1}}((\mathfrak{l}_\mathds{k})_{1},\bar e_{1})\otimes U_{\chi_{2}}((\mathfrak{l}_\mathds{k})_{2},\bar e_{2})$-module $\Pi(V_{1}\circledast V_{2})$ is the same underlying vector space as $V_{1}\circledast V_{2}$ but with the opposite $\mathbb{Z}_2$-grading, we have $\text{dim}\,V_{1}\circledast V_{2}=\text{dim}\,\Pi(V_{1}\circledast V_{2})$ as vector spaces. Recall that the (outer) tensor product $V_{1}\boxtimes V_{2}$ is the same underlying vector space as $V_{1}\otimes V_{2}$, then $\text{dim}\,V_{1}\boxtimes V_{2}=\text{dim}\,V_{1}\otimes V_{2}=4$. From all above we can conclude that $V_{1}\circledast V_{2}$ is an irreducible $U_{\chi_{1}}((\mathfrak{l}_\mathds{k})_{1},\bar e_{1})\otimes U_{\chi_{2}}((\mathfrak{l}_\mathds{k})_{2},\bar e_{2})$-module of type $M$ with dimension $2=2^{\frac{2}{2}}$.

Denote by $V_{3}$ a $2$-dimensional irreducible representation (of type $Q$) of the $\mathds{k}$-algebra $U_{\chi_{3}}((\mathfrak{l}_\mathds{k})_{3},\bar e_{3})$ (if occurs). It follows from Lemma~\ref{AB}(ii) and the discussion above that $(V_{1}\circledast V_{2})\boxtimes V_{3}$ is an irreducible $(U_{\chi_{1}}((\mathfrak{l}_\mathds{k})_{1},\bar e_{1})\otimes U_{\chi_{2}}((\mathfrak{l}_\mathds{k})_{2},\bar e_{2}))\otimes U_{\chi_{3}}((\mathfrak{l}_\mathds{k})_{3},\bar e_{3})$-module of type $Q$. Hence $(V_{1}\circledast V_{2})\boxtimes V_{3}$ is an irreducible $U_{\chi_{1}}((\mathfrak{l}_\mathds{k})_{1},\bar e_{1})\otimes U_{\chi_{2}}((\mathfrak{l}_\mathds{k})_{2},\bar e_{2})\otimes U_{\chi_{3}}((\mathfrak{l}_\mathds{k})_{3},\bar e_{3})$-module (of type $Q$) of dimension $2\cdot2=4=2^{\frac{3+1}{2}}$.

(2) In light of Lemma~\ref{AB}(ii) \& (iii), induction on the number of the terms for $\bigotimes\limits_{i=1}^l U_{\chi_i}((\mathfrak{l}_\mathds{k})_i,e_i)$ with the same discussion as (1) one can show that

(i) when $l$ is odd, the $\mathds{k}$-algebra $\bigotimes\limits_{i=1}^l U_{\chi_i}((\mathfrak{l}_\mathds{k})_i,\bar e_i)$ admits an irreducible representation of type $Q$ with dimension $2^{\frac{l+1}{2}}$, and set it as $V$;

(ii) when $l$ is even, the $\mathds{k}$-algebra $\bigotimes\limits_{i=1}^l U_{\chi_i}((\mathfrak{l}_\mathds{k})_i,\bar e_i)$ admits an irreducible representation of type $M$ with dimension $2^{\frac{l}{2}}$ (assume that $l=0$ when the $(d_1)_i's$ are all even for $1\leqslant i\leqslant r$), and set it as $V'$.

Under the assumption of the theorem, Corollary~\ref{1122}(i) shows that the $\mathds{k}$-algebra $U_{\chi_i}((\mathfrak{l}_\mathds{k})_i,\bar e_i)$ admits $1$-dimensional representations of type $M$ for $l+1\leqslant i\leqslant r$ (if occurs). Easy induction based on Lemma~\ref{AB}(i) shows that the $\mathds{k}$-algebra $\bigotimes\limits_{l+1}^r U_{\chi_i}((\mathfrak{l}_\mathds{k})_i,\bar e_i)$ admits a $1$-dimensional representation of type $M$, and set it as $W$.

(3) Now we will consider the representations of the $\mathds{k}$-algebra $U_\chi(\mathfrak{l}_\mathds{k},\bar e)\cong\bigotimes\limits_{i=1}^r U_{\chi_i}((\mathfrak{l}_\mathds{k})_i,\bar e_i)$. Based on the parity of $l$, for each case we will consider separately.

(i) When $l$ is odd, (2) shows that the $\mathds{k}$-algebra $\bigotimes\limits_{i=1}^l U_{\chi_i}((\mathfrak{l}_\mathds{k})_i,\bar e_i)$ admits an irreducible representation $V$ of type $Q$ with dimension $2^{\frac{l+1}{2}}$, and the $\mathds{k}$-algebra $\bigotimes\limits_{l+1}^r U_{\chi_i}((\mathfrak{l}_\mathds{k})_i,\bar e_i)$ admits a $1$-dimensional representation $W$ of type $M$. It follows from Lemma~\ref{AB}(ii) that $V\boxtimes W$ is an irreducible $\bigotimes\limits_{i=1}^l U_{\chi_i}((\mathfrak{l}_\mathds{k})_i,\bar e_i)\otimes\bigotimes\limits_{l+1}^r U_{\chi_i}((\mathfrak{l}_\mathds{k})_i,\bar e_i)\cong\bigotimes\limits_{i=1}^r U_{\chi_i}((\mathfrak{l}_\mathds{k})_i,\bar e_i)$-module of type $Q$ with dimension $2^{\frac{l+1}{2}}$.

(ii) When $l$ is even, (2) shows that the $\mathds{k}$-algebra $\bigotimes\limits_{i=1}^l U_{\chi_i}((\mathfrak{l}_\mathds{k})_i,\bar e_i)$ admits an irreducible representation $V'$ of type $M$ with dimension $2^{\frac{l}{2}}$, and the $\mathds{k}$-algebra $\bigotimes\limits_{l+1}^r U_{\chi_i}((\mathfrak{l}_\mathds{k})_i,\bar e_i)$ admits a $1$-dimensional representation $W$ of type $M$. It follows from Lemma~\ref{AB}(i) that $V'\boxtimes W$ is an irreducible $\bigotimes\limits_{i=1}^l U_{\chi_i}((\mathfrak{l}_\mathds{k})_i,\bar e_i)\otimes\bigotimes\limits_{l+1}^r U_{\chi_i}((\mathfrak{l}_\mathds{k})_i,\bar e_i)\cong\bigotimes\limits_{i=1}^r U_{\chi_i}((\mathfrak{l}_\mathds{k})_i,\bar e_i)$-module of type $M$ with dimension $2^{\frac{l}{2}}$.

(4) Recall in ([\cite{WZ}], Remark~4.6) Wang-Zhao shows that
\begin{equation}\label{isosum}
U_\chi(\mathfrak{l}_\mathds{k})\cong \text{Mat}_{p^{\frac{d'_0}{2}}2^{\lceil\frac{d'_1}{2}\rceil}}(U_\chi(\mathfrak{l}_\mathds{k},\bar e))
\end{equation}as $\mathds{k}$-algebras. Based on the parity of $l$, for each case we will consider separately.

(i) When $l$ is odd, it follows from (3)(i) that the $\mathds{k}$-algebra $U_\chi(\mathfrak{l}_\mathds{k})$ affords irreducible representations of dimension $p^{\frac{d'_0}{2}}2^{\frac{d'_1-1}{2}}\cdot 2^{\frac{l+1}{2}}=p^{\frac{d'_0}{2}}2^{\frac{d'_1+l}{2}}$;

(ii) When $l$ is even, it follows from (3)(ii) that the $\mathds{k}$-algebra $U_\chi(\mathfrak{l}_\mathds{k})$ affords irreducible representations of dimension $p^{\frac{d'_0}{2}}2^{\frac{d'_1}{2}}\cdot 2^{\frac{l}{2}}=p^{\frac{d'_0}{2}}2^{\frac{d'_1+l}{2}}$.

All the discussions above complete the proof.
\end{proof}

\begin{rem}\label{refine'}
Recall in Remark~\ref{refine} we have showed that the boundary introduced in Proposition~\ref{sumdivisible} is much larger when more than two $(d_1)_i's$ are odd for $1\leqslant i\leqslant r$. In fact, careful inspection on the proof of Theorem~\ref{sumreachable} shows that the boundary obtained by Wang-Zhao in ([\cite{WZ}], Remark~4.6) can never be reached in this case, and the one introduced in Theorem~\ref{sumreachable} is optimal.
\end{rem}

\subsection{On the lower bound of the Super KW Property with arbitrary $p$-characters}
In this part we will discuss whether the lower bound of the Super Kac-Weisfeiler Property for a basic classical Lie superalgebra with any $p$-characters introduced by Wang-Zhao in ([\cite{WZ}], Theorem 5.6) can be reached.

Let $\mathfrak{g}_\mathds{k}$ be a basic classical Lie superalgebra over positive characteristic field $\mathds{k}=\overline{\mathbb{F}}_p$, and $\xi=\xi_{\bar s}+\xi_{\bar n}$ be the Jordan decomposition of $\xi\in(\mathfrak{g}_\mathds{k}^*)_{\bar0}$ (we regard $\xi\in\mathfrak{g}_\mathds{k}^*$ by letting $\xi((\mathfrak{g}_\mathds{k}^*)_{\bar1})=0$). Under the isomorphism $(\mathfrak{g}_\mathds{k}^*)_{\bar0}\cong(\mathfrak{g}_\mathds{k})_{\bar0}$ induced by the non-degenerated bilinear form $(\cdot,\cdot)$ on $(\mathfrak{g}_\mathds{k})_{\bar0}$, it can be identified with the usual Jordan decomposition $\bar x=\bar s+\bar n$ on $(\mathfrak{g}_\mathds{k})_{\bar0}$. Let $\mathfrak{h}_\mathds{k}$ be a Cartan subalgebra of $\mathfrak{g}_\mathds{k}$ and denote by $\mathfrak{l}_\mathds{k}=\mathfrak{g}_\mathds{k}^{\bar s}$ the centralizer of $\bar s$ in $\mathfrak{g}_\mathds{k}$. Let $\Phi$ be the root system of $\mathfrak{g}_\mathds{k}$ and $\Phi(\mathfrak{l}_\mathds{k}):=\{\alpha\in\Phi~|~\alpha(\bar s)=0\}$. By ([\cite{WZ}], Proposition 5.1) $\mathfrak{l}_\mathds{k}$ is always a direct sum of basic classical Lie superalgebras with a system $\Pi$ of simple roots of $\mathfrak{g}_\mathds{k}$ such that $\Pi\cap \Phi(\mathfrak{l}_\mathds{k})$ is a system of simple roots of $\Phi(\mathfrak{l}_\mathds{k})$ (note that a toral subalgebra of $\mathfrak{g}_\mathds{k}$ may also appear in the summand).

Set $$\mathfrak{l}_\mathds{k}=\mathfrak{g}_\mathds{k}^{\bar s}=\bigoplus\limits_{i=1}^r(\mathfrak{g}_\mathds{k})_i\oplus\mathfrak{t}'_\mathds{k},$$ where $(\mathfrak{g}_\mathds{k})_i$ is a basic classical Lie superalgebra for each $1\leqslant i\leqslant r$, and
$\mathfrak{t}'_\mathds{k}$ is a toral subalgebra of $\mathfrak{g}_\mathds{k}$. Then $\xi_{\bar n}=\xi_1+\cdots+\xi_r$ is a nilpotent $p$-character of $\mathfrak{l}_\mathds{k}$ with $\xi_i\in(\mathfrak{g}_\mathds{k})_i^*$ (which can be viewed in $\mathfrak{l}_\mathds{k}^*$ by letting $\xi_i(\bar y)=0$ for all $\bar y\in\bigoplus\limits_{j\neq i}(\mathfrak{g}_\mathds{k})_j\oplus\mathfrak{t}'_\mathds{k}$) for $1\leqslant i\leqslant r$. Let $\bar n=\bar n_1+\cdots+\bar n_r$ be the corresponding decomposition of $\bar n$ in  $\mathfrak{l}_\mathds{k}$ such that $\xi_i(\cdot)=(\bar n_i,\cdot)$  for $1\leqslant i\leqslant r$. For each $1\leqslant i\leqslant r$, denote by $U_{\xi_i}((\mathfrak{g}_\mathds{k})_i,\bar n_i)$ the reduced $W$-superalgebra of basic classical Lie superalgebra $(\mathfrak{g}_\mathds{k})_i$ associated with nilpotent element $\bar n_i$, then it is easy to verify that
$$U_{\xi_{\bar n}}(\bigoplus\limits_{i=1}^r(\mathfrak{g}_\mathds{k})_i,\bar n)\cong\bigotimes\limits_{i=1}^r U_{\xi_i}((\mathfrak{g}_\mathds{k})_i,\bar n_i)$$ by the same discussion as \eqref{dir w}.
Define
\begin{equation}\label{arbitdim}
\begin{array}{rcl}
d_0&:=&\text{dim}\,(\mathfrak{g}_\mathds{k})_{\bar0}-\text{dim}\,(\mathfrak{g}_\mathds{k}^{\bar x})_{\bar0},\\
d_1&:=&\text{dim}\,(\mathfrak{g}_\mathds{k})_{\bar1}-\text{dim}\,(\mathfrak{g}_\mathds{k}^{\bar x})_{\bar1},\\
(d_0)_i&:=&\text{dim}~((\mathfrak{g}_\mathds{k})_i)_{\bar0}-\text{dim}~((\mathfrak{g}_\mathds{k})^{\bar n_i}_i)_{\bar0},\\
(d_1)_i&:=&\text{dim}~((\mathfrak{g}_\mathds{k})_i)_{\bar1}-\text{dim}~((\mathfrak{g}_\mathds{k})^{\bar n_i}_i)_{\bar1},
\end{array}
\end{equation}where $\mathfrak{g}_\mathds{k}^{\bar x}$ denotes the centralizer of $\bar x$ in $\mathfrak{g}_\mathds{k}$, and $(\mathfrak{g}_\mathds{k})^{\bar n_i}_i$ the centralizer of $\bar n_i$ in $(\mathfrak{g}_\mathds{k})_i$ for each $i\in\{1,\cdots,r\}$, and set
\begin{equation}\label{sumdim}
d'_0:=\sum\limits_{i=1}^r(d_0)_i,\qquad
d'_1:=\sum\limits_{i=1}^r(d_1)_i.
\end{equation}
Rearrange the summands of $\bigoplus\limits_{i=1}^r(\mathfrak{g}_\mathds{k})_i$ such that each $(d_1)_i$ is odd for $1\leqslant i\leqslant l$ (if it occurs), and each $(d_1)_i$ is even for $l+1\leqslant i\leqslant r$ (if it occurs).

Let $\mathfrak{b}_\mathds{k}=\mathfrak{h}_\mathds{k}\oplus\mathfrak{n}_\mathds{k}$ be the Borel subalgebra associated to $\Pi$. Define a parabolic subalgebra $\mathfrak{p}_\mathds{k}=\mathfrak{l}_\mathds{k}+\mathfrak{b}_\mathds{k}=\mathfrak{l}_\mathds{k}\oplus\mathfrak{u}_\mathds{k}$, where $\mathfrak{u}_\mathds{k}$ is the nilradical of $\mathfrak{p}_\mathds{k}$. Since $\xi(\mathfrak{u}_\mathds{k})=0$ and $\xi|_{\mathfrak{l}_\mathds{k}}=\xi_{\bar n}|_{\mathfrak{l}_\mathds{k}}$ is nilpotent by ([\cite{WZ}], Section 5.1), any $U_\xi(\mathfrak{l}_\mathds{k})$-mod can be regarded as a $U_\xi(\mathfrak{p}_\mathds{k})$-mod with a trivial action of $\mathfrak{u}_\mathds{k}$. Wang-Zhao proved that the $\mathds{k}$-algebras $U_\xi(\mathfrak{g}_\mathds{k})$ and $U_\xi(\mathfrak{l}_\mathds{k})$ are Morita equivalent in ([\cite{WZ}], Theorem 5.2), and showed that for any irreducible $U_\xi(\mathfrak{g}_\mathds{k})$-mod $M$, $M^{\mathfrak{u}_\mathds{k}}$ is an irreducible $U_\xi(\mathfrak{l}_\mathds{k})$-mod (which is also a $U_\xi(\mathfrak{p}_\mathds{k})$-mod with a trivial action of $\mathfrak{u}_\mathds{k}$) and the natural map
\begin{equation}\label{gp}
U_\xi(\mathfrak{g}_\mathds{k})\otimes_{U_\xi(\mathfrak{p}_\mathds{k})}M^{\mathfrak{u}_\mathds{k}}\longrightarrow M
\end{equation}is an isomorphism in ([\cite{WZ}], Theorem 5.3)  .

\begin{lemma}\label{main3}
Let $\mathfrak{g}_\mathds{k}$ be a basic classical Lie superalgebra over $\mathds{k}=\overline{\mathbb{F}}_p$, and let $\xi\in (\mathfrak{g}_\mathds{k}^*)_{\bar0}$. Retain the notations as all above.
Then for the primes $p$ which satisfy the restrictions imposed in ([\cite{WZ}], Table 1), the dimension of every $U_\xi(\mathfrak{g}_\mathds{k})$-mod $M$ is divisible by $p^{\frac{d_0}{2}}2^{\frac{d_1+l}{2}}$. 
\end{lemma}

\begin{proof}

(1) First note that
([\cite{WZ}], Theorem 5.6) shows
\begin{equation}\label{gtol}
\begin{split}
\text{\underline{dim}}~\mathfrak{g}_\mathds{k}-\text{\underline{dim}}~\mathfrak{g}_\mathds{k}^{\bar x}&=\text{\underline{dim}}~\mathfrak{g}_\mathds{k}-\text{\underline{dim}}~\mathfrak{l}_\mathds{k}^{\bar n}\\
&=2\text{\underline{dim}}~\mathfrak{u}_\mathds{k}^-+(\text{\underline{dim}}~\mathfrak{l}_\mathds{k}-\text{\underline{dim}}~\mathfrak{l}_\mathds{k}^{\bar n}),
\end{split}
\end{equation} where $\mathfrak{l}_\mathds{k}^{\bar n}$ denotes the centralizer of $\bar n$ in $\mathfrak{l}_\mathds{k}$.
Since $\mathfrak{l}_\mathds{k}=\bigoplus\limits_{i=1}^r(\mathfrak{g}_\mathds{k})_i\oplus\mathfrak{t}'_\mathds{k}$ and $\bar n
\in \bigoplus\limits_{i=1}^r(\mathfrak{g}_\mathds{k})_i$, it is obvious that $(\mathfrak{t}'_\mathds{k})^{\bar n}=\mathfrak{t}'_\mathds{k}$, then
\begin{equation}\label{ltosum}
\begin{split}
\text{\underline{dim}}~\mathfrak{l}_\mathds{k}-\text{\underline{dim}}~\mathfrak{l}_\mathds{k}^{\bar n}=&\sum\limits_{i=1}^r
(\text{\underline{dim}}~(\mathfrak{g}_\mathds{k})_i-\text{\underline{dim}}~(\mathfrak{g}_\mathds{k})_i^{\bar n_i})+
\text{\underline{dim}}~\mathfrak{t}'_\mathds{k}-\text{\underline{dim}}~(\mathfrak{t}'_\mathds{k})^{\bar n}\\=&(\sum\limits_{i=1}^r(d_0)_i,\sum\limits_{i=1}^r(d_1)_i).\end{split}
\end{equation} As $\text{\underline{dim}}~\mathfrak{u}_\mathds{k}^-=\text{\underline{dim}}~\mathfrak{u}_\mathds{k}$, \eqref{gtol} shows that
\begin{equation}\label{dimu}
\begin{array}{rcl}
\text{dim}~(\mathfrak{u}^-_\mathds{k})_{\bar0}&=&\text{dim}~(\mathfrak{u}_\mathds{k})_{\bar0}=\frac{d_0-\sum\limits_{i=1}^r(d_0)_i}{2};\\ \text{dim}~(\mathfrak{u}^-_\mathds{k})_{\bar1}&=&\text{dim}~(\mathfrak{u}_\mathds{k})_{\bar1}=\frac{d_1-\sum\limits_{i=1}^r(d_1)_i}{2}.
\end{array}
\end{equation}

(2) Recall that $\mathfrak{g}_\mathds{k}^{\bar s}=\mathfrak{l}_\mathds{k}=\bigoplus\limits_{i=1}^r(\mathfrak{g}_\mathds{k})_i\oplus\mathfrak{t}'_\mathds{k}$, then we have
\begin{equation}\label{lsumt'}
U_{\xi_{\bar n}}(\mathfrak{l}_\mathds{k})\cong U_{\xi_{\bar n}}(\bigoplus\limits_{i=1}^r(\mathfrak{g}_\mathds{k})_i\oplus\mathfrak{t}'_\mathds{k})\cong
U_{\xi_{\bar n}}(\bigoplus\limits_{i=1}^r(\mathfrak{g}_\mathds{k})_i)\otimes U_0(\mathfrak{t}'_\mathds{k})
\end{equation} as $\mathds{k}$-algebras.

(i) First consider the representations of the $\mathds{k}$-algebra $U_{\xi_{\bar n}}(\bigoplus\limits_{i=1}^r(\mathfrak{g}_\mathds{k})_i)$. Since $\xi_{\bar n}|_{\mathfrak{l}_\mathds{k}}$ is nilpotent and each $(\mathfrak{g}_\mathds{k})_i$ is a basic classical Lie superalgebra for $1\leqslant i\leqslant r$, it follows from Proposition~\ref{sumdivisible} that every $U_{\xi_{\bar n}}(\bigoplus\limits_{i=1}^r(\mathfrak{g}_\mathds{k})_i)$-module is divisible by $p^{\frac{d'_0}{2}}2^{\frac{d'_1+l}{2}}$.

As $U_{\xi_{\bar n}}(\mathfrak{l}_\mathds{k})\cong U_{\xi_{\bar n}}(\bigoplus\limits_{i=1}^r(\mathfrak{g}_\mathds{k})_i)\otimes U_0(\mathfrak{t}'_\mathds{k})$ by \eqref{lsumt'}, the discussion above shows that every $U_{\xi_{\bar n}}(\mathfrak{l}_\mathds{k})$-module is divisible by $p^{\frac{d'_0}{2}}2^{\frac{d'_1+l}{2}}$.

(ii) Now discuss the representations of the $\mathds{k}$-algebra $U_0(\mathfrak{t}'_\mathds{k})$. As $\mathfrak{t}'_\mathds{k}$ is a toral subalgebra of $\mathfrak{g}_\mathds{k}$ with a basis $\{t_1,\cdots,t_d\}$ such that $t_i^{[p]}=t_i$ for all $1\leqslant i\leqslant d$, then $U_0(\mathfrak{t}'_\mathds{k})\cong\bigotimes A_1^{\otimes d}$ where $A_1\cong\mathds{k}[X]/(X^p-X)$ is a $p$-dimensional commutative semisimple $\mathds{k}$-algebra whose irreducible representations are $1$-dimensional (obviously which are of type $M$). Hence we can conclude from Lemma~\ref{AB}(i) that the irreducible representations of the $\mathds{k}$-algebra $U_0(\mathfrak{t}'_\mathds{k})$ are all of type $M$ with dimension $1$.

(3) Recall that $p^{\frac{d'_0}{2}}2^{\frac{d'_1+l}{2}}=p^{\frac{\sum\limits_{i=1}^r(d_0)_i}{2}}2^{\frac{l+\sum\limits_{i=1}^r(d_1)_i}{2}}$ by \eqref{sumdim}. As any $U_{\xi}(\mathfrak{l}_\mathds{k})$-mod can be regarded as a $U_\xi(\mathfrak{p}_\mathds{k})$-mod with a trivial action of $\mathfrak{u}_\mathds{k}$, it is immediate from \eqref{gp} and \eqref{dimu} that the dimension of every $U_\xi(\mathfrak{g}_\mathds{k}$)-mod is divisible by
\begin{equation*}
\begin{array}{rl}
&p^{\frac{d'_0}{2}}2^{\frac{d'_1+l}{2}}\cdot p^{\frac{d_0-\sum\limits_{i=1}^r(d_0)_i}{2}} 2^{\frac{d_1-\sum\limits_{i=1}^r(d_1)_i}{2}}\\
=&p^{\sum\limits_{i=1}^r\frac{(d_0)_i}{2}}2^{\frac{l}{2}+\sum\limits_{i=1}^r\frac{(d_1)_i}{2}}\cdot p^{\frac{d_0-\sum\limits_{i=1}^r(d_0)_i}{2}} 2^{\frac{d_1-\sum\limits_{i=1}^r(d_1)_i}{2}}\\
=&p^{\frac{d_0}{2}}2^{\frac{d_1+l}{2}},
\end{array}\end{equation*} completing the proof.
\end{proof}

Now we will discuss whether the lower bound introduced in Lemma~\ref{main3} can be reached for $p\gg0$. Under the assumption of Conjecture~\ref{con}, the following lemma shows that the minimal dimension for the representations of reduced enveloping algebra $U_\xi(\mathfrak{g}_\mathds{k})$ with arbitrary $p$-character $\xi\in(\mathfrak{g}_\mathds{k}^*)_{\bar0}$ in Lemma~\ref{main3} is reachable.

\begin{lemma}\label{main4}
Let $\mathfrak{g}_\mathds{k}$ be a basic classical Lie superalgebra over $\mathds{k}=\overline{\mathbb{F}}_p$, and let $\xi\in (\mathfrak{g}_\mathds{k}^*)_{\bar0}$. Retain the notations as Lemma~\ref{main3}. If Conjecture~\ref{con} is established, then for $p\gg0$ the reduced enveloping algebra $U_\xi(\mathfrak{g}_\mathds{k})$ admits irreducible representations of dimension $p^{\frac{d_0}{2}}2^{\frac{d_1+l}{2}}$.
\end{lemma}

\begin{proof}
Recall that Theorem~\ref{sumreachable} shows that the $\mathds{k}$-algebra $U_{\xi_{\bar n}}(\bigoplus\limits_{i=1}^r(\mathfrak{g}_\mathds{k})_i)$ admits an irreducible representation of dimension $p^{\sum\limits_{i=1}^r\frac{(d_0)_i}{2}}2^{\frac{l}{2}+\sum\limits_{i=1}^r\frac{(d_1)_i}{2}}$ under the assumption of Conjecture~\ref{con}; set it as $V$, and the $\mathds{k}$-algebra $U_{0}(\mathfrak{t}'_\mathds{k})$ affords an irreducible representation of dimension $1$ (note that which is of type $M$) by the proof of Lemma~\ref{main3}; set it as $W$. Thus  Lemma~\ref{AB}(i) \& (ii) imply that $V\boxtimes W$ is an irreducible representation of the $\mathds{k}$-algebra $U_{\xi_{\bar n}}(\bigoplus\limits_{i=1}^r(\mathfrak{g}_\mathds{k})_i)\otimes U_{0}(\mathfrak{t}'_\mathds{k})\cong U_{\xi}(\mathfrak{l}_\mathds{k})$ with dimension $p^{\sum\limits_{i=1}^r\frac{(d_0)_i}{2}}2^{\frac{l}{2}+\sum\limits_{i=1}^r\frac{(d_1)_i}{2}}$.

The remark preceding Lemma~\ref{main3} shows that any $U_\xi(\mathfrak{l}_\mathds{k})$-mod can be regarded as a $U_\xi(\mathfrak{p}_\mathds{k})$-mod with a trivial action of $\mathfrak{u}_\mathds{k}$.  It is immediate from \eqref{gp} that the $\mathds{k}$-algebra $U_\xi(\mathfrak{g}_\mathds{k})$ admits an irreducible representation which is isomorphic to
$U_\xi(\mathfrak{g}_\mathds{k})\otimes_{U_\xi(\mathfrak{p}_\mathds{k})}(V\boxtimes W)$ as $U_\xi(\mathfrak{g}_\mathds{k})$-modules. By \eqref{dimu} we can conclude that whose dimension is
\begin{equation*}
p^{\sum\limits_{i=1}^r\frac{(d_0)_i}{2}}2^{\frac{l}{2}+\sum\limits_{i=1}^r\frac{(d_1)_i}{2}}\cdot p^{\frac{d_0-\sum\limits_{i=1}^r(d_0)_i}{2}} 2^{\frac{d_1-\sum\limits_{i=1}^r(d_1)_i}{2}}
=p^{\frac{d_0}{2}}2^{\frac{d_1+l}{2}},
\end{equation*}  completing the proof.
\end{proof}

In ([\cite{WZ}], Theorem 5.6), Wang-Zhao introduced the Super Kac-Weisfeiler Property for a basic classical Lie superalgebra with arbitrary $p$-characters. By the same discussion as Remark~\ref{refine} we can conclude that the boundary introduced in Lemma~\ref{main3} coincides with that obtained by Wang-Zhao if at most only one of the $(d_1)_i's$ is odd for $1\leqslant i\leqslant r$. But for the case when more than two $(d_1)_i's$ are odd for $1\leqslant i\leqslant r$, the boundary introduced in Lemma~\ref{main3} is much larger. In fact, the latter case will never happen, i.e.

\begin{theorem}$^{[\cite{WZ}]}$\label{main3'}
Let $\mathfrak{g}_\mathds{k}$ be a basic classical Lie superalgebra over $\mathds{k}=\overline{\mathbb{F}}_p$, and let $\xi\in(\mathfrak{g}_\mathds{k}^*)_{\bar0}$. Retain the notation as \eqref{arbitdim}.
Then for the primes $p$ which satisfy the restrictions imposed in ([\cite{WZ}], Table 1), the dimension of every $U_\xi(\mathfrak{g}_\mathds{k})$-mod $M$ is divisible by $p^{\frac{d_0}{2}}2^{\lfloor\frac{d_1}{2}\rfloor}$.
\end{theorem}

Theorem~\ref{main3'} has been verified by Wang-Zhao in ([\cite{WZ}], Theorem 5.6). Here we give a new proof based on Lemma~\ref{main3}, which will be necessary for the arguments in the proof of the forthcoming Theorem \ref{main4'}.

\begin{proof}
(1) Recall that $\mathfrak{l}_\mathds{k}=\mathfrak{g}_\mathds{k}^{\bar s}=\bigoplus\limits_{i=1}^r(\mathfrak{g}_\mathds{k})_i\oplus\mathfrak{t}'_\mathds{k}$ by ([\cite{WZ}], Proposition 5.1), where $(\mathfrak{g}_\mathds{k})_i$ is a basic classical Lie superalgebra for each $1\leqslant i\leqslant r$, and $\mathfrak{t}'_\mathds{k}$ is a toral subalgebra of $\mathfrak{g}_\mathds{k}$. The discussion preceding the theorem shows that it is sufficient to prove that at most only one of the $(d_1)_i's$ (see \eqref{arbitdim}) is odd for $1\leqslant i\leqslant r$ (recall that $d'_1=\sum\limits_{i=1}^r(d_1)_i$ by \eqref{arbitdim}; $d'_1$ and $d_1$ have the same parity by \eqref{gtol} \& \eqref{ltosum}; and $(d_1)_i's$ are odd for $1\leqslant i \leqslant l$) since when $d_1$ is odd, we have $l=1$ and $\lfloor\frac{d_1}{2}\rfloor=\frac{d_1+1}{2}$; when $d_1$ is even, we have $l=0$ and $\lfloor\frac{d_1}{2}\rfloor=\frac{d_1}{2}$. To achieve this, one just needs to consider the summands in the decomposition of $\mathfrak{g}_\mathds{k}^{\bar s}$ with non-zero odd parts.

(2) First note that $(d_1)_i\,(1\leqslant i\leqslant r)$ is always even for the summand $(\mathfrak{g}_\mathds{k})_i$ which is isomorphic to the basic classical Lie superalgebra of type $A(m,n)$ by the remark at the beginning of Section 8.1.  Recall that an explicit list of non-$W$-equivalent systems of positive roots was found by Kac in ([\cite{K}], Section 2.5.4) (a system of simple roots for $F(4)$ is missing; see the remark above ([\cite{WZ}], Proposition 5.1)). Note that in the examples given by Kac the Cartan subalgebra $\mathfrak{h}_\mathds{k}$ is a subspace of the space $D$ of diagonal matrices; the roots are expressed in terms of the standard basis $\epsilon_i$ of $D^*$ (more accurately, the restrictions of the $\epsilon_i$ to $\mathfrak{h}_\mathds{k}$). In the following we assume that the semisimple element $\bar s\in \mathfrak{h}_\mathds{k}$.

(3) Based on all the discussions above, for each case we will consider separately (all the results below are obtained by completely elementary yet tedious case-by-case calculations, thus the proof will be omitted).

(i) For the case when $\mathfrak{l}_\mathds{k}$ is isomorphic to $A(M,N)$, the summands of $\bigoplus\limits_{i=1}^r(\mathfrak{g}_\mathds{k})_i$ for $1\leqslant i\leqslant r$ with non-zero odd parts are always isomorphic to the basic classical Lie superalgebras of type $A(m,n)$, thus the $(d_1)_i's$ are all even in this case.

(ii) For the case when $\mathfrak{l}_\mathds{k}$ is isomorphic to $B(M,N), C(M,N)$ or $D(M,N)$, the summands of $\bigoplus\limits_{i=1}^r(\mathfrak{g}_\mathds{k})_i$ for $1\leqslant i\leqslant r$ with non-zero odd parts are either isomorphic to $A(m,n)$, where the $(d_1)_i's$ are all even; or at most only one summand is isomorphic to $B(m,n),\,C(m,n)$ or $D(m,n)$ respectively (which is of the same type as $\mathfrak{l}_\mathds{k}$) with $(d_1)_i$ being even or odd.
Hence at most only one of the $(d_1)_i's$ is odd for $1\leqslant i\leqslant r$ in this case.

(iii) For the case when $\mathfrak{l}_\mathds{k}$ is isomorphic to $D(2,1;\bar a)$ or $G(3)$, the summands of $\bigoplus\limits_{i=1}^r(\mathfrak{g}_\mathds{k})_i$ for $1\leqslant i\leqslant r$ with non-zero odd parts are either isomorphic to $A(m,n)$, where the $(d_1)_i's$ are all even; or at most only one summand is isomorphic to $B(m,n)$ with $(d_1)_i$ being even or odd. At extreme, $\mathfrak{l}_\mathds{k}=\mathfrak{g}_\mathds{k}^{\bar s}$ is isomorphic to  $D(2,1;\bar a)$ or $G(3)$ respectively when $\bar s=0$. Hence at most only one of the $(d_1)_i's$ is odd for $1\leqslant i\leqslant r$ in this case.

(iv) For the case when $\mathfrak{l}_\mathds{k}$ is isomorphic to $F(4)$, the summands of $\bigoplus\limits_{i=1}^r(\mathfrak{g}_\mathds{k})_i$ for $1\leqslant i\leqslant r$ with non-zero odd parts are either isomorphic to $A(m,n)$, where the $(d_1)_i's$ are all even; or at most only one summand either is isomorphic to $B(m,n)$, or to $D(2,1;\bar a)$, with $(d_1)_i$ being even or odd. At extreme, $\mathfrak{l}_\mathds{k}=\mathfrak{g}_\mathds{k}^{\bar s}\cong F(4)$ when $\bar s=0$. Hence at most only one of the $(d_1)_i's$ is odd for $1\leqslant i\leqslant r$ in this case.

All the discussions in (3)(i)-(iv) show that at most only one of the $(d_1)_i's$ ($1\leqslant i\leqslant r$) is odd in the summands of $\bigoplus\limits_{i=1}^r(\mathfrak{g}_\mathds{k})_i$, completing the proof.
\end{proof}
\begin{rem}
Compared with the approach carried above, Wang-Zhao's original proof in ([\cite{WZ}], Theorem 5.6) is more concise since they did not consider the parity of the $(d_1)_i's~(1\leqslant i\leqslant r)$ for the summands of $\mathfrak{l}_{\mathds{k}}=\mathfrak{g}_{\mathds{k}}^{\bar s}\cong \bigoplus\limits_{i=1}^r(\mathfrak{g}_\mathds{k})_i\oplus\mathfrak{t}'_{\mathds{k}}$ in their proof. The reason why we take such a complicated approach is based on the following considerations:

(1) from the detailed proof of Lemma~\ref{main3} one can conclude that the lower bound for the dimension of $U_\xi(\mathfrak{g}_\mathds{k})$-modules is critically depending on the parity of the $(d_1)_i's$ for $1\leqslant i\leqslant r$. Without careful inspection on the summands of $\mathfrak{l}_{\mathds{k}}\cong \bigoplus\limits_{i=1}^r(\mathfrak{g}_\mathds{k})_i\oplus\mathfrak{t}'_{\mathds{k}}$, we can not guarantee that the lower bound introduced in Theorem~\ref{main3'} is optimal.

(2) another important reason lies in that in this part we mainly concentrate on the discussion whether the lower bound introduced in Theorem~\ref{main3'} can be reached under the assumption of Conjecture~\ref{con}. In the proof of Theorem~\ref{main4'}, we can see that the representation theory of the $\mathds{k}$-algebra
$U_{\xi}(\mathfrak{l}_\mathds{k})$ plays the key role for the realization of the $U_{\xi}(\mathfrak{g}_\mathds{k})$-modules with minimal dimension.
\end{rem}

Now we are in a position to discuss whether the lower bound in the Super Kac-Weisfeiler Property introduced in ([\cite{WZ}], Theorem 5.6) can be reached for $p\gg0$. Under the assumption of Conjecture~\ref{con}, the following theorem shows that the minimal dimension for the representations of reduced enveloping algebra $U_\xi(\mathfrak{g}_\mathds{k})$ with arbitrary $p$-character $\xi\in(\mathfrak{g}_\mathds{k}^*)_{\bar0}$ in the Super Kac-Weisfeiler Property is reachable.

\begin{theorem}\label{main4'}
Let $\mathfrak{g}_\mathds{k}$ be a basic classical Lie superalgebra, and let $\xi\in(\mathfrak{g}_\mathds{k}^*)_{\bar0}$. Retain the notation as \eqref{arbitdim}. If Conjecture~\ref{con} is established, then for $p\gg0$ the reduced enveloping algebra $U_\xi(\mathfrak{g}_\mathds{k})$ admits irreducible representations of dimension $p^{\frac{d_0}{2}}2^{\lfloor\frac{d_1}{2}\rfloor}$.
\end{theorem}

\begin{proof}
Recall that Lemma~\ref{main4} shows that the $\mathds{k}$-algebra $U_{\xi}(\mathfrak{l}_\mathds{k})$ (where $\mathfrak{l}_\mathds{k}=\bigoplus\limits_{i=1}^r(\mathfrak{g}_\mathds{k})_i\oplus\mathfrak{t}'_\mathds{k}$) admits irreducible representations of dimension $p^{\sum\limits_{i=1}^r\frac{(d_0)_i}{2}}2^{\frac{l}{2}+\sum\limits_{i=1}^r\frac{(d_1)_i}{2}}$ under the assumption of Conjecture~\ref{con}. By the proof of Theorem~\ref{main3'} one can conclude that $$p^{\sum\limits_{i=1}^r\frac{(d_0)_i}{2}}2^{\frac{l}{2}+\sum\limits_{i=1}^r\frac{(d_1)_i}{2}}=p^{\sum\limits_{i=1}^r\frac{(d_0)_i}{2}}2^
{\lfloor\sum\limits_{i=1}^r\frac{(d_1)_i}{2}\rfloor}$$since at most only one of the $(d_1)_i's$ ($1\leqslant i\leqslant r$) is odd in the summands of $\bigoplus\limits_{i=1}^r(\mathfrak{g}_\mathds{k})_i$ (i.e. $0\leqslant l\leqslant 1$), and $\sum\limits_{i=1}^r\frac{(d_1)_i}{2}$ \& $l$ have the same parity.
In view of our earlier remark, any $U_\xi(\mathfrak{l}_\mathds{k})$-mod can be regarded as a $U_\xi(\mathfrak{p}_\mathds{k})$-mod with a trivial action of $\mathfrak{u}_\mathds{k}$. It is immediate from \eqref{gp} and \eqref{dimu} that the $\mathds{k}$-algebra $U_\xi(\mathfrak{g}_\mathds{k})$ admits irreducible representations of dimension
\begin{equation}\label{final}
p^{\sum\limits_{i=1}^r\frac{(d_0)_i}{2}}2^{\lfloor\sum\limits_{i=1}^r\frac{(d_1)_i}{2}\rfloor}\cdot p^{\frac{d_0-\sum\limits_{i=1}^r(d_0)_i}{2}} 2^{\frac{d_1-\sum\limits_{i=1}^r(d_1)_i}{2}}
=p^{\frac{d_0}{2}}2^{\frac{d_1}{2}+(\lfloor\sum\limits_{i=1}^r\frac{(d_1)_i}{2}\rfloor-\sum\limits_{i=1}^r\frac{(d_1)_i}{2})}.
\end{equation}

Recall that $\sum\limits_{i=1}^r\frac{(d_1)_i}{2}$ and $d_1$ have the same parity by \eqref{gtol} \& \eqref{ltosum}. For the case when $d_1$ is odd, we have $p^{\frac{d_0}{2}}2^{\frac{d_1}{2}+(\lfloor\sum\limits_{i=1}^r\frac{(d_1)_i}{2}\rfloor-\sum\limits_{i=1}^r\frac{(d_1)_i}{2})}=p^{\frac{d_0}{2}}2^
{\frac{d_1+1}{2}}=p^{\frac{d_0}{2}}2^{\lfloor\frac{d_1}{2}\rfloor}$; for the case when $d_1$ is even, we have $p^{\frac{d_0}{2}}2^{\frac{d_1}{2}+(\lfloor\sum\limits_{i=1}^r\frac{(d_1)_i}{2}\rfloor-\sum\limits_{i=1}^r\frac{(d_1)_i}{2})}=p^{\frac{d_0}{2}}2^{\frac{d_1}{2}}=p^{\frac{d_0}{2}}2^{\lfloor\frac{d_1}{2}\rfloor}$. Hence the desired result follows from \eqref{final}.
\end{proof}

In particular, for the special case when $\mathfrak{g}_\mathds{k}$ is of type $A(M,N)$, we have
\begin{corollary}
Let $\mathfrak{g}_\mathds{k}$ be a basic classical Lie superalgebra of type $A(M,N)$, and let $\xi\in(\mathfrak{g}_\mathds{k}^*)_{\bar0}$. Retain the notation as \eqref{arbitdim}. If Conjecture~\ref{con}(i) is established for the basic classical Lie superalgebras of type $A(m,n)$, then for $p\gg0$ the reduced enveloping algebra $U_\xi(\mathfrak{g}_\mathds{k})$ admits irreducible representations of dimension $p^{\frac{d_0}{2}}2^{\frac{d_1}{2}}$.
\end{corollary}

\begin{proof}
Recall that the remark at the beginning of Section 8.1 shows that $d_1$ is always even for the basic classical Lie superalgebra of type $A(m,n)$, thus we have $p^{\frac{d_0}{2}}2^{\lfloor\frac{d_1}{2}\rfloor}=p^{\frac{d_0}{2}}2^{\frac{d_1}{2}}$. Then the corollary follows from Theorem~\ref{main4'} and (3)(i) in the proof of Theorem~\ref{main3'}.
\end{proof}

{\bf Acknowledgements}\quad
The authors would like to thank Weiqiang Wang and Lei Zhao whose work on super version of Kac-Weisfeiler property stimulated them to do the present research.  The authors got much help  from the discussion with Hao Chang and Weiqiang Wang, as well as from Yung-Ning Peng and Lei Zhao who  explained some results in their papers [\cite{Peng3}] and [\cite{WZ}] respectively. The authors express great thanks to them.

\end{document}